\documentclass[11pt,oneside,reqno]{amsart}
\usepackage[latin9]{inputenc}
\usepackage{geometry}
\usepackage{mathrsfs}
\usepackage{amsbsy}
\usepackage{amstext}
\usepackage{amsthm}
\usepackage{amssymb}
\usepackage{setspace}
\usepackage[unicode=true,
 bookmarks=false,
 breaklinks=false,pdfborder={0 0 1},backref=false,colorlinks=false]
 {hyperref}

\makeatletter
\numberwithin{equation}{section}
\numberwithin{figure}{section}
\theoremstyle{plain}
\newtheorem{thm}{\protect\theoremname}[section]
\theoremstyle{remark}
\newtheorem{notation}[thm]{\protect\notationname}
\theoremstyle{plain}
\newtheorem{prop}[thm]{\protect\propositionname}
\theoremstyle{remark}
\newtheorem{rem}[thm]{\protect\remarkname}
\theoremstyle{plain}
\newtheorem{assumption}[thm]{\protect\assumptionname}
\theoremstyle{definition}
\newtheorem{defn}[thm]{\protect\definitionname}
\theoremstyle{plain}
\newtheorem{lem}[thm]{\protect\lemmaname}
\theoremstyle{remark}
\newtheorem*{acknowledgement*}{\protect\acknowledgementname}

\usepackage{amsthm}
\usepackage{mathrsfs}
\usepackage[noadjust]{cite}
\usepackage{stmaryrd}
\usepackage{enumitem}
\setlist[itemize]{leftmargin=*}
\setlist[enumerate]{leftmargin=*}
\usepackage{bbm}

\usepackage{tikz}
\usetikzlibrary{patterns}
\usetikzlibrary{decorations}
\usetikzlibrary{decorations.pathreplacing}

\DeclareFontFamily{U}{matha}{\hyphenchar\font45}
\DeclareFontShape{U}{matha}{m}{n}{
      <5> <6> <7> <8> <9> <10> gen * matha
      <10.95> matha10 <12> <14.4> <17.28> <20.74> <24.88> matha12
      }{}
\DeclareSymbolFont{matha}{U}{matha}{m}{n}
\DeclareFontSubstitution{U}{matha}{m}{n}

\DeclareFontFamily{U}{mathx}{\hyphenchar\font45}
\DeclareFontShape{U}{mathx}{m}{n}{
      <5> <6> <7> <8> <9> <10>
      <10.95> <12> <14.4> <17.28> <20.74> <24.88>
      mathx10
      }{}
\DeclareSymbolFont{mathx}{U}{mathx}{m}{n}
\DeclareFontSubstitution{U}{mathx}{m}{n}

\DeclareMathDelimiter{\vvvert}{0}{matha}{"7E}{mathx}{"17}

\DeclareMathAlphabet{\scal}{U}{dutchcal}{m}{n}

\def\th@plain{\thm@notefont{}\itshape}
\def\th@definition{\thm@notefont{}\normalfont}

\makeatother

\providecommand{\acknowledgementname}{Acknowledgement}
\providecommand{\assumptionname}{Assumption}
\providecommand{\definitionname}{Definition}
\providecommand{\lemmaname}{Lemma}
\providecommand{\notationname}{Notation}
\providecommand{\propositionname}{Proposition}
\providecommand{\remarkname}{Remark}
\providecommand{\theoremname}{Theorem}

\begin{document}
\title[Hierarchical Structure of Metastability in the Reversible Inclusion
Process]{Hierarchical Structure of Metastability in the Reversible Inclusion
Process: Third Time Scale and Complete Characterization}
\author{Seonwoo Kim}
\address{S. Kim. Department of Mathematical Sciences, Seoul National University,
Republic of Korea.}
\email{ksw6leta@snu.ac.kr}
\begin{abstract}
In this article, we study the hierarchical structure of metastability
in the reversible inclusion process. We fully characterize the third
time scale of metastability subject to any underlying geometry of
the system and prove that this is the last time scale. We also demonstrate
that there are no other meaningful time scales except the three identified
ones. This work completes the verification of the conjecture made
in \cite{BDG} which was partially resolved on the first time scale
in \cite{BDG} and on the second time scale in \cite{Kim RIP-2nd}.
Main tools are potential-theoretic approach and martingale approach
to metastability; we thoroughly investigate the highly-complicated
energy landscape of the system to construct suitable test objects
to provide sharp asymptotics on capacities.
\end{abstract}

\maketitle
\tableofcontents{}

\section{\label{sec1}Introduction}

\emph{Metastability} is a widespread phenomenon that occurs in numerous
dynamical systems possessing two or more locally stable states. In
the context of statistical physics, metastability corresponds to the
so-called first-order phase transition, in that certain observables
of a system exhibit discontinuity with respect to the intensive variables
of the system. Typical examples are found in small random perturbations
of dynamical systems \cite{BEGK Spec,BEGK,BGK,LeeSeo NR1,LeeSeo NR2,RezSeo},
interacting particle systems with sticky interactions \cite{BDG,GroReVa 13,Kim RIP-2nd,KS NRIP,LMS critical,LMS resolvent},
ferromagnetic spin systems in low temperatures \cite{BAC,CGOV,KS IsingPotts-growing,NZ,NS Ising1},
etc. From a dynamical point of view, metastability can be explained
as follows. Starting from one locally stable state, the system tends
to stay near the initial state since it is surrounded by a certain
\emph{energy barrier} which blocks any transitions to other stable
states. However, after a long time, the system attains sufficient
amount of randomness which triggers an energetic fluctuation to overcome
the energy barrier and make a transition to another stable state.

If a system possesses a complex energy landscape with multiple levels
of the stability of states and the energy barriers between them, one
may expect that metastability also occurs in a highly complicated
manner. This is indeed the case; the system enjoys the so-called \emph{hierarchical
structure} of metastability depending especially on the different
values of energy barriers between the locally stable states existing
in the system. This phenomenon was treated in a systematic way in
\cite{BL rev-MP,LanXu}. However, mathematical verifications of this
hierarchical behavior of metastability for actual models are limited
in the literature, and the only existing result at this moment is
presented for random walks in potential fields \cite{LanMisTsu,LanSeo RW-pot-field}.

In this article, we verify the hierarchical structure of metastability
occurring in a concrete particle system, the \emph{reversible inclusion
process}. Inclusion process is an interacting particle system first
introduced in \cite{GKR,GKRV,GRV} as a discrete dual process of certain
interacting diffusion models known as Brownian energy and momentum
processes. It was also recognized as a bosonic counterpart to the
already-well-known exclusion process, where the exclusive interactions
therein were replaced by sticky interactions between the particles.

\emph{Condensation} refers to the phenomenon such that large amount
of particles concentrate at a single site, which is triggered by sticky
behavior of the particles. It was verified in \cite{GroReVa 11} that
the phenomenon of condensation occurs in a certain class of inclusion
processes in the condensing regime, i.e. when the factor $d_{N}$
controlling the diffusive behavior of the system vanishes in a certain
manner as the particle number $N$ tends to infinity (see \eqref{e_LN-def}
and \eqref{e_dN-logN}). The results of condensation were further
extended to inclusion processes under various conditions on $d_{N}$
and with different geometries of the underlying space \cite{ACR,JCG}.
We remark that non-condensing regimes were also thoroughly investigated
in the community during the past few years, regarding the topics of
duality \cite{CGGR,CGR20,CGR21}, local equilibrium properties \cite{KR,OpokuRedig},
non-equilibrium limit theorems \cite{FRS,FGS}, etc.

If there are multiple sites on which the condensation phenomenon takes
place, we are likely to observe a slow transition from one condensed
state to another on a longer time scale. In this case, suppose that
the inclusion process starts from a condensed state. Then, the dynamics
spends a long time before escaping it, thereby revealing a metastable
nature of the initial condensed state. Moreover, in a random moment
later, the process eventually makes a sudden transition to another
condensed state of the system. This behavior fits perfectly well into
the framework of metastability.

The first result on the metastability of inclusion processes was presented
in \cite{GroReVa 13}, in which the authors considered the \emph{symmetric}
inclusion process, such that the underlying particle movements between
the sites are completely symmetric. The main idea was to accelerate
the original process by a certain time scale $\theta_{1}:=1/d_{N}$,
and prove that this accelerated process converges to a certain scaling
limit which explains the macroscopic metastable transitions between
the condensed states. It was further generalized in \cite{BDG} to
the \emph{reversible} setting, where the particle movements are not
necessarily symmetric but satisfy the detailed balance condition (cf.
\eqref{e_cxy-def}). A remarkable observation was that even though
the underlying particle movements are assumed to be irreducible, the
scaling limit process representing the metastable transitions between
the condensed states is not necessarily irreducible (see Section \ref{sec2.3}
for details). In other words, there exist pairs of condensable sites
such that the dynamics cannot make metastable transitions between
those sites on the time scale $\theta_{1}=1/d_{N}$.

Due to the ergodicity of the original inclusion process, it is natural
to claim that the time scale $\theta_{1}$ is too short to detect
such non-observable metastable transitions. Thus, the next step is
to investigate longer time scales. This was conjectured in \cite{BDG},
further claiming that there exist precisely two more time scales,
namely $\theta_{2}:=N/d_{N}^{2}$ and $\theta_{3}:=N^{2}/d_{N}^{3}$,
and all metastable transitions are observable on the third and last
time scale $\theta_{3}$. The authors justified the argument by demonstrating
it on one-dimensional simple geometry. This conjecture was partially
proved in \cite{Kim RIP-2nd} regarding the second time scale $\theta_{2}$,
in that for any underlying geometry of particle movements, the $\theta_{2}$-accelerated
inclusion process converges to a new scaling limit process, which
is still not necessarily irreducible. Thus, it was also conjectured
in \cite{Kim RIP-2nd} that there remains one more time scale.

The main result of this article verifies that this conjecture is indeed
true; \emph{we fully characterize the behavior of metastable transitions
on the third time scale $\theta_{3}=N^{2}/d_{N}^{3}$ subject to any
geometry of the underlying system}, thereby completing the investigation
of the hierarchical structure of metastability occurring in the reversible
inclusion process. In particular, we prove that the third scaling
limit process becomes \emph{irreducible}, which implies that the third
time scale is indeed the last one; all metastable transitions become
observable hereon. Moreover, we also demonstrate that there are no
other meaningful intermediate time scales between the three scales
$\theta_{1}$, $\theta_{2}$, and $\theta_{3}$.

The hierarchy of metastable time scales is closely related to the
so-called \emph{$\Gamma$-expansion approach} to metastability \cite{BGL,Landim Gamma-exp,LanMisSau},
a topic recently well recognized in the community. Ongoing work tries
to characterize this $\Gamma$-expansion of level-two large deviation
rate functionals subject to the reversible inclusion process; refer
to Remark \ref{r_3rd}-(4) for more details.

Two essential features regarding the metastability phenomenon to be
proved in this article are: \emph{Markov chain model reduction} and
\emph{Eyring--Kramers formula}. The Markov chain model reduction
explains the successive movements of the condensate of particles in
terms of a simple macroscopic limit Markov chain on the site space
by employing certain coarse-graining arguments to the complicated
configuration space; details can be found in the recent survey \cite{Landim Meta-MC}.
The Eyring--Kramers formula provides a precise asymptotic estimate
of the mean metastable transition time, on which a deep literature
is built starting from the mid 20th century. Interested readers are
referred to the monograph \cite{BdH} for intensive discussions and
rich references on this topic.

To prove the main results, we employ the so-called \emph{potential-theoretic
approach} \cite{BEGK Spec,BEGK,BGK} and \emph{martingale approach}
\cite{BL TM,BL MG} to metastability. We briefly summarize in Appendix
\ref{appenA} the essential background. For complete references on
these fruitful methodologies, we refer the readers to \cite{BdH,Landim Meta-MC}.

The main mathematical difficulty of this work lies on the analysis
of the highly-complicated energy landscape of the system. In the previous
researches, the energy landscape was mainly one-dimensional on the
first time scale \cite{BDG} and two-dimensional on the second time
scale \cite{Kim RIP-2nd}. In contrast, on the third time scale, the
energy landscape is essentially \emph{three-dimensional}, and in particular
we must conduct a precise analysis regarding all three dimensions
of the relevant simplexes of configurations. The analysis is divided
into two parts. First, we reduce the dimensions of the simplexes from
three to one or two depending on the underlying geometry. Second,
we construct an infinite ladder graph along with a resolvent equation
defined thereon, such that a solution to the resolvent equation becomes
a suitable test object after certain rescaling. A more detailed explanation
of this procedure is given in Section \ref{sec4.1}.

The rest of the article is organized as follows. In Section \ref{sec2},
we settle the notation of the model and present our main results.
In Section \ref{sec3}, we explain our main strategy and reduce the
proofs of the main results to certain precise estimates of capacities.
In Section \ref{sec4}, we conduct a precise analysis of the energy
landscape of our model, considering the intuitive case of two metastable
states. Then, in Sections \ref{sec5} and \ref{sec6}, we construct
a suitable test function (in Section \ref{sec5}) and a test flow
(in Section \ref{sec6}) which are the key objects for the capacity
estimate. Finally, in Section \ref{sec7}, we briefly discuss the
general case of three or more metastable states. In the appendix,
we summarize the potential theory used throughout the article (Appendix
\ref{appenA}) and provide some technical proofs omitted along the
way (Appendix \ref{appenB}).

\section{\label{sec2}Notation and Main Results}

In this section we define the model, mathematically formulate the
phenomena of condensation and metastability, and state the main results.
Those already familiar with the context may proceed to Section \ref{sec2.5}
for the main results.

\subsection{\label{sec2.1}Reversible condensing inclusion process}

We fix a finite site space $S$. Thereon, an underlying random walk
structure is defined as a continuous-time irreducible Markov chain,
which is characterized by some transition rate function $r(\cdot,\,\cdot):S\times S\to[0,\,\infty)$,
where for simplicity we define $r(x,\,x)=0$ for $x\in S$. In addition,
we assume that the underlying random walk is \emph{reversible} with
respect to its unique stationary measure $(m_{x})_{x\in S}$:
\begin{equation}
c_{xy}:=m_{x}r(x,\,y)=m_{y}r(y,\,x)\text{ for all }x,\,y\in S.\label{e_cxy-def}
\end{equation}
We normalize the measure $(m_{x})_{x\in S}$ such that the maximal
value is $1$, i.e., $\max_{x\in S}m_{x}=1$. See Figure \ref{fig2.1}-left
for an illustration of such structure.

\begin{figure}
\begin{tikzpicture}
\fill[rounded corners,white] (-2.25,2.25) rectangle (2.25,-2.25); 
\fill[black!30!white] (-2,2) circle (0.1); \draw (-2,2) circle (0.1);
\fill[black!30!white] (-1,2) circle (0.1); \draw (-1,2) circle (0.1);
\fill[black!30!white] (0,2) circle (0.1); \draw (0,2) circle (0.1);
\fill[black!30!white] (1,2) circle (0.1); \draw (1,2) circle (0.1);
\fill[black!30!white] (2,2) circle (0.1); \draw (2,2) circle (0.1);
\fill[black!30!white] (-2,1) circle (0.1); \draw (-2,1) circle (0.1);
\fill[black!30!white] (-1,1) circle (0.1); \draw (-1,1) circle (0.1);
\fill[black!30!white] (0,1) circle (0.1); \draw (0,1) circle (0.1);
\fill[black!30!white] (1,1) circle (0.1); \draw (1,1) circle (0.1);
\fill[black!30!white] (2,1) circle (0.1); \draw (2,1) circle (0.1);
\fill[black!30!white] (-2,0) circle (0.1); \draw (-2,0) circle (0.1);
\fill[black!30!white] (-1,0) circle (0.1); \draw (-1,0) circle (0.1);
\fill[black!30!white] (0,0) circle (0.1); \draw (0,0) circle (0.1);
\fill[black!30!white] (1,0) circle (0.1); \draw (1,0) circle (0.1);
\fill[black!30!white] (2,0) circle (0.1); \draw (2,0) circle (0.1);
\fill[black!30!white] (-2,-1) circle (0.1); \draw (-2,-1) circle (0.1);
\fill[black!30!white] (-1,-1) circle (0.1); \draw (-1,-1) circle (0.1);
\fill[black!30!white] (0,-1) circle (0.1); \draw (0,-1) circle (0.1);
\fill[black!30!white] (1,-1) circle (0.1); \draw (1,-1) circle (0.1);
\fill[black!30!white] (2,-1) circle (0.1); \draw (2,-1) circle (0.1);
\fill[black!30!white] (-2,-2) circle (0.1); \draw (-2,-2) circle (0.1);
\fill[black!30!white] (-1,-2) circle (0.1); \draw (-1,-2) circle (0.1);
\fill[black!30!white] (0,-2) circle (0.1); \draw (0,-2) circle (0.1);
\fill[black!30!white] (1,-2) circle (0.1); \draw (1,-2) circle (0.1);
\fill[black!30!white] (2,-2) circle (0.1); \draw (2,-2) circle (0.1);
\draw[very thick,to-] (-1.9,2)--(-1.1,2); \draw[very thick] (-0.9,2)--(-0.1,2); \draw[very thick,-to] (0.1,2)--(0.9,2); \draw[very thick] (1.1,2)--(1.9,2);
\draw[very thick] (-2,1.9)--(-2,1.1); \draw[very thick,to-] (1,1.9)--(1,1.1); \draw[very thick, to-] (2,1.9)--(2,1.1);
\draw[very thick,to-] (-1.9,1)--(-1.1,1); \draw[very thick] (0.1,1)--(0.9,1); \draw[very thick] (1.1,1)--(1.9,1);
\draw[very thick,to-] (-2,0.9)--(-2,0.1); \draw[very thick,-to] (-1,0.9)--(-1,0.1); \draw[very thick,-to] (0,0.9)--(0,0.1); \draw[very thick,-to] (1,0.9)--(1,0.1);
\draw[very thick] (-1.9,0)--(-1.1,0); \draw[very thick,-to] (-0.9,0)--(-0.1,0); \draw[very thick,to-] (1.1,0)--(1.9,0);
\draw[very thick,-to] (-2,-0.1)--(-2,-0.9); \draw[very thick,-to] (-1,-0.1)--(-1,-0.9); \draw[very thick,-to] (0,-0.1)--(0,-0.9); \draw[very thick,-to] (2,-0.1)--(2,-0.9);
\draw[very thick] (-1.9,-1)--(-1.1,-1); \draw[very thick,to-] (-0.9,-1)--(-0.1,-1); \draw[very thick,-to] (0.1,-1)--(0.9,-1);
\draw[very thick] (-2,-1.1)--(-2,-1.9); \draw[very thick] (-1,-1.1)--(-1,-1.9); \draw[very thick,-to] (2,-1.1)--(2,-1.9);
\draw[very thick] (-1.9,-2)--(-1.1,-2); \draw[very thick,to-] (-0.9,-2)--(-0.1,-2); \draw[very thick,to-] (0.1,-2)--(0.9,-2); \draw[very thick,-to] (1.1,-2)--(1.9,-2);
\end{tikzpicture}
\hspace{20mm}
\begin{tikzpicture}
\fill[rounded corners,white] (-2.25,2.25) rectangle (2.25,-2.25); 
\fill[black] (-2,2) circle (0.1); \draw (-2,2) circle (0.1);
\fill[white] (-1,2) circle (0.1); \draw[black!30!white] (-1,2) circle (0.1);
\fill[white] (0,2) circle (0.1); \draw[black!30!white] (0,2) circle (0.1);
\fill[black] (1,2) circle (0.1); \draw (1,2) circle (0.1);
\fill[black] (2,2) circle (0.1); \draw (2,2) circle (0.1);
\fill[black] (-2,1) circle (0.1); \draw (-2,1) circle (0.1);
\fill[white] (-1,1) circle (0.1); \draw[black!30!white] (-1,1) circle (0.1);
\fill[white] (0,1) circle (0.1); \draw[black!30!white] (0,1) circle (0.1);
\fill[white] (1,1) circle (0.1); \draw[black!30!white] (1,1) circle (0.1);
\fill[white] (2,1) circle (0.1); \draw[black!30!white] (2,1) circle (0.1);
\fill[white] (-2,0) circle (0.1); \draw[black!30!white] (-2,0) circle (0.1);
\fill[white] (-1,0) circle (0.1); \draw[black!30!white] (-1,0) circle (0.1);
\fill[white] (0,0) circle (0.1); \draw[black!30!white] (0,0) circle (0.1);
\fill[black] (1,0) circle (0.1); \draw (1,0) circle (0.1);
\fill[white] (2,0) circle (0.1); \draw[black!30!white] (2,0) circle (0.1);
\fill[black] (-2,-1) circle (0.1); \draw (-2,-1) circle (0.1);
\fill[black] (-1,-1) circle (0.1); \draw (-1,-1) circle (0.1);
\fill[white] (0,-1) circle (0.1); \draw[black!30!white] (0,-1) circle (0.1);
\fill[white] (1,-1) circle (0.1); \draw[black!30!white] (1,-1) circle (0.1);
\fill[white] (2,-1) circle (0.1); \draw[black!30!white] (2,-1) circle (0.1);
\fill[black] (-2,-2) circle (0.1); \draw (-2,-2) circle (0.1);
\fill[black] (-1,-2) circle (0.1); \draw (-1,-2) circle (0.1);
\fill[white] (0,-2) circle (0.1); \draw[black!30!white] (0,-2) circle (0.1);
\fill[white] (1,-2) circle (0.1); \draw[black!30!white] (1,-2) circle (0.1);
\fill[black] (2,-2) circle (0.1); \draw (2,-2) circle (0.1);
\draw[very thick,to-] (-1.9,2)--(-1.1,2); \draw[very thick] (-0.9,2)--(-0.1,2); \draw[very thick,-to] (0.1,2)--(0.9,2); \draw[very thick] (1.1,2)--(1.9,2);
\draw[very thick] (-2,1.9)--(-2,1.1); \draw[very thick,to-] (1,1.9)--(1,1.1); \draw[very thick, to-] (2,1.9)--(2,1.1);
\draw[very thick,to-] (-1.9,1)--(-1.1,1); \draw[very thick] (0.1,1)--(0.9,1); \draw[very thick] (1.1,1)--(1.9,1);
\draw[very thick,to-] (-2,0.9)--(-2,0.1); \draw[very thick,-to] (-1,0.9)--(-1,0.1); \draw[very thick,-to] (0,0.9)--(0,0.1); \draw[very thick,-to] (1,0.9)--(1,0.1);
\draw[very thick] (-1.9,0)--(-1.1,0); \draw[very thick,-to] (-0.9,0)--(-0.1,0); \draw[very thick,to-] (1.1,0)--(1.9,0);
\draw[very thick,-to] (-2,-0.1)--(-2,-0.9); \draw[very thick,-to] (-1,-0.1)--(-1,-0.9); \draw[very thick,-to] (0,-0.1)--(0,-0.9); \draw[very thick,-to] (2,-0.1)--(2,-0.9);
\draw[very thick] (-1.9,-1)--(-1.1,-1); \draw[very thick,to-] (-0.9,-1)--(-0.1,-1); \draw[very thick,-to] (0.1,-1)--(0.9,-1);
\draw[very thick] (-2,-1.1)--(-2,-1.9); \draw[very thick] (-1,-1.1)--(-1,-1.9); \draw[very thick,-to] (2,-1.1)--(2,-1.9);
\draw[very thick] (-1.9,-2)--(-1.1,-2); \draw[very thick,to-] (-0.9,-2)--(-0.1,-2); \draw[very thick,to-] (0.1,-2)--(0.9,-2); \draw[very thick,-to] (1.1,-2)--(1.9,-2);
\end{tikzpicture}\caption{\label{fig2.1}In the left figure, $25$ circles indicate the sites
in $S$ and edges indicate the underlying random walk $r(\cdot,\,\cdot)$.
Here, $x,\,y\in S$ are connected with a directed edge from $x$ to
$y$ if $r(x,\,y)>r(y,\,x)>0$ and with an undirected edge if $r(x,\,y)=r(y,\,x)>0$.
In the right figure, black circles denote the sites in $S_{\star}$
and white circles denote the sites in $S_{0}=S\setminus S_{\star}$.}
\end{figure}
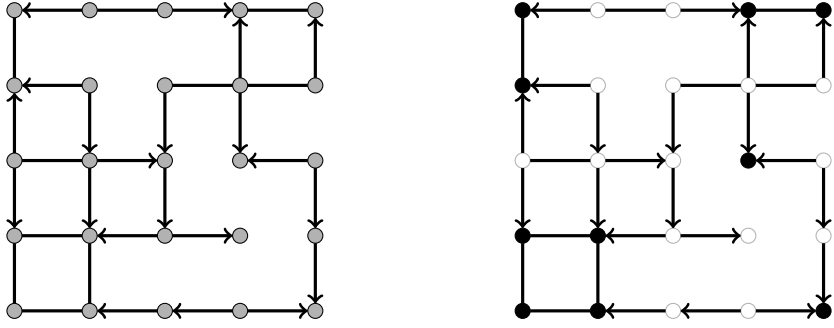

\begin{notation}
\label{n_different}In this article, for a set $A$ and elements $\alpha$
and $\beta$, writing $\alpha,\,\beta\in A$ or $\{\alpha,\,\beta\}\subseteq A$
implies implicitly that $\alpha\ne\beta$. The same is also true for
three or more elements.
\end{notation}

We define the configuration space $\mathcal{H}_{N}$ as
\begin{equation}
\mathcal{H}_{N}:=\Big\{\eta=(\eta_{x})_{x\in S}\in\mathbb{N}^{S}:\sum_{x\in S}\eta_{x}=N\Big\}.\label{e_HN-def}
\end{equation}
Here, each $\eta\in\mathcal{H}_{N}$ represents the configuration
of $N$ particles in $S$. Notice that since $S$ is finite, $\mathcal{H}_{N}$
is also finite. Then, the \emph{inclusion process} is the continuous-time
Markov chain $\{\eta_{N}(t)\}_{t\ge0}$ on $\mathcal{H}_{N}$, characterized
by its transition rate function $r_{N}(\cdot,\,\cdot)$ defined by
\begin{equation}
r_{N}(\eta,\,\zeta):=\begin{cases}
\eta_{x}(d_{N}+\eta_{y})r(x,\,y) & \text{if }\zeta=\eta^{x,\,y}\text{ for some }x,\,y\in S,\\
0 & \text{otherwise}.
\end{cases}\label{e_rN-def}
\end{equation}
Here, $\eta^{x,\,y}\in\mathcal{H}_{N}$ is the configuration obtained
from $\eta$ by sending a particle, if possible, from $x\in S$ to
$y\in S$, and $d_{N}$ is a positive control parameter which depends
only on the number of particles $N$. Denote by $\mathcal{L}_{N}$
the infinitesimal generator corresponding to the process $\eta_{N}(\cdot)$,
such that we have the following representation:
\begin{equation}
(\mathcal{L}_{N}f)(\eta)=\sum_{x,\,y\in S}\eta_{x}(d_{N}+\eta_{y})r(x,\,y)(f(\eta^{x,\,y})-f(\eta)).\label{e_LN-def}
\end{equation}
Further, we assume that
\begin{equation}
\lim_{N\to\infty}d_{N}\log N=0,\label{e_dN-logN}
\end{equation}
which guarantees that the diffusive characteristics of the particle
system driven by $\eta_{x}d_{N}r(x,\,y)$ is much smaller than the
sticky characteristics driven by $\eta_{x}\eta_{y}r(x,\,y)$. Consequently,
the inclusion process exhibits the so-called \emph{condensation} phenomenon,
which is rigorously formulated in Section \ref{sec2.2}.

We denote by $\mathbb{P}_{\eta}=\mathbb{P}_{\eta}^{N}$ and $\mathbb{E}_{\eta}=\mathbb{E}_{\eta}^{N}$
the law and the corresponding expectation of the process starting
from $\eta\in\mathcal{H}_{N}$.

\subsection{\label{sec2.2}Condensation phenomenon}

Since the underlying random walk $r(\cdot,\,\cdot)$ is irreducible,
it is straightforward to see that the inclusion process $\eta_{N}(\cdot)$
is irreducible on the configuration space $\mathcal{H}_{N}$. Thus,
it admits a unique stationary distribution, which is denoted as $\mu_{N}$.
The following proposition gives an explicit characterization of $\mu_{N}$.
\begin{prop}[{\cite[Theorem 2.1]{GroReVa 11}}]
\label{p_muN-prod}The inclusion process $\eta_{N}(\cdot)$ is reversible
with respect to $\mu_{N}$. Moreover,
\begin{equation}
\mu_{N}(\eta)=\frac{1}{Z_{N}}\prod_{x\in S}\big[w_{N}(\eta_{x})m_{x}^{\eta_{x}}\big]\text{ for all }\eta\in\mathcal{H}_{N},\label{e_muN-prod}
\end{equation}
where $w_{N}(n):=\frac{\Gamma(d_{N}+n)}{n!\Gamma(d_{N})}$, $\Gamma(\cdot)$
is the usual Gamma function, and $Z_{N}$ is the normalization factor
such that $\mu_{N}$ is a probability measure.
\end{prop}

\begin{rem}
\label{r_non-rev}Proposition \ref{p_muN-prod} relies heavily on
the fact that the underlying random walk $r(\cdot,\,\cdot)$ is reversible.
Indeed, if $r(\cdot,\,\cdot)$ is more generally non-reversible, then
the inclusion process is not necessarily reversible; even worse, $\mu_{N}$
does not admit a product-type formula as in \eqref{e_muN-prod}. Thus,
it is far more challenging to describe the condensation and metastable
behaviors. Refer to \cite{KS NRIP} for a detailed investigation under
this more general setting.
\end{rem}

According to the formula in \eqref{e_muN-prod}, in terms of the stationary
distribution $\mu_{N}$, it is natural to expect that the sites in
$S$ with maximal value of $m$ (which is $1$) are the more important
ones. Therefore, we define
\begin{equation}
S_{\star}:=\{x\in S:m_{x}=1\}\quad\text{and}\quad S_{0}:=S\setminus S_{\star}.\label{e_S-star-S0-def}
\end{equation}
Then, it is obvious that $m_{a}<1$ for all $a\in S_{0}$. For an
illustration, refer to Figure \ref{fig2.1}-right.

As stated in Section \ref{sec2.1}, owing to the decaying condition
\eqref{e_dN-logN} of $d_{N}$, we observe the (static) condensation
phenomenon in the system. This is mathematically formulated as follows.
For each $x\in S$, we define
\begin{equation}
\mathcal{E}_{N}^{x}:=\{\xi^{x}\}:=\{\eta\in\mathcal{H}_{N}:\eta_{x}=N\},\label{e_ENx-xix-def}
\end{equation}
so that $\xi^{x}$ represents the configuration of which all $N$
particles are located at the single site $x\in S$. Moreover, for
notational convenience, we write
\begin{equation}
\mathcal{E}_{N}(A):=\bigcup_{x\in A}\mathcal{E}_{N}^{x}\text{ for each }A\subseteq S.\label{e_ENA-def}
\end{equation}

\begin{prop}[{Condensation, \cite[Proposition 2.1]{BDG}}]
\label{p_condensation}It holds that
\[
\lim_{N\to\infty}\mu_{N}(\mathcal{E}_{N}^{x})=\frac{1}{|S_{\star}|}\text{ for every }x\in S_{\star}.
\]
Consequently, $\lim_{N\to\infty}\mu_{N}(\mathcal{H}_{N}\setminus\mathcal{E}_{N}(S_{\star}))=0$.
\end{prop}

According to Proposition \ref{p_condensation}, in the long-time limit,
the system is most likely to stay in a condensed state in $S_{\star}$.
Moreover, starting from such a condensed state $\xi^{x}$ for some
$x\in S_{\star}$, it is very unlikely to observe a transition to
other states, thereby explaining the \emph{metastable} feature of
the condensed states.

However, provided that $|S_{\star}|\ge2$, due to the ergodic nature
of the dynamics, starting from $\xi^{x}$ the system will eventually
make a metastable transition to another condensed state $\xi^{y}$.
This feature cannot be explained by the static result provided in
Proposition \ref{p_condensation}.

\subsection{\label{sec2.3}First time scale of metastability}

As explained in the introduction, the inclusion process exhibits the
phenomenon of metastable transitions regarding the location of the
condensate of particles. To formulate this in a rigorous manner, it
is useful to introduce a projection operator $\Psi_{1}=\Psi_{N,\,1}:\mathcal{H}_{N}\to S_{\star}\cup\{\mathfrak{0}\}$
defined as
\[
\Psi_{1}(\eta):=\begin{cases}
x & \text{if }\eta=\xi^{x}\text{ for some }x\in S_{\star},\\
\mathfrak{0} & \text{otherwise}.
\end{cases}
\]
Then, define the order process $\{X_{N}(t)\}_{t\ge0}$ as $X_{N}(t):=\Psi_{1}(\eta_{N}(t))$
for $t\ge0$.
\begin{notation}
\label{n_fN-gN}Throughout the article, for functions $f(N)$ and
$g(N)$, we write
\[
\begin{cases}
f=O(g) & \text{if }|f(N)|\le Cg(N),\ \forall N\ge1\text{ for some constant }C>0,\\
f=o(g)\text{ or }f\ll g & \text{if }\lim_{N\to\infty}f(N)/g(N)=0,\\
f\simeq g & \text{if }\lim_{N\to\infty}f(N)/g(N)=1.
\end{cases}
\]
In particular, in the third case, we also say that $f$ and $g$ are
\emph{asymptotically equal}.
\end{notation}

Now, we are ready to formulate the first-level metastability. For
a subset $\mathcal{A}\subseteq\mathcal{H}_{N}$, we denote by $\mathcal{T}_{\mathcal{A}}$
the (random) hitting time of $\mathcal{A}$ with respect to the dynamics:
\begin{equation}
\mathcal{T}_{\mathcal{A}}:=\inf\{t\ge0:\eta_{N}(t)\in\mathcal{A}\}.\label{e_TA-def}
\end{equation}

\begin{figure}
\begin{tikzpicture}
\fill[rounded corners,white] (-2.25,2.25) rectangle (2.25,-2.25); 
\fill[black] (-2,2) circle (0.1); \draw (-2,2) circle (0.1);
\fill[white] (-1,2) circle (0.1); \draw[black!30!white] (-1,2) circle (0.1);
\fill[white] (0,2) circle (0.1); \draw[black!30!white] (0,2) circle (0.1);
\fill[black] (1,2) circle (0.1); \draw (1,2) circle (0.1);
\fill[black] (2,2) circle (0.1); \draw (2,2) circle (0.1);
\fill[black] (-2,1) circle (0.1); \draw (-2,1) circle (0.1);
\fill[white] (-1,1) circle (0.1); \draw[black!30!white] (-1,1) circle (0.1);
\fill[white] (0,1) circle (0.1); \draw[black!30!white] (0,1) circle (0.1);
\fill[white] (1,1) circle (0.1); \draw[black!30!white] (1,1) circle (0.1);
\fill[white] (2,1) circle (0.1); \draw[black!30!white] (2,1) circle (0.1);
\fill[white] (-2,0) circle (0.1); \draw[black!30!white] (-2,0) circle (0.1);
\fill[white] (-1,0) circle (0.1); \draw[black!30!white] (-1,0) circle (0.1);
\fill[white] (0,0) circle (0.1); \draw[black!30!white] (0,0) circle (0.1);
\fill[black] (1,0) circle (0.1); \draw (1,0) circle (0.1);
\fill[white] (2,0) circle (0.1); \draw[black!30!white] (2,0) circle (0.1);
\fill[black] (-2,-1) circle (0.1); \draw (-2,-1) circle (0.1);
\fill[black] (-1,-1) circle (0.1); \draw (-1,-1) circle (0.1);
\fill[white] (0,-1) circle (0.1); \draw[black!30!white] (0,-1) circle (0.1);
\fill[white] (1,-1) circle (0.1); \draw[black!30!white] (1,-1) circle (0.1);
\fill[white] (2,-1) circle (0.1); \draw[black!30!white] (2,-1) circle (0.1);
\fill[black] (-2,-2) circle (0.1); \draw (-2,-2) circle (0.1);
\fill[black] (-1,-2) circle (0.1); \draw (-1,-2) circle (0.1);
\fill[white] (0,-2) circle (0.1); \draw[black!30!white] (0,-2) circle (0.1);
\fill[white] (1,-2) circle (0.1); \draw[black!30!white] (1,-2) circle (0.1);
\fill[black] (2,-2) circle (0.1); \draw (2,-2) circle (0.1);
\draw[very thick,to-,black!10!white] (-1.9,2)--(-1.1,2); \draw[very thick,black!10!white] (-0.9,2)--(-0.1,2); \draw[very thick,-to,black!10!white] (0.1,2)--(0.9,2); \draw[very thick,latex-latex] (1.1,2)--(1.9,2);
\draw[very thick,latex-latex] (-2,1.9)--(-2,1.1); \draw[very thick,to-,black!10!white] (1,1.9)--(1,1.1); \draw[very thick, to-,black!10!white] (2,1.9)--(2,1.1);
\draw[very thick,to-,black!10!white] (-1.9,1)--(-1.1,1); \draw[very thick,black!10!white] (0.1,1)--(0.9,1); \draw[very thick,black!10!white] (1.1,1)--(1.9,1);
\draw[very thick,to-,black!10!white] (-2,0.9)--(-2,0.1); \draw[very thick,-to,black!10!white] (-1,0.9)--(-1,0.1); \draw[very thick,-to,black!10!white] (0,0.9)--(0,0.1); \draw[very thick,-to,black!10!white] (1,0.9)--(1,0.1);
\draw[very thick,black!10!white] (-1.9,0)--(-1.1,0); \draw[very thick,-to,black!10!white] (-0.9,0)--(-0.1,0); \draw[very thick,to-,black!10!white] (1.1,0)--(1.9,0);
\draw[very thick,-to,black!10!white] (-2,-0.1)--(-2,-0.9); \draw[very thick,-to,black!10!white] (-1,-0.1)--(-1,-0.9); \draw[very thick,-to,black!10!white] (0,-0.1)--(0,-0.9); \draw[very thick,-to,black!10!white] (2,-0.1)--(2,-0.9);
\draw[very thick,latex-latex] (-1.9,-1)--(-1.1,-1); \draw[very thick,to-,black!10!white] (-0.9,-1)--(-0.1,-1); \draw[very thick,-to,black!10!white] (0.1,-1)--(0.9,-1);
\draw[very thick,latex-latex] (-2,-1.1)--(-2,-1.9); \draw[very thick,latex-latex] (-1,-1.1)--(-1,-1.9); \draw[very thick,-to,black!10!white] (2,-1.1)--(2,-1.9);
\draw[very thick,latex-latex] (-1.9,-2)--(-1.1,-2); \draw[very thick,to-,black!10!white] (-0.9,-2)--(-0.1,-2); \draw[very thick,to-,black!10!white] (0.1,-2)--(0.9,-2); \draw[very thick,-to,black!10!white] (1.1,-2)--(1.9,-2);
\end{tikzpicture}
\hspace{20mm}
\begin{tikzpicture}
\fill[rounded corners,white] (-2.25,2.25) rectangle (2.25,-2.25); 
\fill[rounded corners,orange!50!white] (-2.25,2.25) rectangle (-1.75,0.75);
\fill[rounded corners,orange!50!white] (0.75,2.25) rectangle (2.25,1.75);
\fill[rounded corners,orange!50!white] (0.75,0.25) rectangle (1.25,-0.25);
\fill[rounded corners,orange!50!white] (-2.25,-0.75) rectangle (-0.75,-2.25);
\fill[rounded corners,orange!50!white] (1.75,-1.75) rectangle (2.25,-2.25);
\fill[black] (-2,2) circle (0.1); \draw (-2,2) circle (0.1);
\fill[white] (-1,2) circle (0.1); \draw[black!30!white] (-1,2) circle (0.1);
\fill[white] (0,2) circle (0.1); \draw[black!30!white] (0,2) circle (0.1);
\fill[black] (1,2) circle (0.1); \draw (1,2) circle (0.1);
\fill[black] (2,2) circle (0.1); \draw (2,2) circle (0.1);
\fill[black] (-2,1) circle (0.1); \draw (-2,1) circle (0.1);
\fill[white] (-1,1) circle (0.1); \draw[black!30!white] (-1,1) circle (0.1);
\fill[white] (0,1) circle (0.1); \draw[black!30!white] (0,1) circle (0.1);
\fill[white] (1,1) circle (0.1); \draw[black!30!white] (1,1) circle (0.1);
\fill[white] (2,1) circle (0.1); \draw[black!30!white] (2,1) circle (0.1);
\fill[white] (-2,0) circle (0.1); \draw[black!30!white] (-2,0) circle (0.1);
\fill[white] (-1,0) circle (0.1); \draw[black!30!white] (-1,0) circle (0.1);
\fill[white] (0,0) circle (0.1); \draw[black!30!white] (0,0) circle (0.1);
\fill[black] (1,0) circle (0.1); \draw (1,0) circle (0.1);
\fill[white] (2,0) circle (0.1); \draw[black!30!white] (2,0) circle (0.1);
\fill[black] (-2,-1) circle (0.1); \draw (-2,-1) circle (0.1);
\fill[black] (-1,-1) circle (0.1); \draw (-1,-1) circle (0.1);
\fill[white] (0,-1) circle (0.1); \draw[black!30!white] (0,-1) circle (0.1);
\fill[white] (1,-1) circle (0.1); \draw[black!30!white] (1,-1) circle (0.1);
\fill[white] (2,-1) circle (0.1); \draw[black!30!white] (2,-1) circle (0.1);
\fill[black] (-2,-2) circle (0.1); \draw (-2,-2) circle (0.1);
\fill[black] (-1,-2) circle (0.1); \draw (-1,-2) circle (0.1);
\fill[white] (0,-2) circle (0.1); \draw[black!30!white] (0,-2) circle (0.1);
\fill[white] (1,-2) circle (0.1); \draw[black!30!white] (1,-2) circle (0.1);
\fill[black] (2,-2) circle (0.1); \draw (2,-2) circle (0.1);
\draw[very thick,to-,black!10!white] (-1.9,2)--(-1.1,2); \draw[very thick,black!10!white] (-0.9,2)--(-0.1,2); \draw[very thick,-to,black!10!white] (0.1,2)--(0.9,2); \draw[very thick,latex-latex] (1.1,2)--(1.9,2);
\draw[very thick,latex-latex] (-2,1.9)--(-2,1.1); \draw[very thick,to-,black!10!white] (1,1.9)--(1,1.1); \draw[very thick, to-,black!10!white] (2,1.9)--(2,1.1);
\draw[very thick,to-,black!10!white] (-1.9,1)--(-1.1,1); \draw[very thick,black!10!white] (0.1,1)--(0.9,1); \draw[very thick,black!10!white] (1.1,1)--(1.9,1);
\draw[very thick,to-,black!10!white] (-2,0.9)--(-2,0.1); \draw[very thick,-to,black!10!white] (-1,0.9)--(-1,0.1); \draw[very thick,-to,black!10!white] (0,0.9)--(0,0.1); \draw[very thick,-to,black!10!white] (1,0.9)--(1,0.1);
\draw[very thick,black!10!white] (-1.9,0)--(-1.1,0); \draw[very thick,-to,black!10!white] (-0.9,0)--(-0.1,0); \draw[very thick,to-,black!10!white] (1.1,0)--(1.9,0);
\draw[very thick,-to,black!10!white] (-2,-0.1)--(-2,-0.9); \draw[very thick,-to,black!10!white] (-1,-0.1)--(-1,-0.9); \draw[very thick,-to,black!10!white] (0,-0.1)--(0,-0.9); \draw[very thick,-to,black!10!white] (2,-0.1)--(2,-0.9);
\draw[very thick,latex-latex] (-1.9,-1)--(-1.1,-1); \draw[very thick,to-,black!10!white] (-0.9,-1)--(-0.1,-1); \draw[very thick,-to,black!10!white] (0.1,-1)--(0.9,-1);
\draw[very thick,latex-latex] (-2,-1.1)--(-2,-1.9); \draw[very thick,latex-latex] (-1,-1.1)--(-1,-1.9); \draw[very thick,-to,black!10!white] (2,-1.1)--(2,-1.9);
\draw[very thick,latex-latex] (-1.9,-2)--(-1.1,-2); \draw[very thick,to-,black!10!white] (-0.9,-2)--(-0.1,-2); \draw[very thick,to-,black!10!white] (0.1,-2)--(0.9,-2); \draw[very thick,-to,black!10!white] (1.1,-2)--(1.9,-2);
\end{tikzpicture}\caption{\label{fig2.2}The left figure describes the first level of metastable
transitions between the sites in $S_{\star}$. The right figure describes
five irreducible components of $S_{\star}$ with respect to $X(\cdot)$.}
\end{figure}
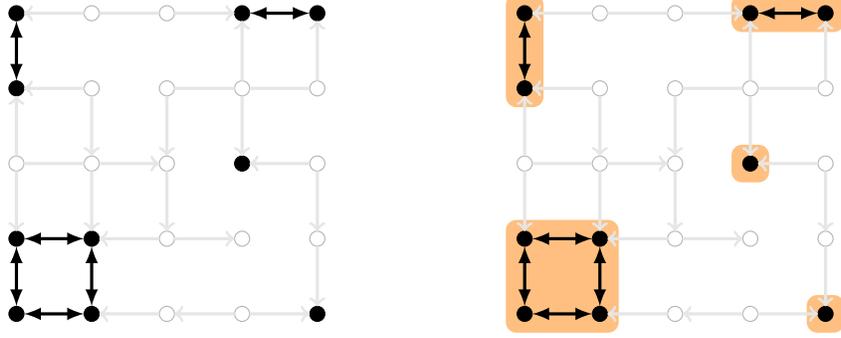

\begin{thm}[{1st time scale of metastability, \cite[Theorem 2.3]{BDG}}]
\label{t_1st}Define $\theta_{1}=\theta_{N,\,1}:=1/d_{N}$. Fix a
site $x\in S_{\star}$ and suppose that the process starts from $\mathcal{E}_{N}^{x}=\{\xi^{x}\}$.
\begin{enumerate}
\item In the limit $N\to\infty$, the law of the accelerated order process
$\{X_{N}(\theta_{1}t)\}_{t\ge0}$ converges, in the sense of finite-dimensional
marginal distributions, to the law of a Markov chain $\{X(t)\}_{t\ge0}$
on $S_{\star}$ determined by the original transition rate $r(\cdot,\,\cdot)$
(See Figure \ref{fig2.2}).
\item The Eyring--Kramers formula is given as (cf. \eqref{e_ENA-def} and
Notation \ref{n_fN-gN})
\[
\mathbb{E}_{\xi^{x}}[\mathcal{T}_{\mathcal{E}_{N}(S_{\star}\setminus\{x\})}]\simeq\frac{1}{\sum_{y\in S_{\star}}r(x,\,y)}\cdot\frac{1}{d_{N}}.
\]
If $\sum_{y\in S_{\star}}r(x,\,y)=0$, this should be read as $\lim_{N\to\infty}\theta_{1}^{-1}\cdot\mathbb{E}_{\xi^{x}}[\mathcal{T}_{\mathcal{E}_{N}(S_{\star}\setminus\{x\})}]=\infty$.
\item On the time scale $\theta_{1}$, excursions outside $\mathcal{E}_{N}(S_{\star})$
are negligible:
\[
\lim_{N\to\infty}\mathbb{E}_{\xi^{x}}\Big[\int_{0}^{t}\mathbbm{1}\{\eta_{N}(\theta_{1}s)\notin\mathcal{E}_{N}(S_{\star})\}\mathrm{d}s\Big]=0\text{ for all }t\ge0.
\]
\end{enumerate}
\end{thm}

Originally in \cite[Theorem 2.3-(ii)]{BDG}, the mode of convergence
in part (1) was alternatively chosen as the convergence of trace processes
in the Skorokhod topology. Combining this result with an additional
technicality presented in \cite[Proposition 2.1]{LanLouMou}, we may
easily deduce the convergence of finite-dimensional marginal distributions
as presented in Theorem \ref{t_1st}-(1).

Even though the underlying random walk determined by $r(\cdot,\,\cdot)$
is irreducible on the whole $S$, it is \emph{not necessarily irreducible}
when restricted to a subset $S_{\star}$ (cf. Figure \ref{fig2.2}).
Thus, the limit process $X(\cdot)$ is not necessarily irreducible.
However, since the original inclusion process is ergodic, it is natural
to expect to eventually observe the metastable transitions also between
the disconnected irreducible components of $S_{\star}$. Therefore,
we pursue to find a bigger time scale on which we can indeed observe
such new metastable transitions.

Moving further, we decompose $S_{\star}$ into $\kappa_{2}$ irreducible
components with respect to $X(\cdot)$:
\[
S_{\star}=\bigcup_{i=1}^{\kappa_{2}}S_{\star i}^{2}.
\]
For instance, in Figure \ref{fig2.2}, it holds that $\kappa_{2}=5$.
We are able to slightly generalize Theorem \ref{t_1st}-(3) as follows.
For real numbers $\alpha$ and $\beta$, we write $\llbracket\alpha,\,\beta\rrbracket:=[\alpha,\,\beta]\cap\mathbb{Z}$
throughout the article.
\begin{thm}[1st time scale, negligibility of excursions]
\label{t_1st-neg}Starting from a site $x\in S_{\star i}^{2}$ for
$i\in\llbracket1,\,\kappa_{2}\rrbracket$, on the first time scale
$\theta_{1}$, the original process spends negligible time outside
the collection $\mathcal{E}_{N}(S_{\star i}^{2})$:
\[
\lim_{N\to\infty}\mathbb{E}_{\xi^{x}}\Big[\int_{0}^{t}\mathbbm{1}\{\eta_{N}(\theta_{1}s)\notin\mathcal{E}_{N}(S_{\star i}^{2})\}\mathrm{d}s\Big]=0\text{ for all }t\ge0.
\]
\end{thm}

\subsection{\label{sec2.4}Second time scale of metastability}

To formulate the second-level metastability, we need a new projection
operator $\Psi_{2}=\Psi_{N,\,2}:\mathcal{H}_{N}\to\llbracket1,\,\kappa_{2}\rrbracket\cup\{\mathfrak{0}\}$
defined as
\[
\Psi_{2}(\eta):=\begin{cases}
i & \text{if }\eta=\xi^{x}\text{ for some }x\in S_{\star i}^{2}\text{ and }i\in\llbracket1,\,\kappa_{2}\rrbracket,\\
\mathfrak{0} & \text{otherwise}.
\end{cases}
\]
Then, the second order process $\{Y_{N}(t)\}_{t\ge0}$ is defined
as $Y_{N}(t):=\Psi_{2}(\eta_{N}(t))$. Hereafter, we additionally
assume that $d_{N}$ \emph{decays subexponentially}, i.e.,
\begin{equation}
\lim_{N\to\infty}d_{N}e^{\epsilon N}=\infty\text{ for all }\epsilon>0.\label{e_dN-decay-subexp}
\end{equation}
See Remark \ref{r_2nd}-(2) for a discussion on this condition.

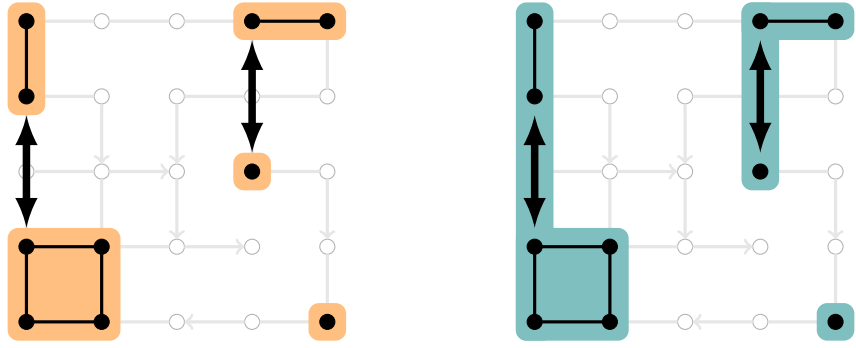
\begin{figure}
\begin{tikzpicture}
\fill[rounded corners,white] (-2.25,2.25) rectangle (2.25,-2.25); 
\fill[white] (-1,2) circle (0.1); \draw[black!30!white] (-1,2) circle (0.1);
\fill[white] (0,2) circle (0.1); \draw[black!30!white] (0,2) circle (0.1);
\fill[white] (-1,1) circle (0.1); \draw[black!30!white] (-1,1) circle (0.1);
\fill[white] (0,1) circle (0.1); \draw[black!30!white] (0,1) circle (0.1);
\fill[white] (1,1) circle (0.1); \draw[black!30!white] (1,1) circle (0.1);
\fill[white] (2,1) circle (0.1); \draw[black!30!white] (2,1) circle (0.1);
\fill[white] (-2,0) circle (0.1); \draw[black!30!white] (-2,0) circle (0.1);
\fill[white] (-1,0) circle (0.1); \draw[black!30!white] (-1,0) circle (0.1);
\fill[white] (0,0) circle (0.1); \draw[black!30!white] (0,0) circle (0.1);
\fill[white] (2,0) circle (0.1); \draw[black!30!white] (2,0) circle (0.1);
\fill[white] (0,-1) circle (0.1); \draw[black!30!white] (0,-1) circle (0.1);
\fill[white] (1,-1) circle (0.1); \draw[black!30!white] (1,-1) circle (0.1);
\fill[white] (2,-1) circle (0.1); \draw[black!30!white] (2,-1) circle (0.1);
\fill[white] (0,-2) circle (0.1); \draw[black!30!white] (0,-2) circle (0.1);
\fill[white] (1,-2) circle (0.1); \draw[black!30!white] (1,-2) circle (0.1);
\draw[very thick,to-,black!10!white] (-1.9,2)--(-1.1,2); \draw[very thick,black!10!white] (-0.9,2)--(-0.1,2); \draw[very thick,-to,black!10!white] (0.1,2)--(0.9,2);
\draw[very thick,to-,black!10!white] (1,1.9)--(1,1.1); \draw[very thick, to-,black!10!white] (2,1.9)--(2,1.1);
\draw[very thick,to-,black!10!white] (-1.9,1)--(-1.1,1); \draw[very thick,black!10!white] (0.1,1)--(0.9,1); \draw[very thick,black!10!white] (1.1,1)--(1.9,1);
\draw[very thick,to-,black!10!white] (-2,0.9)--(-2,0.1); \draw[very thick,-to,black!10!white] (-1,0.9)--(-1,0.1); \draw[very thick,-to,black!10!white] (0,0.9)--(0,0.1); \draw[very thick,-to,black!10!white] (1,0.9)--(1,0.1);
\draw[very thick,black!10!white] (-1.9,0)--(-1.1,0); \draw[very thick,-to,black!10!white] (-0.9,0)--(-0.1,0); \draw[very thick,to-,black!10!white] (1.1,0)--(1.9,0);
\draw[very thick,-to,black!10!white] (-2,-0.1)--(-2,-0.9); \draw[very thick,-to,black!10!white] (-1,-0.1)--(-1,-0.9); \draw[very thick,-to,black!10!white] (0,-0.1)--(0,-0.9); \draw[very thick,-to,black!10!white] (2,-0.1)--(2,-0.9);
\draw[very thick,to-,black!10!white] (-0.9,-1)--(-0.1,-1); \draw[very thick,-to,black!10!white] (0.1,-1)--(0.9,-1);
\draw[very thick,-to,black!10!white] (2,-1.1)--(2,-1.9);
\draw[very thick,to-,black!10!white] (-0.9,-2)--(-0.1,-2); \draw[very thick,to-,black!10!white] (0.1,-2)--(0.9,-2); \draw[very thick,-to,black!10!white] (1.1,-2)--(1.9,-2);
\fill[rounded corners,orange!50!white] (-2.25,2.25) rectangle (-1.75,0.75);
\fill[rounded corners,orange!50!white] (0.75,2.25) rectangle (2.25,1.75);
\fill[rounded corners,orange!50!white] (0.75,0.25) rectangle (1.25,-0.25);
\fill[rounded corners,orange!50!white] (-2.25,-0.75) rectangle (-0.75,-2.25);
\fill[rounded corners,orange!50!white] (1.75,-1.75) rectangle (2.25,-2.25);
\fill[black] (-2,2) circle (0.1); \draw (-2,2) circle (0.1);
\fill[black] (1,2) circle (0.1); \draw (1,2) circle (0.1);
\fill[black] (2,2) circle (0.1); \draw (2,2) circle (0.1);
\fill[black] (-2,1) circle (0.1); \draw (-2,1) circle (0.1);
\fill[black] (1,0) circle (0.1); \draw (1,0) circle (0.1);
\fill[black] (-2,-1) circle (0.1); \draw (-2,-1) circle (0.1);
\fill[black] (-1,-1) circle (0.1); \draw (-1,-1) circle (0.1);
\fill[black] (-2,-2) circle (0.1); \draw (-2,-2) circle (0.1);
\fill[black] (-1,-2) circle (0.1); \draw (-1,-2) circle (0.1);
\fill[black] (2,-2) circle (0.1); \draw (2,-2) circle (0.1);
\draw[very thick] (1.1,2)--(1.9,2);
\draw[very thick] (-2,1.9)--(-2,1.1);
\draw[very thick] (-1.9,-1)--(-1.1,-1);
\draw[very thick] (-2,-1.1)--(-2,-1.9); \draw[very thick] (-1,-1.1)--(-1,-1.9);
\draw[very thick] (-1.9,-2)--(-1.1,-2);
\draw[line width=1mm,latex-latex] (-2,0.75)--(-2,-0.75);
\draw[line width=1mm,latex-latex] (1,1.75)--(1,0.25);
\end{tikzpicture}
\hspace{20mm}
\begin{tikzpicture}
\fill[rounded corners,white] (-2.25,2.25) rectangle (2.25,-2.25); 
\fill[white] (-1,2) circle (0.1); \draw[black!30!white] (-1,2) circle (0.1);
\fill[white] (0,2) circle (0.1); \draw[black!30!white] (0,2) circle (0.1);
\fill[white] (-1,1) circle (0.1); \draw[black!30!white] (-1,1) circle (0.1);
\fill[white] (0,1) circle (0.1); \draw[black!30!white] (0,1) circle (0.1);
\fill[white] (1,1) circle (0.1); \draw[black!30!white] (1,1) circle (0.1);
\fill[white] (2,1) circle (0.1); \draw[black!30!white] (2,1) circle (0.1);
\fill[white] (-2,0) circle (0.1); \draw[black!30!white] (-2,0) circle (0.1);
\fill[white] (-1,0) circle (0.1); \draw[black!30!white] (-1,0) circle (0.1);
\fill[white] (0,0) circle (0.1); \draw[black!30!white] (0,0) circle (0.1);
\fill[white] (2,0) circle (0.1); \draw[black!30!white] (2,0) circle (0.1);
\fill[white] (0,-1) circle (0.1); \draw[black!30!white] (0,-1) circle (0.1);
\fill[white] (1,-1) circle (0.1); \draw[black!30!white] (1,-1) circle (0.1);
\fill[white] (2,-1) circle (0.1); \draw[black!30!white] (2,-1) circle (0.1);
\fill[white] (0,-2) circle (0.1); \draw[black!30!white] (0,-2) circle (0.1);
\fill[white] (1,-2) circle (0.1); \draw[black!30!white] (1,-2) circle (0.1);
\draw[very thick,to-,black!10!white] (-1.9,2)--(-1.1,2); \draw[very thick,black!10!white] (-0.9,2)--(-0.1,2); \draw[very thick,-to,black!10!white] (0.1,2)--(0.9,2);
\draw[very thick,to-,black!10!white] (1,1.9)--(1,1.1); \draw[very thick, to-,black!10!white] (2,1.9)--(2,1.1);
\draw[very thick,to-,black!10!white] (-1.9,1)--(-1.1,1); \draw[very thick,black!10!white] (0.1,1)--(0.9,1); \draw[very thick,black!10!white] (1.1,1)--(1.9,1);
\draw[very thick,to-,black!10!white] (-2,0.9)--(-2,0.1); \draw[very thick,-to,black!10!white] (-1,0.9)--(-1,0.1); \draw[very thick,-to,black!10!white] (0,0.9)--(0,0.1); \draw[very thick,-to,black!10!white] (1,0.9)--(1,0.1);
\draw[very thick,black!10!white] (-1.9,0)--(-1.1,0); \draw[very thick,-to,black!10!white] (-0.9,0)--(-0.1,0); \draw[very thick,to-,black!10!white] (1.1,0)--(1.9,0);
\draw[very thick,-to,black!10!white] (-2,-0.1)--(-2,-0.9); \draw[very thick,-to,black!10!white] (-1,-0.1)--(-1,-0.9); \draw[very thick,-to,black!10!white] (0,-0.1)--(0,-0.9); \draw[very thick,-to,black!10!white] (2,-0.1)--(2,-0.9);
\draw[very thick,to-,black!10!white] (-0.9,-1)--(-0.1,-1); \draw[very thick,-to,black!10!white] (0.1,-1)--(0.9,-1);
\draw[very thick,-to,black!10!white] (2,-1.1)--(2,-1.9);
\draw[very thick,to-,black!10!white] (-0.9,-2)--(-0.1,-2); \draw[very thick,to-,black!10!white] (0.1,-2)--(0.9,-2); \draw[very thick,-to,black!10!white] (1.1,-2)--(1.9,-2);
\fill[rounded corners,teal!50!white] (-2.25,2.25) rectangle (-1.75,-2.25);
\fill[rounded corners,teal!50!white] (0.75,2.25) rectangle (2.25,1.75);
\fill[rounded corners,teal!50!white] (0.75,2.25) rectangle (1.25,-0.25);
\fill[rounded corners,teal!50!white] (-2.25,-0.75) rectangle (-0.75,-2.25);
\fill[rounded corners,teal!50!white] (1.75,-1.75) rectangle (2.25,-2.25);
\fill[black] (-2,2) circle (0.1); \draw (-2,2) circle (0.1);
\fill[black] (1,2) circle (0.1); \draw (1,2) circle (0.1);
\fill[black] (2,2) circle (0.1); \draw (2,2) circle (0.1);
\fill[black] (-2,1) circle (0.1); \draw (-2,1) circle (0.1);
\fill[black] (1,0) circle (0.1); \draw (1,0) circle (0.1);
\fill[black] (-2,-1) circle (0.1); \draw (-2,-1) circle (0.1);
\fill[black] (-1,-1) circle (0.1); \draw (-1,-1) circle (0.1);
\fill[black] (-2,-2) circle (0.1); \draw (-2,-2) circle (0.1);
\fill[black] (-1,-2) circle (0.1); \draw (-1,-2) circle (0.1);
\fill[black] (2,-2) circle (0.1); \draw (2,-2) circle (0.1);
\draw[very thick] (1.1,2)--(1.9,2);
\draw[very thick] (-2,1.9)--(-2,1.1);
\draw[very thick] (-1.9,-1)--(-1.1,-1);
\draw[very thick] (-2,-1.1)--(-2,-1.9); \draw[very thick] (-1,-1.1)--(-1,-1.9);
\draw[very thick] (-1.9,-2)--(-1.1,-2);
\draw[line width=1mm,latex-latex] (-2,0.75)--(-2,-0.75);
\draw[line width=1mm,latex-latex] (1,1.75)--(1,0.25);
\end{tikzpicture}\caption{\label{fig2.3}The left figure illustrates the second level of metastable
transitions between the collections $S_{\star i}^{2}$ for $i\in\llbracket1,\,\kappa_{2}\rrbracket$.
The right figure illustrates three irreducible components with respect
to $Y(\cdot)$.}
\end{figure}

\begin{thm}[2nd time scale of metastability]
\label{t_2nd}Define $\theta_{2}=\theta_{N,\,2}:=N/d_{N}^{2}$. Fix
a starting position $i\in\llbracket1,\,\kappa_{2}\rrbracket$.
\begin{enumerate}
\item (\cite[Theorem 2.10-(1)]{Kim RIP-2nd}) The collection $\mathcal{E}_{N}(S_{\star i}^{2})$
thermalizes before the metastable transition:
\[
\lim_{N\to\infty}\inf_{x,\,y\in S_{\star i}^{2}}\mathbb{P}_{\xi^{x}}[\mathcal{T}_{\xi^{y}}<\mathcal{T}_{\mathcal{E}_{N}(S_{\star}\setminus S_{\star i}^{2})}]=1.
\]
\item (\cite[Theorem 2.10-(2)]{Kim RIP-2nd}) As $N\to\infty$, the law
of the accelerated second order process $\{Y_{N}(\theta_{2}t)\}_{t\ge0}$
converges to the law of a Markov chain $\{Y(t)\}_{t\ge0}$ on $\llbracket1,\,\kappa_{2}\rrbracket$
which is determined by a transition rate function defined as $r^{\textup{2nd}}(i,\,i):=0$
and
\[
r^{\textup{2nd}}(i,\,j):=\frac{1}{|S_{\star i}^{2}|\cdot\mathfrak{R}_{ij}}\text{ for }i,\,j\in\llbracket1,\,\kappa_{2}\rrbracket.
\]
Here, $\mathfrak{R}_{ij}=\mathfrak{R}_{ji}\in(0,\,\infty]$ is a (possibly
infinite) positive constant given in \eqref{e_Rij-def} below, so
that if $\mathfrak{R}_{ij}=\infty$ then $r^{\textup{2nd}}(i,\,j)$
is naturally defined as $0$. Refer to Figure \ref{fig2.3}.
\item For every $x\in S_{\star i}^{2}$, the Eyring--Kramers formula is
given as
\[
\mathbb{E}_{\xi^{x}}[\mathcal{T}_{\mathcal{E}_{N}(S_{\star}\setminus S_{\star i}^{2})}]\simeq\frac{1}{\sum_{j\in\llbracket1,\,\kappa_{2}\rrbracket}r^{\textup{2nd}}(i,\,j)}\cdot\frac{N}{d_{N}^{2}}.
\]
If $\sum_{j\in\llbracket1,\,\kappa_{2}\rrbracket}r^{\textup{2nd}}(i,\,j)=0$,
this should be read as $\lim_{N\to\infty}\theta_{2}^{-1}\cdot\mathbb{E}_{\xi^{x}}[\mathcal{T}_{\mathcal{E}_{N}(S_{\star}\setminus S_{\star i}^{2})}]=\infty$.
\end{enumerate}
\end{thm}

The Eyring--Kramers formula subject to the second time scale given
in Theorem \ref{t_2nd}-(3) is not explicitly formulated in \cite{Kim RIP-2nd}.
Nevertheless, we will also present the Eyring--Kramers formula subject
to the third time scale in Theorem \ref{t_3rd}-(3) of which the idea
of proof is the same. Thus, we do not provide an explicit proof of
Theorem \ref{t_2nd}-(3) and refer the readers to the proof of Theorem
\ref{t_3rd}-(3) to be given in Section \ref{sec3}.

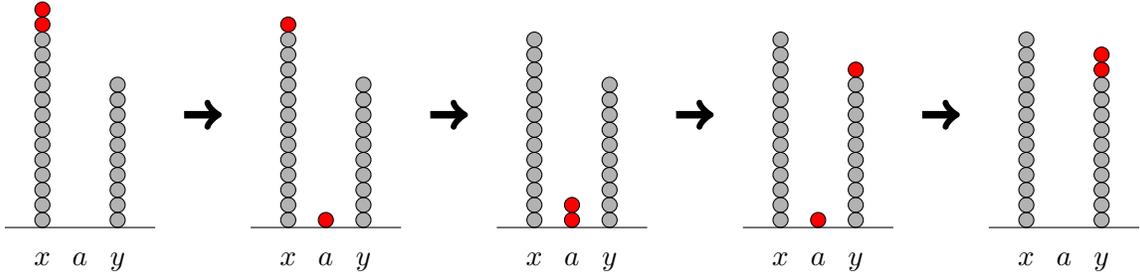
\begin{figure}
\begin{tikzpicture}
\fill[rounded corners,white] (-1,3) rectangle (1,-0.8); 
\fill[black!30!white] (-0.5,0.1) circle (0.1); \fill[black!30!white] (-0.5,0.3) circle (0.1); \fill[black!30!white] (-0.5,0.5) circle (0.1); \fill[black!30!white] (-0.5,0.7) circle (0.1); \fill[black!30!white] (-0.5,0.9) circle (0.1); \fill[black!30!white] (-0.5,1.1) circle (0.1); \fill[black!30!white] (-0.5,1.3) circle (0.1); \fill[black!30!white] (-0.5,1.5) circle (0.1); \fill[black!30!white] (-0.5,1.7) circle (0.1); \fill[black!30!white] (-0.5,1.9) circle (0.1); \fill[black!30!white] (-0.5,2.1) circle (0.1); \fill[black!30!white] (-0.5,2.3) circle (0.1); \fill[black!30!white] (-0.5,2.5) circle (0.1); \fill[red] (-0.5,2.7) circle (0.1); \fill[red] (-0.5,2.9) circle (0.1);
\fill[black!30!white] (0.5,0.1) circle (0.1); \fill[black!30!white] (0.5,0.3) circle (0.1); \fill[black!30!white] (0.5,0.5) circle (0.1); \fill[black!30!white] (0.5,0.7) circle (0.1); \fill[black!30!white] (0.5,0.9) circle (0.1); \fill[black!30!white] (0.5,1.1) circle (0.1); \fill[black!30!white] (0.5,1.3) circle (0.1); \fill[black!30!white] (0.5,1.5) circle (0.1); \fill[black!30!white] (0.5,1.7) circle (0.1); \fill[black!30!white] (0.5,1.9) circle (0.1);
\draw (-1,0)--(1,0);
\draw (-0.5,0.1) circle (0.1); \draw (-0.5,0.3) circle (0.1); \draw (-0.5,0.5) circle (0.1); \draw (-0.5,0.7) circle (0.1); \draw (-0.5,0.9) circle (0.1); \draw (-0.5,1.1) circle (0.1); \draw (-0.5,1.3) circle (0.1); \draw (-0.5,1.5) circle (0.1); \draw (-0.5,1.7) circle (0.1); \draw (-0.5,1.9) circle (0.1); \draw (-0.5,2.1) circle (0.1); \draw (-0.5,2.3) circle (0.1); \draw (-0.5,2.5) circle (0.1); \draw (-0.5,2.7) circle (0.1); \draw (-0.5,2.9) circle (0.1);
\draw (0.5,0.1) circle (0.1); \draw (0.5,0.3) circle (0.1); \draw (0.5,0.5) circle (0.1); \draw (0.5,0.7) circle (0.1); \draw (0.5,0.9) circle (0.1); \draw (0.5,1.1) circle (0.1); \draw (0.5,1.3) circle (0.1); \draw (0.5,1.5) circle (0.1); \draw (0.5,1.7) circle (0.1); \draw (0.5,1.9) circle (0.1);
\draw (-0.5,-0.2) node[below]{$x$};
\draw (0,-0.2) node[below]{$a$};
\draw (0.5,-0.2) node[below]{$y$};
\end{tikzpicture}
\begin{tikzpicture}
\fill[rounded corners,white] (-0.5,3) rectangle (0.5,-0.8); 
\draw[line width=1mm,-to] (-0.25,1.5)->(0.25,1.5);
\end{tikzpicture}
\begin{tikzpicture}
\fill[rounded corners,white] (-1,3) rectangle (1,-0.8); 
\fill[black!30!white] (-0.5,0.1) circle (0.1); \fill[black!30!white] (-0.5,0.3) circle (0.1); \fill[black!30!white] (-0.5,0.5) circle (0.1); \fill[black!30!white] (-0.5,0.7) circle (0.1); \fill[black!30!white] (-0.5,0.9) circle (0.1); \fill[black!30!white] (-0.5,1.1) circle (0.1); \fill[black!30!white] (-0.5,1.3) circle (0.1); \fill[black!30!white] (-0.5,1.5) circle (0.1); \fill[black!30!white] (-0.5,1.7) circle (0.1); \fill[black!30!white] (-0.5,1.9) circle (0.1); \fill[black!30!white] (-0.5,2.1) circle (0.1); \fill[black!30!white] (-0.5,2.3) circle (0.1); \fill[black!30!white] (-0.5,2.5) circle (0.1); \fill[red] (-0.5,2.7) circle (0.1);
\fill[red] (0,0.1) circle (0.1);
\fill[black!30!white] (0.5,0.1) circle (0.1); \fill[black!30!white] (0.5,0.3) circle (0.1); \fill[black!30!white] (0.5,0.5) circle (0.1); \fill[black!30!white] (0.5,0.7) circle (0.1); \fill[black!30!white] (0.5,0.9) circle (0.1); \fill[black!30!white] (0.5,1.1) circle (0.1); \fill[black!30!white] (0.5,1.3) circle (0.1); \fill[black!30!white] (0.5,1.5) circle (0.1); \fill[black!30!white] (0.5,1.7) circle (0.1); \fill[black!30!white] (0.5,1.9) circle (0.1);
\draw (-1,0)--(1,0);
\draw (-0.5,0.1) circle (0.1); \draw (-0.5,0.3) circle (0.1); \draw (-0.5,0.5) circle (0.1); \draw (-0.5,0.7) circle (0.1); \draw (-0.5,0.9) circle (0.1); \draw (-0.5,1.1) circle (0.1); \draw (-0.5,1.3) circle (0.1); \draw (-0.5,1.5) circle (0.1); \draw (-0.5,1.7) circle (0.1); \draw (-0.5,1.9) circle (0.1); \draw (-0.5,2.1) circle (0.1); \draw (-0.5,2.3) circle (0.1); \draw (-0.5,2.5) circle (0.1); \draw (-0.5,2.7) circle (0.1);
\draw (0,0.1) circle (0.1);
\draw (0.5,0.1) circle (0.1); \draw (0.5,0.3) circle (0.1); \draw (0.5,0.5) circle (0.1); \draw (0.5,0.7) circle (0.1); \draw (0.5,0.9) circle (0.1); \draw (0.5,1.1) circle (0.1); \draw (0.5,1.3) circle (0.1); \draw (0.5,1.5) circle (0.1); \draw (0.5,1.7) circle (0.1); \draw (0.5,1.9) circle (0.1);
\draw (-0.5,-0.2) node[below]{$x$};
\draw (0,-0.2) node[below]{$a$};
\draw (0.5,-0.2) node[below]{$y$};
\end{tikzpicture}
\begin{tikzpicture}
\fill[rounded corners,white] (-0.5,3) rectangle (0.5,-0.8); 
\draw[line width=1mm,-to] (-0.25,1.5)->(0.25,1.5);
\end{tikzpicture}
\begin{tikzpicture}
\fill[rounded corners,white] (-1,3) rectangle (1,-0.8); 
\fill[black!30!white] (-0.5,0.1) circle (0.1); \fill[black!30!white] (-0.5,0.3) circle (0.1); \fill[black!30!white] (-0.5,0.5) circle (0.1); \fill[black!30!white] (-0.5,0.7) circle (0.1); \fill[black!30!white] (-0.5,0.9) circle (0.1); \fill[black!30!white] (-0.5,1.1) circle (0.1); \fill[black!30!white] (-0.5,1.3) circle (0.1); \fill[black!30!white] (-0.5,1.5) circle (0.1); \fill[black!30!white] (-0.5,1.7) circle (0.1); \fill[black!30!white] (-0.5,1.9) circle (0.1); \fill[black!30!white] (-0.5,2.1) circle (0.1); \fill[black!30!white] (-0.5,2.3) circle (0.1); \fill[black!30!white] (-0.5,2.5) circle (0.1);
\fill[red] (0,0.1) circle (0.1); \fill[red] (0,0.3) circle (0.1);
\fill[black!30!white] (0.5,0.1) circle (0.1); \fill[black!30!white] (0.5,0.3) circle (0.1); \fill[black!30!white] (0.5,0.5) circle (0.1); \fill[black!30!white] (0.5,0.7) circle (0.1); \fill[black!30!white] (0.5,0.9) circle (0.1); \fill[black!30!white] (0.5,1.1) circle (0.1); \fill[black!30!white] (0.5,1.3) circle (0.1); \fill[black!30!white] (0.5,1.5) circle (0.1); \fill[black!30!white] (0.5,1.7) circle (0.1); \fill[black!30!white] (0.5,1.9) circle (0.1);
\draw (-1,0)--(1,0);
\draw (-0.5,0.1) circle (0.1); \draw (-0.5,0.3) circle (0.1); \draw (-0.5,0.5) circle (0.1); \draw (-0.5,0.7) circle (0.1); \draw (-0.5,0.9) circle (0.1); \draw (-0.5,1.1) circle (0.1); \draw (-0.5,1.3) circle (0.1); \draw (-0.5,1.5) circle (0.1); \draw (-0.5,1.7) circle (0.1); \draw (-0.5,1.9) circle (0.1); \draw (-0.5,2.1) circle (0.1); \draw (-0.5,2.3) circle (0.1); \draw (-0.5,2.5) circle (0.1);
\draw (0,0.1) circle (0.1); \draw (0,0.3) circle (0.1);
\draw (0.5,0.1) circle (0.1); \draw (0.5,0.3) circle (0.1); \draw (0.5,0.5) circle (0.1); \draw (0.5,0.7) circle (0.1); \draw (0.5,0.9) circle (0.1); \draw (0.5,1.1) circle (0.1); \draw (0.5,1.3) circle (0.1); \draw (0.5,1.5) circle (0.1); \draw (0.5,1.7) circle (0.1); \draw (0.5,1.9) circle (0.1);
\draw (-0.5,-0.2) node[below]{$x$};
\draw (0,-0.2) node[below]{$a$};
\draw (0.5,-0.2) node[below]{$y$};
\end{tikzpicture}
\begin{tikzpicture}
\fill[rounded corners,white] (-0.5,3) rectangle (0.5,-0.8); 
\draw[line width=1mm,-to] (-0.25,1.5)->(0.25,1.5);
\end{tikzpicture}
\begin{tikzpicture}
\fill[rounded corners,white] (-1,3) rectangle (1,-0.8); 
\fill[black!30!white] (-0.5,0.1) circle (0.1); \fill[black!30!white] (-0.5,0.3) circle (0.1); \fill[black!30!white] (-0.5,0.5) circle (0.1); \fill[black!30!white] (-0.5,0.7) circle (0.1); \fill[black!30!white] (-0.5,0.9) circle (0.1); \fill[black!30!white] (-0.5,1.1) circle (0.1); \fill[black!30!white] (-0.5,1.3) circle (0.1); \fill[black!30!white] (-0.5,1.5) circle (0.1); \fill[black!30!white] (-0.5,1.7) circle (0.1); \fill[black!30!white] (-0.5,1.9) circle (0.1); \fill[black!30!white] (-0.5,2.1) circle (0.1); \fill[black!30!white] (-0.5,2.3) circle (0.1); \fill[black!30!white] (-0.5,2.5) circle (0.1);
\fill[red] (0,0.1) circle (0.1);
\fill[black!30!white] (0.5,0.1) circle (0.1); \fill[black!30!white] (0.5,0.3) circle (0.1); \fill[black!30!white] (0.5,0.5) circle (0.1); \fill[black!30!white] (0.5,0.7) circle (0.1); \fill[black!30!white] (0.5,0.9) circle (0.1); \fill[black!30!white] (0.5,1.1) circle (0.1); \fill[black!30!white] (0.5,1.3) circle (0.1); \fill[black!30!white] (0.5,1.5) circle (0.1); \fill[black!30!white] (0.5,1.7) circle (0.1); \fill[black!30!white] (0.5,1.9) circle (0.1); \fill[red] (0.5,2.1) circle (0.1);
\draw (-1,0)--(1,0);
\draw (-0.5,0.1) circle (0.1); \draw (-0.5,0.3) circle (0.1); \draw (-0.5,0.5) circle (0.1); \draw (-0.5,0.7) circle (0.1); \draw (-0.5,0.9) circle (0.1); \draw (-0.5,1.1) circle (0.1); \draw (-0.5,1.3) circle (0.1); \draw (-0.5,1.5) circle (0.1); \draw (-0.5,1.7) circle (0.1); \draw (-0.5,1.9) circle (0.1); \draw (-0.5,2.1) circle (0.1); \draw (-0.5,2.3) circle (0.1); \draw (-0.5,2.5) circle (0.1);
\draw (0,0.1) circle (0.1);
\draw (0.5,0.1) circle (0.1); \draw (0.5,0.3) circle (0.1); \draw (0.5,0.5) circle (0.1); \draw (0.5,0.7) circle (0.1); \draw (0.5,0.9) circle (0.1); \draw (0.5,1.1) circle (0.1); \draw (0.5,1.3) circle (0.1); \draw (0.5,1.5) circle (0.1); \draw (0.5,1.7) circle (0.1); \draw (0.5,1.9) circle (0.1); \draw (0.5,2.1) circle (0.1);
\draw (-0.5,-0.2) node[below]{$x$};
\draw (0,-0.2) node[below]{$a$};
\draw (0.5,-0.2) node[below]{$y$};
\end{tikzpicture}
\begin{tikzpicture}
\fill[rounded corners,white] (-0.5,3) rectangle (0.5,-0.8); 
\draw[line width=1mm,-to] (-0.25,1.5)->(0.25,1.5);
\end{tikzpicture}
\begin{tikzpicture}
\fill[rounded corners,white] (-1,3) rectangle (1,-0.8); 
\fill[black!30!white] (-0.5,0.1) circle (0.1); \fill[black!30!white] (-0.5,0.3) circle (0.1); \fill[black!30!white] (-0.5,0.5) circle (0.1); \fill[black!30!white] (-0.5,0.7) circle (0.1); \fill[black!30!white] (-0.5,0.9) circle (0.1); \fill[black!30!white] (-0.5,1.1) circle (0.1); \fill[black!30!white] (-0.5,1.3) circle (0.1); \fill[black!30!white] (-0.5,1.5) circle (0.1); \fill[black!30!white] (-0.5,1.7) circle (0.1); \fill[black!30!white] (-0.5,1.9) circle (0.1); \fill[black!30!white] (-0.5,2.1) circle (0.1); \fill[black!30!white] (-0.5,2.3) circle (0.1); \fill[black!30!white] (-0.5,2.5) circle (0.1);
\fill[black!30!white] (0.5,0.1) circle (0.1); \fill[black!30!white] (0.5,0.3) circle (0.1); \fill[black!30!white] (0.5,0.5) circle (0.1); \fill[black!30!white] (0.5,0.7) circle (0.1); \fill[black!30!white] (0.5,0.9) circle (0.1); \fill[black!30!white] (0.5,1.1) circle (0.1); \fill[black!30!white] (0.5,1.3) circle (0.1); \fill[black!30!white] (0.5,1.5) circle (0.1); \fill[black!30!white] (0.5,1.7) circle (0.1); \fill[black!30!white] (0.5,1.9) circle (0.1); \fill[red] (0.5,2.1) circle (0.1); \fill[red] (0.5,2.3) circle (0.1);
\draw (-1,0)--(1,0);
\draw (-0.5,0.1) circle (0.1); \draw (-0.5,0.3) circle (0.1); \draw (-0.5,0.5) circle (0.1); \draw (-0.5,0.7) circle (0.1); \draw (-0.5,0.9) circle (0.1); \draw (-0.5,1.1) circle (0.1); \draw (-0.5,1.3) circle (0.1); \draw (-0.5,1.5) circle (0.1); \draw (-0.5,1.7) circle (0.1); \draw (-0.5,1.9) circle (0.1); \draw (-0.5,2.1) circle (0.1); \draw (-0.5,2.3) circle (0.1); \draw (-0.5,2.5) circle (0.1);
\draw (0.5,0.1) circle (0.1); \draw (0.5,0.3) circle (0.1); \draw (0.5,0.5) circle (0.1); \draw (0.5,0.7) circle (0.1); \draw (0.5,0.9) circle (0.1); \draw (0.5,1.1) circle (0.1); \draw (0.5,1.3) circle (0.1); \draw (0.5,1.5) circle (0.1); \draw (0.5,1.7) circle (0.1); \draw (0.5,1.9) circle (0.1); \draw (0.5,2.1) circle (0.1); \draw (0.5,2.3) circle (0.1);
\draw (-0.5,-0.2) node[below]{$x$};
\draw (0,-0.2) node[below]{$a$};
\draw (0.5,-0.2) node[below]{$y$};
\end{tikzpicture}\caption{\label{fig2.4}Typical movements of particles during the metastable
transition from $\xi^{x}$ to $\xi^{y}$ on the second time scale,
where $x,\,y\in S_{\star}$ and $a\in S_{0}$. There are three sites
occupied in each movement.}
\end{figure}

\begin{rem}
\label{r_2nd}A few remarks are in order.
\begin{enumerate}
\item The constant $\mathfrak{R}_{ij}$ that appears in Theorem \ref{t_2nd}-(2)
has an explicit formula, which is
\begin{equation}
\mathfrak{R}_{ij}=\int_{0}^{1}\Big[\sum_{x\in S_{\star i}^{2}}\sum_{y\in S_{\star j}^{2}}\sum_{a\in S_{0}}\frac{1}{(1-m_{a})(\frac{1-t}{r(x,\,a)}+\frac{t}{r(y,\,a)})}\Big]^{-1}\mathrm{d}t.\label{e_Rij-def}
\end{equation}
Thus, $\mathfrak{R}_{ij}<\infty$ if and only if there exists a triple
$(x,\,a,\,y)\in S_{\star i}^{2}\times S_{0}\times S_{\star j}^{2}$
such that $r(x,\,a)>0$ and $r(y,\,a)>0$. This is equivalent to saying
that $S_{\star i}^{2}$ and $S_{\star j}^{2}$ are exactly two graph
distance away from each other (cf. Figure \ref{fig2.3}). Thus, the
second limit process $Y(\cdot)$ is still not necessarily irreducible,
in that collections $S_{\star i}^{2}$, $i\in\llbracket1,\,\kappa_{2}\rrbracket$
with three or more graph distance away from each other are not connected.
This suggests yet another bigger time scale to emerge; \emph{this
is the main result of this article} stated in Section \ref{sec2.5}.
\item The subexponential decaying condition \eqref{e_dN-decay-subexp} of
$d_{N}$ is natural in the following sense. Consider the case where
$x,\,y\in S_{\star}$ and $a\in S_{0}$ with $r(x,\,y)=0$, $r(x,\,a)>0$,
and $r(y,\,a)>0$, such that $x$ and $y$ are not connected on the
first time scale but connected on the second time scale. Then, instead
of following the complicated mechanism of the second-time-scale transitions,
we may just simply move the condensate from $x$ to $y$ by moving
all $N$ particles at once from $x$ to $a$ and then again from $a$
to $y$, thereby following the simple mechanism of the first-time-scale
transitions. According to the one-dimensional estimates conducted
in \cite[Section 4.2]{KS NRIP} and \cite[Section 4]{Kim RIP-2nd},
the amount of time we have to wait to observe such transitions has
the scale of
\[
\frac{1}{d_{N}N}\cdot\Big(\frac{1}{m_{a}}\Big)^{N}=\frac{1}{d_{N}Nm_{a}^{N}}.
\]
To neglect these uninteresting types of metastable transitions from
our picture, the second time scale $\theta_{2}$ must be smaller than
this time scale; namely,
\[
\frac{N}{d_{N}^{2}}\ll\frac{1}{d_{N}Nm_{a}^{N}},\quad\text{which is equivalent to}\quad d_{N}\frac{1}{N^{2}}\Big(\frac{1}{m_{a}}\Big)^{N}\gg1.
\]
Requiring this to hold for all $m_{a}\in(0,\,1)$ is equivalent to
assigning the subexponential decaying condition \eqref{e_dN-decay-subexp}
to $d_{N}$.
\item Continuing the previous remark, given the subexponential decaying
condition of $d_{N}$, during the metastable transitions one has to
control the number of particles in $S_{0}$ to be properly bounded.
Then, the number of sites that are occupied during each particle movement
is $3$ (see Figure \ref{fig2.4}), which gives the exponent $3-1=2$
of $d_{N}$ in $\theta_{2}=N/d_{N}^{2}$.
\item Originally, in \cite[Theorem 2.10]{Kim RIP-2nd}, a rather suboptimal
condition $d_{N}N^{2}(\log N)^{2}\ll1$ was required. This technicality
is settled, and the results are valid under the minimal optimal condition
$d_{N}\log N\ll1$. Interested readers are referred to the latest
version in arXiv.
\end{enumerate}
\end{rem}

According to Remark \ref{r_2nd}-(1), the second limit process $Y(\cdot)$
is still not necessarily irreducible. Thus, we further decompose $S_{\star}$
into irreducible components with respect to $Y(\cdot)$:
\[
S_{\star}=\bigcup_{i=1}^{\kappa_{3}}S_{\star i}^{3}.
\]
For instance, in Figure \ref{fig2.3}, it holds that $\kappa_{3}=3$.
Then, the excursions are negligible also on the second time scale.
\begin{thm}[{2nd time scale, negligibility of excursions, \cite[Theorem 2.10-(3)]{Kim RIP-2nd}}]
\label{t_2nd-neg}Starting from a site $x\in S_{\star i}^{3}$ for
$i\in\llbracket1,\,\kappa_{3}\rrbracket$, on the second time scale
$\theta_{2}$, the original process spends negligible time outside
the collection $\mathcal{E}_{N}(S_{\star i}^{3})$:
\[
\lim_{N\to\infty}\mathbb{E}_{\xi^{x}}\Big[\int_{0}^{t}\mathbbm{1}\{\eta_{N}(\theta_{2}s)\notin\mathcal{E}_{N}(S_{\star i}^{3})\}\mathrm{d}s\Big]=0\text{ for all }t\ge0.
\]
\end{thm}

\subsection{\label{sec2.5}Main result: third time scale of metastability}

We are ready to state our main result. We define the third projection
operator $\Psi_{3}=\Psi_{N,\,3}:\mathcal{H}_{N}\to\llbracket1,\,\kappa_{3}\rrbracket\cup\{\mathfrak{0}\}$
as
\[
\Psi_{3}(\eta):=\begin{cases}
x & \text{if }\eta=\xi^{x}\text{ for some }x\in S_{\star i}^{3}\text{ and }i\in\llbracket1,\,\kappa_{3}\rrbracket,\\
\mathfrak{0} & \text{otherwise}.
\end{cases}
\]
Then, the third and \emph{final} order process $\{Z_{N}(t)\}_{t\ge0}$
is defined as $Z_{N}(t):=\Psi_{3}(\eta_{N}(t))$ for $t\ge0$. Refer
to Figure \ref{fig2.5} for an illustration of Theorem \ref{t_3rd}.

\begin{figure}
\begin{tikzpicture}
\fill[rounded corners,white] (-2.25,2.25) rectangle (2.25,-2.25); 
\fill[white] (-1,2) circle (0.1); \draw[black!30!white] (-1,2) circle (0.1);
\fill[white] (0,2) circle (0.1); \draw[black!30!white] (0,2) circle (0.1);
\fill[white] (-1,1) circle (0.1); \draw[black!30!white] (-1,1) circle (0.1);
\fill[white] (0,1) circle (0.1); \draw[black!30!white] (0,1) circle (0.1);
\fill[white] (1,1) circle (0.1); \draw[black!30!white] (1,1) circle (0.1);
\fill[white] (2,1) circle (0.1); \draw[black!30!white] (2,1) circle (0.1);
\fill[white] (-2,0) circle (0.1); \draw[black!30!white] (-2,0) circle (0.1);
\fill[white] (-1,0) circle (0.1); \draw[black!30!white] (-1,0) circle (0.1);
\fill[white] (0,0) circle (0.1); \draw[black!30!white] (0,0) circle (0.1);
\fill[white] (2,0) circle (0.1); \draw[black!30!white] (2,0) circle (0.1);
\fill[white] (0,-1) circle (0.1); \draw[black!30!white] (0,-1) circle (0.1);
\fill[white] (1,-1) circle (0.1); \draw[black!30!white] (1,-1) circle (0.1);
\fill[white] (2,-1) circle (0.1); \draw[black!30!white] (2,-1) circle (0.1);
\fill[white] (0,-2) circle (0.1); \draw[black!30!white] (0,-2) circle (0.1);
\fill[white] (1,-2) circle (0.1); \draw[black!30!white] (1,-2) circle (0.1);
\draw[very thick,to-,black!10!white] (-1.9,2)--(-1.1,2); \draw[very thick,black!10!white] (-0.9,2)--(-0.1,2); \draw[very thick,-to,black!10!white] (0.1,2)--(0.9,2);
\draw[very thick,to-,black!10!white] (1,1.9)--(1,1.1); \draw[very thick, to-,black!10!white] (2,1.9)--(2,1.1);
\draw[very thick,to-,black!10!white] (-1.9,1)--(-1.1,1); \draw[very thick,black!10!white] (0.1,1)--(0.9,1); \draw[very thick,black!10!white] (1.1,1)--(1.9,1);
\draw[very thick,to-,black!10!white] (-2,0.9)--(-2,0.1); \draw[very thick,-to,black!10!white] (-1,0.9)--(-1,0.1); \draw[very thick,-to,black!10!white] (0,0.9)--(0,0.1); \draw[very thick,-to,black!10!white] (1,0.9)--(1,0.1);
\draw[very thick,black!10!white] (-1.9,0)--(-1.1,0); \draw[very thick,-to,black!10!white] (-0.9,0)--(-0.1,0); \draw[very thick,to-,black!10!white] (1.1,0)--(1.9,0);
\draw[very thick,-to,black!10!white] (-2,-0.1)--(-2,-0.9); \draw[very thick,-to,black!10!white] (-1,-0.1)--(-1,-0.9); \draw[very thick,-to,black!10!white] (0,-0.1)--(0,-0.9); \draw[very thick,-to,black!10!white] (2,-0.1)--(2,-0.9);
\draw[very thick,to-,black!10!white] (-0.9,-1)--(-0.1,-1); \draw[very thick,-to,black!10!white] (0.1,-1)--(0.9,-1);
\draw[very thick,-to,black!10!white] (2,-1.1)--(2,-1.9);
\draw[very thick,to-,black!10!white] (-0.9,-2)--(-0.1,-2); \draw[very thick,to-,black!10!white] (0.1,-2)--(0.9,-2); \draw[very thick,-to,black!10!white] (1.1,-2)--(1.9,-2);
\fill[rounded corners,teal!50!white] (-2.25,2.25) rectangle (-1.75,-2.25);
\fill[rounded corners,teal!50!white] (0.75,2.25) rectangle (2.25,1.75);
\fill[rounded corners,teal!50!white] (0.75,2.25) rectangle (1.25,-0.25);
\fill[rounded corners,teal!50!white] (-2.25,-0.75) rectangle (-0.75,-2.25);
\fill[rounded corners,teal!50!white] (1.75,-1.75) rectangle (2.25,-2.25);
\fill[black] (-2,2) circle (0.1); \draw (-2,2) circle (0.1);
\fill[black] (1,2) circle (0.1); \draw (1,2) circle (0.1);
\fill[black] (2,2) circle (0.1); \draw (2,2) circle (0.1);
\fill[black] (-2,1) circle (0.1); \draw (-2,1) circle (0.1);
\fill[black] (1,0) circle (0.1); \draw (1,0) circle (0.1);
\fill[black] (-2,-1) circle (0.1); \draw (-2,-1) circle (0.1);
\fill[black] (-1,-1) circle (0.1); \draw (-1,-1) circle (0.1);
\fill[black] (-2,-2) circle (0.1); \draw (-2,-2) circle (0.1);
\fill[black] (-1,-2) circle (0.1); \draw (-1,-2) circle (0.1);
\fill[black] (2,-2) circle (0.1); \draw (2,-2) circle (0.1);
\draw[very thick] (1.1,2)--(1.9,2);
\draw[very thick] (-2,1.9)--(-2,1.1);
\draw[very thick] (-1.9,-1)--(-1.1,-1);
\draw[very thick] (-2,-1.1)--(-2,-1.9); \draw[very thick] (-1,-1.1)--(-1,-1.9);
\draw[very thick] (-1.9,-2)--(-1.1,-2);
\draw[very thick] (1,1.9)--(1,0.1);
\draw[very thick] (-2,0.9)--(-2,-0.9);
\draw[line width=1mm,latex-latex] (-1.75,2)--(0.75,2);
\draw[line width=1mm,latex-latex] (-0.75,-2)--(1.75,-2);
\draw[line width=1mm,latex-latex] (1.25,0)--(2,0)--(2,-1.75);
\end{tikzpicture}\caption{\label{fig2.5}Third level of metastable transitions between the collections
$S_{\star i}^{3}$ for $i\in\llbracket1,\,\kappa_{3}\rrbracket$.
The third limit process $Z(\cdot)$ is irreducible.}
\end{figure}
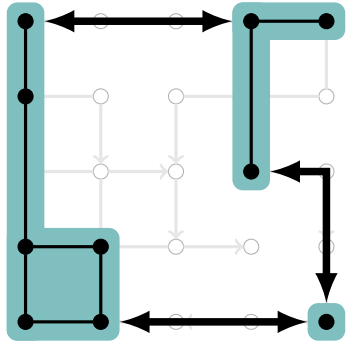

\begin{thm}[Main result: 3rd time scale of metastability]
\label{t_3rd}Define $\theta_{3}=\theta_{N,\,3}:=N^{2}/d_{N}^{3}$
and fix a starting position $i\in\llbracket1,\,\kappa_{3}\rrbracket$.
\begin{enumerate}
\item The collection $\mathcal{E}_{N}(S_{\star i}^{3})$ thermalizes before
the metastable transition:
\[
\lim_{N\to\infty}\inf_{x,\,y\in S_{\star i}^{3}}\mathbb{P}_{\xi^{x}}[\mathcal{T}_{\xi^{y}}<\mathcal{T}_{\mathcal{E}_{N}(S_{\star}\setminus S_{\star i}^{3})}]=1.
\]
\item The law of the accelerated third order process $\{Z_{N}(\theta_{3}t)\}_{t\ge0}$
converges, as $N\to\infty$, to the law of an \textbf{irreducible}
Markov chain $\{Z(t)\}_{t\ge0}$ on $\llbracket1,\,\kappa_{3}\rrbracket$
which is determined by a transition rate function defined as $r^{\textup{3rd}}(i,\,i):=0$
and
\[
r^{\textup{3rd}}(i,\,j):=\frac{1}{|S_{\star i}^{3}|\cdot\mathfrak{K}_{ij}}\text{ for }i,\,j\in\llbracket1,\,\kappa_{3}\rrbracket.
\]
Here, $\mathfrak{K}_{ij}\in(0,\,\infty]$ is an explicit (possibly
infinite) constant defined in \eqref{e_Kxy-def}.
\item For every $x\in S_{\star i}^{3}$, the Eyring--Kramers formula is
given as
\[
\mathbb{E}_{\xi^{x}}[\mathcal{T}_{\mathcal{E}_{N}(S_{\star}\setminus S_{\star i}^{3})}]\simeq\frac{1}{\sum_{j\in\llbracket1,\,\kappa_{3}\rrbracket}r^{\textup{3rd}}(i,\,j)}\cdot\frac{N^{2}}{d_{N}^{3}}.
\]
\item On the third time scale $\theta_{3}$, excursions outside $\mathcal{E}_{N}(S_{\star})$
are negligible:
\[
\lim_{N\to\infty}\mathbb{E}_{\xi^{x}}\Big[\int_{0}^{t}\mathbbm{1}\{\eta_{N}(\theta_{3}s)\notin\mathcal{E}_{N}(S_{\star})\}\mathrm{d}s\Big]=0\text{ for all }t\ge0.
\]
\end{enumerate}
\end{thm}

\begin{figure}
\begin{tikzpicture}
\fill[rounded corners,white] (-1.2,3) rectangle (1.2,-0.8); 
\fill[black!30!white] (-0.8,0.1) circle (0.1); \fill[black!30!white] (-0.8,0.3) circle (0.1); \fill[black!30!white] (-0.8,0.5) circle (0.1); \fill[black!30!white] (-0.8,0.7) circle (0.1); \fill[black!30!white] (-0.8,0.9) circle (0.1); \fill[black!30!white] (-0.8,1.1) circle (0.1); \fill[black!30!white] (-0.8,1.3) circle (0.1); \fill[black!30!white] (-0.8,1.5) circle (0.1); \fill[black!30!white] (-0.8,1.7) circle (0.1); \fill[black!30!white] (-0.8,1.9) circle (0.1); \fill[black!30!white] (-0.8,2.1) circle (0.1); \fill[black!30!white] (-0.8,2.3) circle (0.1); \fill[black!30!white] (-0.8,2.5) circle (0.1); \fill[black!30!white] (-0.8,2.7) circle (0.1); \fill[red] (-0.8,2.9) circle (0.1);
\fill[black!30!white] (0.8,0.1) circle (0.1); \fill[black!30!white] (0.8,0.3) circle (0.1); \fill[black!30!white] (0.8,0.5) circle (0.1); \fill[black!30!white] (0.8,0.7) circle (0.1); \fill[black!30!white] (0.8,0.9) circle (0.1); \fill[black!30!white] (0.8,1.1) circle (0.1); \fill[black!30!white] (0.8,1.3) circle (0.1); \fill[black!30!white] (0.8,1.5) circle (0.1); \fill[black!30!white] (0.8,1.7) circle (0.1); \fill[black!30!white] (0.8,1.9) circle (0.1);
\draw (-1.2,0)--(1.2,0);
\draw (-0.8,0.1) circle (0.1); \draw (-0.8,0.3) circle (0.1); \draw (-0.8,0.5) circle (0.1); \draw (-0.8,0.7) circle (0.1); \draw (-0.8,0.9) circle (0.1); \draw (-0.8,1.1) circle (0.1); \draw (-0.8,1.3) circle (0.1); \draw (-0.8,1.5) circle (0.1); \draw (-0.8,1.7) circle (0.1); \draw (-0.8,1.9) circle (0.1); \draw (-0.8,2.1) circle (0.1); \draw (-0.8,2.3) circle (0.1); \draw (-0.8,2.5) circle (0.1); \draw (-0.8,2.7) circle (0.1); \draw (-0.8,2.9) circle (0.1);
\draw (0.8,0.1) circle (0.1); \draw (0.8,0.3) circle (0.1); \draw (0.8,0.5) circle (0.1); \draw (0.8,0.7) circle (0.1); \draw (0.8,0.9) circle (0.1); \draw (0.8,1.1) circle (0.1); \draw (0.8,1.3) circle (0.1); \draw (0.8,1.5) circle (0.1); \draw (0.8,1.7) circle (0.1); \draw (0.8,1.9) circle (0.1);
\draw (0,-0.2) node[below]{$x$\hspace{0.22cm}$a$\hspace{0.22cm}$b$\hspace{0.22cm}$c$\hspace{0.22cm}$y$};
\end{tikzpicture}
\begin{tikzpicture}
\fill[rounded corners,white] (-0.25,3) rectangle (0.25,-0.8); 
\draw[ultra thick,-to] (-0.15,1.5)->(0.15,1.5);
\end{tikzpicture}
\begin{tikzpicture}
\fill[rounded corners,white] (-1.2,3) rectangle (1.2,-0.8); 
\fill[black!30!white] (-0.8,0.1) circle (0.1); \fill[black!30!white] (-0.8,0.3) circle (0.1); \fill[black!30!white] (-0.8,0.5) circle (0.1); \fill[black!30!white] (-0.8,0.7) circle (0.1); \fill[black!30!white] (-0.8,0.9) circle (0.1); \fill[black!30!white] (-0.8,1.1) circle (0.1); \fill[black!30!white] (-0.8,1.3) circle (0.1); \fill[black!30!white] (-0.8,1.5) circle (0.1); \fill[black!30!white] (-0.8,1.7) circle (0.1); \fill[black!30!white] (-0.8,1.9) circle (0.1); \fill[black!30!white] (-0.8,2.1) circle (0.1); \fill[black!30!white] (-0.8,2.3) circle (0.1); \fill[black!30!white] (-0.8,2.5) circle (0.1); \fill[black!30!white] (-0.8,2.7) circle (0.1);
\fill[red] (-0.4,0.1) circle (0.1);
\fill[black!30!white] (0.8,0.1) circle (0.1); \fill[black!30!white] (0.8,0.3) circle (0.1); \fill[black!30!white] (0.8,0.5) circle (0.1); \fill[black!30!white] (0.8,0.7) circle (0.1); \fill[black!30!white] (0.8,0.9) circle (0.1); \fill[black!30!white] (0.8,1.1) circle (0.1); \fill[black!30!white] (0.8,1.3) circle (0.1); \fill[black!30!white] (0.8,1.5) circle (0.1); \fill[black!30!white] (0.8,1.7) circle (0.1); \fill[black!30!white] (0.8,1.9) circle (0.1);
\draw (-1.2,0)--(1.2,0);
\draw (-0.8,0.1) circle (0.1); \draw (-0.8,0.3) circle (0.1); \draw (-0.8,0.5) circle (0.1); \draw (-0.8,0.7) circle (0.1); \draw (-0.8,0.9) circle (0.1); \draw (-0.8,1.1) circle (0.1); \draw (-0.8,1.3) circle (0.1); \draw (-0.8,1.5) circle (0.1); \draw (-0.8,1.7) circle (0.1); \draw (-0.8,1.9) circle (0.1); \draw (-0.8,2.1) circle (0.1); \draw (-0.8,2.3) circle (0.1); \draw (-0.8,2.5) circle (0.1); \draw (-0.8,2.7) circle (0.1);
\draw (-0.4,0.1) circle (0.1);
\draw (0.8,0.1) circle (0.1); \draw (0.8,0.3) circle (0.1); \draw (0.8,0.5) circle (0.1); \draw (0.8,0.7) circle (0.1); \draw (0.8,0.9) circle (0.1); \draw (0.8,1.1) circle (0.1); \draw (0.8,1.3) circle (0.1); \draw (0.8,1.5) circle (0.1); \draw (0.8,1.7) circle (0.1); \draw (0.8,1.9) circle (0.1);
\draw (0,-0.2) node[below]{$x$\hspace{0.22cm}$a$\hspace{0.22cm}$b$\hspace{0.22cm}$c$\hspace{0.22cm}$y$};
\end{tikzpicture}
\begin{tikzpicture}
\fill[rounded corners,white] (-0.25,3) rectangle (0.25,-0.8); 
\draw[ultra thick,-to] (-0.15,1.5)->(0.15,1.5);
\end{tikzpicture}
\begin{tikzpicture}
\fill[rounded corners,white] (-1.2,3) rectangle (1.2,-0.8); 
\fill[black!30!white] (-0.8,0.1) circle (0.1); \fill[black!30!white] (-0.8,0.3) circle (0.1); \fill[black!30!white] (-0.8,0.5) circle (0.1); \fill[black!30!white] (-0.8,0.7) circle (0.1); \fill[black!30!white] (-0.8,0.9) circle (0.1); \fill[black!30!white] (-0.8,1.1) circle (0.1); \fill[black!30!white] (-0.8,1.3) circle (0.1); \fill[black!30!white] (-0.8,1.5) circle (0.1); \fill[black!30!white] (-0.8,1.7) circle (0.1); \fill[black!30!white] (-0.8,1.9) circle (0.1); \fill[black!30!white] (-0.8,2.1) circle (0.1); \fill[black!30!white] (-0.8,2.3) circle (0.1); \fill[black!30!white] (-0.8,2.5) circle (0.1); \fill[black!30!white] (-0.8,2.7) circle (0.1);
\fill[red] (0,0.1) circle (0.1);
\fill[black!30!white] (0.8,0.1) circle (0.1); \fill[black!30!white] (0.8,0.3) circle (0.1); \fill[black!30!white] (0.8,0.5) circle (0.1); \fill[black!30!white] (0.8,0.7) circle (0.1); \fill[black!30!white] (0.8,0.9) circle (0.1); \fill[black!30!white] (0.8,1.1) circle (0.1); \fill[black!30!white] (0.8,1.3) circle (0.1); \fill[black!30!white] (0.8,1.5) circle (0.1); \fill[black!30!white] (0.8,1.7) circle (0.1); \fill[black!30!white] (0.8,1.9) circle (0.1);
\draw (-1.2,0)--(1.2,0);
\draw (-0.8,0.1) circle (0.1); \draw (-0.8,0.3) circle (0.1); \draw (-0.8,0.5) circle (0.1); \draw (-0.8,0.7) circle (0.1); \draw (-0.8,0.9) circle (0.1); \draw (-0.8,1.1) circle (0.1); \draw (-0.8,1.3) circle (0.1); \draw (-0.8,1.5) circle (0.1); \draw (-0.8,1.7) circle (0.1); \draw (-0.8,1.9) circle (0.1); \draw (-0.8,2.1) circle (0.1); \draw (-0.8,2.3) circle (0.1); \draw (-0.8,2.5) circle (0.1); \draw (-0.8,2.7) circle (0.1);
\draw (0,0.1) circle (0.1);
\draw (0.8,0.1) circle (0.1); \draw (0.8,0.3) circle (0.1); \draw (0.8,0.5) circle (0.1); \draw (0.8,0.7) circle (0.1); \draw (0.8,0.9) circle (0.1); \draw (0.8,1.1) circle (0.1); \draw (0.8,1.3) circle (0.1); \draw (0.8,1.5) circle (0.1); \draw (0.8,1.7) circle (0.1); \draw (0.8,1.9) circle (0.1);
\draw (0,-0.2) node[below]{$x$\hspace{0.22cm}$a$\hspace{0.22cm}$b$\hspace{0.22cm}$c$\hspace{0.22cm}$y$};
\end{tikzpicture}
\begin{tikzpicture}
\fill[rounded corners,white] (-0.25,3) rectangle (0.25,-0.8); 
\draw[ultra thick,-to] (-0.15,1.5)->(0.15,1.5);
\end{tikzpicture}
\begin{tikzpicture}
\fill[rounded corners,white] (-1.2,3) rectangle (1.2,-0.8); 
\fill[black!30!white] (-0.8,0.1) circle (0.1); \fill[black!30!white] (-0.8,0.3) circle (0.1); \fill[black!30!white] (-0.8,0.5) circle (0.1); \fill[black!30!white] (-0.8,0.7) circle (0.1); \fill[black!30!white] (-0.8,0.9) circle (0.1); \fill[black!30!white] (-0.8,1.1) circle (0.1); \fill[black!30!white] (-0.8,1.3) circle (0.1); \fill[black!30!white] (-0.8,1.5) circle (0.1); \fill[black!30!white] (-0.8,1.7) circle (0.1); \fill[black!30!white] (-0.8,1.9) circle (0.1); \fill[black!30!white] (-0.8,2.1) circle (0.1); \fill[black!30!white] (-0.8,2.3) circle (0.1); \fill[black!30!white] (-0.8,2.5) circle (0.1); \fill[black!30!white] (-0.8,2.7) circle (0.1);
\fill[red] (0.4,0.1) circle (0.1);
\fill[black!30!white] (0.8,0.1) circle (0.1); \fill[black!30!white] (0.8,0.3) circle (0.1); \fill[black!30!white] (0.8,0.5) circle (0.1); \fill[black!30!white] (0.8,0.7) circle (0.1); \fill[black!30!white] (0.8,0.9) circle (0.1); \fill[black!30!white] (0.8,1.1) circle (0.1); \fill[black!30!white] (0.8,1.3) circle (0.1); \fill[black!30!white] (0.8,1.5) circle (0.1); \fill[black!30!white] (0.8,1.7) circle (0.1); \fill[black!30!white] (0.8,1.9) circle (0.1);
\draw (-1.2,0)--(1.2,0);
\draw (-0.8,0.1) circle (0.1); \draw (-0.8,0.3) circle (0.1); \draw (-0.8,0.5) circle (0.1); \draw (-0.8,0.7) circle (0.1); \draw (-0.8,0.9) circle (0.1); \draw (-0.8,1.1) circle (0.1); \draw (-0.8,1.3) circle (0.1); \draw (-0.8,1.5) circle (0.1); \draw (-0.8,1.7) circle (0.1); \draw (-0.8,1.9) circle (0.1); \draw (-0.8,2.1) circle (0.1); \draw (-0.8,2.3) circle (0.1); \draw (-0.8,2.5) circle (0.1); \draw (-0.8,2.7) circle (0.1);
\draw (0.4,0.1) circle (0.1);
\draw (0.8,0.1) circle (0.1); \draw (0.8,0.3) circle (0.1); \draw (0.8,0.5) circle (0.1); \draw (0.8,0.7) circle (0.1); \draw (0.8,0.9) circle (0.1); \draw (0.8,1.1) circle (0.1); \draw (0.8,1.3) circle (0.1); \draw (0.8,1.5) circle (0.1); \draw (0.8,1.7) circle (0.1); \draw (0.8,1.9) circle (0.1);
\draw (0,-0.2) node[below]{$x$\hspace{0.22cm}$a$\hspace{0.22cm}$b$\hspace{0.22cm}$c$\hspace{0.22cm}$y$};
\end{tikzpicture}
\begin{tikzpicture}
\fill[rounded corners,white] (-0.25,3) rectangle (0.25,-0.8); 
\draw[ultra thick,-to] (-0.15,1.5)->(0.15,1.5);
\end{tikzpicture}
\begin{tikzpicture}
\fill[rounded corners,white] (-1.2,3) rectangle (1.2,-0.8); 
\fill[black!30!white] (-0.8,0.1) circle (0.1); \fill[black!30!white] (-0.8,0.3) circle (0.1); \fill[black!30!white] (-0.8,0.5) circle (0.1); \fill[black!30!white] (-0.8,0.7) circle (0.1); \fill[black!30!white] (-0.8,0.9) circle (0.1); \fill[black!30!white] (-0.8,1.1) circle (0.1); \fill[black!30!white] (-0.8,1.3) circle (0.1); \fill[black!30!white] (-0.8,1.5) circle (0.1); \fill[black!30!white] (-0.8,1.7) circle (0.1); \fill[black!30!white] (-0.8,1.9) circle (0.1); \fill[black!30!white] (-0.8,2.1) circle (0.1); \fill[black!30!white] (-0.8,2.3) circle (0.1); \fill[black!30!white] (-0.8,2.5) circle (0.1); \fill[black!30!white] (-0.8,2.7) circle (0.1);
\fill[black!30!white] (0.8,0.1) circle (0.1); \fill[black!30!white] (0.8,0.3) circle (0.1); \fill[black!30!white] (0.8,0.5) circle (0.1); \fill[black!30!white] (0.8,0.7) circle (0.1); \fill[black!30!white] (0.8,0.9) circle (0.1); \fill[black!30!white] (0.8,1.1) circle (0.1); \fill[black!30!white] (0.8,1.3) circle (0.1); \fill[black!30!white] (0.8,1.5) circle (0.1); \fill[black!30!white] (0.8,1.7) circle (0.1); \fill[black!30!white] (0.8,1.9) circle (0.1); \fill[red] (0.8,2.1) circle (0.1);
\draw (-1.2,0)--(1.2,0);
\draw (-0.8,0.1) circle (0.1); \draw (-0.8,0.3) circle (0.1); \draw (-0.8,0.5) circle (0.1); \draw (-0.8,0.7) circle (0.1); \draw (-0.8,0.9) circle (0.1); \draw (-0.8,1.1) circle (0.1); \draw (-0.8,1.3) circle (0.1); \draw (-0.8,1.5) circle (0.1); \draw (-0.8,1.7) circle (0.1); \draw (-0.8,1.9) circle (0.1); \draw (-0.8,2.1) circle (0.1); \draw (-0.8,2.3) circle (0.1); \draw (-0.8,2.5) circle (0.1); \draw (-0.8,2.7) circle (0.1);
\draw (0.8,0.1) circle (0.1); \draw (0.8,0.3) circle (0.1); \draw (0.8,0.5) circle (0.1); \draw (0.8,0.7) circle (0.1); \draw (0.8,0.9) circle (0.1); \draw (0.8,1.1) circle (0.1); \draw (0.8,1.3) circle (0.1); \draw (0.8,1.5) circle (0.1); \draw (0.8,1.7) circle (0.1); \draw (0.8,1.9) circle (0.1); \draw (0.8,2.1) circle (0.1);
\draw (0,-0.2) node[below]{$x$\hspace{0.22cm}$a$\hspace{0.22cm}$b$\hspace{0.22cm}$c$\hspace{0.22cm}$y$};
\end{tikzpicture}\caption{\label{fig2.6}Typical movements of a particle during the metastable
transition from $\xi^{x}$ to $\xi^{y}$ on the third time scale,
where $x,\,y\in S_{\star}$ and $a,\,b,\,c\in S_{0}$. In each movement,
there are at most $4$ sites occupied before and after. This property
continues to hold given any graph distance between $x$ and $y$.}
\end{figure}
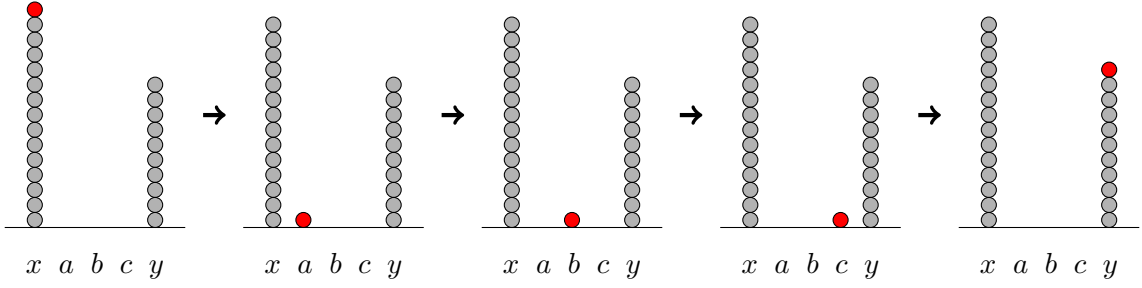

\begin{rem}
\label{r_3rd}We present some remarks regarding the third time scale.
\begin{enumerate}
\item As stated in Theorem \ref{t_3rd}-(2), the third scaling limit $Z(\cdot)$
is irreducible on $\llbracket1,\,\kappa_{3}\rrbracket$, which means
that we can observe all types of remaining metastable transitions
on this scale $\theta_{3}$. Thus, this is the last time scale. 
\item After inspecting the second time scale of metastability (in particular
Remark \ref{r_2nd}-(3)), one would have naturally predicted that
there exists a time scale for each value of graph distance between
the connected components. This is not the case; on the third time
scale $\theta_{3}$, all components with graph distance bigger than
$2$ become connected. This is due to the specific mechanism of particle
movements on the third time scale, illustrated in Figure \ref{fig2.6},
which is that given any graph distance (bigger than $2$) between
$x,\,y\in S_{\star}$, the number of occupied sites during a particle
movement is always at most $4$, which gives the exponent $4-1=3$
of $d_{N}$ in $\theta_{3}=N^{2}/d_{N}^{3}$.
\item Note that the metastable decomposition of $S_{\star}$ becomes coarser
from the first time scale to the third time scale:
\[
\bigcup_{x\in S_{\star}}\{x\}\longrightarrow\bigcup_{i\in\llbracket1,\,\kappa_{2}\rrbracket}S_{\star i}^{2}\longrightarrow\bigcup_{j\in\llbracket1,\,\kappa_{3}\rrbracket}S_{\star j}^{3}.
\]
On each time scale, the irreducible components merge with each other
and become a single collection on the next time scale.
\item Continuing the previous remark, this hierarchical structure of metastability
is closely related to the so-called \emph{$\Gamma$-expansion approach}
to metastability \cite{BGL,Landim Gamma-exp,LanMisSau}. Briefly,
the idea is as follows. Denote by $\mathcal{I}_{N}$ the \emph{level-two
large deviation rate functional} \cite{DV} of the inclusion process
$\eta_{N}(\cdot)$. Moreover, denote by $p_{N}(t)$ the empirical
measure of the process: $p_{N}(t):=\frac{1}{t}\int_{0}^{t}\delta_{\eta_{N}(s)}\mathrm{d}s$.
Then, the following large deviation principle holds:
\[
\begin{aligned}-\inf_{\nu\in A^{o}}\mathcal{I}_{N}(\nu) & \le\liminf_{t\to\infty}\min_{\eta\in\mathcal{H}_{N}}\frac{1}{t}\log\mathbb{P}_{\eta}[p_{N}(t)\in A]\\
 & \le\limsup_{t\to\infty}\max_{\eta\in\mathcal{H}_{N}}\frac{1}{t}\log\mathbb{P}_{\eta}[p_{N}(t)\in A]\le-\inf_{\nu\in\overline{A}}\mathcal{I}_{N}(\nu),
\end{aligned}
\]
where $A$ is any subset of $\mathcal{P}(\mathcal{H}_{N})$, the space
of probability measures on $\mathcal{H}_{N}$, and $A^{o}$ (resp.
$\overline{A}$) denotes the interior (resp. closure) of $A$ with
respect to the weak topology. Then, it is expected that the following
$\Gamma$-expansion of $\mathcal{I}_{N}$ holds:
\[
\mathcal{I}_{N}=\mathcal{I}^{(0)}+\frac{1}{\theta_{1}}\mathcal{I}^{(1)}+\frac{1}{\theta_{2}}\mathcal{I}^{(2)}+\frac{1}{\theta_{3}}\mathcal{I}^{(3)},
\]
which means that $\mathcal{I}_{N}\xrightarrow{N\to\infty}\mathcal{I}^{(0)}$
and $\theta_{p}\mathcal{I}_{N}\xrightarrow{N\to\infty}\mathcal{I}^{(p)}$
for $p\in\{1,\,2,\,3\}$, in the sense of $\Gamma$-convergence, where
$\mathcal{I}^{(p)}$ is the corresponding level-two large deviation
rate functional of the limit process on each $p$-time scale. An ongoing
research attempts to link this general theory of $\Gamma$-expansion
approach to the hierarchical structure of metastability occurring
in the reversible inclusion process.
\end{enumerate}
\end{rem}

\subsection{\label{sec2.6}Other time scales}

In this final subsection of Section \ref{sec2}, we deal with the
remaining time scales other than $\theta_{1}$, $\theta_{2}$, and
$\theta_{3}$, thereby completing the full analysis of the hierarchy
of metastable behavior.
\begin{thm}[Other time scales]
\label{t_other-scales}Fix a sequence $\{\alpha_{N}\}_{N\ge1}$ of
positive real numbers and consider a starting location $x\in S_{\star}$
such that $\{x\}\subseteq S_{\star i}^{2}\subseteq S_{\star j}^{3}$
for $i\in\llbracket1,\,\kappa_{2}\rrbracket$ and $j\in\llbracket1,\,\kappa_{3}\rrbracket$.
\begin{enumerate}
\item If $\alpha_{N}\ll\theta_{1}$, then $\lim_{N\to\infty}\mathbb{P}_{\xi^{x}}[\eta_{N}(\alpha_{N})=\xi^{x}]=1$.
\item If $\theta_{1}\ll\alpha_{N}\ll\theta_{2}$, then $\lim_{N\to\infty}\mathbb{P}_{\xi^{x}}[\eta_{N}(\alpha_{N})=\xi^{y}]=|S_{\star i}^{2}|^{-1}$
for all $y\in S_{\star i}^{2}$.
\item If $\theta_{2}\ll\alpha_{N}\ll\theta_{3}$, then $\lim_{N\to\infty}\mathbb{P}_{\xi^{x}}[\eta_{N}(\alpha_{N})=\xi^{y}]=|S_{\star j}^{3}|^{-1}$
for all $y\in S_{\star j}^{3}$.
\item If $\alpha_{N}\gg\theta_{3}$, then $\lim_{N\to\infty}\mathbb{P}_{\xi^{x}}[\eta_{N}(\alpha_{N})=\xi^{y}]=|S_{\star}|^{-1}$
for all $y\in S_{\star}$.
\end{enumerate}
\end{thm}

According to Theorem \ref{t_other-scales}, on each intermediate time
scale, the law of the inclusion process is asymptotically distributed
as the stationary distribution of the limit process on the previous
time scale, which is the uniform distribution restricted to the irreducible
component containing the initial location. In this sense, we may conclude
that $\theta_{1}$, $\theta_{2}$, and $\theta_{3}$ are the three
\emph{effective} time scales, and all the other time scales are irrelevant
to the metastable transitions.

\section{\label{sec3}Outline of Proof}

In this section, we present an outline of proofs of the results in
Section \ref{sec2}. In particular, we reduce the task of proving
the main results to verifying two capacity estimates to be presented
in Theorems \ref{t_Capacity} and \ref{t_Capacity-specific}.

Denote by $h_{\cdot,\,\cdot}$ and $\mathrm{Cap}_{N}(\cdot,\,\cdot)$
the equilibrium potential and capacity (cf. Appendix \ref{appenA}),
respectively, subject to the inclusion process $\eta_{N}(\cdot)$.
First, we provide precise estimates of the relevant capacities which
will be proved in the remainder of the article from Section \ref{sec4}.
\begin{thm}
\label{t_Capacity}For every non-trivial partition $A\cup B=\llbracket1,\,\kappa_{3}\rrbracket$,
it holds that
\[
\lim_{N\to\infty}\frac{N^{2}}{d_{N}^{3}}\cdot\mathrm{Cap}_{N}(\mathcal{E}_{N}(A),\,\mathcal{E}_{N}(B))=\frac{1}{|S_{\star}|}\cdot\sum_{i\in A}\sum_{j\in B}\frac{1}{\mathfrak{K}_{ij}},
\]
where $\mathfrak{K}_{ij}$ is the constant that appears in Theorem
\ref{t_3rd}-(2).
\end{thm}

In the proof of the Eyring--Kramers formula in Theorem \ref{t_3rd}-(3),
we also need the following specific capacity estimate.
\begin{thm}
\label{t_Capacity-specific}For $i\in\llbracket1,\,\kappa_{3}\rrbracket$
and $x\in S_{\star i}^{3}$, it holds that
\[
\lim_{N\to\infty}\frac{N^{2}}{d_{N}^{3}}\cdot\mathrm{Cap}_{N}(\xi^{x},\,\mathcal{E}_{N}(S_{\star}\setminus S_{\star i}^{3}))=\frac{1}{|S_{\star}|}\cdot\sum_{j\in\llbracket1,\,\kappa_{3}\rrbracket:\,j\ne i}\frac{1}{\mathfrak{K}_{ij}}.
\]
\end{thm}

We assume that Theorems \ref{t_Capacity} and \ref{t_Capacity-specific}
hold and prove the main results: Theorems \ref{t_1st-neg}, \ref{t_3rd},
and \ref{t_other-scales}. The following asymptotics will be frequently
used hereafter. For $k\in\llbracket0,\,N-1\rrbracket$,
\begin{equation}
1\le\frac{(d_{N}+k)w_{N}(k)}{d_{N}}=\frac{(k+1)w_{N}(k+1)}{d_{N}}\le e^{d_{N}\log N}\quad\text{and}\quad\lim_{N\to\infty}\frac{NZ_{N}}{d_{N}}=|S_{\star}|.\label{e_wN-ZN}
\end{equation}
In particular, $(d_{N}+k)w_{N}(k)=(k+1)w_{N}(k+1)\simeq d_{N}$ uniformly.
We refer the readers to \cite[Section 3]{BDG} for a detailed proof
of these estimates.
\begin{proof}[Proof of Theorem \ref{t_1st-neg}]
 Suppose that $x\in S_{\star i}^{2}$ for some $i\in\llbracket1,\,\kappa_{2}\rrbracket$.
By Theorem \ref{t_1st}-(3), it suffices to prove that
\begin{equation}
\lim_{N\to\infty}\mathbb{E}_{\xi^{x}}\Big[\int_{0}^{t}\mathbbm{1}\{\eta_{N}(\theta_{1}s)\in\mathcal{E}_{N}(S_{\star}\setminus S_{\star i}^{2})\}\mathrm{d}s\Big]=0\text{ for all }t\ge0.\label{e_t_1st-neg}
\end{equation}
According to Theorem \ref{t_1st}-(1), for all $s\ge0$ it holds that
\[
\lim_{N\to\infty}\mathbb{P}_{\xi^{x}}[\eta_{N}(\theta_{1}s)\in\mathcal{E}_{N}(S_{\star}\setminus S_{\star i}^{2})]=\mathrm{P}_{x}[X(s)\in S_{\star}\setminus S_{\star i}^{2}]=0,
\]
where $\mathrm{P}_{x}$ denotes the law of $X(\cdot)$ starting from
$x$. The second equality holds since $S_{\star i}^{2}$ is the irreducible
component containing $x$. This result implies \eqref{e_t_1st-neg}
via the dominated convergence theorem.
\end{proof}
\begin{proof}[Proof of Theorem \ref{t_3rd}]
 We denote by $\{\eta_{N}^{\star}(t)\}_{t\ge0}$ and $r_{N}^{\star}(\cdot,\,\cdot)$
the trace process on $\mathcal{E}_{N}(S_{\star})$ and its transition
rate function, respectively, as explained in Appendix \ref{appenA}.
By \cite[Theorem 2.7]{BL TM} and \cite[Proposition 2.1]{LanLouMou},
parts (1) and (2) of Theorem \ref{t_3rd} hold if the following four
conditions are verified:
\begin{itemize}
\item \textbf{(C1)} For $i,\,j\in\llbracket1,\,\kappa_{3}\rrbracket$, it
holds that $\lim_{N\to\infty}\theta_{3}\cdot r_{N}^{\star}(\mathcal{E}_{N}(S_{\star i}^{3}),\,\mathcal{E}_{N}(S_{\star j}^{3}))=r^{\textup{3rd}}(i,\,j)$.
\item \textbf{(C2)} For each $i\in\llbracket1,\,\kappa_{3}\rrbracket$,
it holds that
\[
\lim_{N\to\infty}\frac{\mathrm{Cap}_{N}(\mathcal{E}_{N}(S_{\star i}^{3}),\,\mathcal{E}_{N}(S_{\star}\setminus S_{\star i}^{3}))}{\inf_{x,\,y\in S_{\star i}^{3}}\mathrm{Cap}_{N}(\xi^{x},\,\xi^{y})}=0.
\]
\item \textbf{(C3)} For each $x\in S_{\star}$, it holds that
\[
\lim_{N\to\infty}\mathbb{E}_{\xi^{x}}\Big[\int_{0}^{t}\mathbbm{1}\{\eta_{N}(\theta_{3}s)\notin\mathcal{E}_{N}(S_{\star})\}\mathrm{d}s\Big]=0\text{ for all }t\ge0.
\]
\item \textbf{(C4)} It holds that
\[
\lim_{\delta\to0}\limsup_{N\to\infty}\max_{x\in S_{\star}}\sup_{s\in[2\delta,\,3\delta]}\mathbb{P}_{\xi^{x}}[\eta_{N}(\theta_{3}s)\notin\mathcal{E}_{N}(S_{\star})]=0.
\]
\end{itemize}
Conditions \textbf{(C3)} and \textbf{(C4)} are straightforward from
the following estimate for all $t\ge0$:
\[
\mathbb{P}_{\xi^{x}}[\eta_{N}(t)\notin\mathcal{E}_{N}(S_{\star})]\le\frac{\mathbb{P}_{\mu_{N}}[\eta_{N}(t)\notin\mathcal{E}_{N}(S_{\star})]}{\mu_{N}(\xi^{x})}=\frac{\mu_{N}(\mathcal{H}_{N}\setminus\mathcal{E}_{N}(S_{\star}))}{\mu_{N}(\xi^{x})}=o(1),
\]
where $\mathbb{P}_{\mu_{N}}$ is the law of the process starting from
$\mu_{N}$. The last equality follows from Proposition \ref{p_condensation}.
In particular, Theorem \ref{t_3rd}-(4) is equivalent to \textbf{(C3)}
and thus verified.

For condition \textbf{(C1)}, we apply \eqref{e_cap-cap-cap}. Combining
with Theorem \ref{t_Capacity}, we obtain that
\[
\lim_{N\to\infty}\frac{N^{2}}{d_{N}^{3}}\cdot\mu_{N}(\mathcal{E}_{N}(S_{\star i}^{3}))r_{N}^{\star}(\mathcal{E}_{N}(S_{\star i}^{3}),\,\mathcal{E}_{N}(S_{\star j}^{3}))=\frac{1}{|S_{\star}|}\cdot\frac{1}{\mathfrak{K}_{ij}}.
\]
Applying Proposition \ref{p_condensation}, we obtain that
\[
\lim_{N\to\infty}\frac{N^{2}}{d_{N}^{3}}\cdot r_{N}^{\star}(\mathcal{E}_{N}(S_{\star i}^{3}),\,\mathcal{E}_{N}(S_{\star j}^{3}))=\frac{1}{|S_{\star i}^{3}|}\cdot\frac{1}{\mathfrak{K}_{ij}}=r^{\textup{3rd}}(i,\,j),
\]
which verifies \textbf{(C1)}.

To prove condition \textbf{(C2)}, we adopt the idea from \cite[Section 9]{Kim RIP-2nd}.
For the numerator, Theorem \ref{t_Capacity} implies that
\begin{equation}
\mathrm{Cap}_{N}(\mathcal{E}_{N}(S_{\star i}^{3}),\,\mathcal{E}_{N}(S_{\star}\setminus S_{\star i}^{3}))\le\frac{Cd_{N}^{3}}{N^{2}}.\label{e_t_3rd-1}
\end{equation}
For the denominator, fix $x,\,y\in S_{\star i}^{3}$. Since $x$ and
$y$ belong to the same collection $S_{\star i}^{3}$, it holds that
either $r(x,\,y)>0$, or there exists $a\in S_{0}$ such that $r(x,\,a)>0$
and $r(y,\,a)>0$. We divide into these two cases.
\begin{itemize}
\item If $r(x,\,y)>0$, we construct a sequence $\xi^{x}=\omega_{0}\sim\omega_{1}\sim\cdots\sim\omega_{N}=\xi^{y}$
such that for $n\in\llbracket0,\,N\rrbracket$,
\[
\omega_{n}(x)=N-n\quad\text{and}\quad\omega_{n}(y)=n.
\]
Namely, we send each particle from $x$ to $y$ consecutively. Denote
by $\varphi$ the flow from $\xi^{x}$ to $\xi^{y}$ of value $1$
(see \eqref{e_flow-A-B-def}) along this path:
\[
\varphi(\omega_{n},\,\omega_{n+1}):=1\text{ for each }n\in\llbracket0,\,N-1\rrbracket.
\]
Then, by Proposition \ref{p_muN-prod}, we may write the flow norm
as
\[
\|\varphi\|^{2}=\sum_{n=0}^{N-1}\frac{\varphi(\omega_{n},\,\omega_{n+1})^{2}}{\mu_{N}(\omega_{n})r_{N}(\omega_{n},\,\omega_{n+1})}=\sum_{n=0}^{N-1}\frac{Z_{N}}{w_{N}(N-n)w_{N}(n)\cdot(N-n)(n+d_{N})r(x,\,y)}.
\]
By \eqref{e_wN-ZN}, this is bounded by
\[
C\sum_{n=0}^{N-1}\frac{d_{N}}{Nd_{N}^{2}}=\frac{C}{d_{N}},
\]
and thus the Thomson principle (Proposition \ref{p_TP}) implies that
\[
\mathrm{Cap}_{N}(\xi^{x},\,\xi^{y})\ge\frac{d_{N}}{C}.
\]
\item If there exists $a\in S_{0}$ such that $r(x,\,a)>0$ and $r(y,\,a)>0$,
then we construct a sequence $\xi^{x}=\omega_{0}\sim\omega_{1}\sim\cdots\sim\omega_{2N}=\xi^{y}$
such that for $n\in\llbracket1,\,N\rrbracket$,
\[
\omega_{2n-1}(v)=\begin{cases}
N-n & \text{if }v=x,\\
1 & \text{if }v=a,\\
n-1 & \text{if }v=y,
\end{cases}\quad\text{and}\quad\omega_{2n}(v)=\begin{cases}
N-n & \text{if }v=x,\\
0 & \text{if }v=a,\\
n & \text{if }v=y.
\end{cases}
\]
Namely, we send each particle through $x\to a\to y$ consecutively.
Similarly, we define $\varphi$ as the flow from $\xi^{x}$ to $\xi^{y}$
of value $1$ along this path. Then, we may calculate as above that
\[
\|\varphi\|^{2}=\sum_{n=0}^{2N-1}\frac{\varphi(\omega_{n},\,\omega_{n+1})^{2}}{\mu_{N}(\omega_{n})r_{N}(\omega_{n},\,\omega_{n+1})}\le\frac{CN}{d_{N}^{2}}.
\]
Thus, by the Thomson principle (Proposition \ref{p_TP}), we obtain
\[
\mathrm{Cap}_{N}(\xi^{x},\,\xi^{y})\ge\frac{d_{N}^{2}}{CN}.
\]
\end{itemize}
According to the two cases, we conclude that
\begin{equation}
\mathrm{Cap}_{N}(\xi^{x},\,\xi^{y})\ge\frac{d_{N}^{2}}{CN}.\label{e_t_3rd-2}
\end{equation}
Thus, by \eqref{e_t_3rd-1} and \eqref{e_t_3rd-2}, we deduce that
\[
\frac{\mathrm{Cap}_{N}(\mathcal{E}_{N}(S_{\star i}^{3}),\,\mathcal{E}_{N}(S_{\star}\setminus S_{\star i}^{3}))}{\inf_{x,\,y\in S_{\star i}^{3}}\mathrm{Cap}_{N}(\xi^{x},\,\xi^{y})}\le\frac{Cd_{N}}{N}=o(1),
\]
which proves \textbf{(C2)}. We have checked all four conditions \textbf{(C1)}--\textbf{(C4)},
and thus parts (1) and (2) of Theorem \ref{t_3rd} hold.

We are left to prove Theorem \ref{t_3rd}-(3), the Eyring--Kramers
formula. By \eqref{e_magic-formula}, it holds that
\[
\mathbb{E}_{\xi^{x}}[\mathcal{T}_{\mathcal{E}_{N}(S_{\star}\setminus S_{\star i}^{3})}]=\frac{1}{\mathrm{Cap}_{N}(\xi^{x},\,\mathcal{E}_{N}(S_{\star}\setminus S_{\star i}^{3}))}\cdot\sum_{\eta\in\mathcal{H}_{N}}\mu_{N}(\eta)h_{\xi^{x},\,\mathcal{E}_{N}(S_{\star}\setminus S_{\star i}^{3})}(\eta).
\]
By Theorem \ref{t_Capacity-specific}, the reciprocal of the capacity
is asymptotically equal to
\[
|S_{\star}|\cdot\Big(\sum_{j\in\llbracket1,\,\kappa_{3}\rrbracket:\,j\ne i}\frac{1}{\mathfrak{K}_{ij}}\Big)^{-1}\cdot\frac{N^{2}}{d_{N}^{3}}.
\]
Moreover, by Proposition \ref{p_condensation} and part (1), we may
calculate
\[
\sum_{\eta\in\mathcal{H}_{N}}\mu_{N}(\eta)h_{\xi^{x},\,\mathcal{E}_{N}(S_{\star}\setminus S_{\star i}^{3})}(\eta)=\frac{1}{|S_{\star}|}\Big(1+\sum_{y\in S_{\star i}^{3}:\,y\ne x}h_{\xi^{x},\,\mathcal{E}_{N}(S_{\star}\setminus S_{\star i}^{3})}(\xi^{y})\Big)+o(1)=\frac{|S_{\star i}^{3}|}{|S_{\star}|}+o(1).
\]
Therefore, we conclude that
\[
\mathbb{E}_{\xi^{x}}[\mathcal{T}_{\mathcal{E}_{N}(S_{\star}\setminus S_{\star i}^{3})}]=(1+o(1))\cdot|S_{\star i}^{3}|\cdot\Big(\sum_{j\in\llbracket1,\,\kappa_{3}\rrbracket:\,j\ne i}\frac{1}{\mathfrak{K}_{ij}}\Big)^{-1}\cdot\frac{N^{2}}{d_{N}^{3}},
\]
which is equivalent to part (3).
\end{proof}
\begin{proof}[Proof of Theorem \ref{t_other-scales}]
 As stated in the theorem, take $x\in S_{\star}$, $i\in\llbracket1,\,\kappa_{2}\rrbracket$,
and $j\in\llbracket1,\,\kappa_{3}\rrbracket$ such that $\{x\}\subseteq S_{\star i}^{2}\subseteq S_{\star j}^{3}$.

First, assume that $\alpha_{N}\ll\theta_{1}$. Then, applying the
proof of Theorem \ref{t_1st}-(1) given in \cite[Section 4.3]{BDG}
but with the smaller time scale $\alpha_{N}$ instead of $\theta_{1}$,
we obtain that the finite-dimensional marginals of the $\alpha_{N}$-accelerated
inclusion process converge to the marginals of the zero Markov chain
in $S_{\star}$. Part (1) is straightforward from this observation.

Next, assume that $\theta_{1}\ll\alpha_{N}\ll\theta_{2}$. We prove
that
\begin{equation}
\lim_{N\to\infty}\mathbb{P}_{\xi^{x}}[\eta_{N}(\alpha_{N})=\xi^{y}]=\frac{1}{|S_{\star i}^{2}|}\text{ for all }y\in S_{\star i}^{2}.\label{e_t_other-scales-1}
\end{equation}
We fix an arbitrary time $T>0$. First, using the same logic explained
in the previous paragraph, we similarly deduce that the finite-dimensional
marginals of both the $\alpha_{N}$- and $(\alpha_{N}-T\theta_{1})$-accelerated
inclusion processes converge to the marginals of the zero Markov chain
in $\llbracket1,\,\kappa_{2}\rrbracket$. This implies that
\begin{equation}
\lim_{N\to\infty}\mathbb{P}_{\xi^{x}}[\eta_{N}(\alpha_{N})\in\mathcal{E}_{N}(S_{\star i}^{2})]=1\quad\text{and}\quad\lim_{N\to\infty}\mathbb{P}_{\xi^{x}}[\eta_{N}(\alpha_{N}-T\theta_{1})\in\mathcal{E}_{N}(S_{\star i}^{2})]=1.\label{e_t_other-scales-2}
\end{equation}
Now, take any subset $A\subseteq S_{\star i}^{2}$. By the Markov
property at time $\alpha_{N}-T\theta_{1}$, we have
\[
\begin{aligned}\mathbb{P}_{\xi^{x}}[\eta_{N}(\alpha_{N})\in\mathcal{E}_{N}(A)] & =\mathbb{P}_{\xi^{x}}\big[\mathbb{P}_{\eta_{N}(\alpha_{N}-T\theta_{1})}[\eta_{N}(T\theta_{1})\in\mathcal{E}_{N}(A)]\big]\\
 & \le\sup_{z\in S_{\star i}^{2}}\mathbb{P}_{\xi^{z}}[\eta_{N}(T\theta_{1})\in\mathcal{E}_{N}(A)]+o(1),
\end{aligned}
\]
where the inequality holds by the right identity in \eqref{e_t_other-scales-2}.
Moreover, by Theorem \ref{t_1st}-(1), the probability inside the
supremum in the right-hand side converges to $\mathrm{P}_{z}[X(T)\in A]$
as $N\to\infty$. Thus, what we have proved is that
\[
\limsup_{N\to\infty}\mathbb{P}_{\xi^{x}}[\eta_{N}(\alpha_{N})\in\mathcal{E}_{N}(A)]\le\sup_{z\in S_{\star i}^{2}}\mathrm{P}_{z}[X(T)\in A]\text{ for all }T>0\text{ and }A\subseteq S_{\star i}^{2}.
\]
Sending $T\to\infty$, the right-hand side converges to the stationary
measure of $A$ with respect to $X(\cdot)$ in the connected component
$S_{\star i}^{2}$, which is $|A|/|S_{\star i}^{2}|$. Hence, it holds
that
\[
\limsup_{N\to\infty}\mathbb{P}_{\xi^{x}}[\eta_{N}(\alpha_{N})\in\mathcal{E}_{N}(A)]\le\frac{|A|}{|S_{\star i}^{2}|}\text{ for all }A\subseteq S_{\star i}^{2}.
\]
Since this inequality holds for all $A\subseteq S_{\star i}^{2}$,
combining with the left identity in \eqref{e_t_other-scales-2}, we
readily deduce that \eqref{e_t_other-scales-1} is valid for all $y\in S_{\star i}^{2}$.
This completes the proof of part (2). Parts (3) and (4) can be proved
in the same way; thus, we omit the details.
\end{proof}
It remains to prove the capacity estimates in Theorems \ref{t_Capacity}
and \ref{t_Capacity-specific}. The rest of the article is devoted
to these proofs.

\section{\label{sec4}Energy Landscape Analysis and a Resolvent Equation}

In this section, we conduct the analysis of the energy landscape of
configurations regarding the third-level metastable transitions.
\begin{assumption}
\label{a_Sstar-kappa3}In the remainder of the article, we assume
that $|S_{\star}|=\kappa_{3}=2$ and write $S_{\star}=\{x,\,y\}$,
except in Section \ref{sec7} where we discuss additional difficulties
that arise in the proof of the general case of $|S_{\star}|\ge3$.
\end{assumption}

For simplicity, we define
\begin{equation}
m_{\star}:=\max_{a\in S_{0}}m_{a}\in(0,\,1).\label{e_mstar-def}
\end{equation}
We first introduce the notions of simplexes $\mathcal{A}_{N}^{R}$
and $\widehat{\mathcal{A}}_{N}^{R}$ for each $R\subseteq S$.
\begin{defn}[Simplexes $\mathcal{A}_{N}^{R}$ and $\widehat{\mathcal{A}}_{N}^{R}$
for $R\subseteq S$]
\label{d_AN-ANhat-def}$ $
\begin{enumerate}
\item For each subset $R\subseteq S$, define
\[
\mathcal{A}_{N}^{R}:=\{\eta\in\mathcal{H}_{N}:\eta_{a}=0,\ a\in S\setminus R\}\quad\text{and}\quad\widehat{\mathcal{A}}_{N}^{R}:=\{\eta\in\mathcal{A}_{N}^{R}:\eta_{a}\ge1,\ a\in R\}.
\]
We may replace the superscript $R$ with its explicit elements without
commas. For instance, if $R=\{a,\,b\}$ then we write $\mathcal{A}_{N}^{R}=\mathcal{A}_{N}^{ab}$
or $\widehat{\mathcal{A}}_{N}^{R}=\widehat{\mathcal{A}}_{N}^{ab}$.
\item If $|R|=r$ such that $R=\{a_{1},\,\dots,\,a_{r}\}$, then we denote
by $\xi_{n_{1},\,\dots,\,n_{r-1}}^{a_{1}a_{2}\cdots a_{r}}\in\mathcal{A}_{N}^{R}$
the element which satisfies
\begin{equation}
\xi_{n_{1},\,\dots,\,n_{r-1}}^{a_{1}a_{2}\cdots a_{r}}(a)=\begin{cases}
n_{j} & \text{if }a=a_{j}\text{ for }j\in\llbracket1,\,r-1\rrbracket,\\
N-(n_{1}+\cdots+n_{r-1}) & \text{if }a=a_{r}.
\end{cases}\label{e_xi-notation}
\end{equation}
This notation naturally extends the one for the configurations $\xi^{a}\in\mathcal{E}_{N}^{a}$.
\end{enumerate}
\end{defn}

As explained in Remark \ref{r_3rd}-(2), there are four sites involved
in the effective movements of particles on the third time scale. This
implies that our energy landscape analysis should be focused on each
tetrahedral subset of configurations, i.e., $\mathcal{A}_{N}^{xaby}$
for $\{a,\,b\}\subseteq S_{0}$ (cf. \eqref{e_S-star-S0-def} and
Notation \ref{n_different}).

Next, we define a graph $\mathcal{G}=(\mathcal{V},\,\mathcal{E})$
which realizes the underlying geometry given by the random walk $r(\cdot,\,\cdot)$.

\begin{figure}
\begin{tikzpicture}
\fill[rounded corners,white] (-3,2.1) rectangle (3,-0.1); 
\draw (0,0.7) node[above]{$x$\hspace{4.6cm}$y$};
\fill[rounded corners,white] (-0.25,2.25) rectangle (0.25,1.75);
\fill[rounded corners,white] (-2.25,-0.75) rectangle (-1.75,-1.25);
\fill[rounded corners,white] (-1.25,-0.75) rectangle (1.25,-2.25);
\draw[very thick] (-2,2)--(-2,1)--(-1,2)--(1,2)--(2,1)--(-2,1)--(-2,0);
\draw[very thick] (-2,1)--(-1,0)--(-2,-1);
\draw[very thick] (-1,-1)--(-1,-2)--(1,-2)--(1,0)--(-1,0)--(-1,-1)--(0,-1)--(2,1)--(2,0)--(1,0);
\draw[very thick] (-1,0)--(0,-1); \draw[very thick] (2,0)--(2,-1); \draw[very thick] (-1,2)--(-1,1);
\fill[black] (-2,1) circle (0.1); \draw (-2,1) circle (0.1);
\fill[black] (2,1) circle (0.1); \draw (2,1) circle (0.1);
\fill[red!50!white] (-2,2) circle (0.1); \draw[black!30!white] (-2,2) circle (0.1);
\fill[red!50!white] (-1,2) circle (0.1); \draw[black!30!white] (-1,2) circle (0.1);
\fill[white] (0,2) circle (0.1); \draw[black!30!white] (0,2) circle (0.1);
\fill[blue!50!white] (1,2) circle (0.1); \draw[black!30!white] (1,2) circle (0.1);
\fill[red!50!white] (-1,1) circle (0.1); \draw[black!30!white] (-1,1) circle (0.1);
\fill[blue!50!white] (1,1) circle (0.1); \draw[black!30!white] (1,1) circle (0.1);
\fill[red!50!white] (-2,0) circle (0.1); \draw[black!30!white] (-2,0) circle (0.1);
\fill[red!50!white] (-1,0) circle (0.1); \draw[black!30!white] (-1,0) circle (0.1);
\fill[blue!50!white] (1,0) circle (0.1); \draw[black!30!white] (1,0) circle (0.1);
\fill[blue!50!white] (2,0) circle (0.1); \draw[black!30!white] (2,0) circle (0.1);
\fill[white] (-2,-1) circle (0.1); \draw[black!30!white] (-2,-1) circle (0.1);
\fill[white] (-1,-1) circle (0.1); \draw[black!30!white] (-1,-1) circle (0.1);
\fill[white] (0,-1) circle (0.1); \draw[black!30!white] (0,-1) circle (0.1);
\fill[white] (1,-1) circle (0.1); \draw[black!30!white] (1,-1) circle (0.1);
\fill[white] (2,-1) circle (0.1); \draw[black!30!white] (2,-1) circle (0.1);
\fill[white] (-1,-2) circle (0.1); \draw[black!30!white] (-1,-2) circle (0.1);
\fill[white] (0,-2) circle (0.1); \draw[black!30!white] (0,-2) circle (0.1);
\fill[white] (1,-2) circle (0.1); \draw[black!30!white] (1,-2) circle (0.1);
\end{tikzpicture}
\hspace{5mm}
\begin{tikzpicture}
\fill[rounded corners,white] (-3,2.1) rectangle (3,-0.1); 
\draw (0,0.7) node[above]{$x$\hspace{4.6cm}$y$};
\fill[rounded corners,orange!50!white] (-0.25,2.25) rectangle (0.25,1.75);
\fill[rounded corners,orange!50!white] (-2.25,-0.75) rectangle (-1.75,-1.25);
\fill[rounded corners,orange!50!white] (-1.25,-0.75) rectangle (1.25,-2.25);
\fill[rounded corners,orange!50!white] (1.75,-0.75) rectangle (2.25,-1.25);
\draw[very thick] (-2,1)--(0,2)--(2,1)--(-2,1);
\draw[very thick] (-2,-1)--(-2,1)--(-1,-1)--(-1,-2)--(1,-2)--(1,-1)--(2,1)--(0,-1)--(-1,-1);
\draw[very thick] (-2,1)--(0,-1); \draw[very thick] (2,1)--(2,-1);
\fill[black] (-2,1) circle (0.1); \draw (-2,1) circle (0.1);
\fill[black] (2,1) circle (0.1); \draw (2,1) circle (0.1);
\fill[white] (0,2) circle (0.1); \draw[black!30!white] (0,2) circle (0.1);
\fill[white] (-2,-1) circle (0.1); \draw[black!30!white] (-2,-1) circle (0.1);
\fill[white] (-1,-1) circle (0.1); \draw[black!30!white] (-1,-1) circle (0.1);
\fill[white] (0,-1) circle (0.1); \draw[black!30!white] (0,-1) circle (0.1);
\fill[white] (1,-1) circle (0.1); \draw[black!30!white] (1,-1) circle (0.1);
\fill[white] (2,-1) circle (0.1); \draw[black!30!white] (2,-1) circle (0.1);
\fill[white] (-1,-2) circle (0.1); \draw[black!30!white] (-1,-2) circle (0.1);
\fill[white] (0,-2) circle (0.1); \draw[black!30!white] (0,-2) circle (0.1);
\fill[white] (1,-2) circle (0.1); \draw[black!30!white] (1,-2) circle (0.1);
\draw[rounded corners,red!50!white] (-0.25,2.25) rectangle (0.25,1.75);
\end{tikzpicture}\caption{\label{fig4.1}Illustrations of the graphs $\mathcal{G}$ (left) and
$\mathcal{G}'$ (right). In the left figure, red (resp. blue) circles
denote the sites in $\mathcal{N}_{x}$ (resp. $\mathcal{N}_{y}$),
which are merged into $x$ (resp. $y$) in the right figure. The orange
regions in the right figure denote the subgraphs $\mathcal{G}_{j}'$,
$j\in\llbracket1,\,s\rrbracket$. There are $4$ sites in $\mathcal{A}_{x}$
and $4$ sites in $\mathcal{A}_{y}$ such that $2$ sites belong to
both $\mathcal{A}_{x}$ and $\mathcal{A}_{y}$.}
\end{figure}
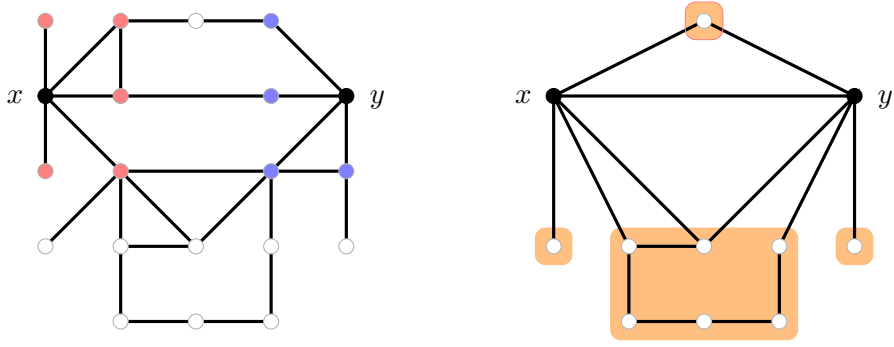

\begin{defn}[Underlying graph $\mathcal{G}$]
\label{d_G-def}We define $\mathcal{G}=(\mathcal{V},\,\mathcal{E})$
(see Figure \ref{fig4.1}-left).
\begin{enumerate}
\item The vertex set $\mathcal{V}$ is defined as $\mathcal{V}:=S$.
\item The edge set $\mathcal{E}$ is defined as follows. For $a,\,b\in\mathcal{V}$,
$\{a,\,b\}\in\mathcal{E}$ if and only if $r(a,\,b)>0$. The collection
$\mathcal{E}$ is well defined since $r(a,\,b)>0$ if and only if
$r(b,\,a)>0$ by reversibility (cf. \eqref{e_cxy-def}). Since the
underlying random walk is irreducible, $\mathcal{G}=(\mathcal{V},\,\mathcal{E})$
becomes a connected graph. We write $a\sim b$ (or $b\sim a$) if
$\{a,\,b\}\in\mathcal{E}$.
\item For a vertex $v\in\mathcal{V}$, we define the \emph{neighborhood}
$\mathcal{N}_{v}$ of $v$ as $\mathcal{N}_{v}:=\{a\in\mathcal{V}:v\sim a\}$.
Since $S_{\star}=\{x,\,y\}$ and $\kappa_{3}=2$, the graph distance
between $x$ and $y$ is at least $3$. This implies that $\mathcal{N}_{x}$
and $\mathcal{N}_{y}$ are disjoint non-empty subsets of $\mathcal{V}\setminus\{x,\,y\}=S_{0}$.
\end{enumerate}
\end{defn}

According to the collections $\mathcal{N}_{x}$ and $\mathcal{N}_{y}$,
we introduce the following six types of $\{a,\,b\}\subseteq S_{0}$:
\begin{equation}
\begin{aligned} & \textbf{(T1)}\ a,\,b\in\mathcal{N}_{x};\quad &  & \textbf{(T2)}\ a,\,b\in\mathcal{N}_{y};\quad &  & \textbf{(T3)}\ a\in\mathcal{N}_{x},\ b\in\mathcal{N}_{y};\\
 & \textbf{(T4)}\ a\in\mathcal{N}_{x},\ b\notin\mathcal{N}_{x}\cup\mathcal{N}_{y};\quad &  & \textbf{(T5)}\ a\notin\mathcal{N}_{x}\cup\mathcal{N}_{y},\ b\in\mathcal{N}_{y};\quad &  & \textbf{(T6)}\ a,\,b\notin\mathcal{N}_{x}\cup\mathcal{N}_{y}.
\end{aligned}
\label{e_ab-types}
\end{equation}

\subsection{\label{sec4.1}Heuristics on energy landscape and test objects}

In this subsection, we present some heuristics of our energy landscape
analysis and the corresponding constructions of test objects. Those
who are already familiar may proceed to Section \ref{sec4.2}.

\subsubsection*{Dimension reduction: 3D simplexes to 1D or 2D simplexes}

We explain the idea of reducing the analysis of three-dimensional
objects $\mathcal{A}_{N}^{xaby}$, $\{a,\,b\}\subseteq S_{0}$ to
the one-dimensional object $\mathcal{A}_{N}^{xy}$ or two-dimensional
objects $\mathcal{A}_{N}^{xay}$, $a\in S_{0}$.

In Section \ref{sec5}, we will construct a test function $F=F_{\textup{test}}:\mathcal{H}_{N}\to\mathbb{R}$
approximating the equilibrium potential $h_{\mathcal{E}_{N}^{x},\,\mathcal{E}_{N}^{y}}$.
Since $h_{\mathcal{E}_{N}^{x},\,\mathcal{E}_{N}^{y}}$ solves the
Dirichlet problem given in \eqref{e_h-Diri-problem}, a nice test
object $F$ would satisfy in some sense that
\[
\mu_{N}(\eta)(\mathcal{L}_{N}F)(\eta)=\sum_{\zeta\in\mathcal{H}_{N}}\mu_{N}(\eta)r_{N}(\eta,\,\zeta)(F(\zeta)-F(\eta))\approx0\text{ for all }\eta\in\mathcal{H}_{N}\setminus(\mathcal{E}_{N}^{x}\cup\mathcal{E}_{N}^{y}).
\]
Thus, it is natural to require $\mu_{N}(\eta)r_{N}(\eta,\,\zeta)\cdot|F(\zeta)-F(\eta)|$
to have a comparable scale as $N\to\infty$ for all $\eta\sim\zeta$;
in particular, $|F(\zeta)-F(\eta)|$ should be small if $\mu_{N}(\eta)r_{N}(\eta,\,\zeta)$
is big. Since we are interested in the asymptotic behavior in the
limit $N\to\infty$, it is our best strategy to design our test function
$F$ so that $F(\eta)=F(\zeta)$ for all pairs $\{\eta,\,\zeta\}$
such that $\mu_{N}(\eta)r_{N}(\eta,\,\zeta)$ has a bigger scale than
the others.

\begin{figure}
\begin{tikzpicture}
\fill[rounded corners,white] (-2.5,3.7) rectangle (2.5,-1.7); 
\draw[densely dashed] (-2,0)--(2,0); \draw[densely dashed] (0.4,-0.8)--(-1.2,0)--(0.4,2.4);
\draw (0,3)--(-2,0)--(0,-1)--(2,0)--(0,3)--(0,-1); \draw (0.4,-0.8)--(0.4,2.4);
\draw[very thick,red] (-0.4,1.2)--(0.4,0.8); \draw[very thick,red] (-0.2,1.5)--(0.4,1.2); \draw[very thick,red] (0,1.8)--(0.4,1.6); \draw[very thick,red] (0.2,2.1)--(0.4,2); 
\draw[very thick,blue] (-0.4,1.2)--(0.4,2.4)--(0.4,0.8);
\draw[very thick,blue] (-0.2,1.1)--(0.4,2); \draw[very thick,blue] (0,1)--(0.4,1.6); \draw[very thick,blue] (0.2,0.9)--(0.4,1.2);
\draw[very thick,blue] (-0.2,1.1)--(-0.2,1.5); \draw[very thick,blue] (0,1)--(0,1.8); \draw[very thick,blue] (0.2,0.9)--(0.2,2.1);
\draw (0,3.1) node[above]{$\xi^x$};
\draw (-2,-0.1) node[below]{$\xi^a$};
\draw (2,-0.1) node[below]{$\xi^y$};
\draw (0,-1.1) node[below]{$\xi^b$};
\end{tikzpicture}
\begin{tikzpicture}
\fill[rounded corners,white] (-2.5,3.7) rectangle (2.5,-1.7); 
\draw[densely dashed] (-2,0)--(2,0); \draw[densely dashed] (0.4,-0.8)--(-1.2,0)--(0.4,2.4); \draw[densely dashed] (-0.4,1.2)--(1.2,1.2); \draw[densely dashed] (-0.2,1.5)--(1,1.5); \draw[densely dashed] (0,1.8)--(0.8,1.8); \draw[densely dashed] (0.2,2.1)--(0.6,2.1);
\draw[very thick,densely dotted,red] (0.4,1.8)--(0.6,1.7);
\draw[very thick,densely dotted,red] (0.2,1.5)--(0.6,1.3); \draw[very thick,densely dotted,red] (0.6,1.5)--(0.8,1.4);
\draw[very thick,densely dotted,red] (0,1.2)--(0.6,0.9); \draw[very thick,densely dotted,red] (0.4,1.2)--(0.8,1); \draw[very thick,densely dotted,red] (0.8,1.2)--(1,1.1);
\draw[very thick,densely dotted,blue] (0.4,1.8)--(0.2,1.7);
\draw[very thick,densely dotted,blue] (0.2,1.5)--(0,1.4); \draw[very thick,densely dotted,blue] (0.6,1.5)--(0.2,1.3);
\draw[very thick,densely dotted,blue] (0,1.2)--(-0.2,1.1); \draw[very thick,densely dotted,blue] (0.4,1.2)--(0,1); \draw[very thick,densely dotted,blue] (0.8,1.2)--(0.2,0.9);
\draw (0,3)--(-2,0)--(0,-1)--(2,0)--(0,3)--(0,-1); \draw (0.4,-0.8)--(0.4,2.4);
\draw[very thick,red] (-0.4,1.2)--(0.4,0.8); \draw[very thick,red] (-0.2,1.5)--(0.4,1.2); \draw[very thick,red] (0,1.8)--(0.4,1.6); \draw[very thick,red] (0.2,2.1)--(0.4,2); 
\draw[very thick,blue] (-0.4,1.2)--(0.4,2.4);
\draw[very thick,blue] (-0.2,1.1)--(0.4,2); \draw[very thick,blue] (0,1)--(0.4,1.6); \draw[very thick,blue] (0.2,0.9)--(0.4,1.2);
\draw[very thick,blue] (0.4,0.8)--(1.2,1.2); \draw[very thick,blue] (0.4,1.2)--(1,1.5); \draw[very thick,blue] (0.4,1.6)--(0.8,1.8); \draw[very thick,blue] (0.4,2)--(0.6,2.1);
\draw (0,3.1) node[above]{$\xi^x$};
\draw (-2,-0.1) node[below]{$\xi^a$};
\draw (2,-0.1) node[below]{$\xi^y$};
\draw (0,-1.1) node[below]{$\xi^b$};
\end{tikzpicture}
\begin{tikzpicture}
\fill[rounded corners,white] (-2.5,3.7) rectangle (2.5,-1.7); 
\draw[densely dashed] (-2,0)--(2,0); \draw[densely dashed] (0.4,-0.8)--(-1.2,0)--(0.4,2.4);
\draw (0,3)--(-2,0)--(0,-1)--(2,0)--(0,3)--(0,-1); \draw (0.4,-0.8)--(0.4,2.4);
\draw[very thick,red] (-0.4,1.2)--(0.4,0.8); \draw[very thick,red] (-0.2,1.5)--(0.4,1.2); \draw[very thick,red] (0,1.8)--(0.4,1.6); \draw[very thick,red] (0.2,2.1)--(0.4,2); 
\draw[very thick,blue] (-0.4,1.2)--(0.4,2.4);
\draw[very thick,blue] (-0.2,1.1)--(0.4,2); \draw[very thick,blue] (0,1)--(0.4,1.6); \draw[very thick,blue] (0.2,0.9)--(0.4,1.2);
\draw (0,3.1) node[above]{$\xi^x$};
\draw (-2,-0.1) node[below]{$\xi^a$};
\draw (2,-0.1) node[below]{$\xi^y$};
\draw (0,-1.1) node[below]{$\xi^b$};
\end{tikzpicture}\\
\vspace{2.5mm}
\begin{tikzpicture}
\fill[rounded corners,white] (-2.5,3) rectangle (2.5,-1); 
\draw[densely dashed] (-2,0)--(2,0);
\draw (0,3)--(-2,0)--(0,-1)--(2,0)--(0,3)--(0,-1);
\fill[blue] (0,3) circle (0.07); \fill[blue] (0.4,2.4) circle (0.07); \fill[blue] (0.8,1.8) circle (0.07); \fill[blue] (1.2,1.2) circle (0.07); \fill[blue] (1.6,0.6) circle (0.07); \fill[blue] (2,0) circle (0.07);
\end{tikzpicture}
\begin{tikzpicture}
\fill[rounded corners,white] (-2.5,3) rectangle (2.5,-1); 
\draw[densely dashed] (-2,0)--(2,0);
\draw (0,3)--(-2,0)--(0,-1)--(2,0)--(0,3)--(0,-1);
\draw[very thick,red] (0,3)--(2,0);
\fill[blue] (0,3) circle (0.07);
\fill[blue] (0.4,2.4) circle (0.07);
\fill[blue] (0.8,1.8) circle (0.07);
\fill[blue] (1.2,1.2) circle (0.07);
\fill[blue] (1.6,0.6) circle (0.07);
\fill[blue] (2,0) circle (0.07);
\end{tikzpicture}
\begin{tikzpicture}
\fill[rounded corners,white] (-2.5,3) rectangle (2.5,-1); 
\fill[red!10!white] (0,3)--(0,-1)--(2,0)--cycle;
\draw[densely dashed] (-2,0)--(2,0);
\draw (0,3)--(-2,0)--(0,-1)--(2,0)--(0,3)--(0,-1);
\draw[very thick,red] (0,3)--(0,-1); \draw[very thick,red] (0.4,2.4)--(0.4,-0.8); \draw[very thick,red] (0.8,1.8)--(0.8,-0.6); \draw[very thick,red] (1.2,1.2)--(1.2,-0.4);  \draw[very thick,red] (1.6,0.6)--(1.6,-0.2);
\fill[blue] (0,3) circle (0.07); \fill[blue] (0,2.2) circle (0.07); \fill[blue] (0,1.4) circle (0.07); \fill[blue] (0,0.6) circle (0.07); \fill[blue] (0,-0.2) circle (0.07); \fill[blue] (0,-1) circle (0.07);
\fill[blue] (0.4,2.4) circle (0.07); \fill[blue] (0.4,1.6) circle (0.07); \fill[blue] (0.4,0.8) circle (0.07); \fill[blue] (0.4,0) circle (0.07); \fill[blue] (0.4,-0.8) circle (0.07);
\fill[blue] (0.8,1.8) circle (0.07); \fill[blue] (0.8,1) circle (0.07); \fill[blue] (0.8,0.2) circle (0.07); \fill[blue] (0.8,-0.6) circle (0.07);
\fill[blue] (1.2,1.2) circle (0.07); \fill[blue] (1.2,0.4) circle (0.07); \fill[blue] (1.2,-0.4) circle (0.07);
\fill[blue] (1.6,0.6) circle (0.07); \fill[blue] (1.6,-0.2) circle (0.07);
\fill[blue] (2,0) circle (0.07);
\end{tikzpicture}\caption{\label{fig4.2}Dimension reductions of $\mathcal{A}_{N}^{xaby}$ for
$\{a,\,b\}\subseteq S_{0}$ of types \textbf{(T1)} (left), \textbf{(T3)}
(middle), and \textbf{(T4)} (right). Blue (resp. red) segments denote
the big-scale (resp. small-scale) edges above, such that configurations
connected with blue segments are merged as blue circles below. The
resulting simplexes are null (left), one-dimensional (middle), or
two-dimensional (right).}
\end{figure}
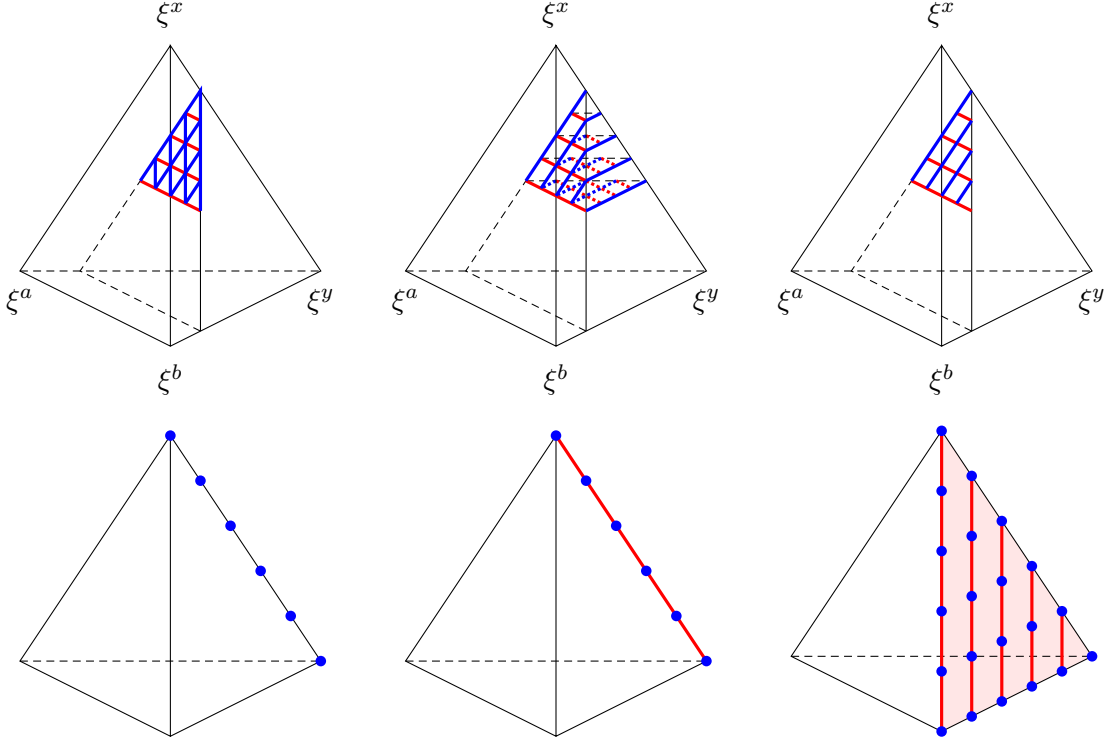

Now, we fix a tetrahedron $\mathcal{A}_{N}^{xaby}$ with $\{a,\,b\}\subseteq S_{0}$
of type \textbf{(T3)} (see \eqref{e_ab-types}) and consider an element
$\xi=\xi_{n,\,k,\,\ell}^{xaby}\in\mathcal{A}_{N}^{xaby}$ (recall
\eqref{e_xi-notation}). For a particle jump from $x$ to $a$, we
calculate $\mu_{N}(\xi)r_{N}(\xi,\,\xi_{n-1,\,k+1,\,\ell}^{xaby})=\mu_{N}(\xi)\cdot n(d_{N}+k)\cdot r(x,\,a)$.
Whereas, for a particle jump from $a$ to $b$, we calculate $\mu_{N}(\xi)r_{N}(\xi,\,\xi_{n,\,k-1,\,\ell+1}^{xaby})=\mu_{N}(\xi)\cdot k(d_{N}+\ell)\cdot r(a,\,b)$.
Moreover, since $a,\,b\in S_{0}$, the dynamics tends to avoid configurations
with lots of particles at $a$ or $b$, since they are $m_{\star}$-exponentially
negligible (cf. \eqref{e_mstar-def}) in the sense of Proposition
\ref{p_muN-prod}. Thus, the values of $k$ and $\ell$ should be
properly bounded during the entire metastable transition, whereas
the value of $n$ varies from $0$ to $N$. From these observations,
we roughly deduce that $\mu_{N}(\xi)r_{N}(\xi,\,\xi_{n-1,\,k+1,\,\ell}^{xaby})$
has a bigger scale than $\mu_{N}(\xi)r_{N}(\xi,\,\xi_{n,\,k-1,\,\ell+1}^{xaby})$.
Thus, by the arguments given in the previous paragraph, \emph{it is
natural to define $F$ to be unchanged subject to the particle movements
between $x$ and $a$.} The same logic remains valid on the other
side as well, such that it is natural to define $F$ as constant subject
to the particle movements between $b$ and $y$.

We may apply the above logic to the other types of $\{a,\,b\}\subseteq S_{0}$
as well, thereby collecting the following reductions for $\mathcal{A}_{N}^{xaby}$
according to the types of $\{a,\,b\}$. Refer to Figure \ref{fig4.2}.
\begin{itemize}
\item \textbf{(T1)} Configurations obtained from $x\leftrightarrow a,\,b$
merge as one (Figure \ref{fig4.2}-left).
\item \textbf{(T2)} Configurations obtained from $a,\,b\leftrightarrow y$
merge.
\item \textbf{(T3)} Configurations obtained from $x\leftrightarrow a$ and
$b\leftrightarrow y$ merge (Figure \ref{fig4.2}-middle).
\item \textbf{(T4)} Configurations obtained from $x\leftrightarrow a$ merge
(Figure \ref{fig4.2}-right).
\item \textbf{(T5)} Configurations obtained from $b\leftrightarrow y$ merge.
\end{itemize}
In particular, $\mathcal{A}_{N}^{xaby}$ for $\{a,\,b\}\subseteq S_{0}$
of type \textbf{(T$\boldsymbol{i}$)} reduces to $\emptyset$ ($i=1,\,2$),
$\mathcal{A}_{N}^{xy}$ ($i=3$), $\mathcal{A}_{N}^{xby}$ ($i=4$),
or $\mathcal{A}_{N}^{xay}$ ($i=5$). Note that $\mathcal{A}_{N}^{xaby}$
for $\{a,\,b\}\subseteq S_{0}$ of type \textbf{(T6)} remains as three-dimensional.

\subsubsection*{Construction of test objects: infinite ladder graph}

Moving further, as in Figure \ref{fig4.3}-left, consider $a,\,b,\,c\notin\mathcal{N}_{x}\cup\mathcal{N}_{y}$
with $a\sim b\sim c$ (red edges) such that there exist $a_{0},\,b_{0}\in\mathcal{N}_{x}$
and $c_{0}\in\mathcal{N}_{y}$ with $a_{0}\sim a$, $b_{0}\sim b$,
and $c\sim c_{0}$. Then, the dimension reductions can be applied
to the simplexes $\mathcal{A}_{N}^{xa_{0}ay}$, $\mathcal{A}_{N}^{xb_{0}by}$,
and $\mathcal{A}_{N}^{xcc_{0}y}$ to produce $\mathcal{A}_{N}^{xay}$
(black), $\mathcal{A}_{N}^{xby}$ (violet), and $\mathcal{A}_{N}^{xcy}$
(blue).

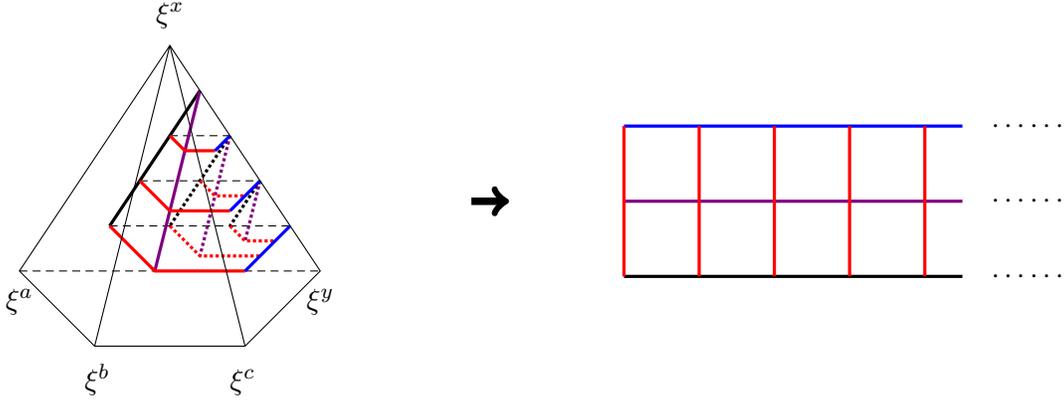
\begin{figure}
\begin{tikzpicture}
\fill[rounded corners,white] (-2.5,3.7) rectangle (2.5,-1.7); 
\draw[densely dashed] (-2,0)--(2,0);
\draw[densely dashed] (-0.8,0.6)--(1.6,0.6); \draw[densely dashed] (-0.4,1.2)--(1.2,1.2); \draw[densely dashed] (0,1.8)--(0.8,1.8);
\draw[very thick,densely dotted] (0.8,1.8)--(0,0.6); \draw[very thick,densely dotted] (1.2,1.2)--(0.8,0.6);
\draw[very thick,densely dotted,violet] (0.8,1.8)--(0.4,0.2);
\draw[very thick,densely dotted,violet] (1.2,1.2)--(1,0.4);
\draw[very thick,densely dotted,red] (0,0.6)--(0.4,0.2)--(1.2,0.2);
\draw[very thick,densely dotted,red] (0.8,0.6)--(1,0.4)--(1.4,0.4);
\draw[very thick,densely dotted,red] (0.4,1.2)--(0.6,1)--(1,1);
\draw[very thick] (0.4,2.4)--(-0.8,0.6);
\draw[very thick,violet] (0.4,2.4)--(-0.2,0);
\draw[very thick,red] (-0.8,0.6)--(-0.2,0)--(1,0); \draw[very thick,blue] (1,0)--(1.6,0.6);
\draw[very thick,red] (-0.4,1.2)--(0,0.8)--(0.8,0.8); \draw[very thick,blue] (0.8,0.8)--(1.2,1.2);
\draw[very thick,red] (0,1.8)--(0.2,1.6)--(0.6,1.6); \draw[very thick,blue] (0.6,1.6)--(0.8,1.8);
\draw (0,3)--(-2,0)--(-1,-1)--(1,-1)--(2,0)--(0,3)--(-1,-1); \draw (0,3)--(1,-1);
\draw (0,3.1) node[above]{$\xi^x$};
\draw (-2,-0.1) node[below]{$\xi^a$};
\draw (2,-0.1) node[below]{$\xi^y$};
\draw (0,-1.1) node[below]{$\xi^b$\hspace{1.6cm}$\xi^c$};
\end{tikzpicture}
\hspace{5mm}
\begin{tikzpicture}
\fill[rounded corners,white] (-1,3.7) rectangle (1,-1.7); 
\draw[line width=1mm,-to] (-0.25,1)->(0.25,1);
\end{tikzpicture}
\hspace{5mm}
\begin{tikzpicture}
\fill[rounded corners,white] (0,3.7) rectangle (6,-1.7); 
\draw[very thick] (0,0)--(4.5,0); \draw[very thick,violet] (0,1)--(4.5,1); \draw[very thick,blue] (0,2)--(4.5,2);
\draw[very thick,red] (0,2)--(0,0);
\draw[very thick,red] (1,2)--(1,0);
\draw[very thick,red] (2,2)--(2,0);
\draw[very thick,red] (3,2)--(3,0);
\draw[very thick,red] (4,2)--(4,0);
\draw (4.75,2) node[right]{$\cdots\cdots$}; \draw (4.75,1) node[right]{$\cdots\cdots$}; \draw (4.75,0) node[right]{$\cdots\cdots$};
\end{tikzpicture}\caption{\label{fig4.3}(Left) Induced graph structure on $\mathcal{A}_{N}^{xabcy}$
where $a,\,b,\,c\protect\notin\mathcal{N}_{x}\cup\mathcal{N}_{y}$.
(Right) Further-induced infinite ladder graph.}
\end{figure}

Note that $r_{N}(\xi_{n,\,k,\,\ell}^{xaby},\,\xi_{n,\,k-1,\,\ell+1}^{xaby})=k(d_{N}+\ell)\cdot r(a,\,b)$
does not depend on the number of particles at $x$ or $y$. Thus,
it is natural to imagine that an appropriate test function $F$ has
a certain local translation-invariance property on the one-dimensional
line $\mathcal{A}_{N}^{xy}$. This observation suggests us to define
$F$ on $\mathcal{A}_{N}^{xaby}$ of the form
\[
F(\xi_{n,\,k,\,\ell}^{xaby})\asymp\sum_{t=1}^{n}\mathfrak{s}(t)+\mathfrak{s}(n)\cdot\mathfrak{g}_{ab}(k,\,\ell)
\]
for some appropriate scale function $\mathfrak{s}:\llbracket0,\,N\rrbracket\to\mathbb{R}$
and a test function $\mathfrak{g}_{ab}:\mathbb{N}\times\mathbb{N}\to\mathbb{R}$.
For the former object $\mathfrak{s}$, it turns out later in Section
\ref{sec5.1} that the correct choice is $\mathfrak{s}(t)=(N-t)t$
(with proper mollification near $\xi^{x}$ and $\xi^{y}$; see \eqref{e_tf-T1}--\eqref{e_tf-T6-2}).
For the latter object $\mathfrak{g}_{ab}$, we construct a suitable
infinite ladder graph (illustrated in Figure \ref{fig4.3}-right)
and a certain resolvent equation thereon (cf. Theorem \ref{t_resolvent-equation}),
and choose $\mathfrak{g}_{ab}$ as a proper rescaling of the solution
to the resolvent equation (cf. Definition \ref{d_g-ghat-def}). This
procedure is explained in full details in Section \ref{sec4.2}.

\subsection{\label{sec4.2}Graph decomposition and a resolvent equation}

From $\mathcal{G}=(\mathcal{V},\,\mathcal{E})$, we define a new graph
$\mathcal{G}'=(\mathcal{V}',\,\mathcal{E}')$. See Figure \ref{fig4.1}-right.
\begin{defn}[Induced graph $\mathcal{G}'=(\mathcal{V}',\,\mathcal{E}')$]
\label{d_G'-def}We define $\mathcal{G}'=(\mathcal{V}',\,\mathcal{E}')$
which is obtained by \emph{contracting} the vertices in $\{x\}\cup\mathcal{N}_{x}$
and $\{y\}\cup\mathcal{N}_{y}$ to two single vertices $x$ and $y$,
respectively. More rigorously, we proceed as follows.
\begin{enumerate}
\item The vertex set is given as $\mathcal{V}':=\mathcal{V}\setminus(\mathcal{N}_{x}\cup\mathcal{N}_{y})$.
\item The edge set $\mathcal{E}'$ is given as follows, where $v,\,w\in\mathcal{V}'\setminus\{x,\,y\}$.
\begin{itemize}
\item $\{v,\,w\}\in\mathcal{E}'$ if and only if $\{v,\,w\}\in\mathcal{E}$.
\item $\{x,\,v\}\in\mathcal{E}'$ if and only if $a\sim v$ for some $a\in\mathcal{N}_{x}$.
\item $\{w,\,y\}\in\mathcal{E}'$ if and only if $w\sim b$ for some $b\in\mathcal{N}_{y}$.
\item $\{x,\,y\}\in\mathcal{E}'$ if and only if $a\sim b$ for some $a\in\mathcal{N}_{x}$
and $b\in\mathcal{N}_{y}$.
\end{itemize}
\item Define $\mathcal{A}_{x}:=\{v\in\mathcal{V}':\{x,\,v\}\in\mathcal{E}'\}$
and $\mathcal{A}_{y}:=\{w\in\mathcal{V}':\{w,\,y\}\in\mathcal{E}'\}$.
Note that it may happen that $\mathcal{A}_{x}\cap\mathcal{A}_{y}\ne\emptyset$.
\end{enumerate}
It is easy to notice that the resulting contracted graph $\mathcal{G}'=(\mathcal{V}',\,\mathcal{E}')$
is still connected.
\end{defn}

If we delete from $\mathcal{G}'$ the vertices $x$ and $y$ and also
the edges that contain $x$ or $y$, then the resulting subgraph is
divided into several connected components. We denote this decomposition
of connected subgraphs as 
\begin{equation}
\mathcal{G}_{1}'=(\mathcal{V}_{1}',\,\mathcal{E}_{1}'),\quad\mathcal{G}_{2}'=(\mathcal{V}_{2}',\,\mathcal{E}_{2}'),\quad\dots,\quad\text{and}\quad\mathcal{G}_{s}'=(\mathcal{V}_{s}',\,\mathcal{E}_{s}').\label{e_Gj'-Vj'-Ej'}
\end{equation}
Then, we have the following decompositions of $\mathcal{V}'$ and
$\mathcal{V}$:
\begin{equation}
\mathcal{V}'=\{x\}\cup\bigcup_{j=1}^{s}\mathcal{V}_{j}'\cup\{y\}\quad\text{and}\quad\mathcal{V}=\{x\}\cup\mathcal{N}_{x}\cup\bigcup_{j=1}^{s}\mathcal{V}_{j}'\cup\mathcal{N}_{y}\cup\{y\}.\label{e_V'-V-dec}
\end{equation}
Now, we fix $j\in\llbracket1,\,s\rrbracket$ and define the infinite
ladder graph $\mathscr{G}_{j}=(\mathscr{V}_{j},\,\mathscr{E}_{j})$
introduced briefly in Section \ref{sec4.1}. Refer to Figure \ref{fig4.4}-left.

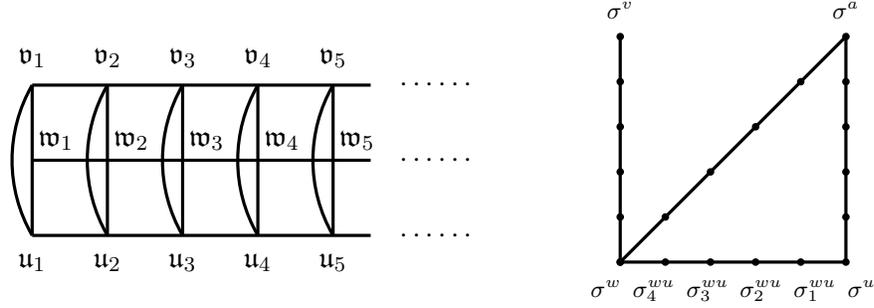
\begin{figure}
\begin{tikzpicture}
\fill[rounded corners,white] (-1,2) rectangle (6,-2); 
\draw (0,1.1) node[above]{$\mathfrak{v}_1$}; \draw (1,1.1) node[above]{$\mathfrak{v}_2$}; \draw (2,1.1) node[above]{$\mathfrak{v}_3$}; \draw (3,1.1) node[above]{$\mathfrak{v}_4$}; \draw (4,1.1) node[above]{$\mathfrak{v}_5$};
\draw (0,-1.1) node[below]{$\mathfrak{u}_1$}; \draw (1,-1.1) node[below]{$\mathfrak{u}_2$}; \draw (2,-1.1) node[below]{$\mathfrak{u}_3$}; \draw (3,-1.1) node[below]{$\mathfrak{u}_4$}; \draw (4,-1.1) node[below]{$\mathfrak{u}_5$};
\draw (0.32,0) node[above]{$\mathfrak{w}_1$}; \draw (1.32,0) node[above]{$\mathfrak{w}_2$}; \draw (2.32,0) node[above]{$\mathfrak{w}_3$}; \draw (3.32,0) node[above]{$\mathfrak{w}_4$}; \draw (4.32,0) node[above]{$\mathfrak{w}_5$};
\draw[very thick] (0,1)--(0,-1); \draw[very thick] (0,1) arc (150:210:2);
\draw[very thick] (1,1)--(1,-1); \draw[very thick] (1,1) arc (150:210:2);
\draw[very thick] (2,1)--(2,-1); \draw[very thick] (2,1) arc (150:210:2);
\draw[very thick] (3,1)--(3,-1); \draw[very thick] (3,1) arc (150:210:2);
\draw[very thick] (4,1)--(4,-1); \draw[very thick] (4,1) arc (150:210:2);
\draw[very thick] (0,1)--(4.5,1); \draw[very thick] (0,0)--(4.5,0); \draw[very thick] (0,-1)--(4.5,-1);
\draw (4.75,1) node[right]{$\cdots\cdots$}; \draw (4.75,0) node[right]{$\cdots\cdots$}; \draw (4.75,-1) node[right]{$\cdots\cdots$};
\end{tikzpicture}
\hspace{10mm}
\begin{tikzpicture}
\draw (-1.5,3.1) node[above]{\footnotesize $\sigma^v$};
\draw (1.5,3.1) node[above]{\footnotesize $\sigma^a$};
\draw (0,-0.1) node[below]{\footnotesize $\sigma^w$\hspace{1.5mm}$\sigma_4^{wu}$\hspace{1.5mm}$\sigma_3^{wu}$\hspace{1.5mm}$\sigma_2^{wu}$\hspace{1.5mm}$\sigma_1^{wu}$\hspace{1.5mm}$\sigma^u$};
\draw[very thick] (-1.5,3)--(-1.5,0)--(1.5,0)--(1.5,3)--(-1.5,0);
\fill (-1.5,0) circle (0.05); \fill (-0.9,0) circle (0.05); \fill (-0.3,0) circle (0.05); \fill (0.3,0) circle (0.05); \fill (0.9,0) circle (0.05); \fill (1.5,0) circle (0.05);
\fill (1.5,0.6) circle (0.05); \fill (1.5,1.2) circle (0.05); \fill (1.5,1.8) circle (0.05); \fill (1.5,2.4) circle (0.05); \fill (1.5,3) circle (0.05);
\fill (0.9,2.4) circle (0.05); \fill (0.3,1.8) circle (0.05); \fill (-0.3,1.2) circle (0.05); \fill (-0.9,0.6) circle (0.05);
\fill (-1.5,3) circle (0.05); \fill (-1.5,2.4) circle (0.05); \fill (-1.5,1.8) circle (0.05); \fill (-1.5,1.2) circle (0.05); \fill (-1.5,0.6) circle (0.05); \fill (-1.5,0) circle (0.05);
\end{tikzpicture}\caption{\label{fig4.4}Infinite graph $\mathscr{G}_{j}=(\mathscr{V}_{j},\,\mathscr{E}_{j})$
where $\mathcal{A}_{x,\,j}\cup\mathcal{A}_{y,\,j}=\{v,\,w,\,u\}$
(left), and an example of the corresponding graph $U_{\ell}$ for
$\ell=5$ with an additional vertex $a\in\mathcal{V}_{j}'\setminus(\mathcal{A}_{x,\,j}\cup\mathcal{A}_{y,\,j})$
(right).}
\end{figure}

\begin{defn}[Infinite graph $\mathscr{G}_{j}=(\mathscr{V}_{j},\,\mathscr{E}_{j})$]
\label{d_infinite-graph-def}We write
\[
\mathcal{A}_{x,\,j}:=\mathcal{A}_{x}\cap\mathcal{V}_{j}'\quad\text{and}\quad\mathcal{A}_{y,\,j}:=\mathcal{A}_{y}\cap\mathcal{V}_{j}'.
\]
Then, we generate a sequence of vertices $\mathfrak{v}_{1},\,\mathfrak{v}_{2},\,\dots$
for each $v\in\mathcal{A}_{x,\,j}\cup\mathcal{A}_{y,\,j}$, and collect
\begin{equation}
\mathscr{V}_{j}:=\bigcup_{v\in\mathcal{A}_{x,\,j}\cup\mathcal{A}_{y,\,j}}\bigcup_{\ell=1}^{\infty}\{\mathfrak{v}_{\ell}\}.\label{e_Vj-def}
\end{equation}
Moreover, the edge set is defined as
\[
\mathscr{E}_{j}:=\bigcup_{v\in\mathcal{A}_{x,\,j}\cup\mathcal{A}_{y,\,j}}\bigcup_{\ell=1}^{\infty}\big\{\{\mathfrak{v}_{\ell},\,\mathfrak{v}_{\ell+1}\}\big\}\cup\bigcup_{\{v,\,w\}\subseteq\mathcal{A}_{x,\,j}\cup\mathcal{A}_{y,\,j}}\bigcup_{\ell=1}^{\infty}\big\{\{\mathfrak{v}_{\ell},\,\mathfrak{w}_{\ell}\}\big\}.
\]
\end{defn}

Now, we wish to define a Markov chain on the infinite graph $\mathscr{G}_{j}=(\mathscr{V}_{j},\,\mathscr{E}_{j})$
characterized by a symmetric transition rate function $\mathfrak{r}=\mathfrak{r}_{j}:\mathscr{V}_{j}\times\mathscr{V}_{j}\to[0,\,\infty)$.
We introduce a parameter
\begin{equation}
\lambda\in(\sqrt{m_{\star}},\,1),\label{e_lambda-def}
\end{equation}
where $m_{\star}\in(0,\,1)$ is defined in \eqref{e_mstar-def}. This
parameter $\lambda$ is a key object in the remainder of the section;
refer to Remark \ref{r_Kxy-lambda-indep} at the end of the section.
We divide the definition of $\mathfrak{r}(\cdot,\,\cdot)$ into two
parts: horizontal transition rates $\mathfrak{r}(\mathfrak{v}_{\ell},\,\mathfrak{v}_{\ell+1})$
for $\ell\ge1$ and $v\in\mathcal{A}_{x,\,j}\cup\mathcal{A}_{y,\,j}$,
and vertical transition rates $\mathfrak{r}(\mathfrak{v}_{\ell},\,\mathfrak{w}_{\ell})$
for $\ell\ge1$ and $\{v,\,w\}\subseteq\mathcal{A}_{x,\,j}\cup\mathcal{A}_{y,\,j}$.
\begin{defn}[Transition rate $\mathfrak{r}$: horizontal direction]
\label{d_trans-hor}Fix $v\in\mathcal{A}_{x,\,j}\cup\mathcal{A}_{y,\,j}$.
Recall that there exists a sequence $\{\mathfrak{v}_{\ell}\}_{\ell\ge1}$
of vertices in $\mathscr{G}_{j}$ that correspond to $v$. For each
$\ell\ge1$, we define (cf. \eqref{e_cxy-def} and \eqref{e_lambda-def})
\[
\mathfrak{r}(\mathfrak{v}_{\ell},\,\mathfrak{v}_{\ell+1})=\mathfrak{r}(\mathfrak{v}_{\ell+1},\,\mathfrak{v}_{\ell}):=\frac{m_{v}^{\ell}}{\lambda^{2\ell+1}}\sum_{a\in\mathcal{N}_{x}\cup\mathcal{N}_{y}}\frac{c_{av}}{1-m_{a}}.
\]
Then, since $v\in\mathcal{A}_{x,\,j}\cup\mathcal{A}_{y,\,j}$, it
is clear that
\begin{equation}
0<\mathfrak{r}(\mathfrak{v}_{\ell},\,\mathfrak{v}_{\ell+1})\le C\frac{m_{v}^{\ell}}{\lambda^{2\ell+1}}\text{ for some constant }C>0.\label{e_trans-hor-bound}
\end{equation}
\end{defn}

\begin{defn}[Transition rate $\mathfrak{r}$: vertical direction]
\label{d_trans-ver}Fix a positive integer $\ell\ge1$. Then, we
consider a symmetric Markov chain $\{\mathscr{X}_{t}^{\ell}\}_{t\ge0}$
as follows. The space $U_{\ell}$ is given as
\[
U_{\ell}:=\Big\{\sigma\in\mathbb{N}^{\mathcal{V}_{j}'}:\sum_{v\in\mathcal{V}_{j}'}\sigma_{v}=\ell,\ \big|\{v\in\mathcal{V}_{j}':\sigma_{v}\ne0\}\big|\le2\Big\},
\]
and the transition rates are given as (cf. \eqref{e_cxy-def})
\begin{equation}
\mathfrak{r}^{\ell}(\sigma_{k}^{vw},\,\sigma_{k-1}^{vw})=\mathfrak{r}^{\ell}(\sigma_{k-1}^{vw},\,\sigma_{k}^{vw}):=m_{v}^{k-1}m_{w}^{\ell-k}c_{vw}\text{ for }v,\,w\in\mathcal{V}_{j}',\ k\in\llbracket1,\,\ell\rrbracket.\label{e_trans-ver-def}
\end{equation}
Here, $\sigma_{k}^{vw}\in U_{\ell}$ is the configuration such that
$\sigma_{k}^{vw}(v)=k$ and $\sigma_{k}^{vw}(w)=\ell-k$. The Markov
chain is irreducible since $\mathcal{V}_{j}'$ is connected. See Figure
\ref{fig4.4}-right.

Note that $\mathcal{A}_{x,\,j}\cup\mathcal{A}_{y,\,j}$ is naturally
embedded into $U_{\ell}$ by a map defined as
\begin{equation}
\mathcal{A}_{x,\,j}\cup\mathcal{A}_{y,\,j}\ni v\mapsto\sigma^{v}\in U_{\ell},\label{e_v-sigma-v}
\end{equation}
where $\sigma^{v}\in U_{\ell}$ is the configuration such that $\sigma^{v}(v)=\ell$.
Thus, with a slight abuse of notation, we may regard $\mathcal{A}_{x,\,j}\cup\mathcal{A}_{y,\,j}$
as a subset of $U_{\ell}$. Then, consider the trace process (cf.
Appendix \ref{appenA}) on $\mathcal{A}_{x,\,j}\cup\mathcal{A}_{y,\,j}$
whose transition rates are notated as $\widehat{\mathfrak{r}}^{\ell}(\sigma^{v},\,\sigma^{w})$
or $\widehat{\mathfrak{r}}_{\sigma^{v}\sigma^{w}}^{\ell}$ for $v$
and $w$ in $\mathcal{A}_{x,\,j}\cup\mathcal{A}_{y,\,j}$. The trace
process is also symmetric. By \eqref{e_RW-UB}, it holds for all $v,\,w\in\mathcal{A}_{x,\,j}\cup\mathcal{A}_{y,\,j}$
that
\begin{equation}
\widehat{\mathfrak{r}}^{\ell}(\sigma^{v},\,\sigma^{w})\le\sum_{\sigma\in U_{\ell}}\mathfrak{r}^{\ell}(\sigma^{v},\,\sigma)\le Cm_{\star}^{\ell-1},\label{e_r-hat-ell-UB}
\end{equation}
where $C>0$ is a global constant which depends only on the structure
of the original graph $\mathcal{G}=(\mathcal{V},\,\mathcal{E})$.

Finally, we define the vertical transition rates as
\[
\mathfrak{r}(\mathfrak{v}_{\ell},\,\mathfrak{w}_{\ell})=\mathfrak{r}(\mathfrak{w}_{\ell},\,\mathfrak{v}_{\ell}):=\frac{1}{\lambda^{2\ell}}\widehat{\mathfrak{r}}^{\ell}(\sigma^{v},\,\sigma^{w})\text{ for }\{v,\,w\}\subseteq\mathcal{A}_{x,\,j}\cup\mathcal{A}_{y,\,j}.
\]
Then, by \eqref{e_r-hat-ell-UB}, there exists a constant $C>0$ such
that
\begin{equation}
\mathfrak{r}(\mathfrak{v}_{\ell},\,\mathfrak{w}_{\ell})\le C\frac{m_{\star}^{\ell-1}}{\lambda^{2\ell-1}}\text{ for all }\{v,\,w\}\subseteq\mathcal{A}_{x,\,j}\cup\mathcal{A}_{y,\,j}.\label{e_trans-ver-bound}
\end{equation}
\end{defn}

Now, we state and solve a resolvent equation on $\mathscr{G}_{j}=(\mathscr{V}_{j},\,\mathscr{E}_{j})$.
We define two auxiliary functions $\mathfrak{f}:\mathscr{V}_{j}\to\mathbb{R}$
and $\mathfrak{h}:\mathscr{V}_{j}\to\mathbb{R}$ as
\[
\mathfrak{f}(\mathfrak{v}_{\ell}):=\begin{cases}
(\frac{1}{\lambda^{2}}-\frac{m_{v}}{\lambda^{3}}+\frac{m_{v}}{\lambda^{2}})\sum_{a\in\mathcal{N}_{x}\cup\mathcal{N}_{y}}\frac{c_{av}}{1-m_{a}} & \text{if }\ell=1,\\
(\frac{1}{\lambda}-1)(1-\frac{m_{v}}{\lambda})\frac{m_{v}^{\ell-1}}{\lambda^{2\ell-1}}\sum_{a\in\mathcal{N}_{x}\cup\mathcal{N}_{y}}\frac{c_{av}}{1-m_{a}} & \text{if }\ell\ge2,
\end{cases}
\]
and
\[
\mathfrak{h}(\mathfrak{v}_{\ell}):=(1-m_{v})\cdot\frac{m_{v}^{\ell-1}}{\lambda^{\ell}}\sum_{a\in\mathcal{N}_{x}}\frac{c_{av}}{1-m_{a}}.
\]
Since $v\in\mathcal{A}_{x,\,j}\cup\mathcal{A}_{y,\,j}$, we see that
$\mathfrak{f}(\mathfrak{v}_{\ell})>0$. Moreover, it is straightforward
that there exists a global constant $C>0$ such that for all $\ell\ge1$,
\begin{equation}
\frac{1}{C}\frac{m_{v}^{\ell-1}}{\lambda^{2\ell-1}}\le\mathfrak{f}(\mathfrak{v}_{\ell})\le C\frac{m_{v}^{\ell-1}}{\lambda^{2\ell-1}}\quad\text{and}\quad0\le\mathfrak{h}(\mathfrak{v}_{\ell})\le C\frac{m_{v}^{\ell-1}}{\lambda^{\ell}}.\label{e_f-h-bound}
\end{equation}
The proof of the following theorem can be skipped at first reading.
\begin{thm}[Resolvent equation on $\mathscr{G}_{j}=(\mathscr{V}_{j},\,\mathscr{E}_{j})$]
\label{t_resolvent-equation}Consider the Markov chain $\{\mathfrak{X}_{t}\}_{t\ge0}$
on $\mathscr{V}_{j}$ corresponding to $\mathfrak{r}(\cdot,\,\cdot)$,
which induces an infinitesimal generator $\mathscr{L}_{j}$.
\begin{enumerate}
\item Fix an integer $L\ge1$ and consider the reflected chain at level
$L$, which is a Markov chain $\{\mathfrak{X}_{t}^{L}\}_{t\ge0}$
on $\mathscr{V}_{j}^{L}:=\{\mathfrak{v}_{\ell}:v\in\mathcal{A}_{x,\,j}\cup\mathcal{A}_{y,\,j},\ \ell\in\llbracket1,\,L\rrbracket\}\subseteq\mathscr{V}_{j}$
with infinitesimal generator $\mathscr{L}_{j}^{L}$ obtained by annihilating
the transition rates outside $\mathscr{V}_{j}^{L}$. Define $\mathfrak{f}^{L}:\mathscr{V}_{j}^{L}\to\mathbb{R}$
and $\mathfrak{h}^{L}:\mathscr{V}_{j}^{L}\to\mathbb{R}$ as
\[
\mathfrak{f}^{L}(\mathfrak{v}_{\ell}):=\begin{cases}
\mathfrak{f}(\mathfrak{v}_{\ell}) & \text{if }\ell\in\llbracket1,\,L-1\rrbracket,\\
(\frac{1}{\lambda}-1)\frac{m_{v}^{L-1}}{\lambda^{2L-1}}\sum_{a\in\mathcal{N}_{x}\cup\mathcal{N}_{y}}\frac{c_{av}}{1-m_{a}} & \text{if }\ell=L,
\end{cases}
\]
and
\[
\mathfrak{h}^{L}(\mathfrak{v}_{\ell}):=\begin{cases}
\mathfrak{h}(\mathfrak{v}_{\ell}) & \text{if }\ell\in\llbracket1,\,L-1\rrbracket,\\
\frac{m_{v}^{L-1}}{\lambda^{L}}\sum_{a\in\mathcal{N}_{x}}\frac{c_{av}}{1-m_{a}} & \text{if }\ell=L.
\end{cases}
\]
Then, there exists a unique solution $\mathfrak{g}_{0}^{L}:\mathscr{V}_{j}^{L}\to\mathbb{R}$
to the following resolvent equation:
\begin{equation}
(\mathfrak{f}^{L}-\mathscr{L}_{j}^{L})\mathfrak{g}_{0}^{L}=\mathfrak{h}^{L}\text{ on }\mathscr{V}_{j}^{L}.\label{e_resolvent-L}
\end{equation}
Moreover, the solutions $\{\mathfrak{g}_{0}^{L}\}_{L\ge1}$ are uniformly
bounded.
\item The limit $\mathfrak{g}_{0}:=\lim_{L\to\infty}\mathfrak{g}_{0}^{L}$
exists, and it solves the following resolvent equation:
\begin{equation}
(\mathfrak{f}-\mathscr{L}_{j})\mathfrak{g}_{0}=\mathfrak{h}\text{ on }\mathscr{V}_{j}.\label{e_resolvent}
\end{equation}
\end{enumerate}
\end{thm}

\begin{proof}
(1) From \eqref{e_f-h-bound} and the fact that $\mathscr{V}_{j}^{L}$
is finite, it holds that $\mathfrak{h}^{L}$ is bounded and $\mathfrak{f}^{L}$
is bounded away from $0$ on $\mathscr{V}_{j}^{L}$. Thus, it is well
known, e.g. \cite[(4.1)]{LMS resolvent}, that the resolvent equation
\eqref{e_resolvent-L} admits a unique solution $\mathfrak{g}_{0}^{L}$
which can be written as
\begin{equation}
\mathfrak{g}_{0}^{L}(\cdot)=\mathrm{E}_{\cdot}^{L}\Big[\int_{0}^{\infty}e^{-\int_{0}^{t}\mathfrak{f}^{L}(\mathfrak{X}_{s}^{L})ds}\mathfrak{h}^{L}(\mathfrak{X}_{t}^{L})\mathrm{d}t\Big],\label{e_t_resolvent-equation-1}
\end{equation}
where $\mathrm{E}_{\cdot}^{L}$ is the expectation corresponding to
the law of $\{\mathfrak{X}_{t}^{L}\}_{t\ge0}$. Thus, to conclude
the proof of part (1), it remains to prove that the solutions are
uniformly bounded. 

It is easy to see that $\mathfrak{f}^{L}$ and $\mathfrak{h}^{L}$
also satisfy the decaying conditions as in \eqref{e_f-h-bound}: for
some constant $C_{1}>0$ independent of $L$, it holds for all $\ell\in\llbracket1,\,L\rrbracket$
that
\begin{equation}
\frac{1}{C_{1}}\frac{m_{v}^{\ell-1}}{\lambda^{2\ell-1}}\le\mathfrak{f}^{L}(\mathfrak{v}_{\ell})\le C_{1}\frac{m_{v}^{\ell-1}}{\lambda^{2\ell-1}}\quad\text{and}\quad0\le\mathfrak{h}^{L}(\mathfrak{v}_{\ell})\le C_{1}\frac{m_{v}^{\ell-1}}{\lambda^{\ell}}.\label{e_fL-hL-bound}
\end{equation}
Moreover, denote by $H^{L}(\mathfrak{v}_{\ell})$ the holding rate
of $\{\mathfrak{X}_{t}^{L}\}_{t\ge0}$ at $\mathfrak{v}_{\ell}$.
Then, by \eqref{e_trans-hor-bound} and \eqref{e_trans-ver-bound},
there exists a constant $C_{2}>0$ such that\textbf{
\begin{equation}
0<H^{L}(\mathfrak{v}_{\ell})\le C_{2}\frac{m_{v}^{\ell-1}}{\lambda^{2\ell-1}}.\label{e_HL-bound}
\end{equation}
}Now, fix $\mathfrak{v}_{\ell}\in\mathscr{V}_{j}^{L}$. Denoting by
$0=\tau_{0}<\tau_{1}<\tau_{2}<\cdots$ the consecutive jump times
of the process, we calculate
\[
\begin{aligned}\mathfrak{g}_{0}^{L}(\mathfrak{v}_{\ell}) & =\mathrm{E}_{\mathfrak{v}_{\ell}}^{L}\Big[\sum_{p=0}^{\infty}\int_{\tau_{p}}^{\tau_{p+1}}e^{-\int_{0}^{t}\mathfrak{f}^{L}(\mathfrak{X}_{s}^{L})ds}\mathfrak{h}^{L}(\mathfrak{X}_{t}^{L})\mathrm{d}t\Big]\\
 & =\sum_{p=0}^{\infty}\mathrm{E}_{\mathfrak{v}_{\ell}}^{L}\Big[\prod_{q=0}^{p-1}e^{-\int_{\tau_{q}}^{\tau_{q+1}}\mathfrak{f}^{L}(\mathfrak{X}_{s}^{L})\mathrm{d}s}\cdot\int_{\tau_{p}}^{\tau_{p+1}}e^{-\int_{\tau_{p}}^{t}\mathfrak{f}^{L}(\mathfrak{X}_{s}^{L})\mathrm{d}s}\mathfrak{h}^{L}(\mathfrak{X}_{t}^{L})\mathrm{d}t\Big].
\end{aligned}
\]
By the strong Markov property (applied $p$ times in a row), this
becomes
\begin{equation}
\sum_{p=0}^{\infty}\mathrm{E}_{\mathfrak{v}_{\ell}}^{L}\Big[\prod_{q=0}^{p-1}\mathrm{E}_{\mathfrak{X}_{\tau_{q}}^{L}}^{L}\big[e^{-\int_{0}^{\tau_{1}}\mathfrak{f}^{L}(\mathfrak{X}_{s}^{L})\mathrm{d}s}\big]\cdot\mathrm{E}_{\mathfrak{X}_{\tau_{p}}^{L}}^{L}\Big[\int_{0}^{\tau_{1}}e^{-\int_{0}^{t}\mathfrak{f}^{L}(\mathfrak{X}_{s}^{L})\mathrm{d}s}\mathfrak{h}^{L}(\mathfrak{X}_{t}^{L})\mathrm{d}t\Big]\Big].\label{e_t_resolvent-equation-2}
\end{equation}
Next, for any $\mathfrak{w}_{\ell}\in\mathscr{V}_{j}^{L}$, since
$\mathfrak{X}_{t}^{L}=\mathfrak{X}_{0}^{L}$ for $0\le t<\tau_{1}$,
we may calculate
\[
\begin{aligned}\mathrm{E}_{\mathfrak{w}_{\ell}}^{L}\Big[\int_{0}^{\tau_{1}}e^{-\int_{0}^{t}\mathfrak{f}^{L}(\mathfrak{X}_{s}^{L})\mathrm{d}s}\mathfrak{h}^{L}(\mathfrak{X}_{t}^{L})\mathrm{d}t\Big] & =\mathfrak{h}^{L}(\mathfrak{w}_{\ell})\cdot\mathrm{E}_{\mathfrak{w}_{\ell}}^{L}\Big[\int_{0}^{\tau_{1}}e^{-\mathfrak{f}^{L}(\mathfrak{w}_{\ell})t}\mathrm{d}t\Big]\\
 & =\mathfrak{h}^{L}(\mathfrak{w}_{\ell})\cdot\mathrm{E}_{\mathfrak{w}_{\ell}}^{L}\Big[\frac{1-e^{-\mathfrak{f}^{L}(\mathfrak{w}_{\ell})\tau_{1}}}{\mathfrak{f}^{L}(\mathfrak{w}_{\ell})}\Big].
\end{aligned}
\]
Since the distribution of $\tau_{1}$ is exponential with mean $(H^{L}(\mathfrak{w}_{L}))^{-1}$,
this can be written as
\begin{equation}
\frac{\mathfrak{h}^{L}(\mathfrak{w}_{\ell})}{\mathfrak{f}^{L}(\mathfrak{w}_{\ell})}\cdot\frac{\mathfrak{f}^{L}(\mathfrak{w}_{\ell})}{H^{L}(\mathfrak{w}_{\ell})+\mathfrak{f}^{L}(\mathfrak{w}_{\ell})}<C\lambda^{\ell-1}\le C.\label{e_t_resolvent-equation-3}
\end{equation}
Here, the first inequality follows from \eqref{e_fL-hL-bound}. Similarly,
by \eqref{e_fL-hL-bound} and \eqref{e_HL-bound},
\begin{equation}
\mathrm{E}_{\mathfrak{w}_{\ell}}^{L}\big[e^{-\int_{0}^{\tau_{1}}\mathfrak{f}^{L}(\mathfrak{X}_{s}^{L})\mathrm{d}s}\big]=\mathrm{E}_{\mathfrak{w}_{\ell}}^{L}\big[e^{-\tau_{1}\mathfrak{f}^{L}(\mathfrak{w}_{\ell})}\big]=\frac{H^{L}(\mathfrak{w}_{\ell})}{H^{L}(\mathfrak{w}_{\ell})+\mathfrak{f}^{L}(\mathfrak{w}_{\ell})}\le\frac{C_{1}C_{2}}{C_{1}C_{2}+1}<1.\label{e_t_resolvent-equation-4}
\end{equation}
Thus, substituting \eqref{e_t_resolvent-equation-3} and \eqref{e_t_resolvent-equation-4}
to \eqref{e_t_resolvent-equation-2}, we obtain that
\[
0\le\mathfrak{g}_{0}^{L}(\mathfrak{v}_{\ell})\le\sum_{p=0}^{\infty}C\Big(\frac{C_{1}C_{2}}{C_{1}C_{2}+1}\Big)^{p}=C(C_{1}C_{2}+1)<\infty,
\]
which concludes the proof of part (1).

\noindent (2) Fix a vertex $\mathfrak{v}_{\ell}\in\mathscr{V}_{j}$
where $v\in\mathcal{A}_{x,\,j}\cup\mathcal{A}_{y,\,j}$ and $\ell\ge1$.
We show that the sequence $\{\mathfrak{g}_{0}^{L}(\mathfrak{v}_{\ell})\}_{L\ge\ell}$
is a Cauchy sequence, and thus converges. To this end, take $L'\ge L\ge\ell$.
Consider the natural coupling of two processes $\{\mathfrak{X}_{t}^{L}\}_{t\ge0}$
and $\{\mathfrak{X}_{t}^{L'}\}_{t\ge0}$, such that the trajectories
remain the same up to the hitting time $\mathcal{T}_{L}$ of the set
$\{\mathfrak{w}_{L}:w\in\mathcal{A}_{x,\,j}\cup\mathcal{A}_{y,\,j}\}$.
Then, by \eqref{e_t_resolvent-equation-1}, it holds that
\[
\begin{aligned}|\mathfrak{g}_{0}^{L}(\mathfrak{v}_{\ell})-\mathfrak{g}_{0}^{L'}(\mathfrak{v}_{\ell})| & =\Big|\mathrm{E}_{\mathfrak{v}_{\ell}}^{L}\Big[\int_{\mathcal{T}_{L}}^{\infty}e^{-\int_{0}^{t}\mathfrak{f}^{L}(\mathfrak{X}_{s}^{L})\mathrm{d}s}\mathfrak{h}^{L}(\mathfrak{X}_{t}^{L})\mathrm{d}t\Big]-\mathrm{E}_{\mathfrak{v}_{\ell}}^{L'}\Big[\int_{\mathcal{T}_{L}}^{\infty}e^{-\int_{0}^{t}\mathfrak{f}^{L'}(\mathfrak{X}_{s}^{L'})\mathrm{d}s}\mathfrak{h}^{L'}(\mathfrak{X}_{t}^{L'})\mathrm{d}t\Big]\Big|\\
 & \le\mathrm{E}_{\mathfrak{v}_{\ell}}^{L}\Big[\int_{\mathcal{T}_{L}}^{\infty}e^{-\int_{0}^{t}\mathfrak{f}^{L}(\mathfrak{X}_{s}^{L})\mathrm{d}s}\mathfrak{h}^{L}(\mathfrak{X}_{t}^{L})\mathrm{d}t\Big]+\mathrm{E}_{\mathfrak{v}_{\ell}}^{L'}\Big[\int_{\mathcal{T}_{L}}^{\infty}e^{-\int_{0}^{t}\mathfrak{f}^{L'}(\mathfrak{X}_{s}^{L'})\mathrm{d}s}\mathfrak{h}^{L'}(\mathfrak{X}_{t}^{L'})\mathrm{d}t\Big].
\end{aligned}
\]
First, we consider the first expectation in the right-hand side. By
the strong Markov property at stopping time $\mathcal{T}_{L}$, we
calculate
\[
\begin{aligned}\mathrm{E}_{\mathfrak{v}_{\ell}}^{L}\Big[\int_{\mathcal{T}_{L}}^{\infty}e^{-\int_{0}^{t}\mathfrak{f}^{L}(\mathfrak{X}_{s}^{L})\mathrm{d}s}\mathfrak{h}^{L}(\mathfrak{X}_{t}^{L})\mathrm{d}t\Big] & =\mathrm{E}_{\mathfrak{v}_{\ell}}^{L}\Big[e^{-\int_{0}^{\mathcal{T}_{L}}\mathfrak{f}^{L}(\mathfrak{X}_{s}^{L})\mathrm{d}s}\cdot\int_{\mathcal{T}_{L}}^{\infty}e^{-\int_{\mathcal{T}_{L}}^{t}\mathfrak{f}^{L}(\mathfrak{X}_{s}^{L})\mathrm{d}s}\mathfrak{h}^{L}(\mathfrak{X}_{t}^{L})\mathrm{d}t\Big]\\
 & =\mathrm{E}_{\mathfrak{v}_{\ell}}^{L}\Big[e^{-\int_{0}^{\mathcal{T}_{L}}\mathfrak{f}^{L}(\mathfrak{X}_{s}^{L})\mathrm{d}s}\cdot\mathrm{E}_{\mathfrak{X}_{\mathcal{T}_{L}}^{L}}^{L}\Big[\int_{0}^{\infty}e^{-\int_{0}^{t}\mathfrak{f}^{L}(\mathfrak{X}_{s}^{L})\mathrm{d}s}\mathfrak{h}^{L}(\mathfrak{X}_{t}^{L})\mathrm{d}t\Big]\Big]\\
 & =\mathrm{E}_{\mathfrak{v}_{\ell}}^{L}\big[e^{-\int_{0}^{\mathcal{T}_{L}}\mathfrak{f}^{L}(\mathfrak{X}_{s}^{L})\mathrm{d}s}\cdot\mathfrak{g}_{0}^{L}(\mathfrak{X}_{\mathcal{T}_{L}}^{L})\big]\le C\cdot\mathrm{E}_{\mathfrak{v}_{\ell}}^{L}\big[e^{-\int_{0}^{\mathcal{T}_{L}}\mathfrak{f}^{L}(\mathfrak{X}_{s}^{L})\mathrm{d}s}\big],
\end{aligned}
\]
where the inequality holds by the uniform boundedness proved in part
(1). By the strong Markov property and \eqref{e_t_resolvent-equation-4},
since $\mathcal{T}_{L}\ge\tau_{L-\ell}$, the last term is bounded
by
\[
C\cdot\mathrm{E}_{\mathfrak{v}_{\ell}}^{L}\Big[\prod_{q=0}^{L-\ell-1}e^{-\int_{\tau_{q}}^{\tau_{q+1}}\mathfrak{f}^{L}(\mathfrak{X}_{s}^{L})\mathrm{d}s}\Big]\le C\Big(\frac{C_{1}C_{2}}{C_{1}C_{2}+1}\Big)^{L-\ell}\to0\text{ as }L\to\infty.
\]
Similarly, we obtain that
\[
\mathrm{E}_{\mathfrak{v}_{\ell}}^{L'}\Big[\int_{\mathcal{T}_{L}}^{\infty}e^{-\int_{0}^{t}\mathfrak{f}^{L'}(\mathfrak{X}_{s}^{L'})\mathrm{d}s}\mathfrak{h}^{L'}(\mathfrak{X}_{t}^{L'})\mathrm{d}t\Big]\le C\cdot\mathrm{E}_{\mathfrak{v}_{\ell}}^{L'}\Big[e^{-\int_{0}^{\mathcal{T}_{L}}\mathfrak{f}^{L'}(\mathfrak{X}_{s}^{L'})\mathrm{d}s}\Big]\le C\Big(\frac{C_{1}C_{2}}{C_{1}C_{2}+1}\Big)^{L'-\ell}\xrightarrow{L\to\infty}0.
\]
Therefore, we have proved that $\{\mathfrak{g}_{0}^{L}(\mathfrak{v}_{\ell})\}_{L\ge\ell}$
is a Cauchy sequence, and thus it converges to some limit function
which is denoted as $\mathfrak{g}_{0}:\mathscr{V}_{j}\to\mathbb{R}$.

It remains to prove that $\mathfrak{g}_{0}$ solves \eqref{e_resolvent}.
For every $v\in\mathcal{A}_{x,\,j}\cup\mathcal{A}_{y,\,j}$ and $\ell\ge1$,
it holds that
\[
\mathfrak{f}(\mathfrak{v}_{\ell})\mathfrak{g}_{0}^{L}(\mathfrak{v}_{\ell})-(\mathscr{L}_{j}\mathfrak{g}_{0}^{L})(\mathfrak{v}_{\ell})=\mathfrak{h}(\mathfrak{v}_{\ell})\text{ for all }L\ge\ell+1.
\]
Sending $L$ to infinity, we obtain that
\[
\mathfrak{f}(\mathfrak{v}_{\ell})\mathfrak{g}_{0}(\mathfrak{v}_{\ell})-(\mathscr{L}_{j}\mathfrak{g}_{0})(\mathfrak{v}_{\ell})=\mathfrak{h}(\mathfrak{v}_{\ell}),
\]
and thus $(\mathfrak{f}-\mathscr{L}_{j})\mathfrak{g}_{0}=\mathfrak{h}$.
\end{proof}
Now, we define several functions which are the building blocks of
the test objects to be defined in Sections \ref{sec5.1} and \ref{sec6.1}.
\begin{defn}
\label{d_g-ghat-def}Define a function $\mathfrak{g}=\mathfrak{g}_{j}:\mathscr{V}_{j}\to\mathbb{R}$
as
\begin{equation}
\mathfrak{g}(\mathfrak{v}_{\ell}):=\frac{1}{\lambda^{\ell}}\mathfrak{g}_{0}(\mathfrak{v}_{\ell})\text{ for all }v\in\mathcal{A}_{x,\,j}\cup\mathcal{A}_{y,\,j},\ \ell\ge1.\label{e_g-def}
\end{equation}
Recall \eqref{e_Vj-def}. We fix $\ell\ge1$ and restrict $\mathfrak{g}$
to $\bigcup_{v\in\mathcal{A}_{x,\,j}\cup\mathcal{A}_{y,\,j}}\{\mathfrak{v}_{\ell}\}$,
which is naturally embedded into $U_{\ell}$ by the map given in \eqref{e_v-sigma-v}.
Thus, $\mathfrak{g}$ can be regarded as defined on a subset of $U_{\ell}$.
We denote by $\widehat{\mathfrak{g}}_{\ell}=\widehat{\mathfrak{g}}_{\ell,\,j}:U_{\ell}\to\mathbb{R}$
the \emph{harmonic extension} (explained in Appendix \ref{appenA})
of $\mathfrak{g}$ to the whole graph $U_{\ell}$ with respect to
the Markov chain $\{\mathscr{X}_{t}^{\ell}\}_{t\ge0}$.

For later use, we also define $\widehat{\mathfrak{g}}_{0}:\mathcal{X}_{0}\to\mathbb{R}$
as $\widehat{\mathfrak{g}}_{0}(\boldsymbol{0}):=0$, where $\boldsymbol{0}$
is the single element of $\mathcal{X}_{0}$.
\end{defn}

\begin{defn}
\label{d_gL-ghatL-def}Similarly, we define $\mathfrak{g}^{L}=\mathfrak{g}_{j}^{L}:\mathscr{V}_{j}^{L}\to\mathbb{R}$
as
\begin{equation}
\mathfrak{g}^{L}(\mathfrak{v}_{\ell}):=\frac{1}{\lambda^{\ell}}\mathfrak{g}_{0}^{L}(\mathfrak{v}_{\ell})\text{ for all }v\in\mathcal{A}_{x,\,j}\cup\mathcal{A}_{y,\,j},\ \ell\in\llbracket1,\,L\rrbracket.\label{e_gL-def}
\end{equation}
Also, for $\ell\in\llbracket1,\,L\rrbracket$, we denote by $\widehat{\mathfrak{g}}_{\ell}^{L}=\widehat{\mathfrak{g}}_{\ell,\,j}^{L}:U_{\ell}\to\mathbb{R}$
the \emph{harmonic extension} of $\mathfrak{g}^{L}$ on $\bigcup_{v\in\mathcal{A}_{x,\,j}\cup\mathcal{A}_{y,\,j}}\{\mathfrak{v}_{\ell}\}$
to the whole graph $U_{\ell}$ with respect to $\{\mathscr{X}_{t}^{\ell}\}_{t\ge0}$,
and define $\widehat{\mathfrak{g}}_{0}^{L}:\mathcal{X}_{0}\to\mathbb{R}$
as $\widehat{\mathfrak{g}}_{0}^{L}(\boldsymbol{0}):=0$.
\end{defn}

The two functions $\mathfrak{g}$ and $\widehat{\mathfrak{g}}$ are
crucially used in the definition of the test function in Section \ref{sec5.1}.
On the other hand, the two functions $\mathfrak{g}^{L}$ and $\widehat{\mathfrak{g}}^{L}$
are constructed for the test flow to be defined in Section \ref{sec6.1}.

Finally, recall from Theorem \ref{t_resolvent-equation} that $\mathfrak{g}_{0}$
and $\mathfrak{g}_{0}^{L}$ are uniformly bounded. Thus, by \eqref{e_g-def}
and \eqref{e_gL-def},
\begin{equation}
0\le\mathfrak{g}(\mathfrak{v}_{\ell}),\,\mathfrak{g}^{L}(\mathfrak{v}_{\ell})\le\frac{C}{\lambda^{\ell}}\text{ for all }v\in\mathcal{A}_{x,\,j}\cup\mathcal{A}_{y,\,j}\text{ and }\ell\ge1.\label{e_g-gL-UB}
\end{equation}
Since $\widehat{\mathfrak{g}}_{\ell}$ (resp. $\widehat{\mathfrak{g}}_{\ell}^{L}$)
is the harmonic extension of $\mathfrak{g}$ (resp. $\mathfrak{g}^{L}$),
by \eqref{e_har-extension}, we also have that
\begin{equation}
0\le\widehat{\mathfrak{g}}_{\ell}(\sigma),\,\widehat{\mathfrak{g}}_{\ell}^{L}(\sigma)\le\frac{C}{\lambda^{\ell}}\text{ for all }\sigma\in U_{\ell}\text{ and }\ell\ge1.\label{e_ghat-ghatL-UB}
\end{equation}

\subsection{\label{sec4.3}Properties of $\mathfrak{g}$ and $\mathfrak{g}^{L}$}

In this subsection, we record some properties of the functions $\mathfrak{g}:\mathscr{V}_{j}\to\mathbb{R}$
and $\mathfrak{g}^{L}:\mathscr{V}_{j}^{L}\to\mathbb{R}$. The proofs
are elementary but technical, so we defer them to Appendix \ref{appenB}.
\begin{lem}
\label{l_g-gL-prop-1}For each $j\in\llbracket1,\,s\rrbracket$ and
$v\in\mathcal{A}_{x,\,j}\cup\mathcal{A}_{y,\,j}$, the following statements
hold.
\begin{enumerate}
\item For $\ell\in\llbracket1,\,L-1\rrbracket$, it holds that
\begin{equation}
\begin{aligned} & (1-m_{v})\cdot\sum_{a\in\mathcal{N}_{x}}\frac{c_{av}m_{v}^{\ell-1}}{1-m_{a}}+\Big(\sum_{a\in\mathcal{N}_{x}\cup\mathcal{N}_{y}}\frac{c_{av}m_{v}^{\ell-1}}{1-m_{a}}\Big)\cdot(\mathfrak{g}^{L}(\mathfrak{v}_{\ell-1})-\mathfrak{g}^{L}(\mathfrak{v}_{\ell}))\\
 & +\Big(\sum_{a\in\mathcal{N}_{x}\cup\mathcal{N}_{y}}\frac{c_{av}m_{v}^{\ell}}{1-m_{a}}\Big)\cdot(\mathfrak{g}^{L}(\mathfrak{v}_{\ell+1})-\mathfrak{g}^{L}(\mathfrak{v}_{\ell}))+\sum_{w\in\mathcal{A}_{x,\,j}\cup\mathcal{A}_{y,\,j}}\widehat{\mathfrak{r}}_{\sigma^{v}\sigma^{w}}^{\ell}((\mathfrak{g}^{L}(\mathfrak{w}_{\ell})-\mathfrak{g}^{L}(\mathfrak{v}_{\ell}))=0,
\end{aligned}
\label{e_l_g-gL-prop-1-1}
\end{equation}
where $\mathfrak{g}^{L}(\mathfrak{v}_{0}):=0$.
\item It holds that
\begin{equation}
\begin{aligned} & \sum_{a\in\mathcal{N}_{x}}\frac{c_{av}m_{v}^{L-1}}{1-m_{a}}+\Big(\sum_{a\in\mathcal{N}_{x}\cup\mathcal{N}_{y}}\frac{c_{av}m_{v}^{L-1}}{1-m_{a}}\Big)\cdot(\mathfrak{g}^{L}(\mathfrak{v}_{L-1})-\mathfrak{g}^{L}(\mathfrak{v}_{L}))\\
 & +\sum_{w\in\mathcal{A}_{x,\,j}\cup\mathcal{A}_{y,\,j}}\widehat{\mathfrak{r}}_{\sigma^{v}\sigma^{w}}^{L}(\mathfrak{g}^{L}(\mathfrak{w}_{L})-\mathfrak{g}^{L}(\mathfrak{v}_{L}))=0.
\end{aligned}
\label{e_l_g-gL-prop-1-2}
\end{equation}
\end{enumerate}
\end{lem}

\begin{lem}
\label{l_g-gL-prop-2}For every $\ell\in\llbracket1,\,L\rrbracket$,
it holds that
\[
\sum_{a\in\mathcal{N}_{x}}\sum_{b\in\mathcal{A}_{x}}\frac{c_{ab}m_{b}^{\ell-1}}{1-m_{a}}\cdot(1+\mathfrak{g}^{L}(\mathfrak{b}_{\ell-1})-\mathfrak{g}^{L}(\mathfrak{b}_{\ell}))=\sum_{b\in\mathcal{N}_{y}}\sum_{a\in\mathcal{A}_{y}}\frac{c_{ab}m_{a}^{\ell-1}}{1-m_{b}}\cdot(\mathfrak{g}^{L}(\mathfrak{a}_{\ell})-\mathfrak{g}^{L}(\mathfrak{a}_{\ell-1})),
\]
where $\mathfrak{g}^{L}(\mathfrak{a}_{0}):=0$ and $\mathfrak{g}^{L}(\mathfrak{b}_{0}):=0$.
\end{lem}

\begin{lem}
\label{l_g-gL-prop-3}For $j\in\llbracket1,\,s\rrbracket$ and $b\in\mathcal{A}_{x,\,j}$,
it holds that
\[
\sum_{\ell=0}^{L-1}m_{b}^{\ell}(1+\mathfrak{g}^{L}(\mathfrak{b}_{\ell})-\mathfrak{g}^{L}(\mathfrak{b}_{\ell+1}))\xrightarrow{L\to\infty}\sum_{\ell=0}^{\infty}m_{b}^{\ell}(1+\mathfrak{g}(\mathfrak{b}_{\ell})-\mathfrak{g}(\mathfrak{b}_{\ell+1})).
\]
Similarly, for $a\in\mathcal{A}_{y,\,j}$, it holds that
\[
\sum_{\ell=0}^{L-1}m_{a}^{\ell}(\mathfrak{g}^{L}(\mathfrak{a}_{\ell+1})-\mathfrak{g}^{L}(\mathfrak{a}_{\ell}))\xrightarrow{L\to\infty}\sum_{\ell=0}^{\infty}m_{a}^{\ell}(\mathfrak{g}(\mathfrak{a}_{\ell+1})-\mathfrak{g}(\mathfrak{a}_{\ell})).
\]
\end{lem}

Now, we define a constant $\mathfrak{K}_{xy}\in(0,\,\infty)$ as
\begin{equation}
\begin{aligned} & \frac{1}{6\mathfrak{K}_{xy}}:=\sum_{a\in\mathcal{N}_{x}}\sum_{b\in\mathcal{N}_{y}}\frac{c_{ab}}{(1-m_{a})(1-m_{b})}+\sum_{a\in\mathcal{N}_{x}}\sum_{b\in\mathcal{A}_{x}}\frac{c_{ab}}{1-m_{a}}\sum_{\ell=0}^{\infty}m_{b}^{\ell}\{1+\mathfrak{g}(\mathfrak{b}_{\ell})-\mathfrak{g}(\mathfrak{b}_{\ell+1})\}^{2}\\
 & +\sum_{a\in\mathcal{A}_{y}}\sum_{b\in\mathcal{N}_{y}}\frac{c_{ab}}{1-m_{b}}\sum_{\ell=0}^{\infty}m_{a}^{\ell}\{\mathfrak{g}(\mathfrak{a}_{\ell+1})-\mathfrak{g}(\mathfrak{a}_{\ell})\}^{2}+\sum_{j=1}^{s}\sum_{\{a,\,b\}\subseteq\mathcal{A}_{x,\,j}\cup\mathcal{A}_{y,\,j}}\sum_{\ell=1}^{\infty}\widehat{\mathfrak{r}}_{\sigma^{a}\sigma^{b}}^{\ell}\{\mathfrak{g}(\mathfrak{b}_{\ell})-\mathfrak{g}(\mathfrak{a}_{\ell})\}^{2},
\end{aligned}
\label{e_Kxy-def}
\end{equation}
where $\mathfrak{g}(\mathfrak{a}_{0}):=0$ and $\mathfrak{g}(\mathfrak{b}_{0}):=0$.
The existence of the infinite summations in the right-hand side are
guaranteed by \eqref{e_g-gL-UB}, since the summands are bounded by
$C(\frac{m_{\star}}{\lambda^{2}})^{\ell}$ where $m_{\star}<\lambda^{2}$.

Since the graph $\mathcal{G}$ is connected, it is straightforward
to see that the right-hand side of \eqref{e_Kxy-def} is strictly
positive, which then implies that $\mathfrak{K}_{xy}$ is a finite
positive real number. In the simple case of $|S_{\star}|=\kappa_{3}=2$,
this automatically implies that the limit Markov chain $Z(\cdot)$
(cf. Theorem \ref{t_3rd}-(2)) is irreducible. A brief discussion
on the irreducibility of $Z(\cdot)$ without Assumption \ref{a_Sstar-kappa3}
is given in Section \ref{sec7}.

According to the next lemma, we have a much simpler expression of
the constant $\mathfrak{K}_{xy}$.
\begin{lem}
\label{l_g-gL-prop-4}It holds that
\[
\begin{aligned}\frac{1}{6\mathfrak{K}_{xy}} & =\sum_{a\in\mathcal{N}_{x}}\sum_{b\in\mathcal{N}_{y}}\frac{c_{ab}}{(1-m_{a})(1-m_{b})}+\sum_{a\in\mathcal{N}_{x}}\sum_{b\in\mathcal{A}_{x}}\frac{c_{ab}}{1-m_{a}}\sum_{\ell=0}^{\infty}m_{b}^{\ell}(1+\mathfrak{g}(\mathfrak{b}_{\ell})-\mathfrak{g}(\mathfrak{b}_{\ell+1}))\\
 & =\sum_{a\in\mathcal{N}_{x}}\sum_{b\in\mathcal{N}_{y}}\frac{c_{ab}}{(1-m_{a})(1-m_{b})}+\sum_{a\in\mathcal{A}_{y}}\sum_{b\in\mathcal{N}_{y}}\frac{c_{ab}}{1-m_{b}}\sum_{\ell=0}^{\infty}m_{a}^{\ell}(\mathfrak{g}(\mathfrak{a}_{\ell+1})-\mathfrak{g}(\mathfrak{a}_{\ell})).
\end{aligned}
\]
\end{lem}

\begin{rem}
\label{r_Kxy-lambda-indep}At this point, one may think that $\mathfrak{K}_{xy}$
depends heavily on the selection of $\lambda\in(\sqrt{m_{\star}},\,1)$.
However, \emph{$\mathfrak{K}_{xy}$ does not depend on $\lambda$};
this will become clear after we prove Theorems \ref{t_Capacity} and
\ref{t_Capacity-specific} since the capacities are independent of
$\lambda$. Thus, $\lambda$ turns out to be just a dummy variable.
Nevertheless, the parameter $\lambda$ is essential in two aspects.
First, it is impossible to guarantee the existence of a solution to
the resolvent equation \eqref{e_resolvent} without the appearance
of $\lambda$. Second, without $\lambda$, it is difficult to guarantee
the well-definedness of the infinite summations that appear in Lemmas
\eqref{l_g-gL-prop-3} and \eqref{l_g-gL-prop-4} and in \eqref{e_Kxy-def}.
\end{rem}

\section{\label{sec5}Upper Bound of Capacity: Test Function}

In this section, we construct a test function $F=F_{\textup{test}}:\mathcal{H}_{N}\to\mathbb{R}$
approximating the equilibrium potential $h_{\mathcal{E}_{N}^{x},\,\mathcal{E}_{N}^{y}}$.
For technical reasons, we fix an \emph{even integer} parameter $N'=N'(N)$
depending only on $N$ such that for all $\epsilon>0$,
\begin{equation}
\lim_{N\to\infty}N'=\infty,\quad\lim_{N\to\infty}\frac{N'}{N}=0,\quad\lim_{N\to\infty}\frac{e^{\epsilon N'}}{N}=\infty,\quad\text{and}\quad\lim_{N\to\infty}d_{N}e^{\epsilon N'}=\infty.\label{e_N'-cond}
\end{equation}
For instance, we may take $N'(N)=2\lfloor(N\log\frac{1}{d_{N}})^{\frac{1}{2}}\rfloor$
where $\lfloor\alpha\rfloor$ is the greatest integer less than or
equal to $\alpha$. The subexponential decay of $d_{N}$ easily implies
that this choice is valid.

\subsection{\label{sec5.1}Test function}

We construct $F:\mathcal{H}_{N}\to\mathbb{R}$ in several steps. First,
we define
\[
\mathcal{M}_{N}:=\bigcup_{\{a,\,b\}\subseteq S_{0}}\mathcal{A}_{N}^{xaby}\quad\text{and}\quad\mathcal{R}_{N}:=\mathcal{H}_{N}\setminus\mathcal{M}_{N}.
\]
It will be proved later that \emph{main (M)} part of the Dirichlet
form takes place in $\mathcal{M}_{N}$, and negligible \emph{remainder
(R)} takes place between $\mathcal{M}_{N}$ and $\mathcal{R}_{N}$
and in $\mathcal{R}_{N}$. 

Recall the six types of $\{a,\,b\}\subseteq S_{0}$ introduced in
\eqref{e_ab-types}. For $\eta\in\mathcal{H}_{N}$ and $A\subseteq S$,
we write
\begin{equation}
\overline{\eta}_{x}:=\eta_{x}+\sum_{v\in\mathcal{N}_{x}}\eta_{v},\quad\overline{\eta}_{y}:=\eta_{y}+\sum_{w\in\mathcal{N}_{y}}\eta_{w},\quad\text{and}\quad\eta(A):=\sum_{v\in A}\eta_{v}.\label{e_eta-not-def}
\end{equation}

\begin{figure}
\begin{tikzpicture}
\fill[rounded corners,white] (-2.5,3.7) rectangle (2.5,-2.5); 
\fill[red!10!white] (0,3)--(-2,0)--(0,-1)--(0.4,-0.8)--(0.4,2.4)--cycle;
\fill[blue!10!white] (1.6,0.6)--(1.2,0)--(1.6,-0.2)--(2,0)--cycle;
\draw[dotted,red] (0.4,2.4)--(-1.2,0)--(0.4,-0.8);
\draw[red] (-2,0)--(-1.2,0);
\draw[blue] (1.2,0)--(2,0);
\draw (-1.2,0)--(1.2,0);
\draw[very thick,dotted,red] (0.4,2.4)--(0.4,-0.8);
\draw[very thick,red] (0,3)--(-2,0)--(0,-1)--(0,3)--(0.4,2.4); \draw[very thick,red] (0.4,-0.8)--(0,-1);
\draw[very thick,dotted,blue] (1.6,0.6)--(1.2,0)--(1.6,-0.2)--cycle;
\draw[very thick,blue] (1.6,0.6)--(2,0)--(1.6,-0.2);
\draw[thick] (0.4,2.4)--(1.6,0.6); \draw[thick] (0.4,-0.8)--(1.6,-0.2);
\draw (0,3.1) node[above]{$\xi^x$};
\draw (0.9,2.4) node[above]{\footnotesize $\xi_{N-N'}^{xy}$};
\draw (1.9,0.6) node[above]{\footnotesize $\xi_{N'}^{xy}$};
\draw (-2,-0.1) node[below]{$\xi^a$};
\draw (2,-0.1) node[below]{$\xi^y$};
\draw (0,-1.1) node[below]{$\xi^b$};
\draw (0,-1.9) node[below]{\textbf{(T1)}};
\end{tikzpicture}
\begin{tikzpicture}
\fill[rounded corners,white] (-2.5,3.7) rectangle (2.5,-2.5); 
\fill[red!10!white] (0,3)--(-0.4,2.4)--(0,2.2)--(0.4,2.4)--cycle;
\fill[blue!10!white] (-2,0)--(0,-1)--(2,0)--(1.6,0.6)--(-1.6,0.6)--cycle;
\draw[dotted,red] (-0.4,2.4)--(0.4,2.4);
\draw[blue] (-2,0)--(2,0);
\draw[very thick,dotted,red] (-0.4,2.4)--(0,2.2)--(0.4,2.4);
\draw[very thick,red] (-0.4,2.4)--(0,3)--(0.4,2.4); \draw[very thick,red] (0,3)--(0,2.2);
\draw[very thick,dotted,blue] (0,-0.2)--(-1.6,0.6)--(1.6,0.6)--cycle;
\draw[very thick,blue] (-1.6,0.6)--(-2,0)--(0,-1)--(0,-0.2); \draw[very thick,blue] (0,-1)--(2,0)--(1.6,0.6);
\draw[thick] (0.4,2.4)--(1.6,0.6); \draw[thick] (-0.4,2.4)--(-1.6,0.6); \draw[thick] (0,2.2)--(0,-0.2);
\draw (0,3.1) node[above]{$\xi^x$};
\draw (0.9,2.4) node[above]{\footnotesize $\xi_{N-N'}^{xy}$};
\draw (1.9,0.6) node[above]{\footnotesize $\xi_{N'}^{xy}$};
\draw (-2,-0.1) node[below]{$\xi^a$};
\draw (2,-0.1) node[below]{$\xi^y$};
\draw (0,-1.1) node[below]{$\xi^b$};
\draw (0,-1.9) node[below]{\textbf{(T2)}, \textbf{(T5)}, and \textbf{(T6)}};
\end{tikzpicture}
\begin{tikzpicture}
\fill[rounded corners,white] (-2.5,3.7) rectangle (2.5,-2.5); 
\fill[red!10!white] (0,3)--(-2,0)--(-1.6,-0.2)--(-1.2,0)--(0.4,2.4)--cycle;
\fill[blue!10!white] (2,0)--(1.6,0.6)--(0,-0.2)--(-0.4,-0.8)--(0,-1)--cycle;
\draw[red] (-2,0)--(-1.2,0);
\draw[blue] (1.2,0)--(2,0);
\draw[dotted,blue] (1.6,0.6)--(1.2,0)--(-0.4,-0.8);
\draw (-1.2,0)--(1.2,0);
\draw[very thick,dotted,red] (0,2.2)--(-1.6,-0.2)--(-1.2,0)--(0.4,2.4)--cycle;
\draw[very thick,red] (0,3)--(-2,0)--(-1.6,-0.2); \draw[very thick,red] (0,2.2)--(0,3)--(0.4,2.4);
\draw[very thick,dotted,blue] (-0.4,-0.8)--(0,-0.2)--(1.6,0.6);
\draw[very thick,blue] (0,-0.2)--(0,-1); \draw[very thick,blue] (-0.4,-0.8)--(0,-1)--(2,0)--(1.6,0.6);
\draw[thick] (0.4,2.4)--(1.6,0.6); \draw[thick] (-0.4,-0.8)--(-1.6,-0.2); \draw[thick] (0,2.2)--(0,-0.2);
\draw (0,3.1) node[above]{$\xi^x$};
\draw (0.9,2.4) node[above]{\footnotesize $\xi_{N-N'}^{xy}$};
\draw (1.9,0.6) node[above]{\footnotesize $\xi_{N'}^{xy}$};
\draw (-2,-0.1) node[below]{$\xi^a$};
\draw (2,-0.1) node[below]{$\xi^y$};
\draw (0,-1.1) node[below]{$\xi^b$};
\draw (0,-1.9) node[below]{\textbf{(T3)} and \textbf{(T4)}};
\end{tikzpicture}\caption{\label{fig5.1}Illustrations of the sets $\mathcal{A}_{N}^{xaby}$
according to the types \textbf{(T1)}--\textbf{(T6)} of $\{a,\,b\}\subseteq S_{0}$.
Red and blue regions (excluding the inner face surrounded by dotted
lines) denote the subset $\mathcal{O}_{N}^{xaby}$. On the red (resp.
blue) region, the test function $F=F_{\textup{test}}$ is defined
as $1$ (resp. $0$).}
\end{figure}
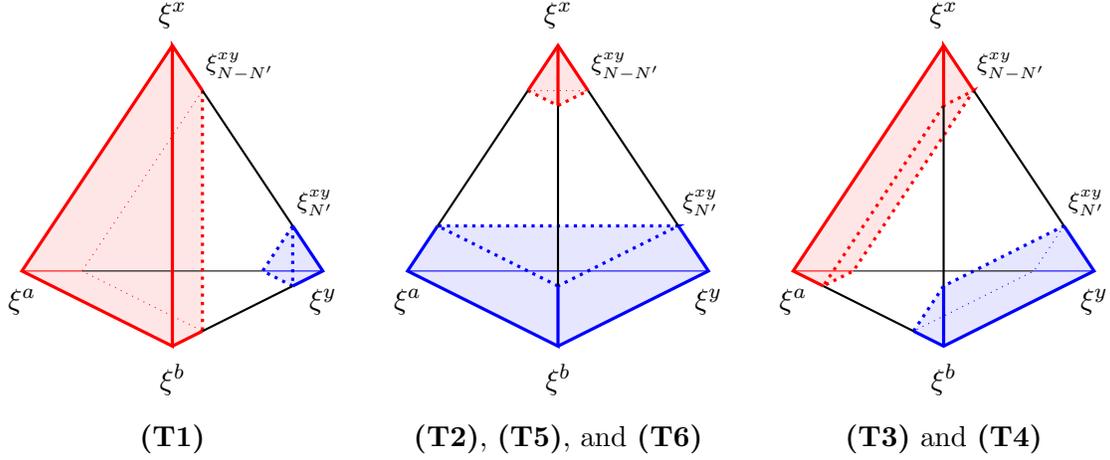

\begin{defn}[\textbf{(T1)} $a,\,b\in\mathcal{N}_{x}$]
\label{d_tf-T1}For $\eta\in\mathcal{A}_{N}^{xaby}$, we have $\overline{\eta}_{x}=\eta_{x}+\eta_{a}+\eta_{b}$
and $\overline{\eta}_{y}=\eta_{y}$.
\begin{enumerate}
\item We collect
\begin{equation}
\begin{aligned}\mathcal{O}_{N}^{xaby} & :=\{\eta\in\mathcal{A}_{N}^{xaby}:\overline{\eta}_{x}\ge N-N'+1\text{ or }\overline{\eta}_{x}\le N'-1\},\\
\mathcal{I}_{N}^{xaby} & :=\{\eta\in\mathcal{A}_{N}^{xaby}:\overline{\eta}_{x}\in\llbracket N',\,N-N'\rrbracket\},
\end{aligned}
\label{e_ON-IN-def}
\end{equation}
and define 
\begin{equation}
F(\eta):=1\text{ if }\overline{\eta}_{x}\ge N-N'+1\quad\text{and}\quad F(\eta):=0\text{ if }\overline{\eta}_{x}\le N'-1.\label{e_F-ON-def}
\end{equation}
\item For each $\eta\in\mathcal{I}_{N}^{xaby}$, we define
\begin{equation}
F(\eta):=\frac{\sum_{t=N'+1}^{\overline{\eta}_{x}}(N-t)t}{\sum_{t=N'+1}^{N-N'}(N-t)t}.\label{e_tf-T1}
\end{equation}
\end{enumerate}
\end{defn}

Also in the following definitions (Definitions \ref{d_tf-T2-T3}--\ref{d_tf-T6}),
we define subsets $\mathcal{O}_{N}^{xaby}$ and $\mathcal{I}_{N}^{xaby}$
as in \eqref{e_ON-IN-def} and define $F$ on $\mathcal{O}_{N}^{xaby}$
as in \eqref{e_F-ON-def}. Refer to Figure \ref{fig5.1} for illustrations
of the subsets $\mathcal{O}_{N}^{xaby}$ and $\mathcal{I}_{N}^{xaby}$,
which differ according to the types \textbf{(T1)}--\textbf{(T6)}.
Next, we define $F(\eta)$ for $\eta\in\mathcal{I}_{N}^{xaby}$ for
the remaining types.
\begin{defn}[\textbf{(T2)}/\textbf{(T3)} $a,\,b\in\mathcal{N}_{y}$ or $a\in\mathcal{N}_{x}$,
$b\in\mathcal{N}_{y}$]
\label{d_tf-T2-T3}Define $F$ on $\mathcal{I}_{N}^{xaby}$ as in
\eqref{e_tf-T1}.
\end{defn}

\begin{defn}[\textbf{(T4)} $a\in\mathcal{N}_{x}$, $b\notin\mathcal{N}_{x}\cup\mathcal{N}_{y}$]
\label{d_tf-T4}Since $b\notin\mathcal{N}_{x}\cup\mathcal{N}_{y}$,
by \eqref{e_V'-V-dec}, $b\in\mathcal{V}_{j}'$ for some $j\in\llbracket1,\,s\rrbracket$.
Then, recalling Definition \ref{d_g-ghat-def}, define for each $\xi_{n,\,k,\,\ell}^{xaby}\in\mathcal{I}_{N}^{xaby}$
as
\begin{equation}
F(\xi_{n,\,k,\,\ell}^{xaby}):=\frac{\sum_{t=N'+1}^{n+k}(N-t)t+(N-n-k)(n+k)\cdot\widehat{\mathfrak{g}}_{\ell}(\sigma^{b})}{\sum_{t=N'+1}^{N-N'}(N-t)t},\label{e_tf-T4}
\end{equation}
where $\sigma^{b}$ is the element of $U_{\ell}$ such that $\sigma^{b}(b)=\ell$;
recall Definition \ref{d_trans-ver}.
\end{defn}

\begin{defn}[\textbf{(T5)} $a\notin\mathcal{N}_{x}\cup\mathcal{N}_{y}$, $b\in\mathcal{N}_{y}$]
\label{d_tf-T5}Since $a\notin\mathcal{N}_{x}\cup\mathcal{N}_{y}$,
by \eqref{e_V'-V-dec}, $a\in\mathcal{V}_{j}'$ for some $j\in\llbracket1,\,s\rrbracket$.
Define for each $\xi_{n,\,k,\,\ell}^{xaby}\in\mathcal{I}_{N}^{xaby}$
as
\begin{equation}
F(\xi_{n,\,k,\,\ell}^{xaby}):=\frac{\sum_{t=N'+1}^{n}(N-t)t+(N-n)n\cdot\widehat{\mathfrak{g}}_{k}(\sigma^{a})}{\sum_{t=N'+1}^{N-N'}(N-t)t}.\label{e_tf-T5}
\end{equation}
\end{defn}

\begin{defn}[\textbf{(T6)} $a,\,b\notin\mathcal{N}_{x}\cup\mathcal{N}_{y}$]
\label{d_tf-T6}Fix $\xi_{n,\,k,\,\ell-k}^{xaby}\in\mathcal{I}_{N}^{xaby}$.
If $a\nsim b$, then we define
\begin{equation}
F(\xi_{n,\,k,\,\ell-k}^{xaby}):=\frac{\sum_{t=N'+1}^{n}(N-t)t}{\sum_{t=N'+1}^{N-N'}(N-t)t}.\label{e_tf-T6-1}
\end{equation}
If $a\sim b$, then by \eqref{e_V'-V-dec}, $a,\,b\in\mathcal{V}_{j}'$
for some $j\in\llbracket1,\,s\rrbracket$. In this case, we define
\begin{equation}
F(\xi_{n,\,k,\,\ell-k}^{xaby}):=\frac{\sum_{t=N'+1}^{n}(N-t)t+(N-n)n\cdot\widehat{\mathfrak{g}}_{\ell}(\sigma_{k}^{ab})}{\sum_{t=N'+1}^{N-N'}(N-t)t},\label{e_tf-T6-2}
\end{equation}
where $\sigma_{k}^{ab}$ is the element of $U_{\ell}$ such that $\sigma_{k}^{ab}(a)=k$
and $\sigma_{k}^{ab}(b)=\ell-k$.
\end{defn}

One can verify that $F$ is well defined on $\mathcal{M}_{N}$ by
checking on the intersections $\mathcal{A}_{N}^{xay}$ for $a\in S_{0}$.
Finally, we define $F$ on the remainder set $\mathcal{R}_{N}$.
\begin{defn}
\label{d_tf-RN}Fix $\eta\in\mathcal{R}_{N}$. If $\overline{\eta}_{x}\ge N-N'+1$
or $\overline{\eta}_{x}\le N'-1$, then we define as in \eqref{e_F-ON-def}.
Otherwise if $\overline{\eta}_{x}\in\llbracket N',\,N-N'\rrbracket$,
define $\eta^{*}\in\mathcal{H}_{N}$ as
\begin{equation}
\eta_{v}^{*}:=\begin{cases}
\overline{\eta}_{x} & \text{if }v=x,\\
\overline{\eta}_{y} & \text{if }v=y,\\
\eta_{v} & \text{if }v\in S_{0}\setminus(\mathcal{N}_{x}\cup\mathcal{N}_{y}).
\end{cases}\label{e_eta-star-def}
\end{equation}
If $\eta^{*}\in\mathcal{M}_{N}$, then define
\begin{equation}
F(\eta):=F(\eta^{*}).\label{e_tf-RN-1}
\end{equation}
If $\eta^{*}\notin\mathcal{M}_{N}$ (i.e., if $\eta^{*}\in\mathcal{R}_{N}$),
define $F(\eta)$ as in \eqref{e_tf-T1}.
\end{defn}

In the remainder of this section, we prove the following proposition.
Recall the definition of $\mathfrak{K}_{xy}$ from \eqref{e_Kxy-def}.
\begin{prop}
\label{p_Diri-form}We have
\[
\limsup_{N\to\infty}\frac{N^{2}}{d_{N}^{3}}\cdot\mathscr{D}_{N}(F)\le\frac{1}{2\mathfrak{K}_{xy}}.
\]
\end{prop}

\subsection{\label{sec5.2}Main part}

First, we record the following properties of the test function $F$.
For every $\eta\in\mathcal{H}_{N}$, define $\eta^{*}\in\mathcal{H}_{N}$
as in \eqref{e_eta-star-def}.
\begin{lem}[Properties of $F$]
\label{l_F-prop}Recall the notation given in \eqref{e_eta-not-def}.
\begin{enumerate}
\item For $\eta,\,\zeta\in\mathcal{H}_{N}$, if $\eta^{*}=\zeta^{*}$ then
$F(\eta)=F(\zeta)$. In other words, $F(\eta)$ depends on $\eta$
only through $\overline{\eta}_{x}$, $\overline{\eta}_{y}$, and $\eta_{a}$
for $a\in S_{0}\setminus(\mathcal{N}_{x}\cup\mathcal{N}_{y})$.
\item If $\eta\in\mathcal{H}_{N}$ satisfies $\overline{\eta}_{x}\in\llbracket N',\,N-N'\rrbracket$,
then
\[
F(\eta)=\frac{\sum_{t=N'+1}^{\overline{\eta}_{x}}(N-t)t+(N-\overline{\eta}_{x})\overline{\eta}_{x}\cdot G(\eta)}{\sum_{t=N'+1}^{N-N'}(N-t)t},\quad\text{where}\quad0\le G(\eta)\le\frac{C}{\lambda^{\eta(S_{0})}}.
\]
\item For every $\eta\in\mathcal{H}_{N}$, it holds that
\[
0\le F(\eta)\le\frac{C}{\lambda^{\eta(S_{0})}}.
\]
\end{enumerate}
\end{lem}

\begin{proof}
(1) If $\eta^{*}\in\mathcal{M}_{N}$, then the explicit formulas given
in Definitions \ref{d_tf-T1}--\ref{d_tf-T6}, along with \eqref{e_tf-RN-1},
imply the result. On the contrary, if $\eta^{*}\in\mathcal{R}_{N}$,
then we automatically have $\eta,\,\zeta\in\mathcal{R}_{N}$, such
that Definitions \ref{d_tf-RN} deduces the result.

\noindent (2) By \eqref{e_tf-T1}--\eqref{e_tf-T6-2}, we may write
\[
F(\eta)=\frac{\sum_{t=N'+1}^{\overline{\eta}_{x}}(N-t)t+(N-\overline{\eta}_{x})\overline{\eta}_{x}\cdot G(\eta)}{\sum_{t=N'+1}^{N-N'}(N-t)t}\quad\text{with}\quad0\le G(\eta)\le\frac{C}{\lambda^{\eta(S_{0})}},
\]
where the second estimate holds by \eqref{e_ghat-ghatL-UB} and the
fact that $N-\overline{\eta}_{x}-\overline{\eta}_{y}\le\eta(S_{0})$.

\noindent (3) This follows directly from part (2) and \eqref{e_F-ON-def}.
\end{proof}
In particular, Lemma \ref{l_F-prop}-(2) implies the following lemma.
\begin{lem}[Difference of $F$]
\label{l_F-diff}For $\eta,\,\zeta\in\mathcal{H}_{N}$ satisfying
$\eta\sim\zeta$ (i.e., $r_{N}(\eta,\,\zeta)>0$),
\[
|F(\zeta)-F(\eta)|\le\frac{C}{N\lambda^{\max\{\eta(S_{0}),\,\zeta(S_{0})\}}}.
\]
\end{lem}

\begin{proof}
By Lemma \ref{l_F-prop}-(2) and \eqref{e_F-ON-def}, it readily holds
that $|F(\zeta)-F(\eta)|$ is bounded by
\[
\frac{\max\{(N-\overline{\eta}_{x})\overline{\eta}_{x},\,(N-\overline{\zeta}_{x})\overline{\zeta}_{x}\}+\max\{(N-\overline{\eta}_{x})\overline{\eta}_{x}\cdot G(\eta),\,(N-\overline{\zeta}_{x})\overline{\zeta}_{x}\cdot G(\zeta)\}}{\sum_{t=N'+1}^{N-N'}(N-t)t},
\]
where
\[
0\le G(\eta)\le\frac{C}{\lambda^{\eta(S_{0})}}\quad\text{and}\quad0\le G(\zeta)\le\frac{C}{\lambda^{\zeta(S_{0})}}.
\]
Note that by \eqref{e_N'-cond},
\begin{equation}
\sum_{t=N'+1}^{N-N'}(N-t)t=\frac{(N-1)N(N+1)}{6}-O(NN'^{2})\simeq\frac{N^{3}}{6}.\label{e_sum-N-cubed}
\end{equation}
Combining these with $(N-\overline{\eta}_{x})\overline{\eta}_{x}\le N^{2}$
and $(N-\overline{\zeta}_{x})\overline{\zeta}_{x}\le N^{2}$, we deduce
the result.
\end{proof}
Now, Proposition \ref{p_Diri-form} will be proved according to a
series of lemmas. For disjoint subsets $\mathcal{A}$ and $\mathcal{B}$
of $\mathcal{H}_{N}$, we write
\begin{equation}
\begin{aligned}\mathscr{D}_{\mathcal{A}}(F) & :=\sum_{\{\eta,\,\zeta\}\subseteq\mathcal{A}}\mu_{N}(\eta)r_{N}(\eta,\,\zeta)\{F(\zeta)-F(\eta)\}^{2},\\
\mathscr{D}_{\mathcal{A}\mathcal{B}}(F) & :=\sum_{\eta\in\mathcal{A}}\sum_{\zeta\in\mathcal{B}}\mu_{N}(\eta)r_{N}(\eta,\,\zeta)\{F(\zeta)-F(\eta)\}^{2}.
\end{aligned}
\label{e_DAf-DABf-def}
\end{equation}
Then, we may represent
\[
\mathscr{D}_{N}(F)=\mathscr{D}_{\mathcal{M}_{N}}(F)+\mathscr{D}_{\mathcal{M}_{N}\mathcal{R}_{N}}(F)+\mathscr{D}_{\mathcal{R}_{N}}(F).
\]
In this subsection, we pursue to estimate $\mathscr{D}_{\mathcal{A}_{N}^{xaby}}(F)$
for each type of $\{a,\,b\}\subseteq S_{0}$ from \textbf{(T1)} to
\textbf{(T6)}.
\begin{lem}[Dirichlet form estimate of \textbf{(T1)}/\textbf{(T2)}]
\label{l_T1-T2}If $a,\,b\in\mathcal{N}_{x}$ or $a,\,b\in\mathcal{N}_{y}$,
it holds that $\mathscr{D}_{\mathcal{A}_{N}^{xaby}}(F)=0$.
\end{lem}

\begin{proof}
By Lemma \ref{l_F-prop}-(1), particle movements in $\{x\}\cup\mathcal{N}_{x}$
or in $\mathcal{N}_{y}\cup\{y\}$ do not change the value of $F$.
This implies that $\mathscr{D}_{\mathcal{A}_{N}^{xaby}}(F)=0$.
\end{proof}
\begin{lem}[Dirichlet form estimate of \textbf{(T3)}]
\label{l_T3}If $a\in\mathcal{N}_{x}$ and $b\in\mathcal{N}_{y}$,
it holds that
\[
\mathscr{D}_{\mathcal{A}_{N}^{xaby}}(F)=(1+o(1))\cdot\frac{3c_{ab}}{(1-m_{a})(1-m_{b})}\cdot\frac{d_{N}^{3}}{N^{2}}.
\]
\end{lem}

\begin{proof}
By \eqref{e_DAf-DABf-def}, we may decompose
\begin{equation}
\mathscr{D}_{\mathcal{A}_{N}^{xaby}}(F)=\mathscr{D}_{\mathcal{O}_{N}^{xaby}}(F)+\mathscr{D}_{\mathcal{O}_{N}^{xaby}\mathcal{I}_{N}^{xaby}}(F)+\mathscr{D}_{\mathcal{I}_{N}^{xaby}}(F).\label{e_AN-ON-ONIN-IN}
\end{equation}
Inspecting Figure \ref{fig5.1}, we readily observe that $\mathscr{D}_{\mathcal{O}_{N}^{xaby}}(F)=0$.
For $\mathscr{D}_{\mathcal{O}_{N}^{xaby}\mathcal{I}_{N}^{xaby}}(F)$,
the contribution $\eta\leftrightarrow\zeta$ for $\eta\in\mathcal{O}_{N}^{xaby}$
and $\zeta\in\mathcal{I}_{N}^{xaby}$ survives only in the following
two cases:
\begin{itemize}
\item $\eta_{x}+\eta_{a}=N-N'+1$, $\zeta_{x}+\zeta_{a}=N-N'$, and $\zeta=\eta^{a,\,b}$;
\item $\eta_{x}+\eta_{a}=N'-1$, $\zeta_{x}+\zeta_{a}=N'$, and $\zeta=\eta^{b,\,a}$.
\end{itemize}
However, by \eqref{e_tf-T1}, in the first case we have $F(\eta)=F(\zeta)=1$,
and in the second case we have $F(\eta)=F(\zeta)=0$. Thus, $\mathscr{D}_{\mathcal{O}_{N}^{xaby}\mathcal{I}_{N}^{xaby}}(F)=0$.

Finally, we calculate $\mathscr{D}_{\mathcal{I}_{N}^{xaby}}(F)$.
Since particle movements of $x\leftrightarrow a$ and $b\leftrightarrow y$
do not affect the test function by Lemma \ref{l_F-prop}-(1), we may
only consider the movements of $a\leftrightarrow b$. Thus, we may
write the summation as
\[
\sum_{n=N'+1}^{N-N'}\sum_{k=1}^{n}\sum_{\ell=0}^{N-n}\mu_{N}(\xi_{n-k,\,k,\,\ell}^{xaby})r_{N}(\xi_{n-k,\,k,\,\ell}^{xaby},\,\xi_{n-k,\,k-1,\,\ell+1}^{xaby})\{F(\xi_{n-k,\,k-1,\,\ell+1}^{xaby})-F(\xi_{n-k,\,k,\,\ell}^{xaby})\}^{2}.
\]
Substituting the exact values from Proposition \ref{p_muN-prod} and
\eqref{e_tf-T1} and applying \eqref{e_wN-ZN}, this is asymptotically
equal to
\begin{equation}
\sum_{n=N'+1}^{N-N'}\sum_{k=1}^{n}\sum_{\ell=0}^{N-n}\frac{Nd_{N}w_{N}(n-k)w_{N}(N-n-\ell)}{2}m_{a}^{k}m_{b}^{\ell}r(a,\,b)\cdot\frac{\{(N-n)n\}^{2}}{\{\sum_{t=N'+1}^{N-N'}(N-t)t\}^{2}}.\label{e_l_T3-1}
\end{equation}
Recall \eqref{e_sum-N-cubed}. If $k>\frac{N'}{2}$ or $\ell>\frac{N'}{2}$,
the summation becomes $O(m_{\star}^{\frac{N'}{2}})$, which is $o(\frac{d_{N}^{3}}{N^{2}})$
by \eqref{e_wN-ZN}, \eqref{e_N'-cond}, and \eqref{e_sum-N-cubed}.
Thus, by \eqref{e_sum-N-cubed}, we may rewrite the remaining part
of \eqref{e_l_T3-1} as
\begin{equation}
\frac{18d_{N}^{3}r(a,\,b)}{N^{5}(1+o(1))}\cdot\sum_{n=N'+1}^{N-N'}\sum_{k=1}^{N'/2}\sum_{\ell=0}^{N'/2}\frac{n^{2}(N-n)^{2}}{(n-k)(N-n-\ell)}m_{a}^{k}m_{b}^{\ell}.\label{e_l_T3-2}
\end{equation}
Note that 
\[
\frac{n^{2}}{n-k}=n+\frac{nk}{n-k}\quad\text{and}\quad\frac{(N-n)^{2}}{N-n-\ell}=N-n+\frac{(N-n)\ell}{N-n-\ell}.
\]
Moreover, if we substitute $\frac{n^{2}}{n-k}$ with $\frac{nk}{n-k}$
in \eqref{e_l_T3-2}, it becomes
\begin{equation}
\begin{aligned} & \frac{18d_{N}^{3}r(a,\,b)}{N^{5}(1+o(1))}\cdot\sum_{n=N'+1}^{N-N'}\sum_{k=1}^{N'/2}\sum_{\ell=0}^{N'/2}\frac{nk(N-n)^{2}}{(n-k)(N-n-\ell)}m_{a}^{k}m_{b}^{\ell}\\
 & \le\frac{Cd_{N}^{3}r(a,\,b)}{N^{4}}\sum_{n=N'+1}^{N-N'}\sum_{k=1}^{N'/2}\sum_{\ell=0}^{N'/2}km_{\star}^{k+\ell}\le\frac{Cd_{N}^{3}r(a,\,b)}{N^{3}}=o\Big(\frac{d_{N}^{3}}{N^{2}}\Big).
\end{aligned}
\label{e_l_T3-3}
\end{equation}
Similarly, if we substitute $\frac{(N-n)^{2}}{N-n-\ell}$ with $\frac{(N-n)\ell}{N-n-\ell}$
in \eqref{e_l_T3-2}, it becomes $o(\frac{d_{N}^{3}}{N^{2}})$. Thus,
\eqref{e_l_T3-2} equals
\[
\begin{aligned} & \frac{18d_{N}^{3}r(a,\,b)}{N^{5}(1+o(1))}\cdot\sum_{n=N'+1}^{N-N'}\sum_{k=1}^{N'/2}\sum_{\ell=0}^{N'/2}n(N-n)m_{a}^{k}m_{b}^{\ell}+o\Big(\frac{d_{N}^{3}}{N^{2}}\Big)\\
 & =\frac{18d_{N}^{3}r(a,\,b)}{N^{5}(1+o(1))}\cdot\frac{N^{3}}{6}\cdot\Big[\frac{m_{a}}{(1-m_{a})(1-m_{b})}+O(m_{\star}^{\frac{N'}{2}})+o(1)\Big].
\end{aligned}
\]
Therefore, we have proved that
\[
\mathscr{D}_{\mathcal{I}_{N}^{xaby}}(F)=(1+o(1))\cdot\frac{3d_{N}^{3}}{N^{2}}\cdot\frac{c_{ab}}{(1-m_{a})(1-m_{b})},
\]
which concludes the proof of Lemma \ref{l_T3}.
\end{proof}
\begin{lem}[Dirichlet form estimate of \textbf{(T4)}]
\label{l_T4}If $a\in\mathcal{N}_{x}$ and $b\notin\mathcal{N}_{x}\cup\mathcal{N}_{y}$,
\[
\mathscr{D}_{\mathcal{A}_{N}^{xaby}}(F)=(1+o(1))\cdot\frac{3c_{ab}}{1-m_{a}}\cdot\sum_{\ell=0}^{\infty}m_{b}^{\ell}\{1+\widehat{\mathfrak{g}}_{\ell}(\sigma^{b})-\widehat{\mathfrak{g}}_{\ell+1}(\sigma^{b})\}^{2}\cdot\frac{d_{N}^{3}}{N^{2}}.
\]
\end{lem}

\begin{proof}
Recall the representation in \eqref{e_AN-ON-ONIN-IN}. It is clear
that $\mathscr{D}_{\mathcal{O}_{N}^{xaby}}(F)=0$. Next, since $a\in\mathcal{N}_{x}$
and $b\notin\mathcal{N}_{x}\cup\mathcal{N}_{y}$, the configurations
between $\mathcal{O}_{N}^{xaby}$ and $\mathcal{I}_{N}^{xaby}$ are
connected only via particle movements of $a\leftrightarrow b$, so
that we may write $\mathscr{D}_{\mathcal{O}_{N}^{xaby}\mathcal{I}_{N}^{xaby}}(F)$
as
\begin{equation}
\begin{aligned} & \sum_{k=0}^{N-N'}\sum_{\ell=1}^{N'}\mu_{N}(\xi_{N-N'-k,\,k,\,\ell}^{xaby})\cdot\ell(d_{N}+k)r(b,\,a)\cdot\{F(\xi_{N-N'-k,\,k+1,\,\ell-1}^{xaby})-F(\xi_{N-N'-k,\,k,\,\ell}^{xaby})\}^{2}\\
 & +\sum_{k=1}^{N'}\sum_{\ell=0}^{N-N'}\mu_{N}(\xi_{N'-k,\,k,\,\ell}^{xaby})\cdot k(d_{N}+\ell)r(a,\,b)\cdot\{F(\xi_{N'-k,\,k-1,\,\ell+1}^{xaby})-F(\xi_{N'-k,\,k,\,\ell}^{xaby})\}^{2}.
\end{aligned}
\label{e_l_T4-1}
\end{equation}
Substituting the exact values from \eqref{e_F-ON-def} and \eqref{e_tf-T4}
that
\[
F(\xi_{N-N'-k,\,k,\,\ell}^{xaby})=1-\frac{N'(N-N')\cdot\widehat{\mathfrak{g}}_{\ell}(\sigma^{b})}{\sum_{t=N'+1}^{N-N'}(N-t)t}\quad\text{and}\quad F(\xi_{N-N'-k,\,k+1,\,\ell-1}^{xaby})=1,
\]
and also applying \eqref{e_wN-ZN} and \eqref{e_sum-N-cubed}, we
see that the first line of \eqref{e_l_T4-1} is bounded by
\[
\frac{CN'^{2}d_{N}}{N^{3}}\sum_{k=0}^{N-N'}\sum_{\ell=1}^{N'}w_{N}(N-N'-k)w_{N}(N'-\ell)m_{\star}^{k+\ell}\cdot\widehat{\mathfrak{g}}_{\ell}(\sigma^{b})^{2}.
\]
By \eqref{e_ghat-ghatL-UB}, $\widehat{\mathfrak{g}}_{\ell}(\sigma^{b})^{2}\le\frac{C}{\lambda^{2\ell}}$
for some constant $C>0$. If $k>\frac{N'}{2}$ or $\ell>\frac{N'}{2}$
in the summation, then the Dirichlet summation is $O((\frac{m_{\star}}{\lambda^{2}})^{\frac{N'}{2}})$,
which is still $o(\frac{d_{N}^{3}}{N^{2}})$ by \eqref{e_N'-cond}
and the fact that $\lambda\in(\sqrt{m_{\star}},\,1)$ (cf. \eqref{e_lambda-def}).
Thus, we may only consider the case when $k\le\frac{N'}{2}$ and $\ell\le\frac{N'}{2}$,
in which case the summation is bounded by
\[
\frac{CN'^{2}d_{N}^{3}}{N^{3}}\sum_{k=0}^{N'/2}\sum_{\ell=1}^{N'/2}\frac{1}{(N-N'-k)(N'-\ell)}\Big(\frac{m_{\star}}{\lambda^{2}}\Big)^{k+\ell}\le O\Big(\frac{N'd_{N}^{3}}{N^{4}}\Big)=o\Big(\frac{d_{N}^{3}}{N^{2}}\Big).
\]
In the inequality, it is again used that $\lambda\in(\sqrt{m_{\star}},\,1)$.
The second line of \eqref{e_l_T4-1} can be similarly estimated as
$o(\frac{d_{N}^{3}}{N^{2}})$. Thus, we have proved that $\mathscr{D}_{\mathcal{O}_{N}^{xaby}\mathcal{I}_{N}^{xaby}}(F)=o(\frac{d_{N}^{3}}{N^{2}})$.

Finally, we consider $\mathscr{D}_{\mathcal{I}_{N}^{xaby}}(F)$, which
is
\[
\sum_{n=N'+1}^{N-N'}\sum_{k=1}^{n}\sum_{\ell=0}^{N-n}\mu_{N}(\xi_{n-k,\,k,\,\ell}^{xaby})r_{N}(\xi_{n-k,\,k,\,\ell}^{xaby},\,\xi_{n-k,\,k-1,\,\ell+1}^{xaby})\{F(\xi_{n-k,\,k-1,\,\ell+1}^{xaby})-F(\xi_{n-k,\,k,\,\ell}^{xaby})\}^{2}.
\]
By \eqref{e_tf-T4}, this is asymptotically equal to
\[
\begin{aligned}\sum_{n=N'+1}^{N-N'}\sum_{k=1}^{n}\sum_{\ell=0}^{N-n} & \frac{Nd_{N}w_{N}(n-k)w_{N}(N-n-\ell)}{2}m_{a}^{k}m_{b}^{\ell}r(a,\,b)\cdot\\
 & \frac{\{(N-n)n+(N-n)n\cdot\widehat{\mathfrak{g}}_{\ell}(\sigma^{b})-(N-n+1)(n-1)\cdot\widehat{\mathfrak{g}}_{\ell+1}(\sigma^{b})\}^{2}}{\{\sum_{t=N'+1}^{N-N'}(N-t)t\}^{2}}.
\end{aligned}
\]
If we substitute $(N-n)n\cdot\widehat{\mathfrak{g}}_{\ell+1}(\sigma^{b})$
for $(N-n+1)(n-1)\cdot\widehat{\mathfrak{g}}_{\ell+1}(\sigma^{b})$,
then we obtain
\begin{equation}
\begin{aligned} & \sum_{n,\,k,\,\ell}\frac{Nd_{N}w_{N}(n-k)w_{N}(N-n-\ell)}{2}m_{a}^{k}m_{b}^{\ell}r(a,\,b)\cdot\frac{(N-n)^{2}n^{2}\{1+\widehat{\mathfrak{g}}_{\ell}(\sigma^{b})-\widehat{\mathfrak{g}}_{\ell+1}(\sigma^{b})\}^{2}}{\{\sum_{t=N'+1}^{N-N'}(N-t)t\}^{2}}\\
 & =(1+o(1))\cdot\frac{3c_{ab}}{1-m_{a}}\cdot\sum_{\ell=0}^{\infty}m_{b}^{\ell}\{1+\widehat{\mathfrak{g}}_{\ell}(\sigma^{b})-\widehat{\mathfrak{g}}_{\ell+1}(\sigma^{b})\}^{2}\cdot\frac{d_{N}^{3}}{N^{2}},
\end{aligned}
\label{e_l_T4-2}
\end{equation}
which can be proved in the same way as in the proof of Lemma \ref{l_T3},
along with the bound \eqref{e_ghat-ghatL-UB} and the fact that $\lambda\in(\sqrt{m_{\star}},\,1)$.
Thus, to conclude the proof, it remains to demonstrate that
\[
\begin{aligned}\sum_{n=N'+1}^{N-N'}\sum_{k=1}^{n}\sum_{\ell=0}^{N-n} & \frac{Nd_{N}w_{N}(n-k)w_{N}(N-n-\ell)}{2}m_{a}^{k}m_{b}^{\ell}r(a,\,b)\cdot\\
 & \frac{\{(N-n)n\cdot\widehat{\mathfrak{g}}_{\ell+1}(\sigma^{b})-(N-n+1)(n-1)\cdot\widehat{\mathfrak{g}}_{\ell+1}(\sigma^{b})\}^{2}}{\{\sum_{t=N'+1}^{N-N'}(N-t)t\}^{2}}=o\Big(\frac{d_{N}^{3}}{N^{2}}\Big).
\end{aligned}
\]
To this end, by \eqref{e_sum-N-cubed}, the left-hand side is bounded
above by
\[
\frac{C}{N^{3}}\sum_{n=N'+1}^{N-N'}\sum_{k=1}^{n}\sum_{\ell=0}^{N-n}d_{N}w_{N}(n-k)w_{N}(N-n-\ell)\cdot m_{\star}^{k+\ell}\cdot\widehat{\mathfrak{g}}_{\ell+1}(\sigma^{b})^{2}.
\]
Since $\widehat{\mathfrak{g}}_{\ell+1}(\sigma^{b})^{2}\le\frac{C}{\lambda^{2\ell}}$,
we can bound this by
\[
\frac{C}{N^{3}}\sum_{n=N'+1}^{N-N'}\sum_{k=1}^{n}\sum_{\ell=0}^{N-n}d_{N}w_{N}(n-k)w_{N}(N-n-\ell)\cdot\Big(\frac{m_{\star}}{\lambda^{2}}\Big)^{k+\ell}.
\]
If $k>\frac{N'}{2}$ or $\ell>\frac{N'}{2}$, this becomes $O((\frac{m_{\star}}{\lambda^{2}})^{\frac{N'}{2}})=o(\frac{d_{N}^{3}}{N^{2}})$,
and thus we may exclude this case. Then, the remaining summation becomes
\[
\frac{Cd_{N}^{3}}{N^{3}}\sum_{n=N'+1}^{N-N'}\sum_{k=1}^{N'/2}\sum_{\ell=0}^{N'/2}\frac{1}{(n-k)(N-n-\ell)}\cdot\Big(\frac{m_{\star}}{\lambda^{2}}\Big)^{k+\ell}.
\]
Simply bounding $\frac{1}{n-k}\le\frac{2}{n}$ and $\frac{1}{N-n-\ell}\le1$,
which is suboptimal but sufficient, this is bounded above by
\[
\frac{Cd_{N}^{3}}{N^{3}}\sum_{n=N'+1}^{N-N'}\sum_{k=1}^{N'/2}\sum_{\ell=0}^{N'/2}\frac{1}{n}\cdot\Big(\frac{m_{\star}}{\lambda^{2}}\Big)^{k+\ell}=O\Big(\frac{d_{N}^{3}}{N^{3}}\log N\Big)=o\Big(\frac{d_{N}^{3}}{N^{2}}\Big).
\]
This concludes the proof.
\end{proof}
\begin{lem}[Dirichlet form estimate of \textbf{(T5)}]
\label{l_T5}If $a\notin\mathcal{N}_{x}\cup\mathcal{N}_{y}$ and
$b\in\mathcal{N}_{y}$,
\[
\mathscr{D}_{\mathcal{A}_{N}^{xaby}}(F)=(1+o(1))\cdot\frac{3c_{ab}}{1-m_{b}}\cdot\sum_{\ell=0}^{\infty}m_{a}^{\ell}\{\widehat{\mathfrak{g}}_{\ell}(\sigma^{a})-\widehat{\mathfrak{g}}_{\ell+1}(\sigma^{a})\}^{2}\cdot\frac{d_{N}^{3}}{N^{2}}.
\]
\end{lem}

\begin{proof}
We omit the proof, since it is almost identical to the proof of Lemma
\ref{l_T4}.
\end{proof}
\begin{lem}[Dirichlet form estimate of \textbf{(T6)}]
\label{l_T6}If $a,\,b\notin\mathcal{N}_{x}\cup\mathcal{N}_{y}$,
it holds that
\[
\mathscr{D}_{\mathcal{A}_{N}^{xaby}}(F)=(1+o(1))\cdot3c_{ab}\cdot\sum_{\ell=1}^{\infty}\sum_{k=1}^{\ell}m_{a}^{k-1}m_{b}^{\ell-k}\{\widehat{\mathfrak{g}}_{\ell}(\sigma_{k-1}^{ab})-\widehat{\mathfrak{g}}_{\ell}(\sigma_{k}^{ab})\}^{2}\cdot\frac{d_{N}^{3}}{N^{2}}.
\]
\end{lem}

\begin{proof}
Almost all parts of the proof are the same with the proof of Lemma
\ref{l_T4}, except for the calculation of the following object in
the case of $a\sim b$:
\[
\sum_{n=N'}^{N-N'-1}\sum_{\ell=1}^{N-n}\sum_{k=1}^{\ell}\frac{Nd_{N}w_{N}(n)w_{N}(N-n-\ell)}{2}m_{a}^{k}m_{b}^{\ell-k}r(a,\,b)\cdot\frac{(N-n)^{2}n^{2}\{\widehat{\mathfrak{g}}_{\ell}(\sigma_{k-1}^{ab})-\widehat{\mathfrak{g}}_{\ell}(\sigma_{k}^{b})\}^{2}}{\{\sum_{t=N'+1}^{N-N'}(N-t)t\}^{2}},
\]
which is an analogue of \eqref{e_l_T4-2} and in turn the main constituent
of the Dirichlet form. Using the same logic, we obtain that this equals
\[
(1+o(1))\cdot3c_{ab}\cdot\frac{d_{N}^{3}}{N^{2}}\cdot\sum_{\ell=1}^{\infty}\sum_{k=1}^{\ell}m_{a}^{k-1}m_{b}^{\ell-k}\cdot\{\widehat{\mathfrak{g}}_{\ell}(\sigma_{k-1}^{ab})-\widehat{\mathfrak{g}}_{\ell}(\sigma_{k}^{ab})\}^{2},
\]
which completes the proof.
\end{proof}

\subsection{\label{sec5.3}Negligible part in $\mathcal{M}_{N}$}

In this subsection, we prove that the remaining Dirichlet summation
in $\mathscr{D}_{\mathcal{M}_{N}}(F)$ that is not handled in the
previous subsection is negligible, i.e., equals $o(\alpha_{N})$.
These contributions are exactly those $\eta\leftrightarrow\zeta$
where $\eta$ and $\zeta$ belong to different sets $\mathcal{A}_{N}^{xaby}$
and $\mathcal{A}_{N}^{xa'b'y}$, respectively, with $\{a,\,b\}\ne\{a',\,b'\}\subseteq S_{0}$.
This contribution $\eta\leftrightarrow\zeta$ happens only if, e.g.,
$b=b'$, at least one particle exists at $b$, and a sole particle
moves between sites $a$ and $a'$. Namely, we are left to prove that
\[
\sum_{n=0}^{N-1}\sum_{\ell=1}^{N-n-1}\mu_{N}(\xi_{n,\,1,\,\ell}^{xaby})\cdot d_{N}r(a,\,a')\cdot\{F(\xi_{n,\,1,\,\ell}^{xa'by})-F(\xi_{n,\,1,\,\ell}^{xaby})\}^{2}=o\Big(\frac{d_{N}^{3}}{N^{2}}\Big).
\]
If $\ell>\frac{N'}{2}$, then by Lemma \ref{l_F-prop}-(4), the summation
is bounded by $O(N(\frac{m_{\star}}{\lambda^{2}})^{\frac{N'}{2}})$,
which is negligible compared to $\frac{d_{N}^{3}}{N^{2}}$ by \eqref{e_N'-cond}.
Thus, it suffices to prove that
\[
\sum_{n=0}^{N-1}\sum_{\ell=1}^{(N-n-1)\wedge(N'/2)}\mu_{N}(\xi_{n,\,1,\,\ell}^{xaby})\cdot d_{N}r(a,\,a')\cdot\{F(\xi_{n,\,1,\,\ell}^{xa'by})-F(\xi_{n,\,1,\,\ell}^{xaby})\}^{2}=o\Big(\frac{d_{N}^{3}}{N^{2}}\Big),
\]
where $\alpha\wedge\beta:=\min\{\alpha,\,\beta\}$. If $n\le\frac{N'}{2}-2$,
then we have $n+\ell+1\le N'-1$, and thus in this case Lemma \ref{l_F-prop}-(2)
implies that $F(\xi_{n,\,1,\,\ell}^{xaby})=F(\xi_{n,\,1,\,\ell}^{xa'by})=0$.
Moreover, if $n\ge N-N'+1$, then again Lemma \ref{l_F-prop}-(2)
implies that $F(\xi_{n,\,1,\,\ell}^{xaby})=F(\xi_{n,\,1,\,\ell}^{xa'by})=1$.
Thus, we are left to prove
\[
\sum_{n=N'/2-1}^{N-N'}\sum_{\ell=1}^{N'/2}\mu_{N}(\xi_{n,\,1,\,\ell}^{xaby})\cdot d_{N}r(a,\,a')\cdot\{F(\xi_{n,\,1,\,\ell}^{xa'by})-F(\xi_{n,\,1,\,\ell}^{xaby})\}^{2}=o\Big(\frac{d_{N}^{3}}{N^{2}}\Big).
\]
By \eqref{e_wN-ZN}, this is equivalent to
\begin{equation}
Nd_{N}^{4}\sum_{n=N'/2-1}^{N-N'}\sum_{\ell=1}^{N'/2}\frac{m_{\star}^{\ell+1}}{n\ell(N-n-\ell-1)}\cdot\{F(\xi_{n,\,1,\,\ell}^{xa'by})-F(\xi_{n,\,1,\,\ell}^{xaby})\}^{2}=o\Big(\frac{d_{N}^{3}}{N^{2}}\Big).\label{e_neg-MN}
\end{equation}
By Lemma \ref{l_F-diff}, it holds that
\[
|F(\xi_{n,\,1,\,\ell}^{xa'by})-F(\xi_{n,\,1,\,\ell}^{xaby})|\le\frac{C}{N\lambda^{\ell}}.
\]
Inserting this to \eqref{e_neg-MN}, the left-hand side is bounded
by
\[
\frac{Cd_{N}^{4}}{N}\sum_{n=N'/2-1}^{N-N'}\sum_{\ell=1}^{N'/2}\frac{1}{n\ell(N-n-\ell-1)}\cdot\Big(\frac{m_{\star}}{\lambda^{2}}\Big)^{\ell}=O\Big(\frac{d_{N}^{4}}{N^{2}}\log N\Big)=o\Big(\frac{d_{N}^{3}}{N^{2}}\Big).
\]
Therefore, we have proved the following lemma.
\begin{lem}
\label{l_MNMN}If $\{a,\,b\}\ne\{a',\,b'\}\subseteq S_{0}$, then
it holds that
\[
\sum_{\eta\in\mathcal{A}_{N}^{xaby}\setminus\mathcal{A}_{N}^{xa'b'y}}\sum_{\zeta\in\mathcal{A}_{N}^{xa'b'y}\setminus\mathcal{A}_{N}^{xaby}}\mu_{N}(\eta)r_{N}(\eta,\,\zeta)\{f(\zeta)-f(\eta)\}^{2}=o\Big(\frac{d_{N}^{3}}{N^{2}}\Big).
\]
\end{lem}

\subsection{\label{sec5.4}Remaining part}

In this subsection, we calculate the remaining negligible parts of
the Dirichlet form, $\mathscr{D}_{\mathcal{M}_{N}\mathcal{R}_{N}}(F)$
and $\mathscr{D}_{\mathcal{R}_{N}}(F)$.
\begin{lem}
\label{l_MNRN-RN}It holds that
\[
\mathscr{D}_{\mathcal{M}_{N}\mathcal{R}_{N}}(F)+\mathscr{D}_{\mathcal{R}_{N}}(F)=o\Big(\frac{d_{N}^{3}}{N^{2}}\Big).
\]
\end{lem}

\begin{proof}
By Lemma \ref{l_F-prop}-(1), any particle movements involving $x$
or $y$ do not change the value of $F$. Thus, we may only consider
the particle movements within $S_{0}$. Precisely, we may bound as
\[
\mathscr{D}_{\mathcal{M}_{N}\mathcal{R}_{N}}(F)+\mathscr{D}_{\mathcal{R}_{N}}(F)\le\sum_{\eta\in\mathcal{R}_{N}}\mu_{N}(\eta)\sum_{b,\,b'\in S_{0}}\eta_{b}(d_{N}+\eta_{b'})r(b,\,b')\{F(\eta^{b,\,b'})-F(\eta)\}^{2}.
\]
Here, we have the inequality instead of an equality, since the connected
pairs in $\mathcal{R}_{N}$ are counted twice. Note that $\eta\in\mathcal{R}_{N}$
if and only if there exist $c_{1},\,c_{2},\,c_{3}\in S_{0}$ such
that $\eta_{c_{1}},\,\eta_{c_{2}},\,\eta_{c_{3}}\ge1$ (recall Notation
\ref{n_different}). Thus, applying the crude bound $\eta_{b}(d_{N}+\eta_{b'})r(b,\,b')\le C\eta(S_{0})^{2}$
and also Lemma \ref{l_F-diff}, we may bound the right-hand side by
\[
C\sum_{n=0}^{N-3}\sum_{\ell=3}^{N-n}\sum_{j=3}^{\ell}\sum_{\{a_{1},\,\dots,\,a_{j}\}\subseteq S_{0}}\sum_{\substack{k_{1},\,\dots,\,k_{j}\ge1:\\
k_{1}+\cdots+k_{j}=\ell
}
}\mu_{N}(\xi_{n,\,k_{1},\,\dots,\,k_{j}}^{xa_{1}\cdots a_{j}y})\cdot\ell^{2}\cdot\Big(\frac{C}{N\lambda^{\ell}}\Big)^{2}.
\]
As done in the proof of Lemma \ref{l_MNMN}, we may restrict the summation
in $\ell\in\llbracket3,\,\frac{N'}{2}\rrbracket$ and $n\in\llbracket\frac{N'}{2}-1,\,N-N'\rrbracket$,
so that we are left to bound
\[
C\sum_{n=N'/2-1}^{N-N'}\sum_{\ell=3}^{N'/2}\sum_{j=3}^{\ell}\sum_{\{a_{1},\,\dots,\,a_{j}\}\subseteq S_{0}}\sum_{\substack{k_{1},\,\dots,\,k_{j}\ge1:\\
k_{1}+\cdots+k_{j}=\ell
}
}\mu_{N}(\xi_{n,\,k_{1},\,\dots,\,k_{j}}^{xa_{1}\cdots a_{j}y})\cdot\ell^{2}\cdot\Big(\frac{1}{N\lambda^{\ell}}\Big)^{2}.
\]
Substituting the exact values, this is bounded by
\[
\frac{C}{N}\sum_{n=N'/2-1}^{N-N'}\sum_{\ell=3}^{N'/2}\sum_{j=3}^{\ell}\sum_{\substack{k_{1},\,\dots,\,k_{j}\ge1:\\
k_{1}+\cdots+k_{j}=\ell
}
}\frac{d_{N}^{j+1}\ell^{2}}{nk_{1}\cdots k_{j}(N-n-\ell)}\cdot\Big(\frac{m_{\star}}{\lambda^{2}}\Big)^{\ell}.
\]
The single summation in $n$ of $\frac{1}{n(N-n-\ell)}$ is bounded
by $O(\frac{1}{N}\log N)$. Thus, we are left to prove that
\begin{equation}
\frac{\log N}{N^{2}}\sum_{\ell=3}^{N'/2}\ell^{2}\Big(\frac{m_{\star}}{\lambda^{2}}\Big)^{\ell}\sum_{j=3}^{\ell}\sum_{\substack{k_{1},\,\dots,\,k_{j}\ge1:\\
k_{1}+\cdots+k_{j}=\ell
}
}\frac{d_{N}^{j+1}}{k_{1}\cdots k_{j}}=o\Big(\frac{d_{N}^{3}}{N^{2}}\Big).\label{e_l_MNRN-RN}
\end{equation}
It is elementary (see, e.g. \cite[Lemma 9.1]{KS NRIP}) to see that
\[
\sum_{\substack{k_{1},\,\dots,\,k_{j}\ge1:\\
k_{1}+\cdots+k_{j}=\ell
}
}\frac{1}{k_{1}\cdots k_{j}}\le\Big(\frac{C\log(\ell+1)}{\ell}\Big)^{j-1}\le C^{j+1}.
\]
Therefore, we bound the left-hand side of \eqref{e_l_MNRN-RN} by
\[
\frac{\log N}{N^{2}}\sum_{\ell=3}^{N'/2}\ell^{2}\Big(\frac{m_{\star}}{\lambda^{2}}\Big)^{\ell}\sum_{j=3}^{\ell}(Cd_{N})^{j+1}\le\frac{C\log N}{N^{2}}\sum_{\ell=3}^{N'/2}\ell^{2}\Big(\frac{m_{\star}}{\lambda^{2}}\Big)^{\ell}\cdot d_{N}^{4}=O\Big(\frac{d_{N}^{4}}{N^{2}}\log N\Big)=o\Big(\frac{d_{N}^{3}}{N^{2}}\Big).
\]
This completes the proof.
\end{proof}

\subsection{\label{sec5.5}Proof of Proposition \ref{p_Diri-form}}

To conclude this section, we present the proof of Proposition \ref{p_Diri-form}.
\begin{proof}[Proof of Proposition \ref{p_Diri-form}]
 By Lemmas \ref{l_MNMN} and \ref{l_MNRN-RN}, we calculate
\[
\mathscr{D}_{N}(F)=\mathscr{D}_{\mathcal{M}_{N}}(F)+\mathscr{D}_{\mathcal{M}_{N}\mathcal{R}_{N}}(F)+\mathscr{D}_{\mathcal{R}_{N}}(F)\le\sum_{i=1}^{6}\sum_{\{a,\,b\}\subseteq S_{0}:\,\text{type }\textup{\textbf{(T}}\boldsymbol{i}\textup{\textbf{)}}}\mathscr{D}_{\mathcal{A}_{N}^{xaby}}(F)+o\Big(\frac{d_{N}^{3}}{N^{2}}\Big).
\]
Then, by Lemmas \ref{l_T1-T2}--\ref{l_T6}, it holds that $\frac{N^{2}}{d_{N}^{3}}\cdot\mathscr{D}_{N}(F)$
is bounded by $(1+o(1))$ times
\begin{equation}
\begin{aligned} & \sum_{a\in\mathcal{N}_{x}}\sum_{b\in\mathcal{N}_{y}}\frac{3c_{ab}}{(1-m_{a})(1-m_{b})}+\sum_{a\in\mathcal{N}_{x}}\sum_{b\in S_{0}\setminus(\mathcal{N}_{x}\cup\mathcal{N}_{y})}\frac{3c_{ab}}{1-m_{a}}\cdot\sum_{\ell=0}^{\infty}m_{b}^{\ell}\{1+\widehat{\mathfrak{g}}_{\ell}(\sigma^{b})-\widehat{\mathfrak{g}}_{\ell+1}(\sigma^{b})\}^{2}\\
 & +\sum_{a\in S_{0}\setminus(\mathcal{N}_{x}\cup\mathcal{N}_{y})}\sum_{b\in\mathcal{N}_{y}}\frac{3c_{ab}}{1-m_{b}}\cdot\sum_{\ell=0}^{\infty}m_{a}^{\ell}\{\widehat{\mathfrak{g}}_{\ell}(\sigma^{a})-\widehat{\mathfrak{g}}_{\ell+1}(\sigma^{a})\}^{2}\\
 & +\sum_{\{a,\,b\}\subseteq S_{0}\setminus(\mathcal{N}_{x}\cup\mathcal{N}_{y}):\,a\sim b}3c_{ab}\cdot\sum_{\ell=1}^{\infty}\sum_{k=1}^{\ell}m_{a}^{k-1}m_{b}^{\ell-k}\{\widehat{\mathfrak{g}}_{\ell}(\sigma_{k-1}^{ab})-\widehat{\mathfrak{g}}_{\ell}(\sigma_{k}^{ab})\}^{2}.
\end{aligned}
\label{e_p_Diri-form-1}
\end{equation}
In terms of the decomposition in \eqref{e_Gj'-Vj'-Ej'}, the second
and third (double) summations in \eqref{e_p_Diri-form-1} can be rewritten
as
\begin{equation}
\begin{aligned} & 3\sum_{j=1}^{s}\sum_{b\in\mathcal{A}_{x,\,j}}\sum_{a\in\mathcal{N}_{x}}\frac{c_{ab}}{1-m_{a}}\cdot\sum_{\ell=0}^{\infty}m_{b}^{\ell}\{1+\mathfrak{g}(\mathfrak{b}_{\ell})-\mathfrak{g}(\mathfrak{b}_{\ell+1})\}^{2}\\
 & +3\sum_{j=1}^{s}\sum_{a\in\mathcal{A}_{y,\,j}}\sum_{b\in\mathcal{N}_{y}}\frac{c_{ab}}{1-m_{b}}\cdot\sum_{\ell=0}^{\infty}m_{a}^{\ell}\{\mathfrak{g}(\mathfrak{a}_{\ell})-\mathfrak{g}(\mathfrak{a}_{\ell+1})\}^{2}.
\end{aligned}
\label{e_p_Diri-form-2}
\end{equation}
Moreover, by \eqref{e_trans-ver-def}, the fourth (triple) summation
in \eqref{e_p_Diri-form-1} can be rewritten as
\begin{equation}
3\sum_{j=1}^{s}\sum_{\ell=1}^{\infty}\sum_{\{a,\,b\}\subseteq\mathcal{V}_{j}'}\sum_{k=1}^{\ell}\mathfrak{r}^{\ell}(\sigma_{k}^{ab},\,\sigma_{k-1}^{ab})\{\widehat{\mathfrak{g}}_{\ell}(\sigma_{k-1}^{ab})-\widehat{\mathfrak{g}}_{\ell}(\sigma_{k}^{ab})\}^{2}.\label{e_p_Diri-form-3}
\end{equation}
Since $\widehat{\mathfrak{g}}_{\ell}$ is the harmonic extension of
$\mathfrak{g}$, by \eqref{e_dual-Diri}, this is equal to
\begin{equation}
3\sum_{j=1}^{s}\sum_{\ell=1}^{\infty}\sum_{\{a,\,b\}\subseteq\mathcal{A}_{x,\,j}\cup\mathcal{A}_{y,\,j}}\widehat{\mathfrak{r}}_{\mathfrak{a}_{\ell}\mathfrak{b}_{\ell}}^{\ell}\{\mathfrak{g}(\mathfrak{b}_{\ell})-\mathfrak{g}(\mathfrak{a}_{\ell})\}^{2}.\label{e_p_Diri-form-4}
\end{equation}
Therefore, by \eqref{e_p_Diri-form-1}, \eqref{e_p_Diri-form-2},
\eqref{e_p_Diri-form-3}, and \eqref{e_p_Diri-form-4}, along with
\eqref{e_Kxy-def}, we have proved that
\[
\frac{N^{2}}{d_{N}^{3}}\cdot\mathscr{D}_{N}(F)\le(1+o(1))\cdot\frac{1}{2\mathfrak{K}_{xy}}.
\]
Sending $N\to\infty$, we conclude the proof of Proposition \ref{p_Diri-form}.
\end{proof}

\section{\label{sec6}Lower Bound of Capacity: Test Flow}

In this section, we provide a matching lower bound of the capacity
$\mathrm{Cap}_{N}(\mathcal{E}_{N}^{x},\,\mathcal{E}_{N}^{y})$. Throughout
this section, we fix a positive integer $L$. Occasionally, we abbreviate
as $h=h_{\mathcal{E}_{N}^{x},\,\mathcal{E}_{N}^{y}}$.

\subsection{\label{sec6.1}Test flow}

We construct a test flow $\psi=\psi_{\textup{test}}^{L}\in\mathfrak{F}$.
The flow $\psi$ is constructed on each $\mathcal{A}_{N}^{xaby}$,
where $\{a,\,b\}\subseteq S_{0}$ is of one of the following four
types:
\[
\begin{aligned} & \textbf{(L3)}\ a\in\mathcal{N}_{x},\ b\in\mathcal{N}_{y},\ a\sim b;\quad &  & \textbf{(L4)}\ a\in\mathcal{N}_{x},\ b\notin\mathcal{N}_{x}\cup\mathcal{N}_{y},\ a\sim b;\\
 & \textbf{(L5)}\ a\notin\mathcal{N}_{x}\cup\mathcal{N}_{y},\ b\in\mathcal{N}_{y},\ a\sim b;\quad &  & \textbf{(L6)}\ a,\,b\notin\mathcal{N}_{x}\cup\mathcal{N}_{y},\ a\sim b.
\end{aligned}
\]
Namely, for $i\in\{3,\,4,\,5,\,6\}$, $\{a,\,b\}\subseteq S_{0}$
is of type \textbf{(L$\boldsymbol{i}$)} if and only if it is of type
\textbf{(T$\boldsymbol{i}$)} and moreover $a\sim b$.

We divide the definition of $\psi$ into bulk flow $\psi^{ab}$ for
each type of $\{a,\,b\}\subseteq S_{0}$, outer correction flows $\varphi^{a}$
for $a\in\mathcal{N}_{x}$ and $\varphi^{b}$ for $b\in\mathcal{N}_{y}$,
and edge correction flows $\varphi^{x}$ and $\varphi^{y}$, such
that
\begin{equation}
\psi:=\sum_{i=3}^{6}\sum_{\{a,\,b\}\subseteq S_{0}:\,\text{type }\textup{\textbf{(L}}\boldsymbol{i}\textup{\textbf{)}}}\psi^{ab}+\sum_{a\in\mathcal{N}_{x}}\varphi^{a}+\sum_{b\in\mathcal{N}_{y}}\varphi^{b}+\varphi^{x}+\varphi^{y}.\label{e_psi-def}
\end{equation}
To define the flows, instead of the objects $\mathfrak{g}$ and $\widehat{\mathfrak{g}}_{\ell}$
used in Section \ref{sec5.1}, we must use the truncated objects $\mathfrak{g}^{L}$
and $\widehat{\mathfrak{g}}_{\ell}^{L}$. This is necessary to control
the small flows on the level $L$ as $0$ rather than just negligible.

\begin{figure}
\begin{tikzpicture}
\fill[rounded corners,white] (-2.5,3.7) rectangle (2.5,-1.7); 
\fill[black!10!white] (0.2,2.7)--(0,2.4)--(1.4,0.3)--(1.6,0.2)--(1.8,0.3)--cycle;
\draw[densely dashed] (-2,0)--(2,0);
\draw (0,3)--(-2,0)--(0,-1)--cycle; \draw (0,-1)--(2,0)--(0,3);
\draw (1.4,0.3)--(1.8,0.3);
\draw[very thick] (0.2,2.7)--(0,2.4)--(1.4,0.3)--(1.6,0.2)--(1.8,0.3)--cycle; \draw[very thick] (0,2.4)--(0.2,2.3)--(0.2,2.7); \draw[very thick] (0.2,2.3)--(1.6,0.2);
\draw (0,3.15) node[above]{\textcolor{white}{$\xi_n^{xy}$}};
\draw (0,-1.15) node[below]{\textcolor{white}{$\xi_{n-\ell,\,\ell}^{xby}$}};
\draw (0,3.1) node[above]{$\xi^x$};
\draw (0.6,2.7) node[above]{\footnotesize $\xi_{N-L}^{xy}$};
\draw (2.1,0.3) node[above]{\footnotesize $\xi_{L}^{xy}$};
\draw (-2,-0.1) node[below]{$\xi^a$};
\draw (2,-0.1) node[below]{$\xi^y$};
\draw (0,-1.1) node[below]{$\xi^b$};
\end{tikzpicture}
\begin{tikzpicture}
\fill[rounded corners,white] (-2.5,3.7) rectangle (2.5,-1.7); 
\draw[densely dashed] (-2,0)--(2,0);
\draw (0,3)--(-2,0)--(0,-1)--cycle; \draw (0,-1)--(2,0)--(0,3);
\draw[very thick,-to,red] (-0.4,2.4)--(0,2.2);
\draw[very thick,-to,red] (-0.8,1.8)--(-0.4,1.6); \draw[very thick,-to,red] (-0.4,1.6)--(0,1.4);
\draw[very thick,-to,red] (-1.2,1.2)--(-0.8,1); \draw[very thick,-to,red] (-0.8,1)--(-0.4,0.8); \draw[very thick,-to,red] (-0.4,0.8)--(0,0.6);
\draw[very thick,-to,red] (-1.6,0.6)--(-1.2,0.4); \draw[very thick,-to,red] (-1.2,0.4)--(-0.8,0.2); \draw[very thick,-to,red] (-0.8,0.2)--(-0.4,0); \draw[very thick,-to,red] (-0.4,0)--(0,-0.2);
\draw[very thick,-to,red] (-2,0)--(-1.6,-0.2); \draw[very thick,-to,red] (-1.6,-0.2)--(-1.2,-0.4); \draw[very thick,-to,red] (-1.2,-0.4)--(-0.8,-0.6); \draw[very thick,-to,red] (-0.8,-0.6)--(-0.4,-0.8); \draw[very thick,-to,red] (-0.4,-0.8)--(0,-1);
\draw[very thick,-to,blue] (-0.4,-0.8)--(0,-0.2);
\draw[very thick,-to,blue] (-0.8,-0.6)--(-0.4,0); \draw[very thick,-to,blue] (-0.4,0)--(0,0.6);
\draw[very thick,-to,blue] (-1.2,-0.4)--(-0.8,0.2); \draw[very thick,-to,blue] (-0.8,0.2)--(-0.4,0.8); \draw[very thick,-to,blue] (-0.4,0.8)--(0,1.4);
\draw[very thick,-to,blue] (-1.6,-0.2)--(-1.2,0.4); \draw[very thick,-to,blue] (-1.2,0.4)--(-0.8,1); \draw[very thick,-to,blue] (-0.8,1)--(-0.4,1.6); \draw[very thick,-to,blue] (-0.4,1.6)--(0,2.2);
\draw[very thick,-to,green] (0,3)--(-0.4,2.4); \draw[very thick,-to,green] (-0.4,2.4)--(-0.8,1.8); \draw[very thick,-to,green] (-0.8,1.8)--(-1.2,1.2); \draw[very thick,-to,green] (-1.2,1.2)--(-1.6,0.6); \draw[very thick,-to,green] (-1.6,0.6)--(-2,0);
\draw[very thick,-to,green] (0,2.2)--(0.4,2.4);
\draw[very thick,-to,green] (0,1.4)--(0.4,1.6); \draw[very thick,-to,green] (0.4,1.6)--(0.8,1.8);
\draw[very thick,-to,green] (0,0.6)--(0.4,0.8); \draw[very thick,-to,green] (0.4,0.8)--(0.8,1); \draw[very thick,-to,green] (0.8,1)--(1.2,1.2);
\draw[very thick,-to,green] (0,-0.2)--(0.4,0); \draw[very thick,-to,green] (0.4,0)--(0.8,0.2); \draw[very thick,-to,green] (0.8,0.2)--(1.2,0.4); \draw[very thick,-to,green] (1.2,0.4)--(1.6,0.6);
\draw[very thick,-to,green] (0,-1)--(0.4,-0.8); \draw[very thick,-to,green] (0.4,-0.8)--(0.8,-0.6); \draw[very thick,-to,green] (0.8,-0.6)--(1.2,-0.4); \draw[very thick,-to,green] (1.2,-0.4)--(1.6,-0.2); \draw[very thick,-to,green] (1.6,-0.2)--(2,0);
\draw (0,3.1) node[above]{$\xi_n^{xy}$};
\draw (-2,-0.1) node[below]{$\xi_{n-\ell,\,\ell}^{xay}$};
\draw (2,-0.1) node[below]{$\xi_{n-\ell}^{xy}$};
\draw (0,-1.1) node[below]{$\xi_{n-\ell,\,\ell}^{xby}$};
\end{tikzpicture}
\begin{tikzpicture}
\fill[rounded corners,white] (-2.5,3.7) rectangle (2.5,-1.7); 
\draw (0,3)--(-2,0)--(0,-1)--cycle;
\draw[very thick,-to,red] (-0.4,2.4)--(0,2.2);
\draw[very thick,-to,red] (-0.8,1.8)--(-0.4,1.6); \draw[very thick,-to,red] (-0.4,1.6)--(0,1.4);
\draw[very thick,-to,red] (-1.2,1.2)--(-0.8,1); \draw[very thick,-to,red] (-0.8,1)--(-0.4,0.8); \draw[very thick,-to,red] (-0.4,0.8)--(0,0.6);
\draw[very thick,-to,red] (-1.6,0.6)--(-1.2,0.4); \draw[very thick,-to,red] (-1.2,0.4)--(-0.8,0.2); \draw[very thick,-to,red] (-0.8,0.2)--(-0.4,0); \draw[very thick,-to,red] (-0.4,0)--(0,-0.2);
\draw[very thick,-to,red] (-2,0)--(-1.6,-0.2); \draw[very thick,-to,red] (-1.6,-0.2)--(-1.2,-0.4); \draw[very thick,-to,red] (-1.2,-0.4)--(-0.8,-0.6); \draw[very thick,-to,red] (-0.8,-0.6)--(-0.4,-0.8); \draw[very thick,-to,red] (-0.4,-0.8)--(0,-1);
\draw[very thick,-to,blue] (-0.4,-0.8)--(0,-0.2);
\draw[very thick,-to,blue] (-0.8,-0.6)--(-0.4,0); \draw[very thick,-to,blue] (-0.4,0)--(0,0.6);
\draw[very thick,-to,blue] (-1.2,-0.4)--(-0.8,0.2); \draw[very thick,-to,blue] (-0.8,0.2)--(-0.4,0.8); \draw[very thick,-to,blue] (-0.4,0.8)--(0,1.4);
\draw[very thick,-to,blue] (-1.6,-0.2)--(-1.2,0.4); \draw[very thick,-to,blue] (-1.2,0.4)--(-0.8,1); \draw[very thick,-to,blue] (-0.8,1)--(-0.4,1.6); \draw[very thick,-to,blue] (-0.4,1.6)--(0,2.2);
\draw[very thick,-to,green] (0,3)--(-0.4,2.4); \draw[very thick,-to,green] (-0.4,2.4)--(-0.8,1.8); \draw[very thick,-to,green] (-0.8,1.8)--(-1.2,1.2); \draw[very thick,-to,green] (-1.2,1.2)--(-1.6,0.6); \draw[very thick,-to,green] (-1.6,0.6)--(-2,0);
\draw (0,3.1) node[above]{$\xi_n^{xy}$};
\draw (-2,-0.1) node[below]{$\xi_{n-\ell,\,\ell}^{xay}$};
\draw (0,-1.1) node[below]{$\xi_{n-\ell,\,\ell}^{xby}$};
\end{tikzpicture}\\
\begin{tikzpicture}
\fill[rounded corners,white] (-2.5,0.2) rectangle (2.5,-1.7); 
\draw (2,0)--(-2,0)--(0,-1)--cycle;
\draw[very thick,-to,red] (1.2,0)--(1.6,-0.2);
\draw[very thick,-to,red] (0.4,0)--(0.8,-0.2); \draw[very thick,-to,red] (0.8,-0.2)--(1.2,-0.4);
\draw[very thick,-to,red] (-0.4,0)--(0,-0.2); \draw[very thick,-to,red] (0,-0.2)--(0.4,-0.4); \draw[very thick,-to,red] (0.4,-0.4)--(0.8,-0.6);
\draw[very thick,-to,red] (-1.2,0)--(-0.8,-0.2); \draw[very thick,-to,red] (-0.8,-0.2)--(-0.4,-0.4); \draw[very thick,-to,red] (-0.4,-0.4)--(0,-0.6); \draw[very thick,-to,red] (0,-0.6)--(0.4,-0.8);
\draw[very thick,-to,red] (-2,0)--(-1.6,-0.2); \draw[very thick,-to,red] (-1.6,-0.2)--(-1.2,-0.4); \draw[very thick,-to,red] (-1.2,-0.4)--(-0.8,-0.6); \draw[very thick,-to,red] (-0.8,-0.6)--(-0.4,-0.8); \draw[very thick,-to,red] (-0.4,-0.8)--(0,-1);
\draw[very thick,-to,blue] (-1.6,-0.2)--(-1.2,0);
\draw[very thick,-to,blue] (-1.2,-0.4)--(-0.8,-0.2); \draw[very thick,-to,blue] (-0.8,-0.2)--(-0.4,0);
\draw[very thick,-to,blue] (-0.8,-0.6)--(-0.4,-0.4); \draw[very thick,-to,blue] (-0.4,-0.4)--(0,-0.2); \draw[very thick,-to,blue] (0,-0.2)--(0.4,0);
\draw[very thick,-to,blue] (-0.4,-0.8)--(0,-0.6); \draw[very thick,-to,blue] (0,-0.6)--(0.4,-0.4); \draw[very thick,-to,blue] (0.4,-0.4)--(0.8,-0.2); \draw[very thick,-to,blue] (0.8,-0.2)--(1.2,0);
\draw[very thick,-to,green] (0,-1)--(0.4,-0.8); \draw[very thick,-to,green] (0.4,-0.8)--(0.8,-0.6); \draw[very thick,-to,green] (0.8,-0.6)--(1.2,-0.4); \draw[very thick,-to,green] (1.2,-0.4)--(1.6,-0.2); \draw[very thick,-to,green] (1.6,-0.2)--(2,0);
\draw (-2,-0.1) node[below]{$\xi_{n,\,\ell}^{xay}$};
\draw (2,-0.1) node[below]{$\xi_n^{xy}$};
\draw (0,-1.1) node[below]{$\xi_{n,\,\ell}^{xby}$};
\end{tikzpicture}
\begin{tikzpicture}
\fill[rounded corners,white] (-2.5,0.2) rectangle (2.5,-1.7); 
\draw (2,0)--(-2,0)--(0,-1)--cycle;
\draw[very thick,-to,red] (1.2,0)--(1.6,-0.2);
\draw[very thick,-to,red] (0.4,0)--(0.8,-0.2); \draw[very thick,-to,red] (0.8,-0.2)--(1.2,-0.4);
\draw[very thick,-to,red] (-0.4,0)--(0,-0.2); \draw[very thick,-to,red] (0,-0.2)--(0.4,-0.4); \draw[very thick,-to,red] (0.4,-0.4)--(0.8,-0.6);
\draw[very thick,-to,red] (-1.2,0)--(-0.8,-0.2); \draw[very thick,-to,red] (-0.8,-0.2)--(-0.4,-0.4); \draw[very thick,-to,red] (-0.4,-0.4)--(0,-0.6); \draw[very thick,-to,red] (0,-0.6)--(0.4,-0.8);
\draw[very thick,-to,red] (-2,0)--(-1.6,-0.2); \draw[very thick,-to,red] (-1.6,-0.2)--(-1.2,-0.4); \draw[very thick,-to,red] (-1.2,-0.4)--(-0.8,-0.6); \draw[very thick,-to,red] (-0.8,-0.6)--(-0.4,-0.8); \draw[very thick,-to,red] (-0.4,-0.8)--(0,-1);
\draw (-2,-0.1) node[below]{$\xi_{n,\,\ell}^{xay}$};
\draw (2,-0.1) node[below]{$\xi_n^{xy}$};
\draw (0,-1.1) node[below]{$\xi_{n,\,\ell}^{xby}$};
\end{tikzpicture}\caption{\label{fig6.1}Diagrams of bulk flow $\psi^{ab}$ and outer correction
flows $\varphi^{a}$ and $\varphi^{b}$ on $\mathcal{A}_{N}^{xaby}$.
The flows are supported on the gray region in the first figure. The
remaining four figures illustrate the flows according to the types
\textbf{(L3)}--\textbf{(L6)}, respectively. Red arrows indicate the
main part $\phi^{ab}$, blue arrows indicate the inner correction
part $\varphi^{ab}$, and green arrows indicate the outer correction
flows $\varphi^{a}$ (\textbf{(L3)} and \textbf{(L4)}) and $\varphi^{b}$
(\textbf{(L3)} and \textbf{(L5)}).}
\end{figure}
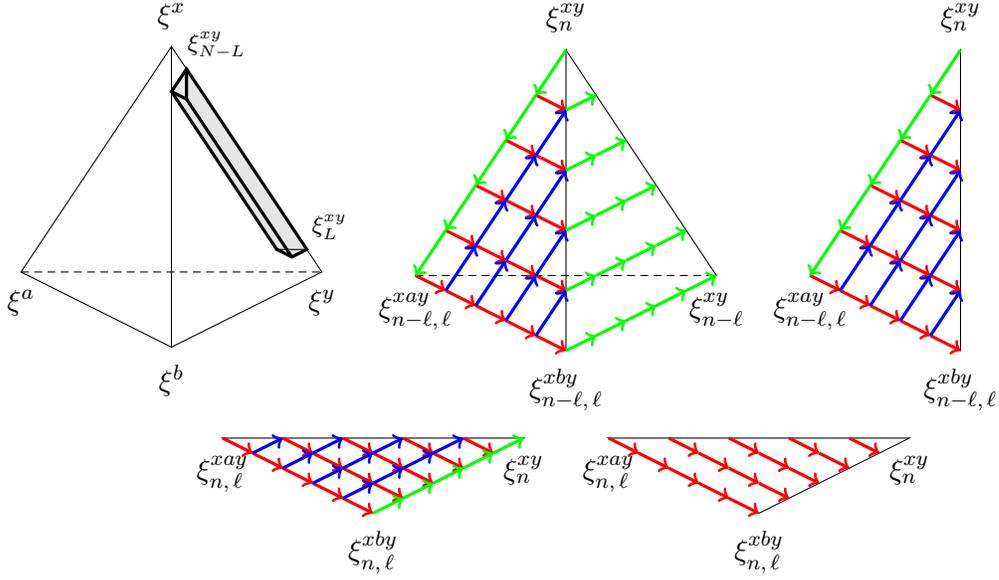

\begin{defn}[\textbf{(L3)} Bulk flow on $\mathcal{A}_{N}^{xaby}$ if $a\in\mathcal{N}_{x}$,
$b\in\mathcal{N}_{y}$, $a\sim b$]
\label{d_tfl-L3}The definition of bulk flow $\psi^{ab}$ is divided
into main part $\phi^{ab}$ and inner correction part $\varphi^{ab}$;
refer to Figure \ref{fig6.1}.
\begin{enumerate}
\item For $n\in\llbracket2L,\,N-L\rrbracket$, $\ell\in\llbracket1,\,L-1\rrbracket$,
and $k\in\llbracket1,\,\ell\rrbracket$, define
\begin{equation}
\phi^{ab}(\xi_{n-\ell,\,k,\,\ell-k}^{xaby},\,\xi_{n-\ell,\,k-1,\,\ell-k+1}^{xaby}):=c_{ab}\cdot m_{a}^{k-1}m_{b}^{\ell-k}.\label{e_tfl-L3-1}
\end{equation}
Moreover, for $k\in\llbracket1,\,L\rrbracket$, define
\begin{equation}
\phi^{ab}(\xi_{n-L,\,k,\,L-k}^{xaby},\,\xi_{n-L,\,k-1,\,L-k+1}^{xaby}):=c_{ab}\cdot\frac{m_{a}^{k-1}m_{b}^{L-k}}{1-m_{a}}.\label{e_tfl-L3-2}
\end{equation}
The main flow $\phi^{ab}$ is depicted as red arrows in Figure \ref{fig6.1}.
\item For $n\in\llbracket2L,\,N-L\rrbracket$, $\ell\in\llbracket2,\,L\rrbracket$,
and $k\in\llbracket1,\,\ell-1\rrbracket$, define
\begin{equation}
\varphi^{ab}(\xi_{n-\ell,\,k,\,\ell-k}^{xaby},\,\xi_{n-\ell+1,\,k-1,\,\ell-k}^{xaby}):=-\sum_{j=\ell}^{L}(\mathrm{div}\,\phi^{ab})(\xi_{n-j,\,k+j-\ell,\,\ell-k}^{xaby}).\label{e_tfl-L3-3}
\end{equation}
Moreover, for $n\in\llbracket N-2L+2,\,N-L\rrbracket$, $\ell\in\llbracket1,\,n-(N-2L+1)\rrbracket$,
and $k\in\llbracket0,\,\ell-1\rrbracket$, define
\begin{equation}
\varphi^{ab}(\xi_{n-\ell,\,k,\,\ell-k}^{xaby},\,\xi_{n-\ell,\,k+1,\,\ell-k-1}^{xaby}):=\phi^{ab}(\xi_{n-\ell,\,1,\,\ell-1}^{xaby},\,\xi_{n-\ell,\,\ell}^{xby})+\varphi^{ab}(\xi_{n-\ell-1,\,1,\,\ell}^{xaby},\,\xi_{n-\ell,\,\ell}^{xaby}).\label{e_tfl-L3-4}
\end{equation}
The inner correction flow $\varphi^{ab}$ is depicted as blue arrows
in Figure \ref{fig6.1}. The exceptional flows defined in \eqref{e_tfl-L3-4},
which are not depicted in Figure \ref{fig6.1}, are designed to handle
small leftovers near $\xi^{x}$.
\item Finally, define $\psi^{ab}:=\phi^{ab}+\varphi^{ab}$.
\end{enumerate}
\end{defn}

\begin{defn}[\textbf{(L4)} Bulk flow on $\mathcal{A}_{N}^{xaby}$ if $a\in\mathcal{N}_{x}$,
$b\notin\mathcal{N}_{x}\cup\mathcal{N}_{y}$, $a\sim b$]
\label{d_tfl-L4}$ $
\begin{enumerate}
\item For $n\in\llbracket2L,\,N-L\rrbracket$, $\ell\in\llbracket1,\,L-1\rrbracket$,
and $k\in\llbracket1,\,\ell\rrbracket$, define
\begin{equation}
\phi^{ab}(\xi_{n-\ell,\,k,\,\ell-k}^{xaby},\,\xi_{n-\ell,\,k-1,\,\ell-k+1}^{xaby}):=c_{ab}\cdot m_{a}^{k-1}m_{b}^{\ell-k}\cdot(1+\widehat{\mathfrak{g}}_{\ell-k}^{L}(\sigma^{b})-\widehat{\mathfrak{g}}_{\ell-k+1}^{L}(\sigma^{b})).\label{e_tfl-L4-1}
\end{equation}
Moreover, for $k\in\llbracket1,\,L\rrbracket$, define
\begin{equation}
\phi^{ab}(\xi_{n-L,\,k,\,L-k}^{xaby},\,\xi_{n-L,\,k-1,\,L-k+1}^{xaby}):=c_{ab}\cdot\frac{m_{a}^{k-1}m_{b}^{L-k}(1+\widehat{\mathfrak{g}}_{L-k}^{L}(\sigma^{b})-\widehat{\mathfrak{g}}_{L-k+1}^{L}(\sigma^{b}))}{1-m_{a}}.\label{e_tfl-L4-2}
\end{equation}
\item Define $\varphi^{ab}$ as in \eqref{e_tfl-L3-3} and \eqref{e_tfl-L3-4}.
\item Finally, define $\psi^{ab}:=\phi^{ab}+\varphi^{ab}$.
\end{enumerate}
\end{defn}

\begin{defn}[\textbf{(L5)} Bulk flow on $\mathcal{A}_{N}^{xaby}$ if $a\notin\mathcal{N}_{x}\cup\mathcal{N}_{y}$,
$b\in\mathcal{N}_{y}$, $a\sim b$]
\label{d_tfl-L5}$ $
\begin{enumerate}
\item For $n\in\llbracket L,\,N-2L\rrbracket$, $\ell\in\llbracket1,\,L-1\rrbracket$,
and $k\in\llbracket1,\,\ell\rrbracket$, define
\[
\phi^{ab}(\xi_{n,\,k,\,\ell-k}^{xaby},\,\xi_{n,\,k-1,\,\ell-k+1}^{xaby}):=c_{ab}\cdot m_{a}^{k-1}m_{b}^{\ell-k}\cdot(\widehat{\mathfrak{g}}_{k}^{L}(\sigma^{a})-\widehat{\mathfrak{g}}_{k-1}^{L}(\sigma^{a})).
\]
Moreover, for $k\in\llbracket1,\,L\rrbracket$, define
\[
\phi^{ab}(\xi_{n,\,k,\,L-k}^{xaby},\,\xi_{n,\,k-1,\,L-k+1}^{xaby}):=c_{ab}\cdot\frac{m_{a}^{k-1}m_{b}^{L-k}(\widehat{\mathfrak{g}}_{k}^{L}(\sigma^{a})-\widehat{\mathfrak{g}}_{k-1}^{L}(\sigma^{a}))}{1-m_{b}}.
\]
\item For $n\in\llbracket L,\,N-2L\rrbracket$, $\ell\in\llbracket2,\,L\rrbracket$,
and $k\in\llbracket1,\,\ell-1\rrbracket$, define
\begin{equation}
\varphi^{ab}(\xi_{n,\,k,\,\ell-k}^{xaby},\,\xi_{n,\,k,\,\ell-k-1}^{xaby}):=-\sum_{j=\ell}^{L}(\mathrm{div}\,\phi^{ab})(\xi_{n,\,k,\,j-k}^{xaby}).\label{e_tfl-L5-1}
\end{equation}
Moreover, for $n\in\llbracket L,\,2L-2\rrbracket$, $\ell\in\llbracket1,\,2L-1-n\rrbracket$,
and $k\in\llbracket1,\,\ell\rrbracket$, define
\begin{equation}
\varphi^{ab}(\xi_{n,\,k,\,\ell-k}^{xaby},\,\xi_{n,\,k-1,\,\ell-k+1}^{xaby}):=-\phi^{ab}(\xi_{n,\,\ell}^{xay},\,\xi_{n,\,\ell-1,\,1}^{xaby})+\varphi^{ab}(\xi_{n,\,\ell,\,1}^{xaby},\,\xi_{n,\,\ell}^{xay}).\label{e_tfl-L5-2}
\end{equation}
The exceptional flows defined in \eqref{e_tfl-L5-2}, which are not
depicted in Figure \ref{fig6.1}, are designed to handle small leftovers
near $\xi^{y}$.
\item Finally, define $\psi^{ab}:=\phi^{ab}+\varphi^{ab}$.
\end{enumerate}
\end{defn}

\begin{defn}[\textbf{(L6)} Bulk flow on $\mathcal{A}_{N}^{xaby}$ if $a,\,b\notin\mathcal{N}_{x}\cup\mathcal{N}_{y}$,
$a\sim b$]
\label{d_tfl-L6}In this final case, we do not need an inner correction
flow. For real numbers $\alpha$ and $\beta$, write $\alpha\vee\beta:=\max\{\alpha,\,\beta\}$.
\begin{enumerate}
\item For $n\in\llbracket L,\,N-2L\rrbracket$, $\ell\in\llbracket1\vee(2L-n),\,L\rrbracket$,
and $k\in\llbracket1,\,\ell\rrbracket$, define
\begin{equation}
\phi^{ab}(\xi_{n,\,k,\,\ell-k}^{xaby},\,\xi_{n,\,k-1,\,\ell-k+1}^{xaby}):=c_{ab}\cdot m_{a}^{k-1}m_{b}^{\ell-k}\cdot(\widehat{\mathfrak{g}}_{\ell}^{L}(\sigma_{k}^{ab})-\widehat{\mathfrak{g}}_{\ell}^{L}(\sigma_{k-1}^{ab})).\label{e_tfl-L6}
\end{equation}
\item Then, define $\psi^{ab}:=\phi^{ab}$.
\end{enumerate}
\end{defn}

Next, we define the outer correction flows $\varphi^{a}$ and $\varphi^{b}$
for $a\in\mathcal{N}_{x}$ and $b\in\mathcal{N}_{y}$.
\begin{defn}[Outer correction flows $\varphi^{a}$, $a\in\mathcal{N}_{x}$ and
$\varphi^{b}$, $b\in\mathcal{N}_{y}$]
\label{d_tfl-outer}The outer correction flows $\varphi^{a}$ and
$\varphi^{b}$ are depicted as green arrows in Figure \ref{fig6.1}.
\begin{enumerate}
\item First, fix $a\in\mathcal{N}_{x}$. For $n\in\llbracket2L,\,N-L\rrbracket$
and $\ell\in\llbracket1,\,L\rrbracket$, define
\begin{equation}
\varphi^{a}(\xi_{n-\ell+1,\,\ell-1}^{xay},\,\xi_{n-\ell,\,\ell}^{xay}):=\sum_{b\in S_{0}\setminus\mathcal{N}_{x}:\,a\sim b}\sum_{j=\ell}^{L}\psi^{ab}(\xi_{n-j,\,j}^{xay},\,\xi_{n-j,\,j-1,\,1}^{xaby}).\label{e_tfl-outer-1}
\end{equation}
\item Next, fix $b\in\mathcal{N}_{y}$. For $n\in\llbracket L,\,N-2L\rrbracket$
and $\ell\in\llbracket1,\,L\rrbracket$, define
\begin{equation}
\varphi^{b}(\xi_{n,\,\ell}^{xby},\,\xi_{n,\,\ell-1}^{xby}):=\sum_{a\in S_{0}\setminus\mathcal{N}_{y}:\,a\sim b}\sum_{j=\ell}^{L}\psi^{ab}(\xi_{n,\,1,\,j-1}^{xaby},\,\xi_{n,\,j}^{xby}).\label{e_tfl-outer-2}
\end{equation}
\end{enumerate}
\end{defn}

Finally, we define the edge correction flows $\varphi^{x}$ and $\varphi^{y}$.

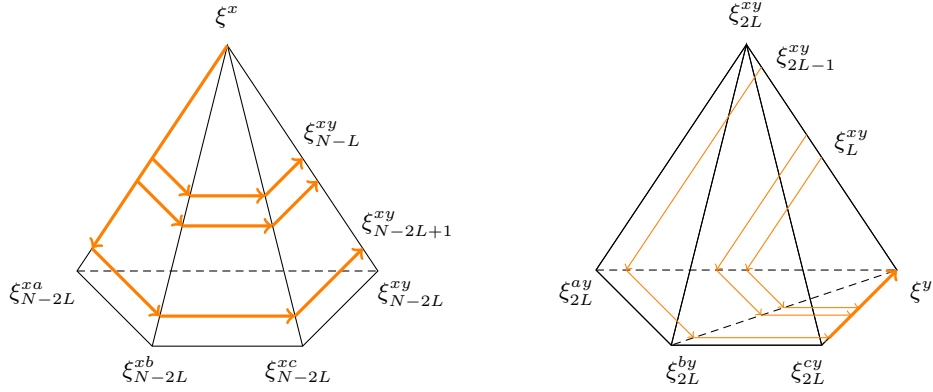
\begin{figure}
\begin{tikzpicture}
\fill[rounded corners,white] (-3,3.6) rectangle (3,-1.6); 
\draw[densely dashed] (-2,0)--(2,0);
\draw (0,3)--(-2,0)--(-1,-1)--(1,-1)--(2,0)--cycle; \draw (-1,-1)--(0,3)--(1,-1);
\draw (0,3.1) node[above]{\footnotesize $\xi^x$};
\draw (1.4,1.5) node[above]{\footnotesize $\xi_{N-L}^{xy}$};
\draw (2.4,0.3) node[above]{\footnotesize $\xi_{N-2L+1}^{xy}$};
\draw (0,0) node[below]{\footnotesize $\xi_{N-2L}^{xa}$\hspace{4cm}$\xi_{N-2L}^{xy}$};
\draw (0,-1) node[below]{\footnotesize $\xi_{N-2L}^{xb}$\hspace{1cm}$\xi_{N-2L}^{xc}$};
\draw[very thick,to-,orange] (1,1.5)--(0.5,1);
\draw[very thick,to-,orange] (0.5,1)--(-0.5,1);
\draw[very thick,to-,orange] (-0.5,1)--(-1,1.5);
\draw[very thick,to-,orange] (1.2,1.2)--(0.6,0.6);
\draw[very thick,to-,orange] (0.6,0.6)--(-0.6,0.6);
\draw[very thick,to-,orange] (-0.6,0.6)--(-1.2,1.2);
\draw[very thick,to-,orange] (1.8,0.3)--(0.9,-0.6);
\draw[very thick,to-,orange] (0.9,-0.6)--(-0.9,-0.6);
\draw[very thick,to-,orange] (-0.9,-0.6)--(-1.8,0.3); 
\draw[very thick,to-,orange] (-1.8,0.3)--(0,3);
\end{tikzpicture}
\hspace{5mm}
\begin{tikzpicture}
\fill[rounded corners,white] (-3,3.6) rectangle (3,-1.6); 
\draw[densely dashed] (-2,0)--(2,0)--(-1,-1);
\draw (0,3)--(-2,0)--(-1,-1)--(1,-1)--(2,0)--cycle; \draw (-1,-1)--(0,3)--(1,-1);
\draw (0,3)--(-2,0)--(-1,-1)--(1,-1)--(2,0)--cycle; \draw (-1,-1)--(0,3)--(1,-1);
\draw (0,3.1) node[above]{\footnotesize $\xi_{2L}^{xy}$};
\draw (0.8,2.5) node[above]{\footnotesize $\xi_{2L-1}^{xy}$};
\draw (1.4,1.4) node[above]{\footnotesize $\xi_L^{xy}$};
\draw (0,0) node[below]{\footnotesize $\xi_{2L}^{ay}$\hspace{4.2cm}$\xi^y$};
\draw (0,-1) node[below]{\footnotesize $\xi_{2L}^{by}$\hspace{1.2cm}$\xi_{2L}^{cy}$};
\draw[-to,orange] (1,1.5)--(0,0);
\draw[-to,orange] (0,0)--(0.5,-0.5);
\draw[-to,orange] (0.5,-0.5)--(1.5,-0.5);
\draw[-to,orange] (0.8,1.8)--(-0.4,0);
\draw[-to,orange] (-0.4,0)--(0.2,-0.6);
\draw[-to,orange] (0.2,-0.6)--(1.4,-0.6);
\draw[-to,orange] (0.2,2.7)--(-1.6,0);
\draw[-to,orange] (-1.6,0)--(-0.7,-0.9);
\draw[-to,orange] (-0.7,-0.9)--(1.1,-0.9);
\draw[very thick,-to,orange] (1.1,-0.9)--(2,0);
\end{tikzpicture}\caption{\label{fig6.2}Edge correction flows $\varphi^{x}$ (left figure)
and $\varphi^{y}$ (right figure) along the path $x\sim a\sim b\sim c\sim y$.}
\end{figure}

\begin{defn}[Edge correction flows $\varphi^{x}$ and $\varphi^{y}$]
\label{d_tfl-edge}Fix a sequence of sites $x=a_{0}\sim a_{1}\sim\cdots\sim a_{p-1}\sim a_{p}=y$
where $p\ge3$ and $a_{1},\,\dots,\,a_{p-1}\in S_{0}$, which is possible
since $\mathcal{G}=(\mathcal{V},\,\mathcal{E})$ is connected (cf.
Definition \ref{d_G-def}). Illustrations of $\varphi^{x}$ and $\varphi^{y}$
are given in Figure \ref{fig6.2}, in the case of $p=4$.
\begin{enumerate}
\item For $\ell\in\llbracket L,\,2L-1\rrbracket$, $i\in\llbracket1,\,p-1\rrbracket$,
and $j\in\llbracket1,\,\ell\rrbracket$, define
\[
\varphi^{x}(\xi_{N-\ell,\,j}^{xa_{i}a_{i+1}},\,\xi_{N-\ell,\,j-1}^{xa_{i}a_{i+1}}):=\sum_{a\in\mathcal{N}_{x}}\varphi^{a}(\xi_{N-\ell}^{xy},\,\xi_{N-\ell-1,\,1}^{xay}).
\]
Then, for $\ell\in\llbracket1,\,2L-1\rrbracket$ define
\[
\varphi^{x}(\xi_{N-\ell+1}^{xa_{1}},\,\xi_{N-\ell}^{xa_{1}}):=\sum_{j=\ell\vee L}^{2L-1}\varphi^{x}(\xi_{N-j}^{xa_{1}},\,\xi_{N-j,\,j-1}^{xa_{1}a_{2}}).
\]
\item Similarly, for $\ell\in\llbracket L,\,2L-1\rrbracket$, $i\in\llbracket0,\,p-2\rrbracket$,
and $j\in\llbracket1,\,\ell\rrbracket$, define
\[
\varphi^{y}(\xi_{j,\,\ell-j}^{a_{i}a_{i+1}y},\,\xi_{j-1,\,\ell-j+1}^{a_{i}a_{i+1}y}):=\sum_{b\in\mathcal{N}_{y}}\varphi^{b}(\xi_{\ell,\,1}^{xby},\,\xi_{\ell}^{xy}).
\]
Then, for $\ell\in\llbracket1,\,2L-1\rrbracket$ define
\[
\varphi^{y}(\xi_{\ell}^{a_{p-1}y},\,\xi_{\ell-1}^{a_{p-1}y}):=\sum_{j=\ell\vee L}^{2L-1}\varphi^{y}(\xi_{1,\,j-1}^{a_{p-2}a_{p-1}y}\,\xi_{j}^{a_{p-1}y}).
\]
\end{enumerate}
\end{defn}

In this section, we prove the following proposition.
\begin{prop}
\label{p_flow}The following statements are valid.
\begin{enumerate}
\item The test flow $\psi$ is a flow from $\xi^{x}$ to $\xi^{y}$ (cf.
\eqref{e_flow-A-B-def}) such that
\[
\lim_{L\to\infty}(\mathrm{div}\,\psi)(\xi^{x})=\frac{1}{6\mathfrak{K}_{xy}}.
\]
\item It holds that
\[
\limsup_{L\to\infty}\limsup_{N\to\infty}\frac{d_{N}^{3}}{N^{2}}\cdot\|\psi\|^{2}\le\frac{1}{18\mathfrak{K}_{xy}}.
\]
\end{enumerate}
\end{prop}

\subsection{\label{sec6.2}Flow norm}

In this subsection, we calculate the (square of the) flow norm, $\|\psi\|^{2}$.
First, we deal with $\|\psi^{ab}\|^{2}$ for $\{a,\,b\}\subseteq S_{0}$.
\begin{lem}
\label{l_flow-norm-1}It holds that
\begin{equation}
\limsup_{L\to\infty}\limsup_{N\to\infty}\frac{d_{N}^{3}}{N^{2}}\cdot\sum_{i=3}^{6}\sum_{\{a,\,b\}\subseteq S_{0}:\,\text{type }\textup{\textbf{(L}}\boldsymbol{i}\textup{\textbf{)}}}\|\psi^{ab}\|^{2}\le\frac{1}{18\mathfrak{K}_{xy}}.\label{e_l_flow-norm-1}
\end{equation}
\end{lem}

\begin{proof}
First, we calculate the flow norm of $\psi^{ab}$ for each type \textbf{(L3)}--\textbf{(L6)}.
\begin{itemize}
\item \textbf{(L3)} $a\in\mathcal{N}_{x}$, $b\in\mathcal{N}_{y}$, and
$a\sim b$: We claim that
\begin{equation}
\limsup_{L\to\infty}\limsup_{N\to\infty}\frac{d_{N}^{3}}{N^{2}}\cdot\|\psi^{ab}\|^{2}\le\frac{c_{ab}}{3(1-m_{a})(1-m_{b})}.\label{e_flow-norm-L3-1}
\end{equation}
First, we prove that 
\begin{equation}
\limsup_{N\to\infty}\frac{d_{N}^{3}}{N^{2}}\cdot\|\phi^{ab}\|^{2}\le\frac{c_{ab}}{3(1-m_{a})(1-m_{b})}.\label{e_flow-norm-L3-2}
\end{equation}
By the definition of $\phi^{ab}$, we calculate
\[
\|\phi^{ab}\|^{2}=\sum_{n=2L}^{N-L}\sum_{\ell=1}^{L}\sum_{k=1}^{\ell}\frac{\phi^{ab}(\xi_{n-\ell,\,k,\,\ell-k}^{xaby},\,\xi_{n-\ell,\,k-1,\,\ell-k+1}^{xaby})^{2}}{\mu_{N}(\xi_{n-\ell,\,k,\,\ell-k}^{xaby})\cdot k(d_{N}+\ell-k)r(a,\,b)}.
\]
By \eqref{e_tfl-L3-1}, \eqref{e_tfl-L3-2}, and \eqref{e_wN-ZN},
the right-hand side equals
\[
\begin{aligned}\frac{2(1+o(1))}{Nd_{N}^{3}}\cdot\Big[ & \sum_{n=2L}^{N-L}\sum_{\ell=1}^{L-1}\sum_{k=1}^{\ell}(n-\ell)(N-n)\cdot m_{a}^{k-1}m_{b}^{\ell-k}\cdot m_{a}r(a,\,b)\\
 & +\sum_{n=2L}^{N-L}\sum_{k=1}^{L}(n-L)(N-n)\cdot\frac{m_{a}^{k-1}m_{b}^{L-k}}{(1-m_{a})^{2}}\cdot m_{a}r(a,\,b)\Big].
\end{aligned}
\]
Applying $n-\ell\le n$ and $n-L\le n$, and enlarging the ranges
of the summations to $n\in\llbracket1,\,N-1\rrbracket$ and $1\le\ell<\infty$,
this is bounded by
\begin{equation}
\frac{2(1+o(1))}{Nd_{N}^{3}}\cdot\frac{(N-1)N(N+1)}{6}\cdot\frac{c_{ab}}{(1-m_{a})(1-m_{b})}=\frac{(1+o(1))\cdot c_{ab}}{3(1-m_{a})(1-m_{b})}\cdot\frac{N^{2}}{d_{N}^{3}},\label{e_flow-norm-L3-3}
\end{equation}
thereby proving \eqref{e_flow-norm-L3-2}. Thus, since $\|\psi^{ab}\|\le\|\phi^{ab}\|+\|\varphi^{ab}\|$,
it remains to prove that
\begin{equation}
\lim_{L\to\infty}\lim_{N\to\infty}\frac{d_{N}^{3}}{N^{2}}\cdot\|\varphi^{ab}\|^{2}=0.\label{e_flow-norm-L3-4}
\end{equation}
By definition,
\begin{equation}
\begin{aligned}\|\varphi^{ab}\|^{2}= & \sum_{n=2L}^{N-L}\sum_{\ell=2}^{L}\sum_{k=1}^{\ell-1}\frac{\varphi^{ab}(\xi_{n-\ell,\,k,\,\ell-k}^{xaby},\,\xi_{n-\ell+1,\,k-1,\,\ell-k}^{xaby})^{2}}{\mu_{N}(\xi_{n-\ell,\,k,\,\ell-k}^{xaby})\cdot k(d_{N}+n-\ell)r(a,\,x)}.\\
 & +\sum_{n=N-2L+2}^{N-L}\sum_{\ell=1}^{n-(N-2L+1)}\sum_{k=0}^{\ell-1}\frac{\varphi^{ab}(\xi_{n-\ell,\,k,\,\ell-k}^{xaby},\,\xi_{n-\ell+1,\,k+1,\,\ell-k-1}^{xaby})^{2}}{\mu_{N}(\xi_{n-\ell,\,k,\,\ell-k}^{xaby})\cdot(\ell-k)(d_{N}+k)r(b,\,a)}.
\end{aligned}
\label{e_flow-norm-L3-5}
\end{equation}
Note that for each $j\in\llbracket\ell,\,L\rrbracket$, $|(\mathrm{div}\,\phi^{ab})(\xi_{n-j,\,k+j-\ell,\,\ell-k}^{xaby})|$
equals
\begin{equation}
\begin{aligned} & \big|\phi^{ab}(\xi_{n-j,\,k+j-\ell,\,\ell-k}^{xaby},\,\xi_{n-j,\,k+j-\ell-1,\,\ell-k+1}^{xaby})-\phi^{ab}(\xi_{n-j,\,k+j-\ell+1,\,\ell-k-1}^{xaby},\,\xi_{n-j,\,k+j-\ell,\,\ell-k}^{xaby})\big|\\
 & =c_{ab}\cdot m_{a}^{k+j-\ell-1}m_{b}^{\ell-k-1}\cdot|m_{a}-m_{b}|\le Cm_{a}^{k+j-\ell-1}m_{b}^{\ell-k},
\end{aligned}
\label{e_flow-norm-L3-6}
\end{equation}
which implies that
\begin{equation}
|\varphi^{ab}(\xi_{n-\ell,\,k,\,\ell-k}^{xaby},\,\xi_{n-\ell+1,\,k+1,\,\ell-k-1}^{xaby})|\le\sum_{j=\ell}^{L}|(\mathrm{div}\,\phi^{ab})(\xi_{n-j,\,k+j-\ell,\,\ell-k}^{xaby})|\le C\cdot m_{a}^{k-1}m_{b}^{\ell-k}.\label{e_flow-norm-L3-7}
\end{equation}
Thus, the first line in the right-hand side of \eqref{e_flow-norm-L3-5}
is bounded above by
\begin{equation}
\sum_{n=2L}^{N-L}\sum_{\ell=2}^{L}\sum_{k=1}^{\ell-1}\frac{C(\ell-k)(N-n)m_{a}^{2k-2}m_{b}^{2\ell-2k}}{d_{N}^{3}N\cdot m_{a}^{k}m_{b}^{\ell-k}}\le\frac{C}{d_{N}^{3}N}\sum_{n=2L}^{N-L}\sum_{\ell=2}^{L}\sum_{k=1}^{\ell-1}(\ell-k)(N-n)m_{\star}^{\ell-1}\le\frac{CN}{d_{N}^{3}}.\label{e_flow-norm-L3-8}
\end{equation}
Moreover, by \eqref{e_tfl-L3-1} and \eqref{e_flow-norm-L3-7}, for
$\ell\in\llbracket1,\,n-(N-2L+1)\rrbracket$,
\begin{equation}
\begin{aligned}|\varphi^{ab}(\xi_{n-\ell,\,k,\,\ell-k}^{xaby},\,\xi_{n-\ell+1,\,k+1,\,\ell-k-1}^{xaby})| & \le|(\mathrm{div}\,\phi^{ab})(\xi_{n-\ell,\,\ell}^{xby})|+|\varphi^{ab}(\xi_{n-\ell-1,\,1,\,\ell}^{xaby},\,\xi_{n-\ell,\,\ell}^{xaby})|\\
 & =|\phi^{ab}(\xi_{n-\ell,\,1,\,\ell-1}^{xaby},\,\xi_{n-\ell,\,\ell}^{xby})|+|\varphi^{ab}(\xi_{n-\ell-1,\,1,\,\ell}^{xaby},\,\xi_{n-\ell,\,\ell}^{xaby})|\\
 & \le C\cdot m_{\star}^{\ell-1}.
\end{aligned}
\label{e_flow-norm-L3-9}
\end{equation}
Thus, letting $m_{\star\star}:=\min_{a\in S_{0}}m_{a}\in(0,\,1)$,
the second line in the right-hand side of \eqref{e_flow-norm-L3-5}
is bounded above by
\begin{equation}
\begin{aligned} & C\sum_{n=N-2L+2}^{N-L}\sum_{\ell=1}^{n-(N-2L+1)}\sum_{k=0}^{\ell-1}\frac{n(N-n)m_{\star}^{\ell-1}}{d_{N}^{3}N\cdot m_{a}^{k}m_{b}^{\ell-k}}\\
 & \le\frac{C}{d_{N}^{3}N}\sum_{n=N-2L+2}^{N-L}\sum_{\ell=1}^{n-(N-2L+1)}\sum_{k=0}^{\ell-1}n(N-n)\cdot\Big(\frac{m_{\star}}{m_{\star\star}}\Big)^{\ell-1}\le\frac{C}{d_{N}^{3}}\cdot L^{3}C^{L}.
\end{aligned}
\label{e_flow-norm-L3-10}
\end{equation}
Therefore, by \eqref{e_flow-norm-L3-8} and \eqref{e_flow-norm-L3-10},
\[
\limsup_{L\to\infty}\limsup_{N\to\infty}\frac{d_{N}^{3}}{N^{2}}\cdot\|\varphi^{ab}\|^{2}\le\limsup_{L\to\infty}\limsup_{N\to\infty}\Big(\frac{C}{N}+\frac{L^{3}C^{L}}{N^{2}}\Big)=0,
\]
so that we have proved \eqref{e_flow-norm-L3-4} and thus \eqref{e_flow-norm-L3-1}.
\item \textbf{(L4)} $a\in\mathcal{N}_{x}$, $b\notin\mathcal{N}_{x}\cup\mathcal{N}_{y}$,
and $a\sim b$: We claim that
\begin{equation}
\limsup_{L\to\infty}\limsup_{N\to\infty}\frac{d_{N}^{3}}{N^{2}}\cdot\|\psi^{ab}\|^{2}\le\frac{c_{ab}}{3(1-m_{a})}\cdot\sum_{\ell=0}^{\infty}m_{b}^{\ell}\{1+\widehat{\mathfrak{g}}_{\ell}(\sigma^{b})-\widehat{\mathfrak{g}}_{\ell+1}(\sigma^{b})\}^{2}.\label{e_flow-norm-L4-1}
\end{equation}
The fact that $\|\phi^{ab}\|^{2}$ is bounded above by the right-hand
side of \eqref{e_flow-norm-L4-1} can be proved in the same way. Moreover,
the negligibility of $\|\varphi^{ab}\|^{2}$ also follows by replacing
\eqref{e_flow-norm-L3-7} and \eqref{e_flow-norm-L3-9} with, respectively,
\[
|\varphi^{ab}(\xi_{n-\ell,\,k,\,\ell-k}^{xaby},\,\xi_{n-\ell+1,\,k+1,\,\ell-k-1}^{xaby})|\le\frac{Cm_{a}^{k-1}m_{b}^{\ell-k}}{\lambda^{\ell-k}}
\]
and
\[
|\varphi^{ab}(\xi_{n-\ell,\,k,\,\ell-k}^{xaby},\,\xi_{n-\ell+1,\,k+1,\,\ell-k-1}^{xaby})|\le C\cdot\frac{m_{\star}^{\ell-1}}{\lambda^{\ell}}.
\]
\item \textbf{(L5)} $a\notin\mathcal{N}_{x}\cup\mathcal{N}_{y}$, $b\in\mathcal{N}_{y}$,
and $a\sim b$: We claim that
\[
\limsup_{L\to\infty}\limsup_{N\to\infty}\frac{d_{N}^{3}}{N^{2}}\cdot\|\psi^{ab}\|^{2}\le\frac{c_{ab}}{3(1-m_{b})}\cdot\sum_{\ell=0}^{\infty}m_{a}^{\ell}\{\widehat{\mathfrak{g}}_{\ell+1}(\sigma^{a})-\widehat{\mathfrak{g}}_{\ell}(\sigma^{a})\}^{2}.
\]
The proof is similar; we omit the details.
\item \textbf{(L6)} $a,\,b\notin\mathcal{N}_{x}\cup\mathcal{N}_{y}$ and
$a\sim b$: We claim that
\[
\limsup_{L\to\infty}\limsup_{N\to\infty}\frac{d_{N}^{3}}{N^{2}}\cdot\|\psi^{ab}\|^{2}\le\frac{c_{ab}}{3}\cdot\sum_{\ell=1}^{\infty}\sum_{k=1}^{\ell}m_{a}^{k-1}m_{b}^{\ell-k}\{\widehat{\mathfrak{g}}_{\ell}(\sigma_{k}^{ab})-\widehat{\mathfrak{g}}_{\ell}(\sigma_{k-1}^{ab})\}^{2}.
\]
Again, we omit the proof.
\end{itemize}
Now, gathering the four inequalities above, the left-hand side of
\eqref{e_l_flow-norm-1} is bounded by
\[
\begin{aligned} & \sum_{a\in\mathcal{N}_{x}}\sum_{b\in\mathcal{N}_{y}}\frac{c_{ab}}{3(1-m_{a})(1-m_{b})}+\sum_{a\in\mathcal{N}_{x}}\sum_{b\in S_{0}\setminus(\mathcal{N}_{x}\cup\mathcal{N}_{y})}\frac{c_{ab}}{3(1-m_{a})}\cdot\sum_{\ell=0}^{\infty}m_{b}^{\ell}\{1+\widehat{\mathfrak{g}}_{\ell}(\sigma^{b})-\widehat{\mathfrak{g}}_{\ell+1}(\sigma^{b})\}^{2}\\
 & +\sum_{a\in S_{0}\setminus(\mathcal{N}_{x}\cup\mathcal{N}_{y})}\sum_{b\in\mathcal{N}_{y}}\frac{c_{ab}}{3(1-m_{b})}\cdot\sum_{\ell=0}^{\infty}m_{a}^{\ell}\{\widehat{\mathfrak{g}}_{\ell+1}(\sigma^{a})-\widehat{\mathfrak{g}}_{\ell}(\sigma^{a})\}^{2}\\
 & +\sum_{\{a,\,b\}\subseteq S_{0}\setminus(\mathcal{N}_{x}\cup\mathcal{N}_{y}):\,a\sim b}\frac{c_{ab}}{3}\cdot\sum_{\ell=1}^{\infty}\sum_{k=1}^{\ell}m_{a}^{k-1}m_{b}^{\ell-k}\{\widehat{\mathfrak{g}}_{\ell}(\sigma_{k}^{ab})-\widehat{\mathfrak{g}}_{\ell}(\sigma_{k-1}^{ab})\}^{2}.
\end{aligned}
\]
As demonstrated in the proof of Proposition \ref{p_Diri-form}, this
equals $(18\mathfrak{K}_{xy})^{-1}$.
\end{proof}
Next, we handle the outer correction flows $\varphi^{a}$, $a\in\mathcal{N}_{x}$
and $\varphi^{b}$, $b\in\mathcal{N}_{y}$.
\begin{lem}
\label{l_flow-norm-2}For $a\in\mathcal{N}_{x}$ and $b\in\mathcal{N}_{y}$,
it holds that
\[
\lim_{L\to\infty}\lim_{N\to\infty}\frac{d_{N}^{3}}{N^{2}}\cdot\|\varphi^{a}\|^{2}=\lim_{L\to\infty}\lim_{N\to\infty}\frac{d_{N}^{3}}{N^{2}}\cdot\|\varphi^{b}\|^{2}=0.
\]
\end{lem}

\begin{proof}
First, we calculate $\|\varphi^{a}\|^{2}$, which is
\begin{equation}
\sum_{n=2L}^{N-L}\sum_{\ell=1}^{L}\frac{\varphi^{a}(\xi_{n-\ell+1,\,\ell-1}^{xay},\,\xi_{n-\ell,\,\ell}^{xay})^{2}}{\mu_{N}(\xi_{n-\ell+1,\,\ell-1}^{xay})\cdot(n-\ell+1)(d_{N}+\ell-1)r(x,\,a)}.\label{e_l_flow-norm-2-1}
\end{equation}
By \eqref{e_tfl-outer-1} and Definitions \ref{d_tfl-L3} and \ref{d_tfl-L4},
we have
\begin{equation}
\varphi^{a}(\xi_{n-\ell+1,\,\ell-1}^{xay},\,\xi_{n-\ell,\,\ell}^{xay})=\sum_{b\in S_{0}\setminus\mathcal{N}_{x}:\,a\sim b}\sum_{j=\ell}^{L}\psi^{ab}(\xi_{n-j,\,j}^{xay},\,\xi_{n-j,\,j-1,\,1}^{xaby})\le C\sum_{j=\ell}^{L}\frac{m_{\star}^{j-1}}{\lambda^{j}}\le C\frac{m_{\star}^{\ell-1}}{\lambda^{\ell}}.\label{e_l_flow-norm-2-2}
\end{equation}
Therefore, we may bound \eqref{e_l_flow-norm-2-1} by
\[
C\sum_{n=2L}^{N-L}\sum_{\ell=1}^{L}\frac{(N-n)m_{\star}^{\ell-1}}{Nd_{N}^{2}m_{\star\star}^{\ell-1}}\le\frac{CN}{d_{N}^{2}}\cdot\Big(\frac{m_{\star}}{m_{\star\star}}\Big)^{L},
\]
which proves the limit for $\varphi^{a}$. The proof for $\varphi^{b}$
is almost identical.
\end{proof}
Finally, we deal with the edge correction flows $\varphi^{x}$ and
$\varphi^{y}$.
\begin{lem}
\label{l_flow-norm-3}It holds that
\[
\lim_{L\to\infty}\lim_{N\to\infty}\frac{d_{N}^{3}}{N^{2}}\cdot\|\varphi^{x}\|^{2}=\lim_{L\to\infty}\lim_{N\to\infty}\frac{d_{N}^{3}}{N^{2}}\cdot\|\varphi^{y}\|^{2}=0.
\]
\end{lem}

\begin{proof}
We first deal with $\varphi^{x}$. By \eqref{e_l_flow-norm-2-2} and
Definitions \ref{d_tfl-outer} and \ref{d_tfl-edge}, for $i\in\llbracket1,\,p-1\rrbracket$
and $j\in\llbracket1,\,\ell\rrbracket$, we have $|\varphi^{x}(\xi_{N-\ell,\,j}^{xa_{i}a_{i+1}},\,\xi_{N-\ell,\,j-1}^{xa_{i}a_{i+1}})|\le C$.
Then, again by Definition \ref{d_tfl-edge},
\[
|\varphi^{x}(\xi_{N-\ell+1}^{xa_{1}},\,\xi_{N-\ell}^{xa_{1}})|\le C\sum_{j=\ell\vee L}^{2L-1}\varphi^{x}(\xi_{N-j}^{xa_{1}},\,\xi_{N-j,\,j-1}^{xa_{1}a_{2}})\le CL.
\]
Therefore, we calculate $\|\varphi^{x}\|^{2}$ as
\[
\begin{aligned} & \sum_{\ell=L}^{2L-1}\sum_{i=1}^{p-1}\sum_{j=1}^{\ell}\frac{\varphi^{x}(\xi_{N-\ell,\,j}^{xa_{i}a_{i+1}},\,\xi_{N-\ell,\,j-1}^{xa_{i}a_{i+1}})^{2}}{\mu_{N}(\xi_{N-\ell,\,j}^{xa_{i}a_{i+1}})\cdot j(d_{N}+\ell-j)r(a_{i},\,a_{i+1})}+\sum_{\ell=1}^{2L-1}\frac{\varphi^{x}(\xi_{N-\ell+1}^{xa_{1}},\,\xi_{N-\ell}^{xa_{1}})^{2}}{\mu_{N}(\xi_{N-\ell}^{xa_{1}})\cdot\ell(d_{N}+N-\ell)r(a_{1},\,x)}\\
 & \le C\sum_{\ell=L}^{2L-1}\sum_{i=1}^{p-1}\sum_{j=1}^{\ell}\frac{N-\ell}{Nd_{N}^{2}m_{\star\star}^{\ell}}+C\sum_{\ell=1}^{2L-1}\frac{L^{2}}{Nd_{N}m_{\star\star}^{\ell}}\le C\Big(\frac{L}{d_{N}^{2}m_{\star\star}^{L}}+\frac{L^{2}}{Nd_{N}m_{\star\star}^{L}}\Big).
\end{aligned}
\]
This proves the result for $\varphi^{x}$. The remaining proof for
$\varphi^{y}$ is analogous.
\end{proof}
Gathering the three lemmas above, we arrive at the conclusion of this
subsection.
\begin{lem}
\label{l_flow-norm-4}It holds that
\[
\limsup_{L\to\infty}\limsup_{N\to\infty}\frac{d_{N}^{3}}{N^{2}}\cdot\|\psi\|^{2}\le\frac{1}{18\mathfrak{K}_{xy}}.
\]
\end{lem}

\begin{proof}
Since the support of each $\psi^{ab}$ for $\{a,\,b\}\subseteq S_{0}$
is disjoint with each other, it suffices to collect \eqref{e_psi-def}
and Lemmas \ref{l_flow-norm-1}, \ref{l_flow-norm-2}, and \ref{l_flow-norm-3}.
\end{proof}

\subsection{\label{sec6.3}Flow divergence}

In this subsection, we first prove that $\psi$ is a flow from $\xi^{x}$
to $\xi^{y}$. Then, we calculate the value of $\psi$, which is by
definition $(\mathrm{div}\,\psi)(\xi^{x})$. Recall the collection
$\widehat{\mathcal{A}}_{N}^{R}$, $R\subseteq S$ defined in Definition
\ref{d_AN-ANhat-def}.
\begin{lem}
\label{l_flow-div-free-1}For all $\{a,\,b\}\subseteq S_{0}$ and
$\eta\in\widehat{\mathcal{A}}_{N}^{xaby}$, it holds that $(\mathrm{div}\,\psi)(\eta)=0$.
\end{lem}

\begin{proof}
The inner correction flow $\varphi^{ab}$ is designed for this lemma.
We divide into three cases.
\begin{itemize}
\item \textbf{(L3)}/\textbf{(L4)}: First, for $n\in\llbracket2L,\,N-2L\rrbracket$,
$\ell\in\llbracket2,\,L\rrbracket$, and $k\in\llbracket1,\,\ell-1\rrbracket$,
it holds that
\[
\begin{aligned}(\mathrm{div}\,\psi)(\xi_{n-\ell,\,k,\,\ell-k}^{xaby})= & \phi^{ab}(\xi_{n-\ell,\,k,\,\ell-k}^{xaby},\,\xi_{n-\ell,\,k-1,\,\ell-k+1}^{xaby})-\phi^{ab}(\xi_{n-\ell,\,k+1,\,\ell-k-1}^{xaby},\,\xi_{n-\ell,\,k,\,\ell-k}^{xaby})\\
 & +\varphi^{ab}(\xi_{n-\ell,\,k,\,\ell-k}^{xaby},\,\xi_{n-\ell+1,\,k-1,\,\ell-k}^{xaby})-\varphi^{ab}(\xi_{n-\ell-1,\,k+1,\,\ell-k}^{xaby},\,\xi_{n-\ell,\,k,\,\ell-k}^{xaby}).
\end{aligned}
\]
By \eqref{e_tfl-L3-1}, \eqref{e_tfl-L3-2}, and \eqref{e_tfl-L3-3},
the second line in the right-hand side equals $[\sum_{j=\ell}^{L}-\sum_{j=\ell+1}^{L}]$
of
\[
(\phi^{ab}(\xi_{n-j,\,k+j-\ell+1,\,\ell-k-1}^{xaby},\,\xi_{n-j,\,k+j-\ell,\,\ell-k}^{xaby})-\phi^{ab}(\xi_{n-j,\,k+j-\ell,\,\ell-k}^{xaby},\,\xi_{n-j,\,k+j-\ell-1,\,\ell-k+1}^{xaby})),
\]
which in turn equals
\[
\phi^{ab}(\xi_{n-\ell,\,k+1,\,\ell-k-1}^{xaby},\,\xi_{n-\ell,\,k,\,\ell-k}^{xaby})-\phi^{ab}(\xi_{n-\ell,\,k,\,\ell-k}^{xaby},\,\xi_{n-\ell,\,k-1,\,\ell-k+1}^{xaby}),
\]
and thus it cancels out with the first line in the right-hand side.
Next, suppose that $n\in\llbracket N-2L+1,\,N-L\rrbracket$. If $\ell\in\llbracket n-N+2L,\,L\rrbracket$,
then the same logic is valid. If $\ell\in\llbracket1,\,n-N+2L-1\rrbracket$,
then the exceptional flow defined in \eqref{e_tfl-L3-4} plays a role,
but they also cancel out with each other since for each $k\in\llbracket1,\,\ell-1\rrbracket$,
\[
\varphi^{ab}(\xi_{n-\ell,\,k,\,\ell-k}^{xaby},\,\xi_{n-\ell,\,k+1,\,\ell-k-1}^{xaby})=\varphi^{ab}(\xi_{n-\ell,\,k-1,\,\ell-k+1}^{xaby},\,\xi_{n-\ell,\,k,\,\ell-k}^{xaby}).
\]
Thus, it holds that $(\mathrm{div}\,\psi^{ab})(\eta)=0$ for any $\eta\in\widehat{\mathcal{A}}_{N}^{xaby}$
if $\{a,\,b\}$ is of type \textbf{(L3)} or \textbf{(L4)}.
\item \textbf{(L5)}: The situation is totally symmetric; thus, the proof
is almost identical.
\item \textbf{(L6)} For $n\in\llbracket L,\,N-2L\rrbracket$, $\ell\in\llbracket2\vee(2L-n),\,L\rrbracket$,
and $k\in\llbracket1,\,\ell-1\rrbracket$, by \eqref{e_tfl-L6},
\begin{equation}
\begin{aligned}(\mathrm{div}\,\psi)(\xi_{n,\,k,\,\ell-k}^{xaby}) & =\phi^{ab}(\xi_{n,\,k,\,\ell-k}^{xaby},\,\xi_{n,\,k-1,\,\ell-k+1}^{xaby})-\phi^{ab}(\xi_{n,\,k+1,\,\ell-k-1}^{xaby},\,\xi_{n,\,k,\,\ell-k}^{xaby})\\
 & =c_{ab}\cdot m_{a}^{k-1}m_{b}^{\ell-k-1}\cdot[m_{b}(\widehat{\mathfrak{g}}_{\ell}^{L}(\sigma_{k}^{ab})-\widehat{\mathfrak{g}}_{\ell}^{L}(\sigma_{k-1}^{ab}))-m_{a}(\widehat{\mathfrak{g}}_{\ell}^{L}(\sigma_{k+1}^{ab})-\widehat{\mathfrak{g}}_{\ell}^{L}(\sigma_{k}^{ab}))].
\end{aligned}
\label{e_l_flow-div-free-1-1}
\end{equation}
Recall from Definition \ref{d_gL-ghatL-def} that $\widehat{\mathfrak{g}}_{\ell}^{L}:U_{\ell}\to\mathbb{R}$
is the harmonic extension of $\mathfrak{g}^{L}$. Thus, by \eqref{e_har-extension},
\[
\widehat{\mathfrak{g}}_{\ell}^{L}(\sigma_{k}^{ab})=\mathrm{E}_{\sigma_{k}^{ab}}^{\ell}[\mathscr{X}_{\mathcal{T}_{\mathcal{A}_{x}\cup\mathcal{A}_{y}}}^{\ell}],
\]
where $\mathrm{E}_{\sigma_{k}^{ab}}^{\ell}$ is the expectation of
the law of $\{\mathscr{X}_{t}^{\ell}\}_{t\ge0}$ (cf. Definition \ref{d_trans-ver})
starting from $\sigma_{k}^{ab}$. Thus, by the strong Markov property,
\[
\widehat{\mathfrak{g}}_{\ell}^{L}(\sigma_{k}^{ab})=\frac{\mathfrak{r}^{\ell}(\sigma_{k}^{vw},\,\sigma_{k-1}^{vw})\cdot\widehat{\mathfrak{g}}_{\ell}^{L}(\sigma_{k-1}^{ab})}{\mathfrak{r}^{\ell}(\sigma_{k}^{vw},\,\sigma_{k-1}^{vw})+\mathfrak{r}^{\ell}(\sigma_{k}^{vw},\,\sigma_{k+1}^{vw})}+\frac{\mathfrak{r}^{\ell}(\sigma_{k}^{vw},\,\sigma_{k+1}^{vw})\cdot\widehat{\mathfrak{g}}_{\ell}^{L}(\sigma_{k+1}^{ab})}{\mathfrak{r}^{\ell}(\sigma_{k}^{vw},\,\sigma_{k-1}^{vw})+\mathfrak{r}^{\ell}(\sigma_{k}^{vw},\,\sigma_{k+1}^{vw})}.
\]
Substituting the exact values from \eqref{e_trans-ver-def}, we obtain
that
\begin{equation}
\widehat{\mathfrak{g}}_{\ell}^{L}(\sigma_{k}^{ab})=\frac{m_{b}}{m_{b}+m_{a}}\cdot\widehat{\mathfrak{g}}_{\ell}^{L}(\sigma_{k-1}^{ab})+\frac{m_{a}}{m_{b}+m_{a}}\cdot\widehat{\mathfrak{g}}_{\ell}^{L}(\sigma_{k+1}^{ab}).\label{e_l_flow-div-free-1-2}
\end{equation}
Therefore, combining \eqref{e_l_flow-div-free-1-1} and \eqref{e_l_flow-div-free-1-2},
we obtain that $(\mathrm{div}\,\psi)(\xi_{n,\,k,\,\ell-k}^{xaby})=0$
and thus complete the proof.
\end{itemize}
\end{proof}
Moreover, by definition, the following lemma is clear.
\begin{lem}
\label{l_flow-div-free-2}For $\eta\in\mathcal{A}_{N}^{xab}\cup\mathcal{A}_{N}^{aby}$
where $\{a,\,b\}\subseteq S_{0}$, it holds that $(\mathrm{div}\,\psi)(\eta)=0$.
\end{lem}

Next, we check that $\psi$ is divergence free on $\widehat{\mathcal{A}}_{N}^{xay}$
for $a\in\mathcal{N}_{x}$ and $\widehat{\mathcal{A}}_{N}^{xby}$
for $b\in\mathcal{N}_{y}$.
\begin{lem}
\label{l_flow-div-free-3}For $\eta\in\widehat{\mathcal{A}}_{N}^{xay}\cup\widehat{\mathcal{A}}_{N}^{xby}$
where $a\in\mathcal{N}_{x}$ and $b\in\mathcal{N}_{y}$, it holds
that $(\mathrm{div}\,\psi)(\eta)=0$.
\end{lem}

\begin{proof}
The outer correction flows $\varphi^{a}$ and $\varphi^{b}$ are designed
for this lemma. Indeed, first assume that $a\in\mathcal{N}_{x}$.
Then, for $n\in\llbracket2L,\,N-L\rrbracket$ and $\ell\in\llbracket1,\,L\rrbracket$,
$(\mathrm{div}\,\psi)(\xi_{n-\ell,\,\ell}^{xay})$ equals
\[
\begin{aligned} & \sum_{b\in S_{0}\setminus\mathcal{N}_{x}:\,a\sim b}\psi^{ab}(\xi_{n-\ell,\,\ell}^{xay},\,\xi_{n-\ell,\,\ell-1,\,1}^{xaby})+\varphi^{a}(\xi_{n-\ell,\,\ell}^{xay},\,\xi_{n-\ell-1,\,\ell+1}^{xay})-\varphi^{a}(\xi_{n-\ell+1,\,\ell-1}^{xay},\,\xi_{n-\ell,\,\ell}^{xay})\\
 & =\sum_{b\in S_{0}\setminus\mathcal{N}_{x}:\,a\sim b}\psi^{ab}(\xi_{n-\ell,\,\ell}^{xay},\,\xi_{n-\ell,\,\ell-1,\,1}^{xaby})-\sum_{b\in S_{0}\setminus\mathcal{N}_{x}:\,a\sim b}\psi^{ab}(\xi_{n-\ell,\,\ell}^{xay},\,\xi_{n-\ell,\,\ell-1,\,1}^{xaby})=0.
\end{aligned}
\]
Similarly, we can prove that $(\mathrm{div}\,\psi)(\xi_{n,\,\ell}^{xby})=0$
for all $n\in\llbracket L,\,N-2L\rrbracket$ and $\ell\in\llbracket1,\,L\rrbracket$.
\end{proof}
\begin{lem}
\label{l_flow-div-free-4}For $\eta\in\widehat{\mathcal{A}}_{N}^{xvy}$
where $v\in S_{0}\setminus(\mathcal{N}_{x}\cup\mathcal{N}_{y})$,
it holds that $(\mathrm{div}\,\psi)(\eta)=0$.
\end{lem}

\begin{proof}
First, assume that $v\in\mathcal{A}_{x}$, such that $v\in\mathcal{A}_{x,\,j}$
for some $j\in\llbracket1,\,s\rrbracket$. Then, we calculate for
$n\in\llbracket2L,\,N-2L\rrbracket$ and $\ell\in\llbracket1,\,L\rrbracket$
that
\[
\begin{aligned}(\mathrm{div}\,\psi)(\xi_{n-\ell,\,\ell}^{xvy})= & -\sum_{a\in\mathcal{N}_{x}:\,a\sim v}(\phi^{ab}(\xi_{n-\ell,\,1,\,\ell-1}^{xavy},\,\xi_{n-\ell,\,\ell}^{xvy})+\varphi^{ab}(\xi_{n-\ell-1,\,1,\,\ell}^{xavy},\,\xi_{n-\ell,\,\ell}^{xvy}))\\
 & +\sum_{w\in\mathcal{V}_{j}':\,v\sim w}\phi^{ab}(\xi_{n-\ell,\,\ell}^{xvy},\,\xi_{n-\ell,\,\ell-1,\,1}^{xvwy}).
\end{aligned}
\]
If $\ell\in\llbracket1,\,L-1\rrbracket$, substituting the exact values,
the right-hand side equals
\[
\begin{aligned} & -\sum_{a\in\mathcal{N}_{x}:\,a\sim v}\frac{c_{av}}{1-m_{a}}\cdot[m_{v}^{\ell-1}(1+\mathfrak{g}^{L}(\mathfrak{v}_{\ell-1})-\mathfrak{g}^{L}(\mathfrak{v}_{\ell}))-m_{v}^{\ell}(1+\mathfrak{g}^{L}(\mathfrak{v}_{\ell})-\mathfrak{g}^{L}(\mathfrak{v}_{\ell+1}))]\\
 & +\sum_{w\in\mathcal{V}_{j}':\,v\sim w}\mathfrak{r}^{\ell}(\sigma_{\ell}^{vw},\,\sigma_{\ell-1}^{vw})\cdot(\widehat{\mathfrak{g}}_{\ell}^{L}(\sigma_{\ell}^{vw})-\widehat{\mathfrak{g}}_{\ell+1}^{L}(\sigma_{\ell-1}^{vw})).
\end{aligned}
\]
By \eqref{e_dual-gen}, this equals
\[
\begin{aligned} & -\sum_{a\in\mathcal{N}_{x}}\frac{c_{av}}{1-m_{a}}\cdot[m_{v}^{\ell-1}(1+\mathfrak{g}^{L}(\mathfrak{v}_{\ell-1})-\mathfrak{g}^{L}(\mathfrak{v}_{\ell}))-m_{v}^{\ell}(1+\mathfrak{g}^{L}(\mathfrak{v}_{\ell})-\mathfrak{g}^{L}(\mathfrak{v}_{\ell+1}))]\\
 & +\sum_{w\in\mathcal{A}_{x,\,j}\cup\mathcal{A}_{y,\,j}}\widehat{\mathfrak{r}}^{\ell}(\sigma^{v},\,\sigma^{w})\cdot(\mathfrak{g}^{L}(\mathfrak{v}_{\ell})-\mathfrak{g}^{L}(\mathfrak{w}_{\ell})).
\end{aligned}
\]
Thus, this is zero by Lemma \ref{l_g-gL-prop-1}-(1). Similarly, $(\mathrm{div}\,\psi)(\xi_{n-L,\,L}^{xvy})=0$.
If $n\in\llbracket N-2L+1,\,N-L\rrbracket$, then the exceptional
flow guarantees that the divergence is zero. Moreover, the case of
$v\in\mathcal{A}_{y}$ can be handled by the same logic.

Finally, assume that $v\in\mathcal{V}_{j}'\setminus(\mathcal{A}_{x,\,j}\cup\mathcal{A}_{y,\,j})$.
Then, for $n\in\llbracket L,\,N-2L\rrbracket$ and $\ell\in\llbracket1\vee(2L-n),\,L\rrbracket$,
\[
(\mathrm{div}\,\psi)(\xi_{n-\ell,\,\ell}^{xvy})=\sum_{w\in\mathcal{V}_{j}'}\phi^{ab}(\xi_{n-\ell,\,\ell}^{xvy},\,\xi_{n-\ell,\,\ell-1,\,1}^{xvwy})=\sum_{w\in\mathcal{V}_{j}'}\mathfrak{r}_{\sigma_{\ell}^{vw}\sigma_{\ell-1}^{vw}}^{\ell}\cdot(\widehat{\mathfrak{g}}_{\ell}^{L}(\sigma_{\ell}^{vw})-\widehat{\mathfrak{g}}_{\ell+1}^{L}(\sigma_{\ell-1}^{vw}))=0,
\]
by \eqref{e_har-extension} and the strong Markov property. Thus,
the proof is completed.
\end{proof}
The only remaining domain is $\mathcal{A}_{N}^{xy}$. We first handle
its interior, $\widehat{\mathcal{A}}_{N}^{xy}$.
\begin{lem}
\label{l_flow-div-free-5}For $\eta\in\widehat{\mathcal{A}}_{N}^{xy}$,
it holds that $(\mathrm{div}\,\psi)(\eta)=0$.
\end{lem}

\begin{proof}
First, we consider $\xi_{n}^{xy}\in\widehat{\mathcal{A}}_{N}^{xy}$
for $n\in\llbracket2L,\,n-2L\rrbracket$. By Definition \ref{d_tfl-outer},
we expand $(\mathrm{div}\,\psi)(\xi_{n}^{xy})$ as
\begin{equation}
\begin{aligned} & \sum_{a\in\mathcal{N}_{x}}\sum_{b\in\mathcal{N}_{y}:\,a\sim b}\sum_{j=1}^{L}[\psi^{ab}(\xi_{n-j,\,j}^{xay},\,\xi_{n-j,\,j-1,\,1}^{xaby})-\psi^{ab}(\xi_{n,\,1,\,j-1}^{xaby},\,\xi_{n,\,j}^{xby})-\psi^{ab}(\xi_{n-1,\,1,\,j}^{xaby},\,\xi_{n,\,j}^{xby})]\\
 & +\sum_{a\in\mathcal{N}_{x}}\sum_{b\in\mathcal{A}_{x}:\,a\sim b}\sum_{j=1}^{L}\psi^{ab}(\xi_{n-j,\,j}^{xay},\,\xi_{n-j,\,j-1,\,1}^{xaby})-\sum_{b\in\mathcal{N}_{y}}\sum_{a\in\mathcal{A}_{y}:\,a\sim b}\sum_{j=1}^{L}\psi^{ab}(\xi_{n,\,1,\,j-1}^{xaby},\,\xi_{n,\,j}^{xby}).
\end{aligned}
\label{e_l_flow-div-free-5-1}
\end{equation}
First, we prove that the first line in the right-hand side of \eqref{e_l_flow-div-free-5-1}
is zero. Indeed, By Definition \ref{d_tfl-L3}, it equals $c_{ab}$
times
\[
\sum_{a\in\mathcal{N}_{x}}\sum_{b\in\mathcal{N}_{y}:\,a\sim b}c_{ab}\cdot\Big[\sum_{j=1}^{L-1}m_{a}^{j-1}+\frac{m_{a}^{L-1}}{1-m_{a}}-\sum_{j=1}^{L-1}\Big(m_{b}^{j-1}-\sum_{i=j+1}^{L}(\mathrm{div}\,\phi^{ab})(\xi_{n+j-i,\,i-j,\,j}^{xaby})\Big)-\frac{m_{b}^{L-1}}{1-m_{a}}\Big].
\]
Rearranging, this becomes
\[
\begin{aligned}\sum_{a\in\mathcal{N}_{x}}\sum_{b\in\mathcal{N}_{y}:\,a\sim b}c_{ab}\cdot\Big[ & \frac{1}{1-m_{a}}-\frac{1-m_{b}^{L-1}}{1-m_{b}}\\
 & +\sum_{j=1}^{L-1}(m_{b}-m_{a})\Big(\sum_{i=j+1}^{L-1}m_{a}^{i-j-1}m_{b}^{j-1}+\frac{m_{a}^{L-j-1}m_{b}^{j-1}}{1-m_{a}}\Big)-\frac{m_{b}^{L-1}}{1-m_{a}}\Big].
\end{aligned}
\]
Further calculating, we rewrite this as
\[
\begin{aligned} & \sum_{a\in\mathcal{N}_{x}}\sum_{b\in\mathcal{N}_{y}:\,a\sim b}c_{ab}\cdot\Big[\frac{1}{1-m_{a}}-\frac{1-m_{b}^{L-1}}{1-m_{b}}+(m_{b}-m_{a})\sum_{j=1}^{L-1}\frac{m_{b}^{j-1}}{1-m_{a}}-\frac{m_{b}^{L-1}}{1-m_{a}}\Big]\\
 & =\sum_{a\in\mathcal{N}_{x}}\sum_{b\in\mathcal{N}_{y}:\,a\sim b}c_{ab}\cdot\Big[\frac{1}{1-m_{a}}-\frac{1-m_{b}^{L-1}}{1-m_{b}}+\frac{(m_{b}-m_{a})(1-m_{b}^{L-1})}{(1-m_{a})(1-m_{b})}-\frac{m_{b}^{L-1}}{1-m_{a}}\Big]=0.
\end{aligned}
\]
Next, we prove that the second line in the right-hand side of \eqref{e_l_flow-div-free-5-1}
is zero. By Definitions \ref{d_tfl-L4} and \ref{d_tfl-L5}, this
equals
\[
\sum_{a\in\mathcal{N}_{x}}\sum_{b\in\mathcal{A}_{x}:\,a\sim b}\frac{c_{ab}\cdot(1-\mathfrak{g}^{L}(\mathfrak{b}_{1}))}{1-m_{a}}-\sum_{b\in\mathcal{N}_{y}}\sum_{a\in\mathcal{A}_{y}:\,a\sim b}\frac{c_{ab}\cdot\mathfrak{g}^{L}(\mathfrak{a}_{1})}{1-m_{b}}.
\]
This is zero by Lemma \ref{l_g-gL-prop-2}. Thus, we have demonstrated
that $(\mathrm{div}\,\psi)(\xi_{n}^{xy})=0$ for $n\in\llbracket2L,\,N-2L\rrbracket$.

Next, assume that $n\in\llbracket N-2L+1,\,N-L\rrbracket$. Then,
by Definitions \ref{d_tfl-outer} and \ref{d_tfl-edge}, we may rewrite
$(\mathrm{div}\,\psi)(\xi_{n}^{xy})$ as
\[
\sum_{a\in\mathcal{N}_{x}}\varphi^{a}(\xi_{n}^{xy},\,\xi_{n-1,\,1}^{xay})-\varphi^{x}(\xi_{n,\,1}^{xa_{p-1}y},\,\xi_{n}^{xy})=0.
\]
Similarly, if $n\in\llbracket L,\,2L-1\rrbracket$, then it holds
that
\[
(\mathrm{div}\,\psi)(\xi_{n}^{xy})=-\sum_{b\in\mathcal{N}_{y}}\varphi^{b}(\xi_{n,\,1}^{xby},\,\xi_{n}^{xy})+\varphi^{y}(\xi_{n}^{xy},\,\xi_{n-1,\,1}^{xa_{1}y})=0.
\]
This concludes the proof.
\end{proof}
Finally, we are ready to prove the following lemma.
\begin{lem}
\label{l_flow-div-free-6}The test flow $\psi$ is a flow from $\xi^{x}$
to $\xi^{y}$ such that
\[
\lim_{L\to\infty}(\mathrm{div}\,\psi)(\xi^{x})=\frac{1}{6\mathfrak{K}_{xy}}.
\]
\end{lem}

\begin{proof}
Combining Lemmas \ref{l_flow-div-free-1}--\ref{l_flow-div-free-5},
we obtain that $(\mathrm{div}\,\psi)(\eta)=0$ for all $\eta\in\mathcal{H}_{N}\setminus\{\xi^{x},\,\xi^{y}\}$.
Thus, it remains to prove that
\[
\lim_{L\to\infty}(\mathrm{div}\,\psi)(\xi^{x})=\frac{1}{6\mathfrak{K}_{xy}}.
\]
To prove this, we collection the exact definitions to calculate $(\mathrm{div}\,\psi)(\xi^{x})$
as
\begin{equation}
\begin{aligned}\sum_{j=L}^{2L-1}\varphi^{x}(\xi_{N-j}^{xa_{1}},\,\xi_{N-j,\,j-1}^{xa_{1}a_{2}}) & =\sum_{j=L}^{2L-1}\sum_{a\in\mathcal{N}_{x}}\varphi^{a}(\xi_{N-j}^{xy},\,\xi_{N-j-1,\,1}^{xay})\\
 & =\sum_{j=L}^{2L-1}\sum_{a\in\mathcal{N}_{x}}\sum_{b\in S_{0}\setminus\mathcal{N}_{x}:\,a\sim b}\sum_{i=1}^{L}\psi^{ab}(\xi_{N-j-i,\,i}^{xay},\,\xi_{N-j-i,\,i-1,\,1}^{xaby}).
\end{aligned}
\label{e_l_flow-div-free-6-1}
\end{equation}
If $b\in\mathcal{N}_{y}$, then the total summation in \eqref{e_l_flow-div-free-6-1}
becomes $\sum_{a\in\mathcal{N}_{x}}\sum_{b\in\mathcal{N}_{y}:\,a\sim b}$
times
\[
\begin{aligned} & \sum_{j=L}^{2L-1}\sum_{i=2L-j}^{L}\phi^{ab}(\xi_{N-j-i,\,i}^{xay},\,\xi_{N-j-i,\,i-1,\,1}^{xaby})+\sum_{j=L}^{2L-1}\sum_{i=1}^{2L-j-1}(\phi^{ab}+\varphi^{ab})(\xi_{N-j-i,\,i}^{xay},\,\xi_{N-j-i,\,i-1,\,1}^{xaby})\\
 & =\sum_{j=L}^{2L-1}\frac{c_{ab}m_{a}^{2L-j-1}}{1-m_{a}}+c_{ab}\cdot\sum_{j=L}^{2L-1}\sum_{i=1}^{2L-j-1}\Big(m_{a}^{i-1}-m_{b}^{i-1}+\frac{m_{b}^{i}}{1-m_{a}}-\frac{m_{a}m_{b}^{i-1}}{1-m_{a}}\Big)\\
 & =c_{ab}\cdot\sum_{j=L}^{2L-1}\Big(\frac{1}{1-m_{a}}-\frac{1-m_{b}^{2L-j-1}}{1-m_{a}}\Big)=c_{ab}\cdot\frac{1-m_{b}^{L}}{(1-m_{a})(1-m_{b})}\xrightarrow{L\to\infty}\frac{c_{ab}}{(1-m_{a})(1-m_{b})}.
\end{aligned}
\]
Similarly, if $b\in\mathcal{A}_{x}$, then the total summation in
\eqref{e_l_flow-div-free-6-1} becomes $\sum_{a\in\mathcal{N}_{x}}\sum_{b\in\mathcal{A}_{x}:\,a\sim b}$
times
\[
\begin{aligned} & \sum_{j=L}^{2L-1}\sum_{i=2L-j}^{L}\phi^{ab}(\xi_{N-j-i,\,i}^{xay},\,\xi_{N-j-i,\,i-1,\,1}^{xaby})+\sum_{j=L}^{2L-1}\sum_{i=1}^{2L-j-1}(\phi^{ab}+\varphi^{ab})(\xi_{N-j-i,\,i}^{xay},\,\xi_{N-j-i,\,i-1,\,1}^{xaby})\\
 & =\sum_{j=L}^{2L-1}\Big[\frac{c_{ab}m_{a}^{2L-j-1}(1-\mathfrak{g}^{L}(\mathfrak{b}_{1}))}{1-m_{a}}+c_{ab}\sum_{i=1}^{2L-j-1}\Big(m_{a}^{i-1}(1-\mathfrak{g}^{L}(\mathfrak{b}_{1}))-m_{b}^{i-1}+\frac{m_{b}^{i}-m_{a}m_{b}^{i-1}}{1-m_{a}}\Big)\Big]\\
 & =\frac{c_{ab}}{1-m_{a}}\cdot\sum_{j=0}^{L-1}m_{b}^{j}(1+\widehat{\mathfrak{g}}_{j}^{L}(\sigma^{b})-\widehat{\mathfrak{g}}_{j+1}^{L}(\sigma^{b}))\xrightarrow{L\to\infty}\frac{c_{ab}}{1-m_{a}}\cdot\sum_{j=0}^{\infty}m_{b}^{j}(1+\widehat{\mathfrak{g}}_{j}(\sigma^{b})-\widehat{\mathfrak{g}}_{j+1}(\sigma^{b})),
\end{aligned}
\]
where the last convergence holds by Lemma \ref{l_g-gL-prop-3}. Thus,
we conclude that $\lim_{L\to\infty}(\mathrm{div}\,\psi)(\xi^{x})$
equals
\[
\sum_{a\in\mathcal{N}_{x}}\sum_{b\in\mathcal{N}_{y}}\frac{c_{ab}}{(1-m_{a})(1-m_{b})}+\sum_{a\in\mathcal{N}_{x}}\sum_{b\in\mathcal{A}_{x}}\frac{c_{ab}}{1-m_{a}}\cdot\sum_{j=0}^{\infty}m_{b}^{j}(1+\widehat{\mathfrak{g}}_{j}(\sigma^{b})-\widehat{\mathfrak{g}}_{j+1}(\sigma^{b})),
\]
which is exactly $\frac{1}{6\mathfrak{K}_{xy}}$ by Lemma \ref{l_g-gL-prop-4}.
This concludes the proof.
\end{proof}

\subsection{\label{sec6.4}Proofs of Proposition \ref{p_flow} and Theorems \ref{t_Capacity}
and \ref{t_Capacity-specific}}
\begin{proof}[Proof of Proposition \ref{p_flow}]
 This is now straightforward by Lemmas \ref{l_flow-norm-4} and \ref{l_flow-div-free-6}.
\end{proof}
\begin{proof}[Proof of Theorem \ref{t_Capacity}]
 By the Dirichlet principle (Proposition \ref{p_DP}) and Proposition
\ref{p_Diri-form},
\[
\limsup_{N\to\infty}\frac{N^{2}}{d_{N}^{3}}\cdot\mathrm{Cap}_{N}(\mathcal{E}_{N}^{x},\,\mathcal{E}_{N}^{y})\le\limsup_{N\to\infty}\frac{N^{2}}{d_{N}^{3}}\cdot\mathscr{D}_{N}(F_{\textup{test}})\le\frac{1}{2\mathfrak{K}_{xy}}.
\]
By the Thomson principle (Proposition \ref{p_TP}) and Proposition
\ref{p_flow}, it holds that
\[
\liminf_{N\to\infty}\frac{N^{2}}{d_{N}^{3}}\cdot\mathrm{Cap}_{N}(\mathcal{E}_{N}^{x},\,\mathcal{E}_{N}^{y})\ge\liminf_{L\to\infty}\liminf_{N\to\infty}\frac{(\mathrm{div}\,\psi)(\xi^{x})^{2}}{\|\psi\|^{2}}\ge\frac{1}{2\mathfrak{K}_{xy}}.
\]
The displayed two inequalities prove Theorem \ref{t_Capacity}\@.
\end{proof}
\begin{proof}[Proof of Theorem \ref{t_Capacity-specific}]
 In the special case of $|S_{\star}|=\kappa_{3}=2$, Theorem \ref{t_Capacity-specific}
is equivalent to Theorem \ref{t_Capacity} and there is nothing more
to prove. Refer to Section \ref{sec7} for a brief explanation on
the proof of the general case.
\end{proof}

\section{\label{sec7}Comments on Proof of the General Case}

In this last section, we settle a few issues regarding the general
case without Assumption \ref{a_Sstar-kappa3}.
\begin{itemize}
\item In the definition \eqref{e_Kxy-def} of $\mathfrak{K}_{xy}$, it was
immediate that the constant $\mathfrak{K}_{xy}$ is a finite positive
real number. This is not true in general; it may happen that there
exist $i,\,j\in\llbracket1,\,\kappa_{3}\rrbracket$ such that $\mathfrak{K}_{xy}=\infty$
for any $x\in S_{\star i}^{3}$ and $y\in S_{\star j}^{3}$, and thus
$r^{\textup{3rd}}(i,\,j)=0$. Nevertheless, since the underlying random
walk $r(\cdot,\,\cdot)$ is irreducible, even for such $i,\,j\in\llbracket1,\,\kappa_{3}\rrbracket$,
there exists a sequence $i=i_{0},\,i_{1},\,\dots,\,i_{n}=j\in\llbracket1,\,\kappa_{3}\rrbracket$
such that $r^{\textup{3rd}}(i_{m},\,i_{m+1})>0$ for all $m\in\llbracket0,\,n-1\rrbracket$.
Thus, in the general case the irreducibility of the scaling limit
$Z(\cdot)$ is still valid.
\item It is quite straightforward to generalize the construction of the
test function $F_{\textup{test}}$ conducted in Section \ref{sec5}.
The main rule to obey is Lemma \ref{l_F-prop}-(1); namely, in the
general case, it must hold that $F_{\textup{test}}(\eta)$ depends
on $\eta$ only through $\sum_{x\in S_{\star i}^{3}}\overline{\eta}_{x}$
for each $i\in\llbracket1,\,\kappa_{3}\rrbracket$, where $\overline{\eta}_{x}$
is defined in \eqref{e_eta-not-def}. After the generalized construction,
similar estimates as the ones given in the remainder of Section \ref{sec5}
are also valid in the general case as well.
\item The construction of the test flow $\psi_{\textup{test}}^{L}$ is even
simpler than the construction of the test function. We can simply
define the test flows on each tetrahedron $\mathcal{A}_{N}^{xaby}$
for all relevant $\{a,\,b\}\subseteq S_{0}$, $x\in S_{\star i}^{3}$,
and $y\in S_{\star j}^{3}$ for $i,\,j\in\llbracket1,\,\kappa_{3}\rrbracket$.
Then, the same strategy works perfectly.
\item In the general case, the proof of Theorem \ref{t_Capacity-specific}
needs a little bit more effort. Suppose as in the theorem that $i\in\llbracket1,\,\kappa_{3}\rrbracket$
and $x\in S_{\star i}^{3}$. Recall that our strategy of calculating
$\mathrm{Cap}_{N}(\mathcal{E}_{N}(S_{\star i}^{3}),\,\mathcal{E}_{N}(S_{\star}\setminus S_{\star i}^{3}))$
is to construct a test function $F=F_{\textup{test}}:\mathcal{H}_{N}\to\mathbb{R}$
and a test flow $\psi=\psi_{\textup{test}}^{L}\in\mathfrak{F}_{\mathcal{E}_{N}(S_{\star i}^{3}),\,\mathcal{E}_{N}(S_{\star}\setminus S_{\star i}^{3})}^{\gamma}$
such that
\[
\limsup_{N\to\infty}\frac{N^{2}}{d_{N}^{3}}\mathrm{Cap}_{N}(\mathcal{E}_{N}(S_{\star i}^{3}),\,\mathcal{E}_{N}(S_{\star}\setminus S_{\star i}^{3}))\le\limsup_{N\to\infty}\frac{N^{2}}{d_{N}^{3}}\mathscr{D}_{N}(F)\le\frac{1}{|S_{\star}|}\cdot\sum_{j\in\llbracket1,\,\kappa_{3}\rrbracket:\,j\ne i}\frac{1}{\mathfrak{K}_{ij}},
\]
and
\[
\liminf_{N\to\infty}\frac{N^{2}}{d_{N}^{3}}\mathrm{Cap}_{N}(\mathcal{E}_{N}(S_{\star i}^{3}),\,\mathcal{E}_{N}(S_{\star}\setminus S_{\star i}^{3}))\ge\liminf_{L\to\infty}\liminf_{N\to\infty}\frac{N^{2}}{d_{N}^{3}}\frac{\gamma^{2}}{\|\psi\|^{2}}\ge\frac{1}{|S_{\star}|}\cdot\sum_{j\in\llbracket1,\,\kappa_{3}\rrbracket:\,j\ne i}\frac{1}{\mathfrak{K}_{ij}}.
\]
The test function $F$ satisfies $F=1$ on $\mathcal{E}_{N}(S_{\star i}^{3})\ni\xi^{x}$
and $F=0$ on $\mathcal{E}_{N}(S_{\star}\setminus S_{\star i}^{3})$,
such that we may still calculate via the Dirichlet principle (Proposition
\ref{p_DP}) that
\begin{equation}
\limsup_{N\to\infty}\frac{N^{2}}{d_{N}^{3}}\mathrm{Cap}_{N}(\xi^{x},\,\mathcal{E}_{N}(S_{\star}\setminus S_{\star i}^{3}))\le\limsup_{N\to\infty}\frac{N^{2}}{d_{N}^{3}}\mathscr{D}_{N}(F)\le\frac{1}{|S_{\star}|}\cdot\sum_{j\in\llbracket1,\,\kappa_{3}\rrbracket:\,j\ne i}\frac{1}{\mathfrak{K}_{ij}}.\label{e_gen-proof-cap-1}
\end{equation}
On the other hand, the test flow $\psi$ cannot be used directly,
since it is a flow from $\mathcal{E}_{N}(S_{\star i}^{3})$ to $\mathcal{E}_{N}(S_{\star}\setminus S_{\star i}^{3})$
but \emph{not} a flow from $\xi^{x}$ to $\mathcal{E}_{N}(S_{\star}\setminus S_{\star i}^{3})$.
This issue can be overcome by sending the remaining divergence of
$\psi$ on $\mathcal{E}_{N}(S_{\star i}^{3})\setminus\{\xi^{x}\}$
to $\xi^{x}$ by adding an additional correction flow, which can be
defined as done in the proof of \eqref{e_t_3rd-2}, whose flow norm
is negligible compared to $\frac{N^{2}}{d_{N}^{3}}$. Therefore, this
modified flow $\psi'$ is now a flow from $\xi^{x}$ to $\mathcal{E}_{N}(S_{\star}\setminus S_{\star i}^{3})$
with $\|\psi'\|^{2}\simeq\|\psi\|^{2}$, so that by the Thomson principle
(Proposition \ref{p_TP}),
\begin{equation}
\liminf_{N\to\infty}\frac{N^{2}}{d_{N}^{3}}\mathrm{Cap}_{N}(\xi^{x},\,\mathcal{E}_{N}(S_{\star}\setminus S_{\star i}^{3}))\ge\liminf_{L\to\infty}\liminf_{N\to\infty}\frac{N^{2}}{d_{N}^{3}}\frac{\gamma^{2}}{\|\psi'\|^{2}}=\frac{1}{|S_{\star}|}\cdot\sum_{j\in\llbracket1,\,\kappa_{3}\rrbracket:\,j\ne i}\frac{1}{\mathfrak{K}_{ij}}.\label{e_gen-proof-cap-2}
\end{equation}
Gathering \eqref{e_gen-proof-cap-1} and \eqref{e_gen-proof-cap-2}
completes the verification of Theorem \ref{t_Capacity-specific} in
the general case.
\end{itemize}

\appendix

\section{\label{appenA}Potential Theory}

We present some essential results from discrete-space potential theory
used in the article. Fix a finite state space $V$ and a continuous-time
irreducible, reversible Markov chain $\{v_{t}\}_{t\ge0}$ on $V$
with corresponding infinitesimal generator $\mathcal{L}$ acting on
functions $f:V\to\mathbb{R}$ by
\[
(\mathcal{L}f)(v)=\sum_{w\in V}R(v,\,w)(f(w)-f(v))\text{ for all }v\in V.
\]
Denote by $\pi$ the stationary distribution of $\{v_{t}\}_{t\ge0}$.
Moreover, we fix two disjoint non-empty subsets $A$ and $B$ of $V$.

For each $f:V\to\mathbb{R}$, the \emph{Dirichlet form} is defined
as $\mathcal{D}(f):=\langle f,\,-\mathcal{L}f\rangle_{\pi}$. It is
straightforward to obtain the following alternative representation:
\[
\mathcal{D}(f)=\frac{1}{2}\sum_{v,\,w\in V}\pi(v)R(v,\,w)(f(w)-f(v))^{2}.
\]
The \emph{equilibrium potential} $h_{A,\,B}:V\to\mathbb{R}$ is defined
as $h_{A,\,B}(v):=\mathbf{P}_{v}[\mathcal{T}_{A}<\mathcal{T}_{B}]$
for $v\in V$, where $\mathcal{T}_{A}$ (resp. $\mathcal{T}_{B}$)
is the hitting time of the set $A$ (resp. $B$) and $\mathbf{P}_{v}$
is the law of the process starting from $v$. Then, $h_{A,\,B}$ is
the unique solution to the following Dirichlet problem:
\begin{equation}
h=1\text{ on }A,\quad h=0\text{ on }B,\quad\text{and}\quad\mathcal{L}h=0\text{ on }(A\cup B)^{c}.\label{e_h-Diri-problem}
\end{equation}
The \emph{capacity} between $A$ and $B$ is defined as $\mathrm{cap}(A,\,B):=\mathcal{D}(h_{A,\,B})$.
It can be easily verified that $\mathrm{cap}(A,\,B)=\mathrm{cap}(B,\,A)$.

\subsubsection*{Variational principles}

As a matter of fact, $\mathrm{cap}(A,\,B)$ can be alternatively defined
in terms of the following variational principle, widely known as the
\emph{Dirichlet principle}.
\begin{prop}[Dirichlet principle]
\label{p_DP}It holds that
\[
\mathrm{cap}(A,\,B)=\inf\{\mathcal{D}(F):F=1\text{ on }A,\ F=0\text{ on }B\}.
\]
Moreover, the equilibrium potential $h_{A,\,B}$ is the unique minimizer.
\end{prop}

By choosing a suitable test function $F_{\textup{test}}$ approximating
$h_{A,\,B}$, we obtain an upper bound for the capacity: $\mathrm{cap}(A,\,B)\le\mathcal{D}(F_{\textup{test}})$.

For a lower bound, we first introduce the flow structure. For $v,\,w\in V$,
we write $v\sim w$ if $R(v,\,w)>0$. By reversibility, it is clear
that $v\sim w$ if and only if $w\sim v$. A function $\phi:V\times V\to\mathbb{R}$
is a \emph{flow} on $V$, collected as $\mathfrak{F}=\mathfrak{F}_{V}$,
if $\phi(v,\,w)=-\phi(w,\,v)$ for all $(v,\,w)\in V\times V$ and
$\phi(v,\,w)\ne0$ only if $v\sim w$.

For each $\phi\in\mathfrak{F}$, the \emph{norm} $\|\phi\|$ and the
\emph{divergence} $\mathrm{div}\,\phi$ are defined as
\[
\Vert\phi\Vert^{2}:=\frac{1}{2}\sum_{v,\,w\in V:\,v\sim w}\frac{\phi(v,\,w)^{2}}{\pi(v)R(v,\,w)}\quad\text{and}\quad(\mathrm{div}\,\phi)(v):=\sum_{w\in V}\phi(v,\,w),\ v\in V.
\]
A flow $\phi\in\mathfrak{F}$ is called a \emph{flow from $A$ to
$B$ of value $\gamma\in\mathbb{R}$} if
\begin{equation}
\sum_{v\in A}(\mathrm{div}\,\phi)(v)=\gamma,\quad\sum_{v\in B}(\mathrm{div}\,\phi)(v)=-\gamma,\quad\text{and}\quad(\mathrm{div}\,\phi)(v)=0\text{ for all }v\notin A\cup B.\label{e_flow-A-B-def}
\end{equation}
We denote by $\mathfrak{F}_{A,\,B}^{\gamma}$ the collection of such
flows. The \emph{harmonic flow} $\Phi_{A,\,B}$ is defined as $\Phi_{A,\,B}(v,\,w):=\pi(v)R(v,\,w)(h_{A,\,B}(v)-h_{A,\,B}(w))$
for $(v,\,w)\in V\times V$. One can easily verify that $\Phi_{A,\,B}$
is a flow from $A$ to $B$.

The following \emph{Thomson principle} is a natural counterpart to
the Dirichlet principle.
\begin{prop}[Thomson principle]
\label{p_TP}It holds that
\[
\mathrm{cap}(A,\,B)=\sup_{\psi\in\mathfrak{F}_{A,\,B}^{\gamma}:\,\gamma\ne0}\frac{\gamma^{2}}{\|\psi\|^{2}}.
\]
Moreover, the supremum is achieved if and only if $\psi=\gamma\Phi_{A,\,B}$
for some $\gamma\ne0$.
\end{prop}

Provided a nice test flow $\psi_{\textup{test}}\in\mathfrak{F}_{A,\,B}^{\gamma}$
approximating $\gamma'\Phi_{A,\,B}$, we obtain $\mathrm{cap}(A,\,B)\ge\gamma^{2}/\|\psi_{\textup{test}}\|^{2}$.
For the proofs of the two variational principles, refer to \cite[Section 7.3]{BdH}.

\subsubsection*{Trace process}

Fix a non-empty proper subset $W$ of $V$ and consider the \emph{trace
process} $\{v_{t}^{W}\}_{t\ge0}$ on $W$. This is defined by considering
the local time $T^{W}(t)$ spent in $W$ and its inverse $S^{W}(t)$,
\[
T_{t}^{W}:=\int_{0}^{t}\mathbbm{1}\{v_{s}\in W\}\mathrm{d}s\quad\text{and}\quad S_{t}^{W}:=\sup\{s\ge0:T_{s}^{W}\le t\},
\]
and then defining $v_{t}^{W}:=v_{S_{t}^{W}}$ for all $t\ge0$. By
\cite[Proposition 6.1 and Lemma 6.7]{BL TM}, $\{v_{t}^{W}\}_{t\ge0}$
is an irreducible, reversible Markov chain on $W$. If $\{v_{t}\}_{t\ge0}$
is symmetric, $\pi$ becomes the uniform distribution and thus $\{v_{t}^{W}\}_{t\ge0}$
is also symmetric.

Denote by $R^{W}$, $\mathcal{L}^{W}$, and $\mathcal{D}^{W}$ the
transition rate, generator, and Dirichlet form of $\{v_{t}^{W}\}_{t\ge0}$,
respectively. With a slight abuse of notation, we also write the mean
transition rate as
\begin{equation}
R^{W}(A,\,B):=\sum_{v\in A}\frac{\pi(v)}{\pi(A)}\sum_{w\in B}R(v,\,w).\label{e_mean-trans-rate}
\end{equation}
By \cite[Corollary 6.2]{BL TM}, it holds for all $w,\,w'\in W$ that
\begin{equation}
R^{W}(w,\,w')=R(w,\,w')+\sum_{v\in V\setminus W}R(w,\,v)\cdot\mathbf{P}_{v}[\mathcal{T}_{W}=\mathcal{T}_{w'}]\le\sum_{v\in V}R(w,\,v),\label{e_RW-UB}
\end{equation}
where the inequality follows from $\mathbf{P}_{v}[\mathcal{T}_{W}=\mathcal{T}_{w'}]\le1$.
By \cite[Lemma 6.8]{BL TM}, for $A,\,B\subseteq W$,
\begin{equation}
\pi(A)\cdot R^{W}(A,\,B)=\frac{1}{2}\big[\mathrm{cap}(A,\,W\setminus A)+\mathrm{cap}(B,\,W\setminus B)-\mathrm{cap}(A\cup B,\,W\setminus(A\cup B))\big].\label{e_cap-cap-cap}
\end{equation}
By \cite[Proposition 6.10]{BL TM}, the following magic formula holds:
for non-empty $B\subseteq V$ and $v\notin B$,
\begin{equation}
\mathbf{E}_{v}[\mathcal{T}_{B}]=\frac{1}{\mathrm{cap}(v,\,B)}\cdot\sum_{w\in V}\pi(w)\cdot h_{v,\,B}(w).\label{e_magic-formula}
\end{equation}

\subsubsection*{Harmonic extension}

For given $g:W\to\mathbb{R}$, we consider the following Dirichlet
problem: find $f:V\to\mathbb{R}$ such that
\[
f=g\text{ on }W\quad\text{and}\quad\mathcal{L}f=0\text{ on }V\setminus W.
\]
It is well known (e.g., \cite[Theorem 7.1]{BdH}) that there exists
a unique solution $\widehat{g}:V\to\mathbb{R}$ to this problem, called
the \emph{harmonic extension} of $g$, which can be represented as
\begin{equation}
\widehat{g}(v)=\mathbf{E}_{v}[g(v_{\mathcal{T}_{W}})]\text{ for all }v\in V,\label{e_har-extension}
\end{equation}
where $\mathbf{E}_{v}$ is the expectation with respect to $\mathbf{P}_{v}$.

\subsubsection*{Duality}

Now, we state the duality between the trace process and the harmonic
extension. For a detailed proof, we refer the readers to \cite[Appendix A]{BGL}.
\begin{enumerate}
\item (Generator) It holds that
\begin{equation}
(\mathcal{L}\widehat{g})(w)=(\mathcal{L}^{W}g)(w)\text{ for all }w\in W.\label{e_dual-gen}
\end{equation}
\item (Dirichlet form) It holds that $\mathcal{D}(\widehat{g})=\pi(W)\cdot\mathcal{D}^{W}(g)$.
This is equivalent to
\begin{equation}
\sum_{\{v,\,v'\}\subseteq V}\pi(v)R(v,\,v')(\widehat{g}(v')-\widehat{g}(v))^{2}=\sum_{\{w,\,w'\}\subseteq W}\pi(w)R^{W}(w,\,w')(g(w')-g(w))^{2}.\label{e_dual-Diri}
\end{equation}
\end{enumerate}

\section{\label{appenB}Proof of results in Section \ref{sec4.3}}

Here, we present the technical proofs of the lemmas given in Section
\ref{sec4.3}.
\begin{proof}[Proof of Lemma \ref{l_g-gL-prop-1}]
 (1) Recall from Theorem \ref{t_resolvent-equation}-(1) that $\mathfrak{g}_{0}^{L}:\mathscr{V}_{j}^{L}\to\mathbb{R}$
solves the resolvent equation \eqref{e_resolvent-L}. If $\ell=1$,
then according to the exact definitions, we calculate $\mathfrak{f}^{L}(\mathfrak{v}_{1})\mathfrak{g}_{0}^{L}(\mathfrak{v}_{1})-(\mathscr{L}_{j}^{L}\mathfrak{g}_{0}^{L})(\mathfrak{v}_{1})-\mathfrak{h}^{L}(\mathfrak{v}_{1})$
(which is zero) as
\[
\begin{aligned} & \Big(\frac{1}{\lambda^{2}}-\frac{m_{v}}{\lambda^{3}}+\frac{m_{v}}{\lambda^{2}}\Big)\sum_{a\in\mathcal{N}_{x}\cup\mathcal{N}_{y}}\frac{c_{av}}{1-m_{a}}\cdot\mathfrak{g}_{0}^{L}(\mathfrak{v}_{1})-\Big(\frac{m_{v}}{\lambda^{3}}\sum_{a\in\mathcal{N}_{x}\cup\mathcal{N}_{y}}\frac{c_{av}}{1-m_{a}}\Big)\cdot(\mathfrak{g}_{0}^{L}(\mathfrak{v}_{2})-\mathfrak{g}_{0}^{L}(\mathfrak{v}_{1}))\\
 & -\sum_{w\in\mathcal{A}_{x,\,j}\cup\mathcal{A}_{y,\,j}}\frac{1}{\lambda^{2}}\widehat{\mathfrak{r}}^{1}(\sigma^{v},\,\sigma^{w})\cdot(\mathfrak{g}_{0}^{L}(\mathfrak{w}_{1})-\mathfrak{g}_{0}^{L}(\mathfrak{v}_{1}))-(1-m_{v})\cdot\frac{1}{\lambda}\sum_{a\in\mathcal{N}_{x}}\frac{c_{av}}{1-m_{a}}.
\end{aligned}
\]
Substituting \eqref{e_gL-def} and multiplying $\lambda$ on both
sides, we obtain that
\[
\begin{aligned} & (1+m_{v})\sum_{a\in\mathcal{N}_{x}\cup\mathcal{N}_{y}}\frac{c_{av}}{1-m_{a}}\cdot\mathfrak{g}^{L}(\mathfrak{v}_{1})-\Big(m_{v}\sum_{a\in\mathcal{N}_{x}\cup\mathcal{N}_{y}}\frac{c_{av}}{1-m_{a}}\Big)\cdot\mathfrak{g}^{L}(\mathfrak{v}_{2})\\
 & =\sum_{w\in\mathcal{A}_{x,\,j}\cup\mathcal{A}_{y,\,j}}\widehat{\mathfrak{r}}^{1}(\sigma^{v},\,\sigma^{w})\cdot(\mathfrak{g}^{L}(\mathfrak{w}_{1})-\mathfrak{g}^{L}(\mathfrak{v}_{1}))-(1-m_{v})\cdot\sum_{a\in\mathcal{N}_{x}}\frac{c_{av}}{1-m_{a}},
\end{aligned}
\]
which is exactly \eqref{e_l_g-gL-prop-1-1} for $\ell=1$. The remaining
cases of $\ell\in\llbracket2,\,L\rrbracket$ can be proved by the
same strategy, which is evaluating $0=\mathfrak{f}^{L}(\mathfrak{v}_{\ell})\mathfrak{g}_{0}^{L}(\mathfrak{v}_{\ell})-(\mathscr{L}_{j}^{L}\mathfrak{g}_{0}^{L})(\mathfrak{v}_{\ell})-\mathfrak{h}^{L}(\mathfrak{v}_{\ell})$.
We omit the tedious repetition of the calculation.
\end{proof}
\begin{proof}[Proof of Lemma \ref{l_g-gL-prop-2}]
 The idea is to sum up the identities in Lemma \ref{l_g-gL-prop-1}-(1)
for all $\ell'\in\llbracket\ell,\,L-1\rrbracket$ and Lemma \ref{l_g-gL-prop-1}-(2),
also for all $v\in\mathcal{A}_{x,\,j}\cup\mathcal{A}_{y,\,j}$ and
$j\in\llbracket1,\,s\rrbracket$. Notice that there are four terms
in \eqref{e_l_g-gL-prop-1-1} and three terms in \eqref{e_l_g-gL-prop-1-2}
to be summed up. Computing the first terms in \eqref{e_l_g-gL-prop-1-1}
and \eqref{e_l_g-gL-prop-1-2}, we obtain
\begin{equation}
\sum_{j=1}^{s}\sum_{v\in\mathcal{A}_{x,\,j}\cup\mathcal{A}_{y,\,j}}\sum_{a\in\mathcal{N}_{x}}\frac{c_{av}}{1-m_{a}}\Big(\sum_{\ell'=\ell}^{L-1}(1-m_{v})\cdot m_{v}^{\ell'-1}+m_{v}^{L-1}\Big)=\sum_{a\in\mathcal{N}_{x}}\sum_{v\in\mathcal{A}_{x}}\frac{c_{av}m_{v}^{\ell-1}}{1-m_{a}}.\label{e_l_g-gL-prop-2-1}
\end{equation}
Computing the second and third terms in \eqref{e_l_g-gL-prop-1-1}
and the second term in \eqref{e_l_g-gL-prop-1-2}, a beautiful interpolation
takes place and we are left with
\begin{equation}
\Big[\sum_{a\in\mathcal{N}_{x}}\sum_{v\in\mathcal{A}_{x}}+\sum_{a\in\mathcal{N}_{y}}\sum_{v\in\mathcal{A}_{y}}\Big]\frac{c_{av}m_{v}^{\ell-1}}{1-m_{a}}\cdot(\mathfrak{g}^{L}(\mathfrak{v}_{\ell-1})-\mathfrak{g}^{L}(\mathfrak{v}_{\ell})).\label{e_l_g-gL-prop-2-2}
\end{equation}
Computing the fourth term in \eqref{e_l_g-gL-prop-1-1} and the third
term in \eqref{e_l_g-gL-prop-1-2}, we obtain
\begin{equation}
\sum_{j=1}^{s}\sum_{v\in\mathcal{A}_{x,\,j}\cup\mathcal{A}_{y,\,j}}\sum_{w\in\mathcal{A}_{x,\,j}\cup\mathcal{A}_{y,\,j}}\sum_{\ell'=\ell}^{L}\widehat{\mathfrak{r}}^{\ell'}(\sigma^{v},\,\sigma^{w})(\mathfrak{g}^{L}(\mathfrak{w}_{\ell'})-\mathfrak{g}^{L}(\mathfrak{v}_{\ell'}))=0,\label{e_l_g-gL-prop-2-3}
\end{equation}
since each term is counted twice with opposite signs with respect
to the summations in $v$ and $w$. Finally, Lemma \ref{l_g-gL-prop-1}
indicates that the summation of the terms in \eqref{e_l_g-gL-prop-2-1},
\eqref{e_l_g-gL-prop-2-2}, and \eqref{e_l_g-gL-prop-2-3} equals
$0$. This is exactly what we wanted.
\end{proof}
\begin{proof}[Proof of Lemma \ref{l_g-gL-prop-3}]
 First, we prove the first statement. The existence of the series
in the right-hand side is guaranteed by \eqref{e_g-gL-UB}, since
$|m_{b}^{\ell}(1+\mathfrak{g}(\mathfrak{b}_{\ell})-\mathfrak{g}(\mathfrak{b}_{\ell+1})|\le C(\frac{m_{b}}{\lambda})^{\ell}$
where $m_{b}<\lambda$. Thus, it suffices to prove that
\[
\lim_{L\to\infty}\Big|\sum_{\ell=0}^{L-1}m_{b}^{\ell}(1+\mathfrak{g}^{L}(\mathfrak{b}_{\ell})-\mathfrak{g}^{L}(\mathfrak{b}_{\ell+1}))-\sum_{\ell=0}^{L-1}m_{b}^{\ell}(1+\mathfrak{g}(\mathfrak{b}_{\ell})-\mathfrak{g}(\mathfrak{b}_{\ell+1}))\Big|=0.
\]
This limit holds since $\mathfrak{g}^{L}$ converges to $\mathfrak{g}$
pointwisely as $L\to\infty$, and $\{m_{b}^{\ell}\cdot\mathfrak{g}^{L}(\mathfrak{v}_{\ell})\}_{\ell\ge L}$
decays exponentially and uniformly as $L\to\infty$. The second statement
of the lemma can be proved by the same logic.
\end{proof}
\begin{proof}[Proof of Lemma \ref{l_g-gL-prop-4}]
 The second equality is immediate by summing up the identities in
Lemma \ref{l_g-gL-prop-2} for $\ell\in\llbracket1,\,L\rrbracket$,
sending $L$ to infinity, and then applying Lemma \ref{l_g-gL-prop-3}.
Thus, we focus on proving the first equality.

We may rewrite the last (triple) summation in the right-hand side
of \eqref{e_Kxy-def} as
\[
\sum_{j=1}^{s}\sum_{\{a,\,b\}\subseteq\mathcal{A}_{x,\,j}\cup\mathcal{A}_{y,\,j}}\sum_{\ell=1}^{\infty}\widehat{\mathfrak{c}}_{\sigma^{a}\sigma^{b}}^{\ell}\{\mathfrak{g}(\mathfrak{b}_{\ell})-\mathfrak{g}(\mathfrak{a}_{\ell})\}^{2}=\sum_{j=1}^{s}\sum_{v,\,w\in\mathcal{A}_{x,\,j}\cup\mathcal{A}_{y,\,j}}\sum_{\ell=1}^{\infty}\widehat{\mathfrak{c}}_{\sigma^{v}\sigma^{w}}^{\ell}\mathfrak{g}(\mathfrak{v}_{\ell})(\mathfrak{g}(\mathfrak{v}_{\ell})-\mathfrak{g}(\mathfrak{w}_{\ell})).
\]
By Lemma \ref{l_g-gL-prop-1}-(1) (with $L\to\infty$ and applying
the dominated convergence theorem), the right-hand side equals
\[
\begin{aligned}\sum_{\ell=1}^{\infty}\sum_{j=1}^{s}\sum_{v\in\mathcal{A}_{x,\,j}\cup\mathcal{A}_{y,\,j}}\mathfrak{g}(\mathfrak{v}_{\ell})\Big[ & (1-m_{v})\cdot\sum_{a\in\mathcal{N}_{x}}\frac{c_{av}m_{v}^{\ell-1}}{1-m_{a}}+\Big(\sum_{a\in\mathcal{N}_{x}\cup\mathcal{N}_{y}}\frac{c_{av}m_{v}^{\ell-1}}{1-m_{a}}\Big)\cdot(\mathfrak{g}(\mathfrak{v}_{\ell-1})-\mathfrak{g}(\mathfrak{v}_{\ell}))\\
 & +\Big(\sum_{a\in\mathcal{N}_{x}\cup\mathcal{N}_{y}}\frac{c_{av}m_{v}^{\ell}}{1-m_{a}}\Big)\cdot(\mathfrak{g}(\mathfrak{v}_{\ell+1})-\mathfrak{g}(\mathfrak{v}_{\ell}))\Big].
\end{aligned}
\]
Rearranging, this becomes
\[
\begin{aligned} & \sum_{a\in\mathcal{N}_{x}}\sum_{b\in\mathcal{A}_{x}}\frac{c_{ab}}{1-m_{a}}\Big[\sum_{\ell=1}^{\infty}m_{b}^{\ell-1}\mathfrak{g}(\mathfrak{b}_{\ell})(1+\mathfrak{g}(\mathfrak{b}_{\ell-1})-\mathfrak{g}(\mathfrak{b}_{\ell}))-\sum_{\ell=1}^{\infty}m_{b}^{\ell}\mathfrak{g}(\mathfrak{b}_{\ell})(1+\mathfrak{g}(\mathfrak{b}_{\ell})-\mathfrak{g}(\mathfrak{b}_{\ell+1}))\Big]\\
 & +\sum_{a\in\mathcal{A}_{y}}\sum_{b\in\mathcal{N}_{y}}\frac{c_{ab}}{1-m_{b}}\Big[\sum_{\ell=1}^{\infty}m_{a}^{\ell-1}\mathfrak{g}(\mathfrak{a}_{\ell})(\mathfrak{g}(\mathfrak{a}_{\ell})-\mathfrak{g}(\mathfrak{a}_{\ell-1}))-\sum_{\ell=1}^{\infty}m_{a}^{\ell}\mathfrak{g}(\mathfrak{a}_{\ell})(\mathfrak{g}(\mathfrak{a}_{\ell+1})-\mathfrak{g}(\mathfrak{a}_{\ell}))\Big].
\end{aligned}
\]
Therefore, we may calculate $(6\mathfrak{K}_{xy})^{-1}$ by substituting
this expression to the definition:
\[
\frac{1}{6\mathfrak{K}_{xy}}=\sum_{a\in\mathcal{N}_{x}}\sum_{b\in\mathcal{N}_{y}}\frac{c_{ab}}{(1-m_{a})(1-m_{b})}+\sum_{a\in\mathcal{N}_{x}}\sum_{b\in\mathcal{A}_{x}}\frac{c_{ab}}{1-m_{a}}\sum_{\ell=0}^{\infty}m_{b}^{\ell}(1+\mathfrak{g}(\mathfrak{b}_{\ell})-\mathfrak{g}(\mathfrak{b}_{\ell+1})).
\]
This completes the proof.
\end{proof}
\begin{acknowledgement*}
SK was supported by the following grants.
\begin{itemize}
\item Samsung Science and Technology Foundation grant (No. SSTF-BA1901-03).
\item National Research Foundation of Korea (NRF) grant funded by the Korean
government (MSIT) (No. 2022R1F1A106366811, 2022R1A5A600084012, 2023R1A2C100517311,
and NRF-2019-Global Ph.D. Fellowship Program).
\end{itemize}
\end{acknowledgement*}

\end{document}